\documentclass[11pt,letterpaper,twoside,reqno,nosumlimits]{amsart}

\usepackage{newclude}

\usepackage[usenames,dvipsnames]{xcolor}
\usepackage{fancyhdr}
\usepackage{amsmath,amsfonts,amsbsy,amsgen,amscd,amssymb,amsthm}
\usepackage{url}

\usepackage{mathtools}

\usepackage[font=small,margin=0.25in,labelfont={bf},labelsep={colon}]{caption}

\usepackage{subcaption}

\usepackage{tikz}
\usepackage{microtype}
\usepackage{enumitem}

\definecolor{dark-gray}{gray}{0.3}
\definecolor{dkgray}{rgb}{.4,.4,.4}
\definecolor{dkblue}{rgb}{0,0,.5}
\definecolor{medblue}{rgb}{0,0,.75}
\definecolor{rust}{rgb}{0.5,0.1,0.1}

\usepackage[colorlinks=true]{hyperref}

\hypersetup{urlcolor=rust}
\hypersetup{citecolor=dkblue}
\hypersetup{linkcolor=dkblue}

\usepackage{setspace}

\usepackage{graphicx}
\usepackage{booktabs,longtable,tabu} 
\usepackage{multirow} 

\usepackage{float}

\usepackage[full]{textcomp}

\usepackage[scaled=.98,sups,osf]{XCharter}
\usepackage[scaled=1.04,varqu,varl]{inconsolata}
\usepackage[type1]{cabin}
\usepackage[charter,vvarbb,scaled=1.07]{newtxmath}
\usepackage[cal=boondoxo]{mathalfa}
\usepackage[T1]{fontenc}

\usepackage{bm} 

\graphicspath{{figures/}}

\newtheorem{theorem}{Theorem}[section]
\newtheorem{lemma}[theorem]{Lemma}

\newtheorem{proposition}[theorem]{Proposition}

\newtheorem{corollary}[theorem]{Corollary}

\theoremstyle{definition}

\newtheorem{definition}[theorem]{Definition}

\numberwithin{equation}{section} 
\numberwithin{figure}{section}
\numberwithin{table}{section}

\floatstyle{plaintop}
\newfloat{recipe}{thp}{lor}
\floatname{recipe}{Recipe}
\numberwithin{recipe}{section}

\renewcommand{\phi}{\varphi}

\newcommand{\eps}{\varepsilon}

\newcommand{\iunit}{\mathrm{i}}

\providecommand{\mathbbm}{\mathbb} 

\newcommand{\sgn}[1]{\operatorname{sgn}{#1}}

\newcommand{\diff}{\mathrm{d}}
\newcommand{\idiff}{\, \diff}

\newcommand{\triplenorm}[1]{{\left\vert\kern-0.25ex\left\vert\kern-0.25ex\left\vert #1
    \right\vert\kern-0.25ex\right\vert\kern-0.25ex\right\vert}}

\newcommand{\ind}{\mathbbm{1}} 

\def\XXint#1#2#3{{%
\setbox0=\hbox{$#1{#2#3}{\int}$}
\vcenter{\hbox{$#2#3$}}\kern-.5\wd0}}

\renewcommand{\leq}{\leqslant}
\renewcommand{\geq}{\geqslant}
\renewcommand{\subset}{\subseteq}

\usepackage{placeins,needspace}

\allowdisplaybreaks

\usepackage[numbers]{natbib}
\mathtoolsset{showonlyrefs}

\newcommand{\notationanchor}[1]{\phantomsection\label{#1}}

\begin{document}
\title{Nonuniqueness of the $2$D Euler Evolutions\\with $L^1\cap L^p$ Vorticity}
\author{Thomas Y. Hou}
\address{Applied and Computational Mathematics,
California Institute of Technology, Pasadena, CA, USA}
\email{hou@cms.caltech.edu}

\author{Xiang Qin}
\address{Applied and Computational Mathematics,
California Institute of Technology, Pasadena, CA, USA}
\email{xqin2@caltech.edu}

\author{Peicong Song}
\address{Applied and Computational Mathematics,
California Institute of Technology, Pasadena, CA, USA}
\email{psong2@caltech.edu}

\author{Wenhan Zhang}
\address{Applied and Computational Mathematics,
California Institute of Technology, Pasadena, CA, USA}
\email{wzhang9@caltech.edu}

\begin{abstract}
For each prescribed $1<p<\infty$, we construct a compactly supported
initial vorticity in $L^1(\mathbb{R}^2)\cap L^p(\mathbb{R}^2)$ that
admits two distinct evolutions under the unforced two-dimensional
incompressible Euler equation on $[0,1]$. Both vorticities are strongly
continuous in $L^1\cap L^p$ and are transported exactly by their
respective almost-everywhere invertible, measure-preserving regular
Lagrangian flows. The construction couples unstable four-vortex
configurations across increasingly small spatial scales, realizing them
through concentrated signed vortex patches. At each finite cutoff, a
prepared reference evolution remains close to rotating configurations,
while a second evolution differs initially only at the finest scale.
In the second evolution, the finest configuration consists of separating
dipoles whose remote strain activates the instability of the next larger
configuration. This process propagates through successive scales with
summable transition times. As the cutoff tends to infinity, the two
families of initial data converge strongly, along a subsequence, to the
same datum, while a nonzero difference survives at a fixed larger scale
and is detected by a fixed smooth observable at a positive limiting
time. Quantitative concentration and relative comparison estimates
establish this cascade for exact Euler patch solutions. Strong
compactness through inverse-flow stability then preserves both exact
vorticity transport and the positive-time distinction between the two
limiting evolutions. The initial datum and construction parameters may
depend on $p$.
\end{abstract}

\maketitle

\tableofcontents
\section{Introduction}\label{sec:introduction}

\subsection{Background}

We study the two-dimensional incompressible Euler system
\begin{equation}\label{eq:euler-velocity-system}
 \left\{\begin{aligned}
  \partial_tu+(u\cdot\nabla)u+\nabla P&=0,\qquad &(t,x)&\in(0,1)\times\mathbb R^2,\\
  \operatorname{div}u&=0,\qquad &u(0,x)&=u_0(x).
 \end{aligned}\right.
\end{equation}
Here $u$ is the velocity and $P$ is the pressure. It has the vorticity form
\begin{equation}\label{eq:euler-veorticity-system}
 \left\{\begin{aligned}
  \partial_t\omega+(u\cdot\nabla)\omega&=0,\quad &(t,x)&\in(0,1)\times\mathbb R^2,\\
  u&=K*\omega,\quad & \omega(0,x)&=\omega_0(x),\quad K(x) = \frac{x^{\perp}}{2\pi|x|^2}.
 \end{aligned}\right.
\end{equation}
We study whether exact transport of vorticity determines a unique
solution of the unforced two-dimensional incompressible Euler equation
beyond the bounded-vorticity class. For each prescribed $1<p<\infty$,
we construct two different evolutions from one compactly supported
initial vorticity in $L^1\cap L^p$. Both evolutions retain this
regularity, are strongly continuous at the initial time, and transport
the datum exactly by their own measure-preserving flows. The mechanism
combines an unstable four-vortex configuration with interactions across
infinitely many scales. A perturbation confined to the finest scale
vanishes in the initial-data limit, while its dynamically amplified
effect remains detectable at a fixed larger scale.

Recent manuscripts
\cite{ChoiNonuniqueness2026,Coiculescu2026I,Hexagon2026} also address
nonuniqueness through multiscale constructions. In particular, Choi's
independently developed work \cite{ChoiNonuniqueness2026} uses
essentially the same core mechanism. We discuss the relationship
between the two constructions in Section~\ref{subsec:choi-comparison}
and the development of our work in Section~\ref{subsec:research-history}.

For smooth vorticity, the planar Euler equation has a unique global
evolution. The classical results of Wolibner and Kato
\cite{Wolibner1933,Kato1967}, and the accounts
\cite{MajdaBertozzi2002,MarchioroPulvirenti1994}, explain the central
role of vorticity transport. Yudovich's theorem
\cite{Yudovich1963} extends global uniqueness to bounded integrable
vorticity: the associated velocity has a logarithmic modulus of
continuity, sufficient to control differences of trajectories and
solutions. The later theorem \cite{Yudovich1995} also allows unbounded
vorticity with suitable quantitative control of the growth of its
$L^p$ norms. Thus an unbounded vorticity need not lead to nonuniqueness,
and membership in every finite $L^p$ is different from a prescribed
growth condition as $p\to\infty$. Localized versions of this theory are
developed in \cite{CrippaStefani2024}.

Weak existence extends well beyond bounded vorticity. DiPerna and
Majda \cite{DiPernaMajda1987} developed compactness methods for
two-dimensional incompressible flow; Delort \cite{Delort1991} obtained
existence for distinguished-sign vortex-sheet data. Existence in these
rough classes leaves open the separate issue of uniqueness. There is
also a distinction between uniqueness for linear transport with a
prescribed velocity and uniqueness for the coupled Euler evolution.
The theories of DiPerna--Lions and Ambrosio
\cite{DiPernaLions1989,Ambrosio2004} give transport and flow uniqueness
under Sobolev or bounded-variation hypotheses, but two Euler solutions
generate different velocities. In the present construction each
limiting velocity has a unique regular Lagrangian flow, and the two
coupled evolutions are nevertheless different.

We use the normalization
\begin{equation}\label{eq:euler-normalization}
 x^\perp=(-x_2,x_1),\qquad K(x)=\frac{x^\perp}{|x|^2},\qquad
 u=K*\omega,\qquad \partial_t\omega+\operatorname{div}(u\omega)=0.
\end{equation}
Accordingly $\operatorname{curl}u=2\pi\omega$: the physical vorticity
is $2\pi\omega$. Multiplying all vorticities by this fixed constant
recovers the usual Biot--Savart convention and does not change any
Lebesgue-space assertion. A canonical velocity means precisely the
convolution in \eqref{eq:euler-normalization}, with no additional
harmonic field. This convention gives a single interaction time
$\tau=d^2/\Gamma$ for a configuration of length $d$ and normalized
circulation magnitude $\Gamma$.

\subsection{Statement of the result}

\begin{definition}\label{def:class}
Fix $1<p<\infty$. A Lagrangian Euler evolution on $[0,1]$ is a real
vorticity $\omega\in C([0,1];L^1(\mathbb R^2)\cap L^p(\mathbb R^2))$,
its canonical velocity $u=K*\omega$, and an almost-everywhere invertible,
measure-preserving regular Lagrangian flow $X$. For almost every label
$x_0$, the path $X(\cdot,x_0)$ is absolutely continuous, starts at $x_0$,
and solves $\partial_tX(t,x_0)=u(t,X(t,x_0))$ almost everywhere in time.
At every time the identity
$\omega(t,X(t,x_0))=\omega(0,x_0)$ holds for almost every label.
The velocity satisfies the unforced Euler momentum equation in
distributions for a suitable pressure.
\end{definition}

\begin{theorem}\label{thm:main}
For every prescribed real $1<p<\infty$, there is a real, compactly
supported $\omega_0\in L^1(\mathbb R^2)\cap L^p(\mathbb R^2)$ and two
Lagrangian Euler evolutions $(\omega^\kappa,u^\kappa,X^\kappa)$,
$\kappa\in\{\mathrm{quiet},\mathrm{trig}\}$, on $[0,1]$, such that
\begin{equation}\label{eq:main-regularity}
 \begin{gathered}
  \omega^\kappa\in C([0,1];L^1\cap L^p),\qquad
  \omega^\kappa(0)=\omega_0,\\
  u^\kappa=K*\omega^\kappa\in C([0,1];L^2_{\rm loc}).
 \end{gathered}
\end{equation}
For suitable pressures $P^\kappa$,
\begin{equation}\label{eq:main-momentum}
 \begin{gathered}
  \partial_tu^\kappa+\operatorname{div}(u^\kappa\otimes u^\kappa)
       +\nabla P^\kappa=0,\\
  \operatorname{div}u^\kappa=0,\qquad u^\kappa(0)=K*\omega_0.
 \end{gathered}
\end{equation}
in distributions on $(0,1)\times\mathbb R^2$.
\notationanchor{nt:forward-flow}The flow starts at the identity and
transports the datum exactly:
\begin{equation}\label{eq:main-transport}
 \begin{gathered}
  X^\kappa(0,x_0)=x_0,\qquad
  \frac{\diff}{\diff t}X^\kappa(t,x_0)
       =u^\kappa(t,X^\kappa(t,x_0)),\\
  \omega^\kappa(t,X^\kappa(t,x_0))=\omega_0(x_0).
 \end{gathered}
\end{equation}
The trajectory equation holds almost everywhere in time for almost every
label; the transport identity holds at every time for almost every label.
Renormalized transport identities hold, and all initially finite
vorticity norms and absolutely integrable Casimirs are conserved.
There are $t_*\in(0,1]$, a real
$\varphi_*\in C_c^\infty(\mathbb R^2)$, and $\gamma_*>0$ such that
\begin{equation}\label{eq:theorem-separation}
 \int_{\mathbb R^2}\varphi_*(x)
  \bigl(\omega^{\rm trig}(t_*,x)-\omega^{\rm quiet}(t_*,x)\bigr)
        \idiff x\geq\gamma_*.
\end{equation}
\notationanchor{nt:target-exponent}The datum and the construction
parameters may depend on $p$, and are fixed before the approximation
cutoff tends to infinity.
\end{theorem}

In fact we carry out the construction with $q=\max(p,3)$ and obtain
strong $L^1\cap L^q$ continuity; interpolation gives the stated
$L^p$ conclusion. The two solutions preserve the same signed value
distribution. Their difference is therefore a difference of spatial
evolutions of the same vorticity, detected by the fixed observable in
\eqref{eq:theorem-separation}. For each prescribed exponent,
Section~\ref{sec:scales} specifies
a scale hierarchy satisfying the required integrability,
concentration, and cascade estimates. The superscripts
\emph{quiet} and \emph{trig} distinguish the
reference and triggered histories constructed below. The reference is
selected to remain close to rotating configurations; the terminology
does not assert stability under arbitrary perturbations.

\subsection{Dynamical generation of the second evolution}
\label{subsec:dynamical-mechanism}

The proof combines a local instability with an infinite family of
interacting configurations. The instability amplifies a small geometric
difference. The scale hierarchy makes the initial difference vanish
in the vorticity norm while its amplified effect survives at positive
time. We sketch these steps with Figures~\ref{fig:initial-array}--\ref{fig:material-histories};
Subsection~\ref{subsec:proof-roadmap} then explains the analytical
estimates that implement them for Euler solutions.

\paragraph{\textbf{The initial array and its scales.}}
At each location $R_ne_1$, place four circular patches, two positive
and two negative, with normalized circulation magnitudes $\Gamma_n$.
Their mutual distances are comparable to $d_n$, and each has radius
$a_n\ll d_n$. As $n$ increases, the configurations accumulate toward
the origin and their interaction times decrease.

\begin{figure}[!htbp]
 \centering
 \includegraphics[width=.70\linewidth]{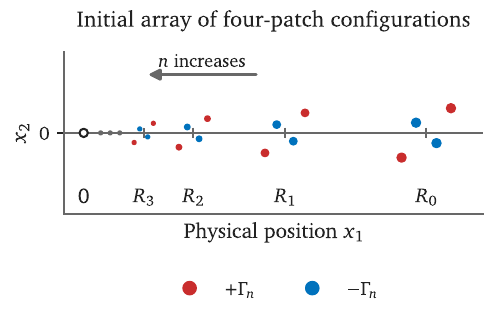}
 \caption{Initial array of four-patch configurations centered at
$R_ne_1$. Red and blue symbols represent positive and negative
patches with normalized circulations $+\Gamma_n$ and $-\Gamma_n$,
respectively; crosses mark configuration centers. Each initial
patch has radius $a_n$, much smaller than the internal
configuration scale $d_n$. Increasing indices approach the
origin, whereas the cascade proceeds outward, from $n+1$ to $n$.
Positions, separations, and patch radii are schematic.}
 \label{fig:initial-array}
\end{figure}

The scales of this array have distinct roles:
\[
 \begin{array}{c|l}
 R_n & \text{initial distance of the configuration center from the origin},\\
 d_n & \text{length scale within the four-patch configuration},\\
 a_n & \text{radius of each initial patch},\\
 \Gamma_n & \text{normalized circulation magnitude of each patch},\\
 \tau_n=d_n^2/\Gamma_n & \text{interaction time of the configuration}.
 \end{array}
\]
Sections~\ref{subsec:scale-choice} and~\ref{subsec:technical-inputs}
give the precise scales and initial disks. These choices make the
infinite array integrable in $L^1\cap L^q$ and are fixed before choosing
a finite cutoff $N$. The arrow toward increasing index in the figure
points toward smaller configurations; the instability will propagate
in the opposite direction.

\paragraph{\textbf{Two finite evolutions and the local instability.}}
The elementary point-vortex configuration consists of four vortices
arranged with central symmetry, two of each sign. At the reference
shape, the configuration rotates without changing shape. After removing
translation and rotation, one shape direction expands and another
contracts. Along the outgoing branch, the four vortices form two
separating neutral pairs, each containing one vortex of each sign.
Section~\ref{sec:cell} derives the
exact shape equations and the outgoing estimates. This reduction has
classical antecedents in four-vortex dynamics
\cite{BorisovPavlov1998,ArefStremler1999}.

For each cutoff $N$, we prepare a reference evolution
$\omega_N^{\rm quiet}$ whose configurations at indices $0,\ldots,N$
stay close to their rotating shapes on $[0,1]$. Superposing isolated
rotating quadrupoles does not account for the velocity induced by the other 
configurations. An exit-map argument then selects the initial expanding
coordinates so that the coupled evolution remains close to the rotating reference configurations. The triggered evolution $\omega_N^{\rm trig}$ starts with
exactly the same patches at every index $n<N$. Only the four centers
at index $N$ are changed, placing that configuration on a fixed
outgoing section. Its radii, circulations, and vorticity values remain
the same. Proposition~\ref{prop:finite} carries out this preparation.

Each finite bounded-vorticity datum has a unique Euler evolution.
The two finite evolutions start from different data, and the isolated
four-vortex instability alone does not give nonuniqueness. The common
initial datum will emerge only after passing to the infinite-scale limit.

\paragraph{\textbf{The signed interaction of adjacent configurations.}}
As the two neutral pairs at index $n+1$ separate, the anisotropic
part of their signed second-moment tensor grows. Its contribution
to the far field induces strain at the adjacent larger
configuration $n$, as illustrated in Figure~\ref{fig:donor-receiver}.
The projection of this strain onto the receiver's expanding
shape direction induces a small perturbation that the instability
amplifies to a fixed deformation. The receiver then enters its
own separating motion and can induce growth in the expanding
coordinate of the next larger configuration.

\begin{figure}[!htbp]
 \centering
 \includegraphics[width=.86\linewidth]{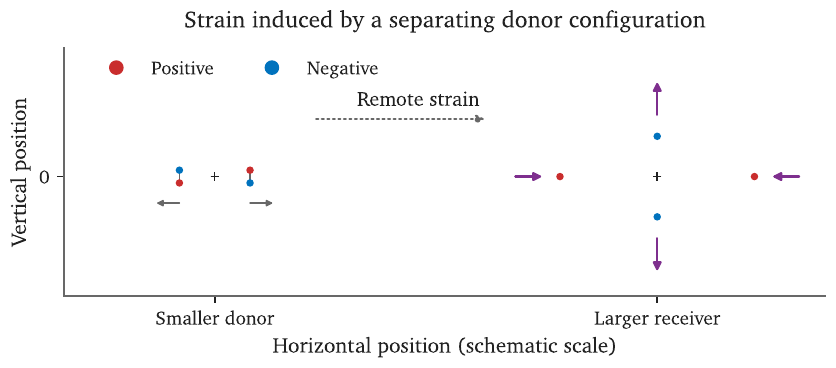}
 \caption{Schematic interaction between the smaller donor $n+1$
and the larger receiver $n$. Gray arrows indicate the motion of
the donor's two separating neutral pairs. The growing anisotropic
part of their signed second-moment tensor induces the strain
shown by the purple arrows near the receiver, after subtraction
of common translation. The dotted arrow denotes remote influence.
For the selected relative phases, the projection of the leading
strain onto the receiver's expanding shape direction has a positive
exponentially weighted integral, although its instantaneous
projection can change sign.}
 \label{fig:donor-receiver}
\end{figure}

The relative orientation matters. In coordinates rotating with the
receiver, the strain oscillates and its instantaneous expanding
component need not be positive. What determines growth is its
exponentially weighted accumulation: an earlier input has more time
to be amplified. The explicit calculation in
Section~\ref{subsec:signed-strain} identifies an interval of relative
phases with a positive accumulated input. Section~\ref{sec:release}
proves that the patch errors and the other interactions are smaller
than this input. The strain comes from the other patches in the same
Euler solution; the equation remains unforced.

\FloatBarrier

\paragraph{\textbf{An entire cascade in bounded time.}}
Repeating this transfer propagates the instability through
\[
 N\longrightarrow N-1\longrightarrow\cdots\longrightarrow1
 \longrightarrow0.
\]
Figure~\ref{fig:cascade-array} displays this order. The purple arrows
represent transmission of strain, not transfer of vorticity between
configurations. Each material patch remains in the full velocity field
throughout the evolution.

\begin{figure}[!htbp]
 \centering
 \includegraphics[width=.95\linewidth]{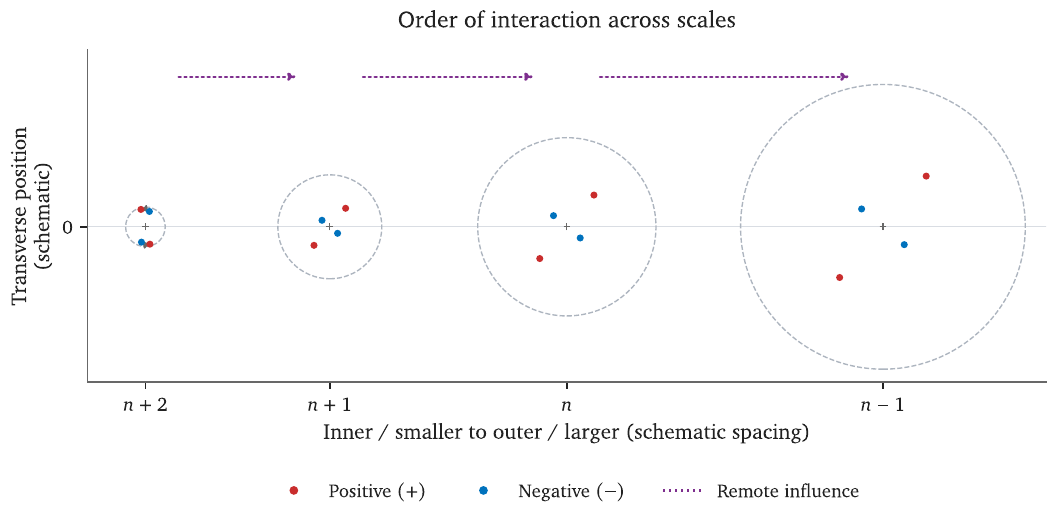}
 \caption{Schematic layout and interaction order
$n+2\to n+1\to n\to n-1$, from smaller, faster configurations
to larger, slower ones. Purple arrows identify the leading
strain interactions that drive the cascade; all configurations
remain coupled through the full Euler velocity. Dashed circles
serve only to group the four patches of each configuration.
The arrows indicate influence, not transport of vorticity
between configurations. The drawing illustrates the interaction
order rather than a simultaneous snapshot of the evolution.
Positions and sizes are schematic.}
 \label{fig:cascade-array}
\end{figure}

Let $\varepsilon_n$ denote the dimensionless input size at
configuration $n$. Amplification to a fixed deformation and
passage to the outgoing section require a physical duration
bounded by
\[
 \Delta t_n\lesssim
 \tau_n\bigl[1+\log(1/\varepsilon_n)\bigr].
\]
The interaction times $\tau_n$ decrease sufficiently rapidly
as $n$ increases that the scale choices in
Section~\ref{sec:scales} make the sum of the required physical
durations less than $1/4$, uniformly in the cutoff $N$.
These choices also control the initial vorticity norms and
the patch errors. Section~\ref{subsec:outgoing-laws} proves
the required passage-time and outgoing-angle estimates.

Favorable phases at successive receivers must also arise from the
initial data. Section~\ref{sec:finite} selects nested intervals of a
single rotation parameter of the finest configuration. The outgoing
time and angle laws ensure that one original parameter gives positive
transfer at every step. No phase is readjusted during the evolution.

\FloatBarrier

\Needspace{5\baselineskip}
\paragraph{\textbf{Why the mechanism persists for material patches.}}
The point-vortex picture must remain accurate while the patches deform
and all scales interact. Concentration estimates control the patches
through rotation, departure, and pair separation. At physically larger configurations, the two histories start from
identical disks, allowing
their difference to be estimated proportionally to the difference of
the exterior fields. Two separate absolute approximation errors would
not suffice to resolve the small transmitted signal. Sections~\ref{sec:geometry}
and~\ref{sec:release} control geometry and growth on the same stopped
interval, using the estimates stated in Section~\ref{sec:scales} and
proved in Section~\ref{sec:technical-proofs}. We explain these three
analytical ingredients in Subsection~\ref{subsec:proof-roadmap}.

\paragraph{\textbf{A common initial datum and a surviving later difference.}}
The infinite array has no finest configuration. At each finite
cutoff, the two initial data differ only at index $N$, whose
contribution to the $L^1\cap L^q$ norms tends to zero as
$N\to\infty$. Summable tails and convergence of the centers
at each fixed index yield a common strong initial limit
along a subsequence. Meanwhile, the cascade produces a fixed
shape difference at index zero. A permutation-invariant shape
statistic and patch concentration turn that difference into
a distinction between the signed vorticities themselves.
One smooth spatial test can be fixed along a subsequence,
and the observation times remain bounded away from zero.
The finite-data conclusions are
\begin{equation}\label{eq:strategy-target}
 \begin{gathered}
 \omega_{0,N}^{\rm quiet},\ \omega_{0,N}^{\rm trig}
       \longrightarrow\omega_0\quad\text{in }L^1\cap L^q,\\
 \int\varphi_*\bigl(\omega_N^{\rm trig}(t_N)
                     -\omega_N^{\rm quiet}(t_N)\bigr)\idiff x
       \geq\gamma>0,\qquad t_N\in[t_0,1/4],\quad t_0>0.
 \end{gathered}
\end{equation}
Here the scales, test, gap, and positive lower time bound are independent
of $N$. Proposition~\ref{prop:criterion} combines compactness of the
canonical velocities with stability of their inverse flows to obtain
strong convergence in $C([0,1];L^1\cap L^q)$. Exact transport and the
fixed observable gap both survive this limit.

\begin{figure}[!htbp]
 \centering
 \includegraphics[width=.80\linewidth]{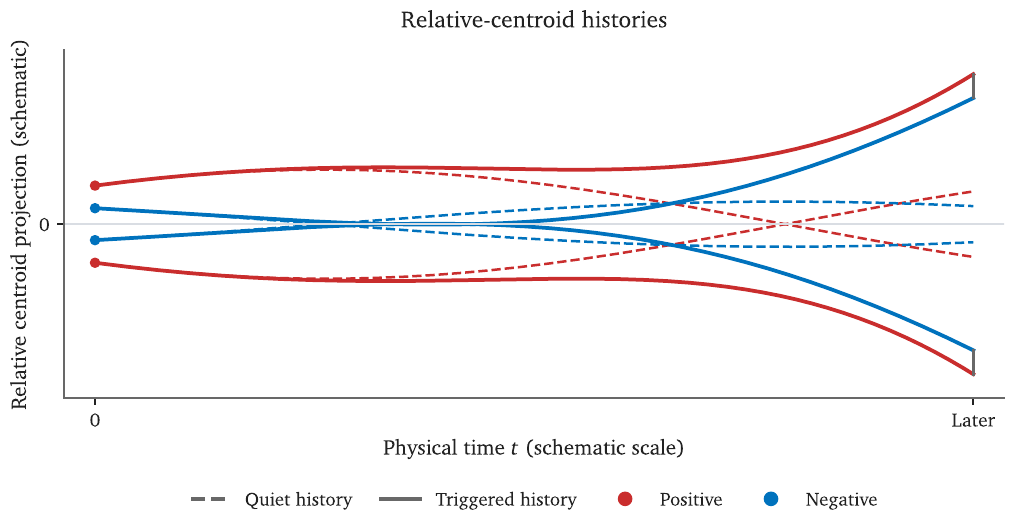}
 \caption{Schematic relative-centroid projections for a fixed
activated inner configuration in the two limiting Euler
evolutions from a common initial datum. Dashed curves represent
the prepared rotating reference; solid curves illustrate
triggered departure toward neutral-pair separation. Red and
blue indicate positive and negative components. Apparent
overlap at early times represents schematic closeness and
does not assert agreement on a positive-time interval.
The depicted separation is not the index-zero observation
used to establish nonuniqueness. The curves are illustrative,
not computed trajectories, and projected intersections do
not indicate collisions.}
 \label{fig:material-histories}
\end{figure}

Figure~\ref{fig:material-histories} illustrates the outgoing
motion of an activated inner configuration. For a fixed index
$n\geq1$, the plotted quantities are the component centroids
relative to their mean, projected onto a fixed unit direction
$\mathbf e$:
\[
 \mathbf e\cdot
 \bigl(b_{n,j}^{\kappa}(t)-\boldsymbol c_n^\kappa(t)\bigr),
 \qquad
 \boldsymbol c_n^\kappa(t)
 =\frac14\sum_{j=1}^4 b_{n,j}^\kappa(t).
\]
The same direction $\mathbf e$ is used for both histories.
The proof of nonuniqueness detects a fixed shape difference
at configuration $0$ while it remains in the small shape
chart; full neutral-pair separation at that configuration
is not required. Section~\ref{subsec:common-datum-observable}
constructs the common datum and the fixed observable.

\FloatBarrier

\subsection{Relation to Choi's independently developed construction}
\label{subsec:choi-comparison}
During the preparation of this manuscript, we became aware of Choi's
preprint \cite{ChoiNonuniqueness2026}, which appeared on arXiv on
October 27, 2026, shortly before our planned submission. Our work was
developed independently. The two works share the following core
mechanism: an unstable quadrupole, its transition into separating
dipoles, and an infinite cascade through which a vanishing finest-scale
perturbation produces a persistent difference at a larger scale.
Both constructions implement this mechanism with exact Euler
evolutions of concentrated vortex patches.

There are differences in scale choices, phase selection, and the
organization of the patch estimates and limiting argument. These are
differences in quantitative implementation and exposition rather than
fundamentally different sources of nonuniqueness. Choi's theorem gives
one compactly supported initial vorticity belonging simultaneously to
every finite $L^p$, with global-in-time evolutions and finite-energy
velocities. Theorem~\ref{thm:main}, as stated and proved here, treats
each prescribed finite exponent on $[0,1]$, with the datum and
parameters allowed to depend on that exponent. This comparison
concerns the stated results of the two manuscripts. Our scale choices
are flexible: letting patch areas shrink faster than every fixed
inverse power of their growing amplitudes suggests a route to the
same common-datum conclusion. Establishing this extension requires
new uniform concentration, interaction and comparison estimates
for the modified hierarchy.

Our presentation emphasizes explicit algebraic shape coordinates,
the signed interaction coefficient, and relative estimates for
patches transported from identical initial disks. The reference
preparation and nested-interval choice of a single trigger angle
appear in Section~\ref{sec:finite}; Section~\ref{sec:completion}
combines a fixed signed observable with inverse-flow stability.
In a companion preprint \cite{HouQinSongZhangDichotomy2026}, which
we post concurrently, we use a different method to establish
a sharp uniqueness--nonuniqueness criterion for arbitrary nondecreasing
vorticity $L^p$-norm envelopes within a class of stationary neutral
vortex arrays.
We describe our independent development in
Section~\ref{subsec:research-history}.

\subsection{Literature review}

\paragraph{\textbf{Instability and forced nonuniqueness.}}
Vishik's work \cite{Vishik2018I,Vishik2018II} established nonuniqueness
for forced planar Euler in finite-exponent vorticity classes. The
mechanism combines an unstable vortex with a similarity transformation
that brings an instability at logarithmic time $-\infty$ to the initial
time. The forcing makes the prescribed background compatible with the
equation. The monograph \cite{AlbrittonEtAl2024} develops this argument
in detail, including the construction of an unstable vortex. Castro,
Faraco, Mengual and Solera \cite{CastroFaracoMengualSolera2024} give a
simpler construction by first finding a piecewise constant unstable
vortex and then regularizing it. Albritton, Colombo and Mescolini
\cite{AlbrittonColomboMescolini2025} study nonunique viscosity limits
near this forced scenario, with a quantitative relation between the
viscosity and the size of initial perturbations.

A related instability method produced nonunique Leray solutions of the
\emph{forced} three-dimensional Navier--Stokes equation in the work of
Albritton, Bru\`e and Colombo \cite{AlbrittonBrueColombo2022}. This
connection is useful conceptually: a finite-dimensional unstable
direction can create nonuniqueness only after the surrounding flow,
initial trace, and nonlinear errors have been controlled. For the
unforced Navier--Stokes equation, Jia and \v Sver\'ak
\cite{JiaSverak2014,JiaSverak2015} developed the forward self-similar
framework and a route from instability to nonuniqueness; Guillod and
\v Sver\'ak \cite{GuillodSverak2023} investigated the bifurcation
numerically. Hou, Wang and Yang \cite{HouWangYang2026} present a
computer-assisted proof of unforced Leray--Hopf nonuniqueness based on
an unstable self-similar profile. Ionescu, Jia and Palasek
\cite{IonescuJiaPalasek2026} construct numerical unstable profiles in
the axisymmetric swirl-free class. The last work supplies approximate
profiles and their stability calculations, rather than a completed
nonlinear nonuniqueness theorem.

\paragraph{\textbf{Euler similarity profiles.}}
Constructing a nonradial self-similar Euler flow from a homogeneous
datum is a distinct problem from producing two flows from that datum.
Elling's algebraic spiral and mixed-sign constructions
\cite{Elling2013,Elling2016}, and the algebraic spiral solutions of
Shao, Wei and Zhang \cite{ShaoWeiZhang2025}, provide important examples
of similarity profiles. More recent existence results include Choi's
large-data construction \cite{ChoiLargeData2026}, the multi-sink
configurations of Choi and Coiculescu \cite{ChoiCoiculescu2026}, and
the Lagrangian construction of Hou, Qin, Song and Zhang
\cite{HouQinSongZhang2026}. The latter retains exact backward particle
transport and permits angular data without smallness or sign conditions
in its stated homogeneous class. These results clarify the available
backgrounds for a similarity approach. The present proof instead uses
spatially separated configurations with decreasing interaction times.

\paragraph{\textbf{Convex integration and vorticity regularity.}}
The weak nonuniqueness constructions of Scheffer and Shnirelman
\cite{Scheffer1993,Shnirelman1997}, followed by the differential-inclusion
formulation of De Lellis and Sz\'ekelyhidi \cite{DeLellisSzekelyhidi2009},
showed how flexible the Euler equation can be at low velocity
regularity. Controlling positive-time vorticity is an additional
restriction. Bru\`e and Colombo \cite{BrueColombo2023} obtained
nonuniqueness with vorticity in a Lorentz space; Bru\`e, Colombo and
Kumar \cite{BrueColomboKumar2024} crossed into $L^p$ for some $p>1$
using a new convex integration scheme with dipole building blocks.
Mengual \cite{Mengual2024} constructed nonunique admissible solutions
from truncated power-law vortices whose initial vorticity belongs to
$L^1\cap L^p$, including any prescribed finite $p>1$ in the extensions
treated there. This last initial-data assertion does not provide an
$L^\infty_tL^p_x$ bound for the vorticity of those solutions at positive
times. In Theorem~\ref{thm:main}, that regularity and exact transport
are retained on the entire closed time interval.

\paragraph{\textbf{Other October 2026 manuscripts.}}
The comparison with Choi is given in Subsection~\ref{subsec:choi-comparison}.
Coiculescu \cite{Coiculescu2026I} studies a radial nonnegative datum
and a threefold-symmetric perturbation. Part I establishes a compactness
criterion and develops the linear analysis of the multiscale building
blocks. Its instability criteria require interval-arithmetic certificates
discussed in the companion Part II, whose working draft is linked from
the repository record. The proposed second evolution departs from a stationary radial
reference, whereas here the reference is assembled from rotating
four-patch configurations. The Hexagon manuscript \cite{Hexagon2026}
claims, for every finite $p>2$, two strong vanishing-viscosity
accumulation paths from one compactly supported zero-circulation datum.
Its construction uses concentric unstable profiles and controls their
interacting forcing histories. Our conclusion is stated for canonical
Euler evolutions and does not assert that viscosity selects either one.
These comparisons concern the stated mechanisms and conclusions of
the cited versions.

\paragraph{\textbf{Cascades and concentrated vortices.}}
The use of successive scales has precedents in dyadic Euler models
\cite{KatzPavlovic2005} and in the averaged Navier--Stokes construction
of Tao \cite{Tao2016}. The September 2026 manuscripts of OpenAI
\cite{OpenAIEuler2026,OpenAINavierStokes2026} use dynamical
amplification in three-dimensional singularity constructions: the Euler
manuscript transfers amplification between localized oscillations, while
the Navier--Stokes manuscript combines a collapsing background with
oscillatory corrections and a smooth external force. The common
structural issue is to make an amplified component supply controlled
input to another component, with errors and initial increments that
remain summable. Here the transfer proceeds from the smaller, faster
configuration to its larger neighbor and ends in two different
Lagrangian limits.

Our passage from point dynamics to transported vorticity belongs to
the theory of concentrated vortices, including Turkington
\cite{Turkington1987}, Marchioro and Pulvirenti
\cite{MarchioroPulvirenti1993}, and the long-time concentration work
of Butt\`a and Marchioro \cite{ButtaMarchioro2018}. Classical patch
regularity results include Chemin and Bertozzi--Constantin
\cite{Chemin1993,BertozziConstantin1993}. The estimates needed here
are quantitative in the core radius and in differences between two
histories. Rearrangement \cite{BrascampLiebLuttinger1974} controls a
logarithmic deficit; harmonic estimates and singular-integral bounds
\cite{Stein1970} control exterior fields; the quantitative flow theory
of Crippa and De Lellis \cite{CrippaDeLellis2008} is used in the final
strong compactness argument. These inputs have separate roles, which
are made explicit in the technical lemmas and their applications.

\subsection{Analytical ingredients and organization of the proof}
\label{subsec:proof-roadmap}

The preceding dynamical picture has to hold for interacting material
patches, not merely for prescribed point-vortex trajectories. Three
analytical issues explain the structure of the proof.

\paragraph{\textbf{Concentration throughout the different stages.}}
Section~\ref{sec:scales} fixes the scales and states three patch
estimates, whose proofs are collected in
Section~\ref{sec:technical-proofs}. Lemma~\ref{lem:concentration}
uses a corrected energy to cancel rigid rotation and a high moment to
control the transported support. Lemma~\ref{lem:neutral-pair} uses
the cancellation of the common mutual translation in a neutral pair.
This permits control to continue after a quadrupole has entered its
separating stage. All estimates retain the full elapsed time and the
same original material patches when the description of a configuration
changes.

\paragraph{\textbf{Relative errors rather than two unrelated absolute errors.}}
An absolute approximation of each evolution by point vortices is not
sufficient to determine the small input that activates the receiver. For a spatially larger configuration at index $k<n$,
its individual approximation error need not be smaller than the
input $\varepsilon_n$ that activates index $n>k$. The two histories,
however, start from identical disks at these indices.
Lemma~\ref{lem:response} bounds the difference of their deformations
by the accumulated difference of centered exterior gradients. This
proportional estimate preserves the smallness of the relative signal.
Section~\ref{sec:release} combines it with signed moment estimates,
control of the impulse generated since time zero, and exponential
weighting to establish positive growth.

\paragraph{\textbf{Geometry and growth on a common interval.}}
Concentration requires control of the accumulated shape deformation,
while the growth estimate uses the current concentration and separation
bounds. Sections~\ref{sec:geometry} and~\ref{sec:release} therefore
work on the same stopped interval. Definition~\ref{def:common-bootstrap}
and Proposition~\ref{prop:bootstrap-closure} organize their simultaneous
strict improvement, without assuming a future crossing. A quantitative
transverse-crossing argument then continues the configuration to the
outgoing section, where the neutral-pair estimate applies.

Section~\ref{sec:completion} completes the construction. It first
establishes the compactness criterion and continuous dependence of each
fixed finite patch family, independently of the cascade selection.
These facts justify the exit map and the nested-interval choices.
The common strong initial limit and the fixed observable then yield
Theorem~\ref{thm:main}. This order separates the analytic inputs to
the finite selection from the selection itself.

\subsection{Development of the construction and use of AI}
\label{subsec:research-history}

Our investigation of nonuniqueness for the unforced two-dimensional
Euler equation began in January 2026 with numerical studies of vortex
sheets and related viscous models. During the summer, these
investigations revealed symmetry-breaking phenomena in nearly
self-similar regimes within families of $n$-fold symmetric
configurations. In particular, observations involving four-vortex
configurations led us to study their unstable shape dynamics as a
possible ingredient in a nonuniqueness construction. These observations
provided motivation and geometric intuition, but at that stage we did
not know how to turn them into a rigorous proof of two Euler evolutions
from the same initial datum.

Following OpenAI's September 8, 2026 announcement of its Euler and
Navier--Stokes analyses \cite{OpenAIEuler2026,OpenAINavierStokes2026},
we revisited the possibility of combining the four-vortex instability
with an infinite multiscale construction. In particular, the Euler
analysis \cite{OpenAIEuler2026} suggested organizing successive
amplification across scales while controlling the accumulated
interactions and the size of the initial contributions. With assistance from GPT-6 in developing
the argument, we obtained the result underlying the present manuscript in mid-September
2026, independently of Choi's work. Since then, we have developed the written presentation of the PDE
analysis and the geometric intuition. The authors
take responsibility for the mathematical statements and proofs in
this manuscript.

\subsection*{Conventions}
All unspecified spatial norms are on $\mathbb R^2$. We use
$q=\max(p,3)$ throughout. The branch index is
$\kappa\in\{\mathrm{quiet},\mathrm{trig}\}$; the nonnegative unit-mass
component densities are $\mu_{n,j}^\kappa$, with signs
$(\epsilon_1,\epsilon_2,\epsilon_3,\epsilon_4)=(1,1,-1,-1)$.
Spatial derivatives are denoted by $\nabla$; $D^+$ denotes only the
upper right Dini derivative. Primes refer to the time specified in the
surrounding statement. The imaginary unit is $\iunit$, and
$\operatorname{Rot}(\theta)$ denotes counterclockwise rotation.
Generic constants depend on the fixed exponent and the explicitly
declared geometric and coefficient bounds, and are independent of the
cutoff depth. The notation $o_L(1)$ means convergence to zero uniformly
in every admitted index and stopped history as $L\to\infty$.
\section{The isolated four-vortex model}\label{sec:cell}\label{app:point}

We construct shape coordinates in which the isolated equations are
$\mathsf U'=a\mathsf U$, $\mathsf V'=-a\mathsf V$, and $\mathsf A'=0$,
with $a(0,0,1)=\lambda=2\sqrt2$. We then compute the strain produced
along the invariant set $\mathsf V=0$, $\mathsf A=1$, $\mathsf U>0$.
At a distant rotating configuration its contribution to the
$\mathsf U$-equation has an exponentially weighted integral
$\varepsilon\mathcal H(\psi)+O(\varepsilon r)$, where
$\mathcal H(\psi)=2(\sqrt2\cos2\psi+\sin2\psi)$.
This formula supplies both the size and the sign required for the
comparison of the two Euler evolutions.

\subsection{The intrinsic shape coordinates}

We use throughout the kernel
\[
 K(x)=\frac{x^\perp}{|x|^2},\qquad
 x^\perp=(-x_2,x_1),\qquad K(z)=\frac{\iunit}{\bar z}.
\]
Consider two vortices of strength $+1$ at $z,-z$ and two of
strength $-1$ at $w,-w$. Central symmetry is preserved by the point
equations. At $z$, the remaining positive vortex contributes $K(2z)$,
whereas the negative vortices contribute $-K(z-w)-K(z+w)$.
Using
$(\bar z-\bar w)^{-1}+(\bar z+\bar w)^{-1}
=2\bar z/(\bar z^2-\bar w^2)$ gives
\begin{equation}\label{eq:complex-cell}
 z'=\frac{\iunit}{2\bar z}-\frac{2\iunit\bar z}{\bar z^2-\bar w^2},
 \qquad
 w'=-\frac{\iunit}{2\bar w}+\frac{2\iunit\bar w}{\bar w^2-\bar z^2}.
\end{equation}
For circulation magnitude $\Gamma$ and length $d$, the interaction time
and normalized time are
\begin{equation}\label{eq:single-clock}
 \tau=\frac{d^2}{\Gamma},\qquad s=\frac{t}{\tau}.
\end{equation}
Indeed, if $x_i=c+dy_i$ in an isolated configuration with fixed center $c$,
then homogeneity gives
\[
 \frac{\diff y_i}{\diff s}
 =\frac{\tau\Gamma}{d^2}\sum_{j\ne i}\epsilon_jK(y_i-y_j)
 =\sum_{j\ne i}\epsilon_jK(y_i-y_j),
 \qquad (\epsilon_1,\epsilon_2,\epsilon_3,\epsilon_4)=(1,1,-1,-1).
\]
Thus point dynamics and patch transport use the same clock.

For the later patch problem the four centroids need not be centrally
symmetric. We therefore separate their common translation, the offset
of the two same-sign midpoints, and their two same-sign half-separations.
With labels $1,2$ for the positive components and $3,4$ for the negative components, define
\begin{equation}\label{eq:cell-geometric-coordinates}
 \begin{gathered}
 \boldsymbol c=\tfrac14(b_1+b_2+b_3+b_4),\\
 \chi=\tfrac12(b_1+b_2-b_3-b_4),\\
 z=(b_1-b_2)/(2d),\qquad
 w=(b_3-b_4)/(2d),\\
 b_{1,2}=\boldsymbol c+\chi/2\pm dz,\qquad
 b_{3,4}=\boldsymbol c-\chi/2\pm dw.
 \end{gathered}
\end{equation}
The vectors $\boldsymbol c,\chi$ have physical length units; $z,w$ are
dimensionless. In particular the signed impulse of these four points
is $2\Gamma\chi$, so setting the mean centroid to zero does not set
the impulse to zero. Equation~\eqref{eq:complex-cell} concerns the
isolated subspace $\chi=0$ after translation and dilation.

We next distinguish deformation from common rotation: the shape
variables describe the relative radii and angle of the two same-sign
axes, while $\vartheta$ records their orientation. The offset $\chi$
remains a separate coordinate for departures from central symmetry.
Let $\alpha_z,\alpha_w$ be local continuous arguments of the labeled
vectors $z,w$. Their squared-radius difference, twice their radius
product, relative angle and mean angle are
\begin{equation}\label{eq:chart}
 \begin{gathered}
 \mathsf A=|z|^2-|w|^2,\qquad \mathsf R=2|z||w|,\qquad
 \phi=\arg z-\arg w,\qquad \vartheta=(\arg z+\arg w)/2,\\
 \mathsf X=\frac{\mathsf R-\mathsf A}{\sqrt{\mathsf A\mathsf R}},
 \qquad \mathsf Y=\sqrt{\mathsf R/\mathsf A}\sin(\phi-\pi/2),
 \qquad \mathsf U=\mathsf X-\mathsf Y,\quad
 \mathsf V=\mathsf X+\mathsf Y.
 \end{gathered}
\end{equation}
Here $\phi=\alpha_z-\alpha_w$ and
$\vartheta=(\alpha_z+\alpha_w)/2$. At the reference shape
$(\mathsf A,\mathsf R,\phi)=(1,1,\pi/2)$ the radii are
\begin{equation}\label{eq:cell-reference}
 |z|^2=(\sqrt2+1)/2,\qquad |w|^2=(\sqrt2-1)/2,
 \qquad \alpha_z=\vartheta+\pi/4,\quad
 \alpha_w=\vartheta-\pi/4.
\end{equation}
Thus the two same-sign axes are perpendicular, but their radii are
unequal. Choosing the positive axis horizontal means
$\vartheta=-\pi/4$ and places the labeled vector $w$ on the negative
vertical axis. A common rotation changes $\vartheta$ and leaves the
three shape variables unchanged.

The reduction of point-vortex dynamics to quantities invariant under
translations and rotations has a long history. In particular,
Borisov--Pavlov~\cite{BorisovPavlov1998} use internal distance and area
variables, and Aref--Stremler~\cite{ArefStremler1999} analyze the integrable
four-vortex problem with zero circulation and impulse. We use that
geometric reduction here, and derive the particular algebraic chart
needed for the expanding-mode calculation below.

The construction of $\mathsf X,\mathsf Y$ can be read backward from
\eqref{eq:cell-energy-factorization}: the conserved energy divided by
$\mathsf A$, minus its reference value, is a difference of two squares.
The first square measures relative radial expansion, and the second
measures departure from perpendicular axes. To first order at
$\mathsf R=\mathsf A$, $\phi=\pi/2$,
\[
 \mathsf X=\frac{\mathsf R-\mathsf A}{\mathsf A}
   +O((\mathsf R/\mathsf A-1)^2),\qquad
 \mathsf Y=\phi-\pi/2
   +O(|\mathsf R/\mathsf A-1||\phi-\pi/2|+|\phi-\pi/2|^3).
\]
Their combinations $\mathsf U,\mathsf V$ diagonalize the exact
saddle equations.
To verify that they really form a chart, put
$\mathsf X=(\mathsf U+\mathsf V)/2$ and
$\mathsf Y=(\mathsf V-\mathsf U)/2$. The full inverse is
\begin{equation}\label{eq:cell-chart-inverse}
 \begin{gathered}
 q_0=\frac{\mathsf X+\sqrt{\mathsf X^2+4}}2>0,\qquad
 \mathsf R=\mathsf A q_0^2,\qquad
 \phi=\frac\pi2+\arcsin(\mathsf Y/q_0),\\
 \mathsf S=\sqrt{\mathsf A^2+\mathsf R^2},\qquad
 z=\sqrt{(\mathsf S+\mathsf A)/2}\,
                  e^{\iunit(\vartheta+\phi/2)},\qquad
 w=\sqrt{(\mathsf S-\mathsf A)/2}\,
                  e^{\iunit(\vartheta-\phi/2)}.
 \end{gathered}
\end{equation}
We use the positive square roots and the branch
$\arcsin\in(-\pi/2,\pi/2)$. For example, all inverse formulas are smooth on the explicit chart
\begin{equation}\label{eq:explicit-shape-chart}
 |\mathsf U|\leq1/16,\quad |\mathsf V|\leq1/16,\quad
 |\mathsf A-1|\leq1/16,\quad |\chi|/d\leq1/16,
 \quad |\vartheta-\vartheta_0|<\pi/8.
\end{equation}
The arguments are lifted on this labeled chart. Put
$y_i=(b_i-\boldsymbol c)/d$ and $\xi=\chi/d$. On the centered
space $\sum_i y_i=0$, the coordinates
\[
 \mathcal X=(\mathsf U,\mathsf V,\vartheta,\mathsf A,\xi)
       =\Phi(y_1,y_2,y_3,y_4)
\]
have six real components. Directly from the inverse formulas,
$|y_i|\leq2$, $|y_i-y_j|\geq1/2$ for $i\ne j$ on
\eqref{eq:explicit-shape-chart}, and there is an absolute $C_{\rm ch}$ with
\begin{equation}\label{eq:chart-derivative-bounds}
 \|\nabla\Phi\|+\|\nabla^2\Phi\|
                  +\|\nabla\Phi^{-1}\|\leq C_{\rm ch}.
\end{equation}
Here the derivatives act on the finite-dimensional configuration
variables, restricted to their centered space; the bounds are
independent of $d$ and $\vartheta_0$.

\begin{proposition}[Exact saddle equations]\label{prop:intrinsic-saddle}
For the centrally symmetric isolated system, let the shape variables be
those of \eqref{eq:chart} and put
\[
 \mathsf E=\frac{\mathsf A^2+\mathsf R^2\sin^2\phi}{\mathsf R}.
\]
Then $\mathsf A'=\mathsf E'=0$ and
\begin{equation}\label{eq:exact-saddle}
 \begin{gathered}
 \mathsf U'=a\mathsf U,\qquad \mathsf V'=-a\mathsf V,
 \qquad \mathsf U\mathsf V=\mathsf E/\mathsf A-2,\\
 a=a(\mathsf U,\mathsf V,\mathsf A)
 =\frac{2(\mathsf R+\mathsf A)\sqrt{\mathsf A^2+\mathsf R^2}\sin\phi}
             {\mathsf R^2\mathsf E}.
 \end{gathered}
\end{equation}
The rate $a$ is smooth and positive on
\eqref{eq:explicit-shape-chart}, with
\[
 a(0,0,\mathsf A)=\frac{2\sqrt2}{\mathsf A},\qquad
 \vartheta'=-1\quad\hbox{at }(\mathsf U,\mathsf V,\mathsf A)=(0,0,1).
\]
\end{proposition}

\begin{proof}
Put $\mathsf S=\sqrt{\mathsf A^2+\mathsf R^2}$ and
$\Delta=\mathsf A^2+\mathsf R^2\sin^2\phi$.
The four opposite-sign interaction distances in the Hamiltonian combine as
$|z-w|^2|z+w|^2=|z^2-w^2|^2=\Delta$. Consequently the Hamiltonian is
\[
 -\log(2|z|)-\log(2|w|)+2\log|z-w|+2\log|z+w|
       =\log(\mathsf E/2),\qquad \mathsf E=\Delta/\mathsf R.
\]
We derive the transformed equations to fix their signs. Put
$u_z=|z|^2$ and $u_w=|w|^2$, so
$u_z+u_w=\mathsf S$, $u_z-u_w=\mathsf A$ and
$4u_zu_w=\mathsf R^2$. Dividing \eqref{eq:complex-cell} by $z,w$
and multiplying the complex denominators by their conjugates gives
\begin{equation}\label{eq:cell-logarithmic-equations}
 \begin{aligned}
 \frac{z'}z&=\frac{2u_w\sin2\phi}{\Delta}
   +\iunit\left(\frac1{2u_z}
                   -\frac{2(u_z-u_w\cos2\phi)}{\Delta}\right),\\
 \frac{w'}w&=\frac{2u_z\sin2\phi}{\Delta}
   +\iunit\left(-\frac1{2u_w}
                   +\frac{2(u_w-u_z\cos2\phi)}{\Delta}\right).
 \end{aligned}
\end{equation}
The real parts imply
$u_z'=u_w'=4u_zu_w\sin2\phi/\Delta$. Their difference is zero,
and their logarithmic sum gives
$\mathsf R'/\mathsf R=2\mathsf S\sin2\phi/\Delta$.
Subtracting and averaging the imaginary parts gives, respectively,
\[
 \phi'=\frac{2\mathsf S}{\mathsf R^2}
               -\frac{4\mathsf S\sin^2\phi}{\Delta},\qquad
 \vartheta'=-\frac{\mathsf A}{\mathsf R^2}
               -\frac{2\mathsf A\cos^2\phi}{\Delta}.
\]
Using $\Delta=\mathsf E\mathsf R$, we obtain
\begin{align}\label{eq:cell-radial-phase}
 \mathsf A'&=0,\qquad
 \mathsf R'=\frac{2\mathsf S\sin(2\phi)}{\mathsf E},\notag\\
 \phi'&=\frac{2\mathsf S(\mathsf A^2-\mathsf R^2\sin^2\phi)}
                   {\mathsf R^2\Delta},\qquad
 \vartheta'=-\frac{\mathsf A}{\mathsf R^2}
       -\frac{2\mathsf A\cos^2\phi}{\mathsf E\mathsf R}.
\end{align}
The energy is also constant directly from these equations:
\[
 \mathsf E'
 =\mathsf R'(\sin^2\phi-\mathsf A^2/\mathsf R^2)
       +\mathsf R\sin2\phi\,\phi'=0.
\]
Its expression explains the choice of the algebraic coordinates:
\begin{equation}\label{eq:cell-energy-factorization}
 \mathsf X^2-\mathsf Y^2
 =\frac{\mathsf R}{\mathsf A}+\frac{\mathsf A}{\mathsf R}-2
                    -\frac{\mathsf R}{\mathsf A}\cos^2\phi
 =\frac{\mathsf E}{\mathsf A}-2.
\end{equation}
To obtain their equations without division by either coordinate, use
$\mathsf Y=-\sqrt{\mathsf R/\mathsf A}\cos\phi$ and
$\mathsf A'=0$. Substitution of \eqref{eq:cell-radial-phase} yields
\[
 \begin{aligned}
 \mathsf X'
 &=\frac{\mathsf R+\mathsf A}{2\sqrt{\mathsf A}\mathsf R^{3/2}}
                                                  \mathsf R'
 =\frac{2(\mathsf R+\mathsf A)\mathsf S\sin\phi\cos\phi}
               {\sqrt{\mathsf A}\mathsf R^{3/2}\mathsf E}
 =-a\mathsf Y,\\
 \mathsf Y'
 &=\sqrt{\mathsf R/\mathsf A}
       \left(-\frac{\mathsf R'}{2\mathsf R}\cos\phi
                                      +\phi'\sin\phi\right)
 =\frac{2\sqrt{\mathsf R/\mathsf A}\mathsf S\sin\phi}
                {\mathsf R^3\mathsf E}
                         (\mathsf A^2-\mathsf R^2)
 =-a\mathsf X.
 \end{aligned}
\]
Thus the matrix acting on $(\mathsf X,\mathsf Y)$ is
$\left(\begin{smallmatrix}0&-a\\-a&0\end{smallmatrix}\right)$.
Its combinations $\mathsf X-\mathsf Y$ and
$\mathsf X+\mathsf Y$ have rates $a$ and $-a$:
These identities give \eqref{eq:exact-saddle}.
The denominator stays positive on a fixed small chart, so $a$ is smooth
there and $a(0,0,\mathsf A)=2\sqrt2/\mathsf A$. Thus $\mathsf U$
grows and $\mathsf V$ decays while $\mathsf A$ stays fixed. At
$(\mathsf U,\mathsf V,\mathsf A)=(0,0,1)$ their rates are
$\lambda=2\sqrt2$ and $-\lambda$, and $\vartheta'=-1$.
The derivation also holds on the invariant coordinate planes
$\mathsf U=0$ and $\mathsf V=0$.
\end{proof}

\subsection{From the invariant set to two separated neutral pairs}

\begin{proposition}[The outgoing branch]\label{prop:isolated-outgoing}
Fix $0<h_*<1/16$ on the branch
$\mathsf V=0$, $\mathsf A=1$, $\mathsf U>0$.
The orbit through $\mathsf U=h_*$ is global forward in time, crosses
each section $\mathsf R=\mathsf R_{\rm out}>\mathsf R_h$ transversely, where
\[
 \mathsf R_h=\left(\frac{h_*+\sqrt{h_*^2+16}}4\right)^2,
\]
and separates into two neutral pairs. For every sufficiently large fixed
$K_{\rm out}$, choose $\mathsf R_{\rm out}=8K_{\rm out}^2$ and set
$s=0$ at that crossing. With
$\boldsymbol d=z-w$ and $\boldsymbol m=(z+w)/2$, the following
integrals converge:
\begin{equation}\label{eq:cell-outgoing-limits}
 \begin{gathered}
 \boldsymbol d_\infty
     =\boldsymbol d(0)+\int_0^\infty\boldsymbol d'(v)\idiff v,
 \qquad e_{\rm out}=-K(\boldsymbol d_\infty),\\
 m_\infty=\boldsymbol m(0)
              +\int_0^\infty(\boldsymbol m'(v)-e_{\rm out})\idiff v.
 \end{gathered}
\end{equation}
The limiting direction satisfies $|e_{\rm out}|=1$, and for every $s\geq0$,
\begin{equation}\label{eq:outgoing-asymptotics}
 \boldsymbol d(s)=e_{\rm out}^\perp+O((K_{\rm out}+s)^{-2}),\qquad
 \boldsymbol m(s)=s e_{\rm out}+m_\infty+O((K_{\rm out}+s)^{-1}).
\end{equation}
Let $y^*=(z,-z,w,-w)$ be the reference centered state at the section.
There is $\zeta_0>0$, depending only on this fixed section, such that the set
\begin{equation}\label{eq:outgoing-neighborhood}
 \mathcal O_{\rm out}(\zeta_0)
 =\left\{(y_i)_{i=1}^4:\ \sum_i y_i=0,\quad
                       \max_i|y_i-y_i^*|<\zeta_0\right\}
\end{equation}
consists of states whose isolated forward evolutions have a nonzero
limiting velocity $v_{1,\infty}$ for the pair $(1,3)$. The map
$\mathcal E(y)=v_{1,\infty}/|v_{1,\infty}|$ satisfies
\[
 \mathcal E(y^*)=e_{\rm out},\qquad
 |\mathcal E(y)-\mathcal E(\widetilde y)|
       \leq C\max_i|y_i-\widetilde y_i|
 \quad(y,\widetilde y\in\mathcal O_{\rm out}(\zeta_0)).
\]
\end{proposition}

\begin{proof}
On $\mathsf V=0$, $\mathsf A=1$, $\mathsf U>0$, the Hamiltonian
identity gives $\mathsf E=2$. More precisely, $\mathsf Y=-\mathsf X$
and $\mathsf Y=-\sqrt{\mathsf R}\cos\phi$ imply
$\mathsf R\cos\phi=\mathsf R-1$ and
$\mathsf U=2\mathsf X>0$, hence $\mathsf R>1$.
We continue the branch with $0<\phi<\pi/2$, so
$\sin\phi=\sqrt{2\mathsf R-1}/\mathsf R$. Substitution gives
\begin{equation}\label{eq:positive-branch}
 \mathsf E=2,\quad \cos\phi=1-\mathsf R^{-1},\quad
 \mathsf U=\frac{2(\mathsf R-1)}{\sqrt{\mathsf R}},\quad
 \mathsf R'=\frac{2\sqrt{1+\mathsf R^2}(\mathsf R-1)
                       \sqrt{2\mathsf R-1}}{\mathsf R^2}>0.
\end{equation}
These formulas continue the branch outside the small saddle chart:
their denominators are nonzero for every $\mathsf R>1$.
For
$\mathsf R_h=((h_*+\sqrt{h_*^2+16})/4)^2$, the exact passage time is
\begin{equation}\label{eq:isolated-transverse-time}
 s_{\rm out}-s_h
 =\int_{\mathsf R_h}^{\mathsf R_{\rm out}}
 \frac{R^2\idiff R}{2\sqrt{1+R^2}(R-1)\sqrt{2R-1}}<\infty
 \qquad(\mathsf R_{\rm out}>\mathsf R_h).
\end{equation}
The positive derivative in \eqref{eq:positive-branch} makes the crossing
transverse. Moreover, on the entire branch $\mathsf R\geq1$,
\[
 |z-w|^2=\sqrt{1+\mathsf R^2}-\mathsf R+1\geq1,\quad
 |z+w|^2=\sqrt{1+\mathsf R^2}+\mathsf R-1\geq1,\quad
 4|w|^2\geq2(\sqrt2-1).
\]
The remaining same-sign distance is still larger. These inequalities
give a uniform positive lower bound for every interaction distance. Beyond the section it is useful to group
one positive and one negative vortex. Write
$\boldsymbol d=z-w$, $\boldsymbol m=(z+w)/2$.
The vector $\boldsymbol d$ points from the negative vortex to the
positive vortex in this pair. Since
$z\cdot w=\mathsf R\cos\phi/2=(\mathsf R-1)/2$ and
$\det(z,w)=-\mathsf R\sin\phi/2$, we obtain
\begin{equation}\label{eq:cell-pair-geometry}
 \begin{gathered}
 |\boldsymbol d|^2=\sqrt{1+\mathsf R^2}-\mathsf R+1,\qquad
 |\boldsymbol m|^2=\tfrac14(\sqrt{1+\mathsf R^2}+\mathsf R-1),\\
 \boldsymbol m\cdot\boldsymbol d=\tfrac12,
 \qquad \det(\boldsymbol m,\boldsymbol d)=\tfrac12\sqrt{2\mathsf R-1}.
 \end{gathered}
\end{equation}
The other neutral pair has midpoint $-\boldsymbol m$ and vector
$-\boldsymbol d$. The identities show that $|\boldsymbol d|\to1$
and $|\boldsymbol m|\to\infty$. Each pair also translates with a
nonzero limiting speed, as quantified below. Choose $K_{\rm out}\geq16$, to be increased below, and set
$\mathsf R_{\rm out}=8K_{\rm out}^2$. Formula
\eqref{eq:cell-pair-geometry} then gives
$3K_{\rm out}\leq2|\boldsymbol m(0)|\leq5K_{\rm out}$.
Set $s=0$ at this section. Dividing the last equation of \eqref{eq:positive-branch} by
$2\sqrt{\mathsf R}$ shows that
\[
 0<c\leq(\sqrt{\mathsf R})'
 =\frac{\sqrt{1+\mathsf R^2}(\mathsf R-1)\sqrt{2\mathsf R-1}}
                 {\mathsf R^{5/2}}\leq2
 \qquad(\mathsf R\geq\mathsf R_{\rm out}).
\]
Therefore $\sqrt{\mathsf R}\asymp K_{\rm out}+s$,
$|\boldsymbol m|\asymp K_{\rm out}+s$ and
$1\leq|\boldsymbol d|^2\leq1+C(K_{\rm out}+s)^{-2}$.
These estimates rule out both collisions and finite-time escape in
the isolated orbit.

The cancellation responsible for its asymptotics is visible in the
exact pair equations:
\begin{equation}\label{eq:cell-exact-pair-equations}
 \begin{aligned}
 \boldsymbol d'
   &=K(2\boldsymbol m+\boldsymbol d)-2K(2\boldsymbol m)
                                      +K(2\boldsymbol m-\boldsymbol d),\\
 \boldsymbol m'
   &=-K(\boldsymbol d)
       +\tfrac12\bigl(K(2\boldsymbol m+\boldsymbol d)
                                     -K(2\boldsymbol m-\boldsymbol d)\bigr).
 \end{aligned}
\end{equation}
Both vortices in the first pair have the same mutual velocity
$-K(\boldsymbol d)$, which cancels from $\boldsymbol d'$.
The companion pair contributes a second difference of $K$ to that
row and a first difference to the midpoint row.
Since $|\nabla^jK(x)|\leq C_j|x|^{-1-j}$, Taylor's formula on
the intervening segments gives
\begin{equation}\label{eq:cell-pair-tail-bounds}
 |\boldsymbol d'|\leq C(K_{\rm out}+s)^{-3},\qquad
 |\boldsymbol m'+K(\boldsymbol d)|\leq C(K_{\rm out}+s)^{-2}.
\end{equation}
More generally a neutral companion has velocity, gradient and Hessian
bounded by these first-difference estimates with powers $-2,-3,-4$.

The integrals defining the limits in \eqref{eq:cell-outgoing-limits} converge as follows.
The first tail is bounded by $C(K_{\rm out}+s)^{-2}$.
Equation~\eqref{eq:cell-pair-geometry} then gives
$|\boldsymbol d_\infty|=1$, hence $|e_{\rm out}|=1$ and
$\boldsymbol d_\infty=e_{\rm out}^\perp$.
On the fixed annulus containing the pair vectors, $K$ is Lipschitz.
Thus \eqref{eq:cell-pair-tail-bounds} gives
$|\boldsymbol m'-e_{\rm out}|\leq C(K_{\rm out}+s)^{-2}$,
which justifies the second integral. Integrating its tail proves \eqref{eq:outgoing-asymptotics}.
The bounds hold for every $s\geq0$. The section and $K_{\rm out}$ are
fixed before any scale ratios are chosen; constants involving
$m_\infty$, in particular, may depend on that fixed section.

We make the section neighborhood and the continuation argument precise.
Use the reference section state $y^*$ and its neighborhood
\eqref{eq:outgoing-neighborhood}.
For a state in this set define
$\boldsymbol d_1=y_1-y_3$, $\boldsymbol d_2=y_2-y_4$ and
$\boldsymbol m_1=(y_1+y_3)/2$, $\boldsymbol m_2=(y_2+y_4)/2$.
Let $e=e_{\rm out}$ be the reference limiting direction and
$q=\boldsymbol m_1-\boldsymbol m_2$. Start with the provisional inequalities
\begin{equation}\label{eq:isolated-pair-bootstrap}
 \tfrac12<|\boldsymbol d_j|<2,\qquad
 q\cdot e>K_{\rm out}+s,\qquad
 |q\cdot e^\perp|<\tfrac14(K_{\rm out}+s).
\end{equation}
For large $K_{\rm out}$ and sufficiently small $\zeta_0$, the initial
state satisfies these inequalities strictly. Before their first failure,
neutrality and Taylor expansion give
\[
 |\boldsymbol d_j'|\leq C(K_{\rm out}+s)^{-3},\qquad
 q'=-K(\boldsymbol d_1)+K(\boldsymbol d_2)
                                 +O((K_{\rm out}+s)^{-2}).
\]
The vector change in the first equation is at most $CK_{\rm out}^{-2}$.
The reference pair vectors at the section differ from their limits by
this amount. Increase $K_{\rm out}$ and decrease $\zeta_0$ so that
$3/4<|\boldsymbol d_j|<5/4$ and
\[
 q'\cdot e\geq3/2,\qquad |q'\cdot e^\perp|\leq1/8,\qquad
 q(0)\cdot e\geq2K_{\rm out},\qquad
 |q(0)\cdot e^\perp|\leq K_{\rm out}/8.
\]
Integration improves each inequality in
\eqref{eq:isolated-pair-bootstrap}. No failure occurs, and these isolated
solutions exist for every $s\geq0$.

Let $E_d(s)$ and $E_m(s)$ be the maximum vector and
midpoint differences of two such isolated solutions whose section states
differ by at most $\zeta$. Initially stop also if $E_d$ or
$E_m/(K_{\rm out}+s)$ reaches $1/64$. The segment
between the two configurations then stays separated. Differentiating
the pair fields along it gives
\begin{equation}\label{eq:cell-section-comparison}
 \begin{aligned}
 D^+E_d&\leq C(K_{\rm out}+s)^{-3}E_d
                      +C(K_{\rm out}+s)^{-4}E_m,\\
 D^+E_m&\leq CE_d+C(K_{\rm out}+s)^{-3}E_m .
 \end{aligned}
\end{equation}
The coefficients are uniform in the section neighborhood. The second
row and an integrating factor bounded above and below give
$E_m(s)\leq C\zeta+C\int_0^sE_d(v)\idiff v$.
Substitution into the first row, followed by a supremum up to any
finite endpoint $T$, gives
\[
 \sup_{s\leq T}E_d(s)
 \leq C\zeta+CK_{\rm out}^{-2}\sup_{s\leq T}E_d(s).
\]
Here we used
$\int_0^\infty(K_{\rm out}+s)^{-3}\idiff s=O(K_{\rm out}^{-2})$
and
$\int_0^\infty s(K_{\rm out}+s)^{-4}\idiff s
=O(K_{\rm out}^{-2})$.
Absorption yields $E_d\leq C\zeta$ and
$E_m\leq C\zeta(1+s)$, uniformly in $T$.
For sufficiently small $\zeta$ these inequalities improve the additional
stops, so the comparison holds for all $s\geq0$.
These are the $M_d=\eta=0$ cases of
\eqref{eq:donor-volterra}, proved here for the isolated equations.

Each nearby isolated pair has a limiting vector
$\boldsymbol d_{j,\infty}$ and velocity
$v_{j,\infty}=-K(\boldsymbol d_{j,\infty})\ne0$, by the same tail
integrals. Its speed need not equal one, and the two limiting vectors
need not be opposite. Define the direction of a section state to be
$v_{1,\infty}/|v_{1,\infty}|$. The comparison just proved makes it
Lipschitz in that state and equal to $e_{\rm out}$ on the reference
orbit. For patch-centroid section data, this direction coordinate is defined by
the auxiliary isolated point evolution. The patch comparison in
Section~\ref{sec:release} takes place on finite stopped intervals.
\end{proof}

\begin{lemma}[Phase primitive and offset errors]\label{lem:phase-offset}
On the positive branch, define $g_*$ and $P$ by
\[
 g_*(\mathsf U(\mathsf R))
    =\frac{(\mathsf R-1)(\mathsf R^2+1)}{\mathsf R^3},\qquad
 \mathsf U(\mathsf R)=\frac{2(\mathsf R-1)}{\sqrt{\mathsf R}},
\]
\begin{equation}\label{eq:phase-primitive}
 P(\mathsf U)=\int_0^{\mathsf U}
       \frac{g_*(v)}{a(v,0,1)v}\idiff v.
\end{equation}
Then $P$ is smooth through $\mathsf U=0$ and
\[
 (\vartheta+s-P(\mathsf U))'=0.
\]
On \eqref{eq:explicit-shape-chart}, allowing a nonzero midpoint offset
$\chi$ changes the isolated shape equations by at most
$C|\chi/d|^2$. The internal signed impulse is exactly conserved and
the internal mean-centroid velocity is bounded by $C|\chi|/d$ in
normalized velocity units.
\end{lemma}

\begin{proof}
We also need the mean phase at a prescribed section crossing.
On $\mathsf V=0$, $\mathsf A=1$, the exact equation reads
\[
 \vartheta'=-\frac1{\mathsf R^2}
                -\frac{(\mathsf R-1)^2}{\mathsf R^3}
            =-1+g_*(\mathsf U),\qquad
 g_*(\mathsf U(\mathsf R))
       =(\mathsf R-1)(\mathsf R^2+1)/\mathsf R^3=O(\mathsf U).
\]
Consequently $P$ in \eqref{eq:phase-primitive} is smooth, and $(\vartheta+s-P(\mathsf U))'=0$. Indeed
$\mathsf U=2(\mathsf R-1)/\sqrt{\mathsf R}$ has derivative $2$ at
$\mathsf R=1$, so $\mathsf R$ is smooth in $\mathsf U$ there and
$g_*(\mathsf U)/\mathsf U$ extends smoothly through zero.
Since $a(0,0,1)=\lambda>0$, the integrand has a smooth extension.
Differentiating $P$ and using $\mathsf U'=a(\mathsf U,0,1)\mathsf U$
gives the asserted invariant. In particular $P'$ is bounded near zero:
small errors in the shape equation produce integrable phase errors
without division by the possibly very small $\mathsf U$.

The actual patch centroids need not have $\chi=0$. With normalized offset
$\chi/d$, the two same-sign half-separation equations become
\begin{equation}\label{eq:cell-offset-equations}
 \begin{aligned}
 z'=K(2z)-\frac12\sum_{\epsilon=\pm1}
    \bigl(K(z-w+\epsilon\chi/d)+K(z+w+\epsilon\chi/d)\bigr),\\
 w'=-K(2w)+\frac12\sum_{\epsilon=\pm1}
    \bigl(K(w-z+\epsilon\chi/d)+K(w+z+\epsilon\chi/d)\bigr).
 \end{aligned}
\end{equation}
These equations concern the internal point velocities; patch averaging
and exterior velocities will be added later.
Both rows are even in $\chi/d$. For any separated vector $v$,
\[
 \tfrac12\bigl(K(v+\chi/d)+K(v-\chi/d)\bigr)-K(v)
  =\tfrac12\int_{-1}^1(1-|a|)
          \nabla^2K(v+a\chi/d)[\chi/d,\chi/d]\idiff a.
\]
Thus the internal offset error in each smooth shape equation is
$O(|\chi/d|^2)$ on a fixed separated chart.
For the signed impulse, pair the terms with indices $i,j$:
$\epsilon_i\epsilon_j(K(y_i-y_j)+K(y_j-y_i))=0$.
Its internal derivative is therefore exactly zero, including when
$\chi\ne0$. The internal mean-centroid velocity instead vanishes by
central symmetry at $\chi=0$; smoothness bounds it by $C|\chi|/d$
in the normalized velocity units.
These different orders will be used separately in the comparison.
\end{proof}

\subsection{The signed remote strain}
\label{subsec:signed-strain}

The outgoing pair vectors have bounded length, while their midpoints
separate linearly. Their signed second moment therefore grows linearly.
Two cancellations remove the mass and impulse terms of the remote
field; the second moment is the first surviving term. The following
calculation retains its phase, since its sign determines which shape
direction is excited. Figure~\ref{fig:donor-receiver} illustrates this
donor--receiver interaction.

\begin{proposition}[Signed expanding-mode coefficient]\label{prop:signed-strain}
Let $\sigma>0$ and $\gamma>0$ be the source-to-receiver length and
circulation-magnitude ratios, and put
\[
 r=\frac{\sigma^2}{\gamma},\qquad \lambda=2\sqrt2.
\]
Let $\zeta_i$, labeled $(z,-z,w,-w)$ with signs
$(\epsilon_1,\epsilon_2,\epsilon_3,\epsilon_4)=(1,1,-1,-1)$,
be the isolated orbit starting on the fixed outgoing section of
Proposition~\ref{prop:isolated-outgoing}, rotated so that its limiting
direction is $e_1$. For $\theta\in\mathbb R$, set
\[
 y_i(s)=\sigma\operatorname{Rot}(\theta)\zeta_i(s/r),\qquad
 M(s)=\gamma\sum_i\epsilon_i y_i(s)\otimes y_i(s)
     =\gamma^2P_{\sigma,\theta}(s).
\]
For every $s\geq0$,
\begin{equation}\label{eq:cell-normalized-moment}
 \begin{gathered}
 P_{\sigma,\theta}(s)=r\sum_i\epsilon_i
      (\operatorname{Rot}(\theta)\zeta_i(s/r))
            \otimes(\operatorname{Rot}(\theta)\zeta_i(s/r)),\\
 P_{\sigma,\theta}(s)=2sS_\theta+O(r),\qquad
 S_\theta=e(\theta)\otimes e(\theta)^\perp+e(\theta)^\perp\otimes e(\theta),
 \quad e(\theta)=(\cos\theta,\sin\theta).
 \end{gathered}
\end{equation}
At a displacement $De_1$, assume $\max_i|y_i(s)|\leq D/2$. Then
\begin{equation}\label{eq:cell-remote-taylor}
 \begin{aligned}
 \gamma\sum_i\epsilon_i\nabla K(De_1-y_i)
   &=\tfrac12\nabla^3K(De_1):M+\mathcal T,\\
 |\mathcal T|&\leq C\gamma D^{-5}\sum_i|y_i|^3 .
 \end{aligned}
\end{equation}
The contraction is
$(\nabla^3K(x):M)_{ab}=\sum_{j,k}\partial_b\partial_j\partial_kK_a(x)M_{jk}$.
Writing the two source pairs as $m_j\pm d_j/2$, one also has
\begin{equation}\label{eq:cell-grouped-taylor}
 |\mathcal T|\leq C\gamma D^{-5}
           \sum_{j=1}^2|d_j|(|m_j|+|d_j|)^2.
\end{equation}
Thus the leading remote gradient is
$\varepsilon\mathcal B(P_{\sigma,\theta})$, where
\[
 \varepsilon=\gamma^2D^{-4},\qquad
 \mathcal B(T)=\tfrac12\nabla^3K(e_1):T.
\]
For the expanding projection in the receiver's rotating frame, define
\begin{equation}\label{eq:cell-rotated-forcing}
 \begin{gathered}
 \mathcal S_{\sigma,\psi}(s)=\operatorname{Rot}(s)
       \mathcal B(P_{\sigma,\psi}(s))\operatorname{Rot}(-s),\qquad
 f_{\sigma,\psi}(s)
       =-2\bigl(\sqrt2(\mathcal S_{\sigma,\psi})_{11}
                         +(\mathcal S_{\sigma,\psi})_{12}\bigr).
 \end{gathered}
\end{equation}
Then
\begin{equation}\label{eq:coefficient}
 \mathcal H_\sigma(\psi):=\int_0^\infty e^{-\lambda s}
                    f_{\sigma,\psi}(s)\idiff s
       =\mathcal H(\psi)+O(r),\qquad
 \mathcal H(\psi)=2(\sqrt2\cos2\psi+\sin2\psi).
\end{equation}
The remainder constants depend only on the chosen outgoing section.
\end{proposition}

\begin{proof}
Normalize the observed configuration to length and circulation one,
using the time variable $s=t/\tau$.
Let a second configuration have length $\sigma$ and circulation
magnitude $\gamma$ in these units. Its time scale divided by the
observed time scale is $r=\sigma^2/\gamma$, because
$(\sigma d)^2/(\gamma\Gamma)\cdot\Gamma/d^2=r$.
Choose the isolated orbit
$\zeta_i$, with labels $(z,-z,w,-w)$ and signs
$\epsilon_i=(1,1,-1,-1)$, whose direction in
\eqref{eq:outgoing-asymptotics} is $e_1$. After rotation through
$\theta$, its positions relative to its center are
$y_i(s)=\sigma\operatorname{Rot}(\theta)\zeta_i(s/r)$.
The signed mass, impulse and second moment are, componentwise,
\begin{equation}\label{eq:cell-signed-moments}
 \sum_i\gamma\epsilon_i=0,\qquad
 \sum_i\gamma\epsilon_i y_i=0,\qquad
 M_{ab}(s)=\gamma\sum_i\epsilon_i(y_i)_a(y_i)_b,\quad a,b\in\{1,2\}.
\end{equation}
Here $(x\otimes y)_{ab}=x_ay_b$, and all moments are taken about
the same configuration center. Zero signed mass makes the impulse
independent of that center; zero impulse also makes this second moment
independent of the choice of center. The corresponding patch moments
will include their component covariances and their nonzero generated
impulse in Section~\ref{sec:release}.

Write $M=\gamma^2P_{\sigma,\theta}$. Since
$\gamma\sigma^2=\gamma^2r$, the exact expression for $P_{\sigma,\theta}$ is the first row of
\eqref{eq:cell-normalized-moment}.
To justify the uniform remainder, the exact pair identity is
\[
 \sum_i\epsilon_i\zeta_i\otimes\zeta_i
  =2(z\otimes z-w\otimes w)
  =2(\boldsymbol m\otimes\boldsymbol d+
                         \boldsymbol d\otimes\boldsymbol m).
\]
Substitute \eqref{eq:outgoing-asymptotics} at $v=s/r$.
The leading term is $2v(e_1\otimes e_1^\perp+e_1^\perp\otimes e_1)$.
The remaining terms are bounded uniformly for $v\geq0$: the constant
$m_\infty$ is fixed, $v(K_{\rm out}+v)^{-2}$ is bounded, and the
midpoint tail is $O((K_{\rm out}+v)^{-1})$. Multiplication by $r$
gives the $O(r)$ term in \eqref{eq:cell-normalized-moment}.
Its constant is independent of $s,\sigma,\gamma,\theta$ after the
section has been fixed.

We next specify the spatial regime in which this moment gives the
remote velocity gradient. At displacement $De_1$, assume
$\max_i|y_i(s)|\leq D/2$. Taylor expansion of the normalized gradient
and \eqref{eq:cell-signed-moments} give \eqref{eq:cell-remote-taylor}.
The colon contracts the two moment indices, leaving the velocity and
gradient indices:
\[
 (\nabla^3K(x):M)_{ab}
       =\sum_{j,k=1}^2\partial_b\partial_j\partial_k K_a(x)M_{jk}.
\]
In \eqref{eq:cell-remote-taylor}, the constant term vanishes by zero
mass, the linear term by zero impulse, and the quadratic term has a
positive sign because it contains two factors of $-y_i$.
The fourth derivatives of $K$ are bounded by $CD^{-5}$ on all Taylor
segments, which proves the remainder estimate.

Grouping the two neutral pairs gives the sharper bound used later.
If their midpoint and vector in the observed units are $m_j,d_j$,
then the derivative with respect to $y$ of the third-order Taylor
remainder of $\nabla K(De_1-y)$ is bounded by
$CD^{-5}|y|^2$. Subtracting its values at $m_j+d_j/2$ and $m_j-d_j/2$ gives
\eqref{eq:cell-grouped-taylor}.
Since $|d_j|=O(\sigma)$ and
$|m_j|=O(\sigma+\gamma s/\sigma)$, each extra Taylor order has an
extra ratio of source diameter to distance. The moment expansion
\eqref{eq:cell-normalized-moment} holds for all $s\geq0$; the spatial
expansions hold only while the stated source-size condition holds.

By homogeneity $\nabla^3K(De_1)=D^{-4}\nabla^3K(e_1)$,
so the leading strain is
$\varepsilon\mathcal B(P_{\sigma,\theta})$, where
$\varepsilon=\gamma^2D^{-4}$ and
$\mathcal B(T)=\tfrac12\nabla^3K(e_1):T$.
The cubic terms of the kernel at $e_1+(h,k)$ are
\[
 \begin{aligned}
 K_1(e_1+(h,k))&=-k+2hk-3h^2k+k^3+O(|(h,k)|^4),\\
 K_2(e_1+(h,k))&=1-h+h^2-k^2-h^3+3hk^2+O(|(h,k)|^4).
 \end{aligned}
\]
They determine all entries of the contraction:
\begin{equation}\label{eq:cell-kernel-contraction}
 \mathcal B(T)=
 \begin{pmatrix}-6T_{12}&3(T_{22}-T_{11})\\
                 3(T_{22}-T_{11})&6T_{12}\end{pmatrix},
 \qquad
 \mathcal B(2sS_\theta)=12s
 \begin{pmatrix}-\cos2\theta&\sin2\theta\\
                 \sin2\theta&\cos2\theta\end{pmatrix}.
\end{equation}
The formula applies to symmetric $T$, as all second moments here are.
In particular scalar multiples of the identity map to zero, in
agreement with $\Delta K=0$ away from the origin.
At the reference shape of the observed configuration, take its positive
axis to be horizontal. An affine strain
$\left(\begin{smallmatrix}a&b\\b&-a\end{smallmatrix}\right)$ then
acts on $z=(|z|,0)$ and $w=(0,-|w|)$ from
\eqref{eq:cell-reference}. Its radial and angular contributions are
\[
 (|z|^2)'=2a|z|^2,\qquad (|w|^2)'=-2a|w|^2,\qquad
 \alpha_z'=b,\qquad\alpha_w'=-b.
\]
Thus $\mathsf R'=0$, $\mathsf A'=2\sqrt2a$ and $\phi'=2b$.
At $\mathsf A=\mathsf R=1$, $\phi=\pi/2$, the coordinate
differentials in \eqref{eq:chart} are
$\diff\mathsf X=\diff\mathsf R-\diff\mathsf A$ and
$\diff\mathsf Y=\diff\phi$. The resulting expanding contribution is
\begin{equation}\label{eq:cell-strain-projection}
 (\mathsf U')_{\rm strain}
       =\mathsf R'-\mathsf A'-\phi'=-2(\sqrt2a+b).
\end{equation}
The minus sign is determined by the choice
$\mathsf U=\mathsf X-\mathsf Y$ and by the labeled relative angle.

We record the actual moving coordinate map before rotating the strain.
Let $c(t)$ be a differentiable origin for the observed configuration,
let $\theta_{\rm rec}(s)$ be its reference positive-axis angle, and set
\begin{equation}\label{eq:cell-moving-frame}
 t=t_{\rm in}+\tau s,\qquad
 y(s)=d^{-1}\operatorname{Rot}(-\theta_{\rm rec}(s))(x(t)-c(t)).
\end{equation}
For a trajectory $x_t=u(t,x)$, differentiation gives
\begin{equation}\label{eq:cell-moving-equation}
 y'=-\theta_{\rm rec}'(s)y^\perp
   +\frac{\tau}{d}\operatorname{Rot}(-\theta_{\rm rec}(s))
       \bigl(u(t,c(t)+d\operatorname{Rot}(\theta_{\rm rec}(s))y)
                                                       -c_t(t)\bigr).
\end{equation}
A spatially constant term changes only translation and cancels from
the half-separation equations. An exterior affine gradient $B_{\rm phys}$
becomes
$\tau\operatorname{Rot}(-\theta_{\rm rec})
B_{\rm phys}\operatorname{Rot}(\theta_{\rm rec})$.
Thus the gradient gains a time factor but no length factor.
The additional term $-\theta_{\rm rec}'y^\perp$ is part of the full
evolution, not of this transformed strain.
For the reference orbit $\theta_{\rm rec}(s)=\theta_{\rm rec}(0)-s$,
so this term is $y^\perp$, cancelling its isolated rotation $-y^\perp$.

If $\theta_{\rm rec}$ is the observed initial positive-axis angle
and $\theta_{\rm line}$ the center-line angle, rotational covariance
leaves the effective phase
$\psi=\theta+\theta_{\rm rec}-2\theta_{\rm line}$.
Indeed, first rotating the line to $e_1$ replaces the source direction
by $\theta-\theta_{\rm line}$. Rotating the resulting gradient from
the line coordinates to the observed coordinates conjugates it by
$\operatorname{Rot}(\theta_{\rm line}-\theta_{\rm rec})$.
For the matrix
$M_\beta=\left(\begin{smallmatrix}-\cos2\beta&\sin2\beta\\
\sin2\beta&\cos2\beta\end{smallmatrix}\right)$ in
\eqref{eq:cell-kernel-contraction}, one has
$\operatorname{Rot}(\alpha)M_\beta\operatorname{Rot}(-\alpha)
=M_{\beta-\alpha}$, which gives precisely this $\psi$.
The same covariance follows from the first formula in
\eqref{eq:cell-kernel-contraction} for the complete moment, including
its bounded remainder.
The observed reference shape rotates with angular velocity $-1$.
In its rotating coordinates use \eqref{eq:cell-rotated-forcing}.
The notation retains the parameter $\sigma$ used later; this normalized
isolated orbit also depends on the time ratio $r=\sigma^2/\gamma$.
Equations~\eqref{eq:cell-normalized-moment},
\eqref{eq:cell-kernel-contraction} and
\eqref{eq:cell-strain-projection} now give
\begin{equation}\label{eq:cell-leading-forcing}
 f_{\sigma,\psi}(s)
 =24s\bigl(\sqrt2\cos(2\psi-2s)-\sin(2\psi-2s)\bigr)+O(r).
\end{equation}
The error is uniform in $s\geq0$ and $\psi$, and
$|f_{\sigma,\psi}(s)|\leq C(s+r)$.

The exponentially weighted coefficient is explicit. For $\lambda=2\sqrt2$,
\[
 \int_0^\infty s e^{-(\lambda-2\iunit)s}\idiff s
     =\frac1{(\lambda-2\iunit)^2}
     =\frac1{36}+\iunit\frac{\sqrt2}{18}.
\]
Taking real and imaginary parts gives
\[
 \int_0^\infty se^{-\lambda s}\cos2s\idiff s=\frac1{36},\qquad
 \int_0^\infty se^{-\lambda s}\sin2s\idiff s=\frac{\sqrt2}{18}.
\]
Expanding the trigonometric functions in
\eqref{eq:cell-leading-forcing}, its weighted integral equals
\[
 24\left[
   \cos2\psi\left(\frac{\sqrt2}{36}+\frac{\sqrt2}{18}\right)
  +\sin2\psi\left(\frac2{18}-\frac1{36}\right)\right]
       =2(\sqrt2\cos2\psi+\sin2\psi).
\]
The error integrates to at most $Cr/\lambda$, proving
\eqref{eq:coefficient}.
This integral uses the isolated leading-moment forcing, which is defined
for all $s\geq0$. The spatial Taylor approximation is used under its
source-size condition. On each stopped Euler interval $[0,S]$, the
omitted forcing tail satisfies
\[
 \left|\int_S^\infty e^{-\lambda s}f_{\sigma,\psi}(s)\idiff s\right|
       \leq C(1+S+r)e^{-\lambda S}.
\]
To see the consequence of its sign,
the scalar equation $u'=\lambda u+\varepsilon f_{\sigma,\psi}$,
$u(0)=0$, gives
\[
 e^{-\lambda s}u(s)
   =\varepsilon\int_0^s e^{-\lambda v}f_{\sigma,\psi}(v)\idiff v
   \longrightarrow\varepsilon\mathcal H_\sigma(\psi).
\]
Thus a positive lower bound for $\mathcal H$ produces a positive
expanding component of order $\varepsilon$ once $r$ is small. This
calculation concerns the isolated equations. Section~\ref{sec:release}
will prove the corresponding estimate for the Euler evolutions, with
a variable positive rate and all patch and exterior errors included.
Its inputs are the factorization \eqref{eq:exact-saddle}, the tensor
and projection \eqref{eq:outgoing-asymptotics}--\eqref{eq:coefficient},
and the smooth phase identity \eqref{eq:phase-primitive}.

\end{proof}
\section{The multiscale construction and three technical estimates}\label{sec:scales}\label{sec:datum}

The radii and circulations must meet three quantitative requirements.
First, $\sum_n\Gamma_n^q a_n^{-2(q-1)}<\infty$ makes the vorticity norm
summable. Second, $1/\tau_n\leq\eta_n^{-1/2}$ puts physical time one
within the support estimate. Third, the normalized variance error
$\eta_n^2$ must be smaller than the second-moment forcing
$\varepsilon_n$ after every fixed polynomial factor in
$\ell_n=\log(1/\varepsilon_n)$. We choose the exponents to satisfy these
three requirements and then verify the remaining time and distance sums.

\subsection{Integrability first, then spatial separation}
\label{subsec:scale-choice}

Fix the exponent $q=\max(p,3)$. A disk of radius $a$ and circulation
magnitude $\Gamma$ contributes $\pi^{1-q}\Gamma^q a^{-2(q-1)}$
to the $q$th power of the vorticity norm. We therefore choose the
circulation and radius exponents together:
\begin{equation}\label{eq:general-exponents}
 \delta=\frac1{10q},\qquad \beta=2-\delta,\qquad E=5\delta.
\end{equation}
If $\Gamma=d^\beta$ and $a=\eta_*d^{1+E}$, the resulting power of $d$ is
\begin{equation}\label{eq:norm-exponent}
 \kappa_q=q\beta-2(q-1)(1+E)=\frac9{10}+\frac1q>0.
\end{equation}
The strict positivity of $\kappa_q$ permits both small patch radii and
a summable $L^q$ norm.
Here is the compatibility problem behind these choices. Writing
$\beta=2-\delta$ makes the interaction time $\tau=d^\delta$ tend
to zero. Writing $\eta=\eta_*d^E$ gives
\[
 \frac{\eta^{1/2}}{\tau}=\eta_*^{1/2}d^{E/2-\delta}.
\]
Thus the support estimate through physical time one requires
$E>2\delta$. On the other hand, norm summability requires
\[
 q\delta+2(q-1)E<2.
\]
With $\delta=1/(10q)$ and $E=5\delta$, the left side is
$1/10+(q-1)/q<2$, and $E/2-\delta=3\delta/2>0$.
These two strict inequalities leave room both for concentration over a
long normalized interval and for an integrable infinite family.

The time at index $n+1$ must be short compared with that at index $n$.
We arrange $\tau_{n+1}/\tau_n
=(R_{n+1}/R_n)^\alpha$ with $\alpha>0$. Set
\begin{equation}\label{eq:radial-exponents}
 \rho=\frac1{1-2\delta},\qquad
 \alpha=\rho\delta=\frac1{10q-2},\qquad
 \nu_\Gamma=\rho\beta=2+3\alpha.
\end{equation}
In particular $0<\alpha\leq1/28$, $\rho=1+2\alpha$,
$\rho(\beta-1)=1+\alpha$, and $\rho E=5\alpha$.
These identities are the source of the powers in the remote-field sums.
Indeed, if $d=hR^\rho$, then $\tau=h^\delta R^\alpha$,
$\Gamma=h^\beta R^{\nu_\Gamma}$, and
$\Gamma/d=h^{\beta-1}R^{1+\alpha}$. The derivatives of the remote
kernel will introduce inverse powers of $R$; the useful residual powers
are
\[
 \nu_\Gamma-2=3\alpha,\qquad
 2\nu_\Gamma-4=6\alpha,\qquad
 \nu_\Gamma+(1+\alpha)-3=4\alpha.
\]
All are positive. This explains the choice of $\rho$, in addition
to the time ratio it produces. Finally $\rho E=5\alpha>4\alpha$
will leave a positive quadratic coefficient in the comparison of
$\eta_n^2$ with $\varepsilon_n$.

Let $L$ be a large finite number, to be fixed after the geometric and
analytic allowances. Put
\begin{equation}\label{eq:scales}
 \begin{gathered}
 h=e^{-L},\qquad R_n=e^{-Ln^2},\qquad d_n=hR_n^\rho,
 \qquad \Gamma_n=d_n^\beta,\\
 \tau_n=d_n^2/\Gamma_n=d_n^\delta,\\
 M=10+\nu_\Gamma+\frac{\nu_\Gamma^2}{\alpha},\qquad
 \eta_*=e^{-ML},\qquad \eta_n=\eta_*d_n^E,\qquad a_n=d_n\eta_n.
 \end{gathered}
\end{equation}
The configuration at index $n$ is initially centered at $R_ne_1$.
Its four disks have radius $a_n$, circulation magnitudes $\Gamma_n$,
signs $(1,1,-1,-1)$ and centroid distances comparable to $d_n$.
This is the array illustrated in Figure~\ref{fig:initial-array};
Figure~\ref{fig:cascade-array} shows the order of its subsequent activation.
Their shape coordinates and rotation angles will be chosen later.
The physical strengths and radii are now fixed independently
of the cutoff. Both patch transport and point dynamics use $t/\tau_n$
and the kernel $K(x)=x^\perp/|x|^2$.

Whenever these initial disks are disjoint, their norm sums are bounded by
the following finite quantities:
\begin{equation}\label{eq:norms}
 4\sum_{n\geq0}\Gamma_n<\infty,\qquad
 4\pi^{1-q}\eta_*^{-2(q-1)}\sum_{n\geq0}d_n^{\kappa_q}<\infty.
\end{equation}
Indeed, for any $r>0$ and $m\geq0$,
\begin{equation}\label{eq:norm-tail}
 \sum_{n\geq m}d_n^r
 \leq\frac{d_m^r}{1-e^{-\rho Lr}}.
\end{equation}
To verify the uniform tail, write $n=m+j$ and use
$(m+j)^2-m^2=2mj+j^2\geq j$. Then
$d_{m+j}^r\leq d_m^r e^{-\rho Lrj}$, whose geometric sum is
\eqref{eq:norm-tail}. For a fixed $L$, each cutoff truncates this one infinite sequence.
These tails will later reduce convergence of the whole datum to
translation continuity of finitely many disks.

\subsection{The forcing size and the comparison times}

To compute the effect of index $n+1$ on index $n$, measure lengths in
units of $d_n$ and time in units of $\tau_n$. The necessary
ratios are
\begin{equation}\label{eq:event-scales-general}
 \begin{gathered}
 q_n=\frac{R_{n+1}}{R_n}=e^{-L(2n+1)},\qquad
 \sigma_n=\frac{d_{n+1}}{d_n}=q_n^\rho,\qquad
 r_n=\frac{\tau_{n+1}}{\tau_n}=q_n^\alpha,\\
 p_n=\frac{d_n}{R_n}=hR_n^{2\alpha},\qquad
 D_n=\frac{R_n-R_{n+1}}{d_n}=\frac{1-q_n}{p_n},\\
 \eps_n=\sigma_n^{2\beta}D_n^{-4},\qquad
 \ell_n=\log(1/\eps_n).
 \end{gathered}
\end{equation}
The normalized circulation at index $n+1$ is
$\Gamma_{n+1}/\Gamma_n=\sigma_n^\beta$.
Its signed second moment therefore contributes
$\eps_n=\sigma_n^{2\beta}D_n^{-4}$ to the gradient at index $n$,
as computed in Sec.~\ref{sec:cell}. A quantity starting at size
$\eps_n$ and growing exponentially reaches a fixed value in time
$O(\ell_n)$. The radial ratio $q_n$ is distinct from the fixed
Lebesgue exponent $q$.

\begin{lemma}[Uniform scale margins]\label{lem:scale-margins}
For $L\geq\log8$ and $3\alpha L\geq\log2$,
\begin{align}
 \ell_n&=L\bigl(4+8\alpha n^2+2\nu_\Gamma(2n+1)\bigr)
                +4\log(1-q_n),\quad
 1+\ell_n\leq10L(n+1)^2,\label{eq:log-signal}\\
 \sum_{k>n}\frac{\tau_k(1+\ell_k)}{\tau_n}
 &\leq120L(n+2)^2e^{-\alpha L(2n+1)},\label{eq:tailclock}\\
 \sum_{n\geq0}\tau_n(1+\ell_n)
 &\leq490Le^{-\delta L},\label{eq:totalclock}\\
 \frac{\eta_n^2}{\eps_n}
 &=(1-q_n)^4\exp\!\left[-L\left(16+2E+
       2\alpha\left(n-\frac{\nu_\Gamma}{\alpha}\right)^2\right)\right].
 \label{eq:coreratio}
\end{align}
For each fixed nonnegative integer $m$,
\begin{equation}\label{eq:polynomial-margin}
 \sup_{n\geq0}(1+\ell_n)^m\frac{\eta_n^2}{\eps_n}
 \leq C_{q,m}L^m e^{-(16+2E)L}\longrightarrow0.
\end{equation}
Moreover $\sup_n(1+\ell_n)/\sqrt{D_n}\to0$, and the elapsed time
$1/\tau_n$ is no larger than $\eta_n^{-1/2}$:
\begin{equation}\label{eq:core-lifetimes}
 \frac{\eta_n^2}{\tau_n}=\eta_*^2d_n^{9\delta},\qquad
 \frac{\eta_n^{1/2}}{\tau_n}
       =\eta_*^{1/2}d_n^{3\delta/2}\leq1.
\end{equation}
All limits are uniform in $n$ for the fixed exponent $q$.
\end{lemma}

\begin{proof}
We first compute the logarithmic time scale without dropping the
factor $1-q_n$:
\[
 \log D_n=L+2\alpha Ln^2+\log(1-q_n),\qquad
 \log\sigma_n=-\rho L(2n+1).
\]
Consequently
$\ell_n=-2\beta\log\sigma_n+4\log D_n$, giving the first
identity in \eqref{eq:log-signal}. For its upper bound use
$\log(1-q_n)\leq0$, $\alpha\leq1/28$, and
$\nu_\Gamma=2+3\alpha\leq59/28$. The constant, linear and
quadratic coefficients in
$4+8\alpha n^2+2\nu_\Gamma(2n+1)$ are bounded above by those
in $9(n+1)^2$. Since $L\geq\log8>1$, adding one is bounded
by the remaining $L(n+1)^2$.

For the tail write $k=n+1+j$. The exact time ratio is
$\tau_k/\tau_n=e^{-\alpha L(k^2-n^2)}$. The inequalities
$(k+1)^2\leq(n+2)^2(j+1)^2$ and
$k^2-n^2\geq2n+1+3j$ give
\[
 \begin{split}
 \sum_{k>n}\frac{\tau_k(1+\ell_k)}{\tau_n}
 &\leq10L(n+2)^2e^{-\alpha L(2n+1)}
       \sum_{j\geq0}(j+1)^2e^{-3\alpha Lj}\\
 &\leq120L(n+2)^2e^{-\alpha L(2n+1)}.
 \end{split}
\]
The last sum is at most
$\sum_{j\geq0}(j+1)^2 2^{-j}=12$; this follows by differentiating
the geometric series twice, or from
$\sum_{j\geq0}(j+1)^2x^j=(1+x)/(1-x)^3$.
Apply the tail bound at $n=0$, add $1+\ell_0\leq10L$,
and use $\tau_0=e^{-\delta L}$ to obtain
$\sum_n\tau_n(1+\ell_n)\leq(10+480)L e^{-\delta L}$.

The error comparison requires keeping the terms linear in $n$ as well
as the Gaussian term. The definitions give
\begin{equation}\label{eq:scale-quadratic}
 \begin{split}
 &\log\frac{\eta_n^2}{\eps_n}
 =4\log(1-q_n)
 -L\bigl[2M+2E-4-2\nu_\Gamma
                   +2\alpha n^2-4\nu_\Gamma n\bigr],\\
 &2M+2E-4-2\nu_\Gamma+2\alpha n^2-4\nu_\Gamma n
 =16+2E+2\alpha(n-\nu_\Gamma/\alpha)^2.
 \end{split}
\end{equation}
The second equality is exactly the reason for including
$\nu_\Gamma^2/\alpha$ in $M$. It proves \eqref{eq:coreratio}.
For fixed $q$ the center $c_q=\nu_\Gamma/\alpha=20q-1$
is fixed, although it need not be small. For $L\geq1$,
\[
 \sup_{n\geq0}(n+1)^{2m}e^{-2\alpha L(n-c_q)^2}
 \leq\sup_{x\geq0}(x+1)^{2m}e^{-2\alpha(x-c_q)^2}
 =:C_{q,m}'<\infty.
\]
Combining this bound with $1+\ell_n\leq10L(n+1)^2$
and $(1-q_n)^4\leq1$ proves \eqref{eq:polynomial-margin}.
This uniform supremum permits the approximation cutoff to tend to infinity.

Similarly, since $q_n\leq1/8$,
\[
 \frac{1+\ell_n}{\sqrt{D_n}}
 \leq C L e^{-L/2}(n+1)^2e^{-\alpha Ln^2}
 \leq C_q L e^{-L/2}.
\]
This proves the uniform distance margin. Finally the two expressions
in \eqref{eq:core-lifetimes} follow by subtracting $\delta$ from
$2E$ and $E/2$, respectively. In particular the second one is
equivalent to $1/\tau_n\leq\eta_n^{-1/2}$.
\end{proof}

These bounds now verify each time comparison explicitly. If the already
completed intervals at indices $k>n$ have lengths
$C\tau_k(1+\ell_k)$, then \eqref{eq:tailclock} makes their sum
$o_L(\tau_n)$. For $k<n$, the interval of length
$C\tau_n(1+\ell_n)$ is also $o_L(\tau_k)$.
Thus the difference estimate for initially identical disks is used within
the fixed small interaction time $T_*$ of Lemma~\ref{lem:response},
within the prescribed interval of available exterior bounds and starting
from physical time zero.
Equation~\eqref{eq:totalclock} makes the sum of all required physical
intervals smaller than $1/4$ after a single choice of $L$.
Finally \eqref{eq:polynomial-margin} makes every error
$C(1+\ell_n)^m\eta_n^2$ negligible relative to $\eps_n$, uniformly in
$n$. These conditions and the summable norm bound are fixed before $N$
increases.

\subsection{Component densities and the three technical inputs}
\label{subsec:technical-inputs}

At a finite cutoff $N$, the two evolutions are indexed by
$\kappa\in\{\mathrm{quiet},\mathrm{trig}\}$. Their vorticities are
\begin{equation}\label{eq:component-density-convention}
 \omega_N^\kappa(t,x)
 =\sum_{n=0}^N\sum_{j=1}^4\epsilon_j\Gamma_n\mu_{n,j}^\kappa(t,x),
 \qquad (\epsilon_1,\epsilon_2,\epsilon_3,\epsilon_4)=(1,1,-1,-1),
 \qquad \mu_{n,j}^\kappa\geq0,\quad\int\mu_{n,j}^\kappa=1.
\end{equation}
The cutoff index on a component density is suppressed. The disks in
Figure~\ref{fig:initial-array} are specified initially by
\begin{equation}\label{eq:initial-component-disks}
 \mu_{n,j}^\kappa(0,x)
  =\frac{\ind_{B_{a_n}(b_{n,j}^\kappa(0))}(x)}{\pi a_n^2},
 \qquad
 b_{n,j}^\kappa(0)=R_ne_1+d_n y_{n,j}^\kappa,
 \qquad \sum_{j=1}^4 y_{n,j}^\kappa=0.
\end{equation}
The shapes and phases of $y_{n,j}^\kappa$ will be selected in
Section~\ref{sec:finite}; their scales have already been fixed. For any prescribed $C_{\rm geom}\geq2$ and $0<c_{\rm geom}\leq1/2$,
the inequalities
\[
 |y_{n,j}^\kappa|\leq C_{\rm geom},\qquad
 |y_{n,i}^\kappa-y_{n,j}^\kappa|\geq c_{\rm geom}\quad(i\ne j)
\]
give disjoint initial disks once
$e^{-L}\leq[8(C_{\rm geom}+1)]^{-1}$ and
$\sup_n\eta_n\leq c_{\rm geom}/4$. Within one configuration this follows
from $2a_n\leq c_{\rm geom}d_n/2$. For $m>n$, its two enclosing balls
have sum of radii at most
$(C_{\rm geom}+1)e^{-L}(R_n+R_m)$, whereas their centers are separated
by $R_n-R_m\geq(1-e^{-L})R_n$. The constants allow both the rotating
reference shapes and the fixed outgoing section.

All three estimates below are local in the configuration and uniform in
the small core radius. We state their precise hypotheses now; their
proofs occupy Section~\ref{sec:technical-proofs}. The first estimate
removes rigid rotation from the concentration cost. The second uses the
cancellation of the common mutual velocity of an outgoing neutral pair.
The third compares two patch shapes starting from disks, so that its
error is proportional to the difference of the fields acting on them.
This proportional estimate is needed for configurations larger than the
receiver: an absolute error of size $O(\eta_k^2)$ need not be smaller
than the forcing $\varepsilon_n$ at a different index. At the receiver
itself, by contrast, \eqref{eq:polynomial-margin} makes absolute errors
of size $\eta_n^2$ negligible relative to $\varepsilon_n$, even after
the fixed polynomial losses in the event duration.

For length $d$, circulation magnitude $\Gamma$ and starting time
$t_{\rm in}$, the normalization used in the first two lemmas is
\begin{equation}\label{eq:material-exterior-scaling}
 s=\frac{\Gamma(t-t_{\rm in})}{d^2},\qquad
 \mu_i(s,y)=d^2\mu_i^{\rm phys}(t,dy),\qquad
 V(s,y)=\frac d\Gamma V^{\rm phys}
       \left(t_{\rm in}+\frac{d^2}{\Gamma}s,dy\right).
\end{equation}
In particular, $b_i=b_i^{\rm phys}/d$ and
$I_i=I_i^{\rm phys}/d^2$. Component densities are always probability
densities; their signed circulations are the separate factors
$\epsilon_i\Gamma$.

\begin{lemma}[Concentration with a quantified rotation cost]
\label{lem:concentration}
Fix $m\in\{2,4\}$, signs $\epsilon_i\in\{-1,1\}$ with
$\sum_{i=1}^m\epsilon_i=0$, and numbers
\[
 0<\ell_-\leq\ell_+<\infty,\qquad
 0<r\leq\ell_-/16,\qquad C_0,C_{\rm init}\geq1.
\]
There are $C\geq1$ and $0<\eta_0<1$, depending only on these data,
with the following property. Let $0<\eta<\eta_0$ and $s_{\rm pre},S\geq0$
satisfy $s_{\rm pre}+S\leq\eta^{-1/2}$. On $[0,S]$, let the
nonnegative unit-mass densities $\mu_i$ be transported by the same
area-preserving bounded-vorticity flow with velocity
\[
 u(s,y)=\sum_{j=1}^m\epsilon_j(K*\mu_j)(s,y)+V(s,y).
\]
Assume each $\mu_i(s)$ is a rearrangement of
$f_\eta=(\pi\eta^2)^{-1}\ind_{B_\eta}$, and define
\[
 \begin{gathered}
 b_i=\int y\mu_i(y)\idiff y,\qquad
 I_i=\int|y-b_i|^2\mu_i(y)\idiff y,\qquad
 \boldsymbol c=m^{-1}\sum_i b_i,\\
 \mathcal S(\mu)=-\iint\log|x-y|\mu(x)\mu(y)\idiff x\idiff y,
 \qquad \mathcal D_i=\mathcal S(f_\eta)-\mathcal S(\mu_i).
 \end{gathered}
\]
For every $s\in[0,S]$ suppose
\begin{equation}\label{eq:concentration-geometry}
 |b_i(s)-b_j(s)|\geq\ell_-\quad(i\ne j),\qquad
 |b_i(s)-\boldsymbol c(s)|\leq\ell_+.
\end{equation}
Set $\mathcal N(s)=\bigcup_i B_{2r}(b_i(s))$. The field $V$ is smooth,
divergence-free and curl-free on an open neighborhood of the closure
of these tubes. With spatial norms on $\mathcal N(s)$, set
\[
 \begin{gathered}
 V_i^{\rm pt}=\sum_{j\ne i}\epsilon_jK(b_i-b_j)+V(s,b_i),\qquad
 A=\|\nabla V\|_{L^\infty(\mathcal N(s))},\qquad
 H=\|\nabla^2V\|_{L^\infty(\mathcal N(s))},\\
 B=\max_i\| (\partial_s+V_i^{\rm pt}\cdot\nabla)\nabla V
                                      \|_{L^\infty(\mathcal N(s))}.
 \end{gathered}
\]
Here $V_i^{\rm pt}(s)$ is held constant in the spatial derivative.
There is an absolutely continuous real $\Omega$ such that, for
$W_i=V_i^{\rm pt}-\boldsymbol c'-\Omega(b_i-\boldsymbol c)^\perp$,
\begin{equation}\label{eq:cost}
 \sup_{0\leq s\leq S}(|\Omega(s)|+A(s))
 +\int_0^S\left(A+B+H+H^2+|\Omega'|+\max_i|W_i|\right)\idiff s
 \leq C_0.
\end{equation}
At the beginning of the interval assume, for each $i$,
\[
 \begin{gathered}
 I_i(0)+\mathcal D_i(0)\leq C_{\rm init}\eta^2,\qquad
 \int|y-b_i(0)|^{24}\mu_i(0,y)\idiff y
        \leq C_{\rm init}[\eta(1+s_{\rm pre})]^8,\\
 \operatorname{supp}\mu_i(0)\subset
 B\bigl(b_i(0),C_{\rm init}\eta^{2/9}(1+s_{\rm pre})^{2/5}\bigr).
 \end{gathered}
\]
Then, for every $s\in[0,S]$,
\begin{equation}\label{eq:concentration}
 \begin{gathered}
 I_i(s)+\mathcal D_i(s)\leq C\eta^2,\qquad
 \int|y-b_i(s)|^{24}\mu_i(s,y)\idiff y
             \leq C[\eta(1+s_{\rm pre}+s)]^8,\\
 \operatorname{supp}\mu_i(s)\subset
 B\bigl(b_i(s),C\eta^{2/9}(1+s_{\rm pre}+s)^{2/5}\bigr)
 \subset B(b_i(s),C2^{2/5}\eta^{1/45})\subset B_r(b_i(s)).
 \end{gathered}
\end{equation}
Moreover,
\begin{equation}\label{eq:concentration-centroid}
 |b_i'-V_i^{\rm pt}|\leq C\eta^2(1+H).
\end{equation}
The constants do not depend on $S$, $s_{\rm pre}$ or $\eta$.
\end{lemma}

The deficit $\mathcal D_i$ is nonnegative by rearrangement. On the
stopped support balls, the proof gives
\[
 \sum_i I_i\leq C\eta^2+o_\eta(1)\sum_i\mathcal D_i.
\]
A corrected energy then absorbs the second term. The remaining moment
and support bounds are proved with level sets of the frozen stream,
whose normal velocity vanishes. This removes the cumulative effect of
many rigid rotations. When the lemma is restarted, $s_{\rm pre}$ records
all preceding time measured in the same units $d^2/\Gamma$.

\begin{lemma}[Concentration and separation of a neutral pair]
\label{lem:neutral-pair}\label{lem:neutral-pair-control}
\label{lem:neutral pair control}
Fix $0<\ell_-\leq\ell_+<\infty$, $0<r\leq\ell_-/32$ and
$A_0,C_{\rm init}\geq1$. There are
$\delta_0>0$, $C\geq1$ and $0<\eta_0<1$, depending only on these
numbers, with the following property. Let
$0<\eta<\eta_0$, $0<\delta\leq\delta_0$, and
$s_{\rm pre},S\geq0$ satisfy $s_{\rm pre}+S\leq\eta^{-1/2}$.
Let $\mu_+,\mu_-$ be unit-mass rearrangements of
$f_\eta=(\pi\eta^2)^{-1}\ind_{B_\eta}$, transported on $[0,S]$ by
the same area-preserving bounded-vorticity flow with velocity
\[
 u=K*(\mu_+-\mu_-)+V.
\]
Write
\[
 \begin{gathered}
 b_\pm=\int y\mu_\pm(y)\idiff y,\qquad
 \boldsymbol d=b_+-b_-,\qquad
 I_\pm=\int|y-b_\pm|^2\mu_\pm(y)\idiff y,\\
 \mathcal D_\pm=-\iint\log|x-y|f_\eta(x)f_\eta(y)\idiff x\idiff y
             +\iint\log|x-y|\mu_\pm(x)\mu_\pm(y)\idiff x\idiff y.
 \end{gathered}
\]
Assume initially
\[
 \begin{gathered}
 \ell_-\leq|\boldsymbol d(0)|\leq\ell_+,\qquad
 I_\pm(0)+\mathcal D_\pm(0)\leq C_{\rm init}\eta^2,\\
 \int|y-b_\pm(0)|^{24}\mu_\pm(0,y)\idiff y
                       \leq C_{\rm init}[\eta(1+s_{\rm pre})]^8,\\
 \operatorname{supp}\mu_\pm(0)\subset
 B\bigl(b_\pm(0),C_{\rm init}\eta^{2/9}(1+s_{\rm pre})^{2/5}\bigr).
 \end{gathered}
\]
On $\mathcal N(s)=B_{2r}(b_+(s))\cup B_{2r}(b_-(s))$ the field $V$
is smooth, divergence-free and curl-free on an open neighborhood of the
closure. For $\boldsymbol d\ne0$ set
\[
 \begin{gathered}
 V_\pm^{\rm pt}=-K(\boldsymbol d)+V(s,b_\pm),\\
 A=\max\left\{\|\nabla V\|_{L^\infty(\mathcal N(s))},
       \frac{|V(s,b_+)-V(s,b_-)|}{|\boldsymbol d|}\right\},\qquad
 H=\|\nabla^2V\|_{L^\infty(\mathcal N(s))},\\
 B=\max_{i=\pm}
 \|(\partial_s+V_i^{\rm pt}\cdot\nabla)\nabla V
                                      \|_{L^\infty(\mathcal N(s))}.
 \end{gathered}
\]
Assume
\begin{equation}\label{eq:neutral-pair-cost}
 \sup_{0\leq s\leq S}A(s)\leq A_0,\qquad
 \int_0^S(A+B+H+H^2)\idiff s\leq\delta.
\end{equation}
These inequalities may be checked on the interval preceding the first
failure of $\ell_-/2<|\boldsymbol d|<2\ell_+$.
Then, throughout $[0,S]$,
\begin{equation}\label{eq:neutral-pair-concentration}
 \begin{gathered}
 I_\pm+\mathcal D_\pm\leq C\eta^2,\qquad
 \int|y-b_\pm(s)|^{24}\mu_\pm(s,y)\idiff y
                         \leq C[\eta(1+s_{\rm pre}+s)]^8,\\
 \operatorname{supp}\mu_\pm(s)\subset
 B\bigl(b_\pm(s),C\eta^{2/9}(1+s_{\rm pre}+s)^{2/5}\bigr)
 \subset B_r(b_\pm(s)),\\
 |\boldsymbol d(s)-\boldsymbol d(0)|
             \leq C(\ell_++\eta^2)\delta\leq\ell_-/4,\qquad
 \ell_-/2<|\boldsymbol d(s)|<2\ell_+,\\
 |b_\pm'+K(\boldsymbol d)-V(s,b_\pm)|\leq C\eta^2(1+H).
 \end{gathered}
\end{equation}
Here $\delta_0$ and $\eta_0$ are chosen so that
$C(\ell_++\eta_0^2)\delta_0\leq\ell_-/4$ and the displayed support
radius is less than $r$.
\end{lemma}

The extra quotient in $A$ is essential: bounds on disjoint centroid
balls do not control the difference of the values of $V$ at their
centers. The mutual velocities of a neutral pair agree exactly, even
for deformed patches. More precisely, the same vector
\[
 \delta v=-\iint K(x-y)\mu_+(x)\mu_-(y)\idiff x\idiff y
                                             +K(\boldsymbol d)
\]
occurs in both centroid equations. Consequently
\begin{equation}\label{eq:neutral-pair-vector-equation}
 \boldsymbol d'=V(b_+)-V(b_-)+\delta V_+-\delta V_-,\qquad
 \delta V_\pm=\int V(y)\mu_\pm(y)\idiff y-V(b_\pm),\qquad
 |\delta V_\pm|\leq\tfrac12H I_\pm.
\end{equation}
Thus a small integrated exterior cost preserves the pair length without
requiring a bound on the total translation of its midpoint. After the
four-component configuration separates, the companion neutral pair is
included in $V$ for this application.

In physical variables the centroid conclusions of the two concentration
lemmas take the common form
\begin{equation}\label{eq:physical-centroid-error}
 \left|\frac{\diff b_i^{\rm phys}}{\diff t}
  -\sum_{j\ne i}\epsilon_j\Gamma K(b_i^{\rm phys}-b_j^{\rm phys})
  -V^{\rm phys}(t,b_i^{\rm phys})\right|
 \leq C\frac\Gamma d\eta^2
       +Ca^2\|\nabla^2V^{\rm phys}\|,\qquad a=d\eta.
\end{equation}
The Hessian norm is taken on the corresponding physical balls of radius
$2rd$. In the rotating reference configuration the appropriate angular
velocity is $\Omega=-1/\mathsf A$, in the same normalized time $s$.

\subsection{A proportional comparison of two transported disks}
\label{subsec:proportional-comparison}

For a single component, include every other component of the same
configuration in its exterior field. A physical patch of circulation
$\epsilon\Gamma$, initial radius $a=\eta d$ and centroid $db(s)$ has
exterior normalization
\begin{equation}\label{eq:packet-exterior-scaling}
 s=t/\tau,\qquad \tau=d^2/\Gamma,\qquad
 V_{\rm ext}(s,y)=\frac d\Gamma V^{\rm phys}(\tau s,dy).
\end{equation}
The core time is $s/\eta^2=\Gamma t/a^2$. Thus a fixed interval in $s$
contains $O(\eta^{-2})$ core rotations. A direct estimate integrating the
size of the self gradient would not be uniform in $\eta$.

Use fixed spatial axes and core coordinates
$\xi=(x-db(s))/a$. Describe the boundary by its squared radial graph,
\begin{equation}\label{eq:centered-analytic-graph}
 Z_{\mathfrak h}(\theta)=\sqrt{1+2\mathfrak h(\theta)}
              (\cos\theta,\sin\theta),\qquad
 \mathfrak h_0=0,\qquad
 \int_0^{2\pi}(\cos\theta,\sin\theta)
                 (1+2\mathfrak h(\theta))^{3/2}\idiff\theta=0.
\end{equation}
The area constraint is linear in $\mathfrak h$. The centroid constraint
uniquely determines its modes $\mathfrak h_{\pm1}$ as quadratic analytic
functions of $g=\Pi_{\geq2}\mathfrak h$ near zero. For these higher modes
write
\begin{equation}\label{eq:analytic-wiener-norm}
 \|g\|_{\varrho,k}
  =\sum_{|m|\geq2}(1+|m|)^k e^{\varrho|m|}|g_m|,
 \qquad k=1,2.
\end{equation}
For a full Fourier series the same norm sums over all integers.

\begin{lemma}[Difference estimate for transported disks]
\label{lem:response}\label{lem:shape estimate}
Fix $r_{\rm ext}>0$, $M_{\rm ext}\geq1$ and $0<\varrho_*\leq1/4$.
There exist $0<\eta_0<r_{\rm ext}/16$, $T_*>0$, $\Lambda>0$ and
$C\geq1$, depending only on these three numbers, with
$\Lambda T_*\leq\varrho_*/4$, such that the following holds.
Let $T_0>0$, $0<\eta<\eta_0$ and $\epsilon\in\{-1,1\}$. On
$[0,T_0]$ let a unit-mass density $\mu$ be transported by
$\epsilon K*\mu+V_{\rm ext}$ from
\[
 \mu(0,y)=(\pi\eta^2)^{-1}\ind_{B_\eta(b(0))}(y),\qquad
 b(s)=\int y\mu(s,y)\idiff y.
\]
Assume it is the bounded-vorticity area-preserving transport and,
provisionally, that
$\operatorname{supp}\mu(s)\subset B_{r_{\rm ext}/4}(b(s))$.
The field $V_{\rm ext}$ is smooth, divergence-free and curl-free on
an open neighborhood of the closure of $B_{r_{\rm ext}}(b(s))$, and
\begin{equation}\label{eq:response-exterior-hypothesis}
 \sup_{0\leq s\leq T_0}
 \left(
 \|\nabla V_{\rm ext}\|_{L^\infty(B_{r_{\rm ext}}(b(s)))}
 +\|(\partial_s+b'(s)\cdot\nabla)\nabla V_{\rm ext}
                         \|_{L^\infty(B_{r_{\rm ext}}(b(s)))}
 \right)\leq M_{\rm ext}.
\end{equation}
Set
\[
 T=\min\{T_0,T_*\},\qquad
 \varrho_{\rm b}(s)=\varrho_*-\Lambda s,\qquad
 \varrho(s)=\varrho_*/2-\Lambda s.
\]
Then the actual patch has a centered analytic graph
\eqref{eq:centered-analytic-graph} on $[0,T]$, with
\begin{equation}\label{eq:shape-absolute}
 \|g(s)\|_{\varrho_{\rm b}(s),1}
       +\|g(s)\|_{\varrho(s),2}\leq C\eta^2,
 \qquad \|\mathfrak h(s)\|_{\varrho(s),1}\leq1/16.
\end{equation}
In particular its support lies in $B_{2\eta}(b(s))$, improving the
provisional support condition.

For two such histories with the same $\eta$, sign $\epsilon$ and
initial centered disk, identify their centroid balls by translation with
$\mathcal B=B_{r_{\rm ext}}(0)$ and define
\[
 \begin{aligned}
 \widetilde V^{[j]}(s,y)
 &=V_{\rm ext}^{[j]}(s,b^{[j]}(s)+y)
                         -V_{\rm ext}^{[j]}(s,b^{[j]}(s)),\\
 \Delta_A(s)
 &=\|\nabla\widetilde V^{[1]}(s,\cdot)
                -\nabla\widetilde V^{[2]}(s,\cdot)\|_{L^\infty(\mathcal B)},\\
 \mathcal E_1(s)
 &=\|g^{[1]}(s)-g^{[2]}(s)\|_{\varrho(s),1}.
 \end{aligned}
\]
Their graph difference satisfies
\begin{equation}\label{eq:response}
 D^+\mathcal E_1(s)\leq C\mathcal E_1(s)+C\Delta_A(s),\qquad
 \mathcal E_1(s)\leq C\int_0^s\Delta_A(v)\idiff v
 \quad(0\leq s\leq T).
\end{equation}
Here $D^+$ is the upper right Dini derivative. After translating the two
physical patches to a common centroid $c_{\rm phys}$, their self-velocity
difference obeys, for every integer $\nu\geq0$,
\begin{equation}\label{eq:remotegraph}
 |\nabla^\nu\delta u(x)|
 \leq C_\nu\Gamma a^2|x-c_{\rm phys}|^{-\nu-3}\mathcal E_1(s),
 \qquad |x-c_{\rm phys}|\geq4a,\qquad a=d\eta.
\end{equation}
For the normalized exterior averaging errors
\[
 \delta_{\rm obs}^{[j]}(s)
   =\int V_{\rm ext}^{[j]}(s,y)\mu^{[j]}(s,y)\idiff y
                    -V_{\rm ext}^{[j]}(s,b^{[j]}(s)),
\]
one has
\begin{equation}\label{eq:observer-response}
 |\delta_{\rm obs}^{[1]}-\delta_{\rm obs}^{[2]}|
 \leq C\eta^2\left(M_{\rm ext}\mathcal E_1+\eta^2\Delta_A\right).
\end{equation}
The constants are independent of $\eta,d,\Gamma$; $C_\nu$ may also
depend on $\nu$. No derivative of $\Delta_A$ is assumed.
\end{lemma}

The separate bound \eqref{eq:shape-absolute} uses the nonzero circular
frequencies of all modes $|m|\geq2$ and one integration by parts in
core time. The uniform bound on the material derivative in
\eqref{eq:response-exterior-hypothesis} controls the resulting forcing
variation. After that absolute bound is established, one subtracts the
two original graph equations. Their common fast circular part is
skew on Fourier modes; a decreasing analytic radius absorbs the one
angular derivative lost by the nonlinear part. This proves the
proportional estimate \eqref{eq:response} using only $\Delta_A$.
The fixed strip gap gives the second norm in
\eqref{eq:shape-absolute}.

Estimate \eqref{eq:remotegraph} is a physical source-field bound, whereas
\eqref{eq:observer-response} is in normalized velocity units and acquires
a factor $\Gamma/d$ when returned to physical coordinates. Centroid
translations are estimated separately. Harmonicity is used here twice:
it gives uniform derivative estimates on the smaller centroid balls,
and the average of the exterior over an exact centered disk equals
its value at the center. The latter cancellation explains the extra
$\eta^2$ multiplying $\Delta_A$ in \eqref{eq:observer-response}.
\section{Separation and patch bounds on a common time interval}\label{sec:geometry}\label{app:geometry}

We prove that each component at index $n$ has physical variance
$I_{n,i}\leq Ca_n^2$ and support radius at most
$C\eta_n^{1/45}d_n$, and that all four components stay within
$B(Z(t)+R_ne_1,\delta_{\rm loc}R_n/2)$. Here $Z$ is the actual
trajectory starting at zero and $\delta_{\rm loc}>0$ is fixed and small.
These estimates keep distinct indices a distance comparable to
$\max(R_n,R_k)$ apart. They are proved on the same stopped interval
as the inequality for $\Delta\mathsf U_n$ in Section~\ref{sec:release},
so the finite construction can continue the estimates together.

Write $\epsilon=(1,1,-1,-1)$. In each of the two finite evolutions,
\[
 \omega^\kappa=\sum_{n=0}^N\sum_{j=1}^4
      \epsilon_j\Gamma_n\mu_{n,j}^\kappa,
 \qquad \kappa\in\{\mathrm{quiet},\mathrm{trig}\},
 \qquad \int\mu_{n,j}^\kappa=1.
\]
Every $\mu_{n,j}^\kappa$ is the push-forward of its original constant-density
disk by the actual Euler flow. In particular, regrouping the four
components into the pairs $(1,3)$ and $(2,4)$ does not change the
velocity. Put $b_{n,j}=\int x\mu_{n,j}(x)\,\mathrm dx$,
$\boldsymbol c_n=\tfrac14\sum_jb_{n,j}$, and
$\chi_n=\tfrac12(b_{n,1}+b_{n,2}-b_{n,3}-b_{n,4})$.
The superscript $\kappa$ is suppressed when an estimate applies to
both histories. Let $Z^\kappa$ be the trajectory starting at zero in
the corresponding velocity.

A triggered configuration passes through three stages: the small
shape chart, a bounded outgoing arc, and two separated neutral pairs.
We specify the bounds in each stage on one common stopped interval,
so that concentration and growth can be improved together. The section
crossings change the applicable stage; a failure of any other bound
would stop the argument. This fixes the interval used in
Section~\ref{sec:release}.
Choose $0<\delta_{\rm loc}<1/64$. Choose a closed small chart
\[
 \mathcal C=\{|\mathsf U|\leq2h_*,\ |\mathsf V|\leq\delta_*,
                  |\mathsf A-1|\leq\delta_*,\ |\chi|/d\leq\delta_*\}
\]
on which $a\geq\lambda/2$, where $\delta_*>0$ and $h_*>0$ are
fixed sufficiently small. Its centered, normalized centroid states
have distances between different components at least $4c_*>0$ and
diameter at most $C_*$; decrease $c_*$ if necessary. The constants
here depend only on the explicit chart of Section~\ref{app:point}.
Let $Y^{\rm arc}(u)$, $0\leq u\leq T_+$, be the isolated orbit from
$\mathsf U=h_*$, extended a fixed positive time beyond its first
$\mathsf R=\mathsf R_{\rm out}$ crossing. Write $T_{\rm arc}$ for
that crossing time. Choose $\delta_{\rm arc}>0$ so that every state
within $2\delta_{\rm arc}$ of
\[
 \{\operatorname{Rot}(\theta)Y^{\rm arc}(u):
                      0\leq u\leq T_+,\ 0\leq\theta\leq2\pi\}
\]
has component distances at least $2c_*$ and diameter at most $2C_*$.
The extension and tube can also be chosen so that
$\nabla\mathsf R\cdot F\geq g_*/2>0$ when
$|\mathsf R-\mathsf R_{\rm out}|\leq\delta_{\rm arc}$ in this tube;
$F$ is the centered four-point field. These choices are possible by
the explicit isolated orbit and its transverse outgoing crossing.

For later use define the dimensionless patch quantities, with their
full elapsed age $t/\tau_n$, by
\[
 \begin{split}
 P_{n,j}(t)=\max\bigg\{\frac{I_{n,j}^{\rm phys}}{a_n^{\,2}},\;
 \frac{\mathcal D_{n,j}}{\eta_n^2},\;
 \frac{\int|x-b_{n,j}|^{24}\mu_{n,j}(x)\,\mathrm dx}
     {d_n^{24}[\eta_n(1+t/\tau_n)]^8},\;
 \frac{\sup_{x\in\operatorname{supp}\mu_{n,j}}|x-b_{n,j}|}
 {d_n\eta_n^{2/9}(1+t/\tau_n)^{2/5}}\bigg\}.
 \end{split}
\]
Here $a_n=d_n\eta_n$ is the physical core radius. The deficit is computed after dilation
by $d_n$ and is independent of the choice of center.

\begin{definition}[Common bootstrap regime]\label{def:common-bootstrap}
Let $J=[0,T]$, $T\leq1$. Fix positive constants
$C_{\rm rot},C_{\rm arc},C_{\rm pair},C_V,C_A,C_\chi,C_W,C_E,C_\vartheta$
and fixed $0<c_w<C_w$. The regime $\mathcal B(J)$ consists of the
following conditions, imposed for every $0\leq n\leq N$, both
$\kappa\in\{\mathrm{quiet},\mathrm{trig}\}$, and every current time
in the indicated stage. Section times are introduced only when the
corresponding first crossing has actually occurred.
The finest triggered configuration is the prescribed initial outgoing
state in the finite construction: set $t_N^{\rm out}=0$ for that
configuration and apply the pair regime from physical time zero.
It has no preceding chart or active-comparison interval.
\begin{enumerate}
\item In every stage,
\[
 \operatorname{supp}\mu_{n,j}^\kappa(t)
  \subset B(Z^\kappa(t)+R_ne_1,\delta_{\rm loc}R_n),\qquad
 t/\tau_n\leq\eta_n^{-1/2}.
\]
The second bound follows from the scales. The patch bound is $P_{n,j}\leq2C_{\rm rot}$,
$2C_{\rm arc}$ or $2C_{\rm pair}$ in, respectively, the small chart,
the arc, or the separated-pair stage. In the first two stages use
$r=c_*/32$ in Lemma~\ref{lem:concentration}. In the pair stage use
$r=(64d_n)^{-1}$ times the smaller of the two entry pair lengths in
Lemma~\ref{lem:neutral-pair}. The provisional component supports lie in
$B_{rd_n}(b_{n,j})$; the coefficient norms are taken on
$\bigcup_j B_{2rd_n}(b_{n,j})$ and the connecting segments required
in the technical lemmas. These choices satisfy $r\leq\ell_-/16$
in the compact stages and $r\leq\ell_-/32$ at pair entry.
\item While a configuration is in $\mathcal C$, put
$b_n=\tau_nh^\beta+\eta_n^2$. Require
\[
 |\mathsf V_n|\leq2C_Vb_n,\quad
 |\mathsf A_n-1|\leq2C_A(h^\beta+\eta_n^2/\tau_n)t,\quad
 |\chi_n|/d_n\leq2C_\chi h^{\beta+1}R_n^{5\alpha}.
\]
For a quiet configuration also require
$|\mathsf U_n^{\rm quiet}|\leq B_{\rm ref}b_n$.
For a triggered configuration awaiting its adjacent source crossing,
require membership in $\mathcal C$. On an active step $n+1\to n$,
require $|\Delta\mathsf U_n|\leq h_*$, until its positive crossing.
The time integral of the rotation defect from physical time zero
over each current chart interval $[0,t]$ is at most $2C_W$.
\item Between the positive crossing $t_n^*$ and the outgoing section,
the centered normalized state remains in the $2\delta_{\rm arc}$
tube just defined, with matched entry rotation and
$0\leq(t-t_n^*)/\tau_n\leq T_+$. On that interval
$|\chi_n|/d_n\leq2C_\chi h^{\beta+1}R_n^{5\alpha}$; changing
the fixed constant accommodates the bounded arc diameter.
\item After $t_n^{\rm out}$, set
\[
 \begin{gathered}
 d_{n,j}^{\rm pair}=b_{n,j}-b_{n,j+2},\qquad
 m_{n,j}^{\rm pair}=\tfrac12(b_{n,j}+b_{n,j+2}),\\
 y_n=m_{n,1}^{\rm pair}-m_{n,2}^{\rm pair},\qquad
 v_n=\Gamma_n/d_n,\\
 L_n(t)=K_{\rm out}d_n+v_n(t-t_n^{\rm out}).
 \end{gathered}
\]
Let $e_n$ be the unit separation direction of the matched isolated
outgoing state. Fix $c_0,C_0>0$ and $\epsilon_0>0$ from that state
so that its entry separation and relative pair velocity have
$e_n$ projections at least $4c_0K_{\rm out}$ and $4c_0$,
respectively; the entry separation divided by $K_{\rm out}$ and
the relative velocity have norms at most $C_0/4$. Require
\[
 \begin{gathered}
 c_0L_n(t)\leq e_n\cdot y_n(t),\quad |y_n(t)|\leq C_0L_n(t),\\
 |e_n^\perp\cdot y_n(t)|\leq\epsilon_0 e_n\cdot y_n(t),\\
 |d_{n,j}^{\rm pair}(t)-d_{n,j}^{\rm pair}(t_n^{\rm out})|
                \leq\epsilon_0d_n,\\
 |\chi_n|/d_n\leq2C_\chi h^{2\beta-1}R_n^{4\alpha}.
 \end{gathered}
\]
Choose $K_{\rm out}$ so large that the isolated entry and velocity
directions satisfy the same cone with $\epsilon_0/4$, and choose
$\epsilon_0$ smaller than one eighth of both entry pair lengths
divided by $d_n$. Thus the pair lengths are explicitly bounded
above and below by their entry lengths plus or minus $\epsilon_0d_n$.
\item On an active comparison interval put
$s=(t-t_{n+1}^{\rm out})/\tau_n$ and
$\mathfrak w(s)=\exp(-\int_0^s\lambda_{\rm eff})$ as in
\eqref{eq:trigger}. With the norms defined in
Section~\ref{sec:release}, require
\[
 \begin{gathered}
 E_U^*\leq2C_w,\quad |\xi|\leq1,\quad
 E_{\rm aux}^*\leq2C_E(\varepsilon_n+\eta_n^2)(1+E_U^*),\\
 E_\vartheta^*\leq2C_\vartheta(\varepsilon_n+\eta_n^2)(1+E_U^*),\\
 \mathfrak w(s)\Delta\mathsf U_n(s)\geq\tfrac12c_w\varepsilon_n
          \quad(s\geq s_{\rm ent}).
 \end{gathered}
\]
The lower bound is imposed only after its initialization on the
fixed interval $[0,s_{\rm ent}]$ in
Proposition~\ref{prop:weighted-crossing}. Before that time only the
upper and auxiliary bounds are imposed. At the larger indices
$k<n$ the centered analytic graphs are required to satisfy
$\|g_{k,j}(t)\|_{\varrho_{\rm b}(t/\tau_k),1}\leq2C_g\eta_k^2$ on the larger strip
of Lemma~\ref{lem:response}; the compared graphs use its smaller
strip. Here $C_g$ is that lemma's absolute conclusion constant.
The event deadline is $s\leq S=C_{\rm sched}(1+\ell_n)$, and the
elapsed-time assumption is $s_0\leq CL(n+2)^2r_n$.
\end{enumerate}
\end{definition}

The quiet unstable faces are intentional exit faces for the shooting
argument, and $\Delta\mathsf U_n=h_*$ and
$\mathsf R_n=\mathsf R_{\rm out}$ are intentional changes of stage.
Every other adjustable bound in $\mathcal B$ has a strict margin.
The weighted lower bound improves from $c_w/2$ to $c_w$.
On each application of a concentration lemma starting at $t_{\rm in}$,
we retain $s_{\rm pre}=t_{\rm in}/\tau_n$, so that
$s_{\rm pre}+(t-t_{\rm in})/\tau_n=t/\tau_n$.

\subsection{Spatial derivatives, material derivatives and their time integrals}

\begin{lemma}[Exterior estimates in every stage]\label{lem:exterior-gradient-controls}
Fix $J=[0,T]$, $T\leq1$, satisfying
Definition~\ref{def:common-bootstrap}. Let $V_n^\kappa$ be the field of all
indices different from $n$ and put
 \begin{align}
 \mathcal N_n^\kappa(t)&=B(Z^\kappa(t)+R_ne_1,2\delta_{\rm loc}R_n),\\
 M_{\nu,n}^{\rm rem}
   &=\max_\kappa\sup_{0\leq t\leq T}
       \|\nabla^\nu V_n^\kappa(t)\|_{L^\infty(\mathcal N_n^\kappa(t))}
       \quad(\nu=1,2),\\
 \ell_{nk}&=\max(R_n,R_k),\qquad v_n=\Gamma_n/d_n.
 \end{align}
To cover both the four-component and pair applications, let
$\mathscr P=\{\{1,2,3,4\},\{1,3\},\{2,4\}\}$. For $P\in\mathscr P$
and $i\in P$, define the physical frozen point velocity
\[
 \begin{aligned}
 \mathcal V_{n,i}^{\kappa,P}(t)
   ={}&\Gamma_n\sum_{j\in P\setminus\{i\}}\epsilon_j
          K(b_{n,i}^\kappa(t)-b_{n,j}^\kappa(t))\\
      &+\Gamma_n\sum_{j\notin P}\epsilon_j
          (K*\mu_{n,j}^\kappa)(t,b_{n,i}^\kappa(t))
          +V_n^\kappa(t,b_{n,i}^\kappa(t)).
 \end{aligned}
\]
Thus the companion pair retains its transported patch field in a pair
application. Set
\[
 B_n^{\rm rem}=\max_{\kappa,P\in\mathscr P,i\in P}\sup_{0\leq t\leq T}
 \| (\partial_t+\mathcal V_{n,i}^{\kappa,P}(t)\cdot\nabla)
                   \nabla V_n^\kappa(t)\|_{L^\infty(\mathcal N_n^\kappa(t))}.
\]
Then
\begin{equation}\label{eq:remote-spatial}
 M_{1,n}^{\rm rem}\leq Ch^{\beta},\qquad
 M_{2,n}^{\rm rem}\leq Ch^{\beta}R_n^{\nu_\Gamma-3},
\end{equation}
and, for $n\geq1$,
\begin{equation}\label{eq:sharp-hessian}
 M_{2,n}^{\rm rem}\leq Ch^{\beta}
   \left(R_{n-1}^{\nu_\Gamma-3}+R_n^{\nu_\Gamma-3}q_n^{\nu_\Gamma}\right)
 \leq Ch^{\beta}R_{n-1}^{\nu_\Gamma-3}.
\end{equation}
The material-gradient norm satisfies
\begin{align}\label{eq:remote-material}
 B_n^{\rm rem}&\leq C\sum_{k\ne n}
 \left(\frac{\Gamma_k^2}{\ell_{nk}^4}
       +\frac{\Gamma_k(v_n+v_k)}{\ell_{nk}^3}
       +\frac{h^{\beta}\Gamma_k}{\ell_{nk}^2}\right)
 \leq Ch^{2\beta-1}.
\end{align}
The four integrated concentration costs obey
\begin{equation}\label{eq:physical-cost}
 \int_0^T\bigl(M_{1,n}^{\rm rem}+\tau_n B_n^{\rm rem}+d_nM_{2,n}^{\rm rem}
                     +\tau_nd_n^2\bigl(M_{2,n}^{\rm rem}\bigr)^2\bigr)\idiff t\leq Ch^\beta.
\end{equation}
For the exterior $V$ of a single component at index $k$, including its
other same-index components, one has for every integer $\nu\geq1$,
\begin{equation}\label{eq:analytic-canonical}
 |\nabla^{\nu}V|\leq\frac{CB_{\rm an}^{\nu} {\nu}!}{\tau_kd_k^{{\nu}-1}},\qquad
 |(\partial_t+b_{k,i}'\cdot\nabla)\nabla^{\nu}V|
        \leq\frac{CB_{\rm an}^{\nu}({\nu}+2)!}{\tau_k^2d_k^{{\nu}-1}}
\end{equation}
on the concentric centroid balls of radii $c_*d_k/8$ and $c_*d_k/4$,
or on the corresponding smaller pair balls specified in
Definition~\ref{def:common-bootstrap}. Here $B_{\rm an}$ is a fixed
kernel-derivative allowance. All constants are independent of $N$, $n$
and $T$.
\end{lemma}

\begin{proof}

Put $v_n=\Gamma_n/d_n$. Require configuration $n$ to lie in
$B(Z+R_ne_1,\delta_{\rm loc} R_n)$, with fixed small $\delta_{\rm loc}$.
This margin and $R_{n+1}\leq R_n/8$
give separation comparable to $\ell_{nk}=\max(R_n,R_k)$ between
distinct configurations, also on fixed enlargements of the disks.
In the small shape chart and on the fixed segment leading to
$\mathsf R=\mathsf R_{\rm out}$, all component separations are bounded
below by a fixed multiple of $d_n$. After that section the two neutral
pair lengths are comparable to $d_n$, and their midpoint separation is
at least a fixed multiple of $K_{\rm out}d_n+v_n(t-t_n^{\rm out})$.
Thus the velocity at one centroid from any other component of the
same configuration is $O(v_n)$.

For every fixed $b>0$, the ratio bound gives
$\sum_{k>n}R_k^b\leq C_bR_n^b$,
$\sum_{k<n}R_k^b\leq C_b$ and
$\sum_{k<n}R_k^{-b}\leq C_bR_n^{-b}$. Since
$\Gamma_k=h^{\beta}R_k^{\nu_\Gamma}$, the absolute-mass kernel sums give
\[
 \begin{split}
 \sum_{k\ne n}\frac{\Gamma_k}{\ell_{nk}^{j+1}}
 &=h^\beta\left(\sum_{k<n}R_k^{\nu_\Gamma-j-1}
       +R_n^{-j-1}\sum_{k>n}R_k^{\nu_\Gamma}\right),\qquad j=1,2.
 \end{split}
\]
For $j=1$ the outer exponent is $3\alpha>0$, so the first
sum is bounded independently of $n$; the second is at most
$CR_n^{3\alpha}$. For $j=2$ the exponent is
$\nu_\Gamma-3=-1+3\alpha<0$. Its outer sum is bounded by
$CR_{n-1}^{\nu_\Gamma-3}$ when $n\geq1$, since the last outer
radius gives the largest summand. The inner sum is bounded by
$CR_n^{\nu_\Gamma-3}q_n^{\nu_\Gamma}$.
The same estimates hold for each of the four patches, since its
total variation is $\Gamma_k$ and its distance from the observed
neighborhood is at least $c\ell_{nk}$. This proves \eqref{eq:remote-spatial}.
The neighborhood $\mathcal N_n^\kappa(t)$ contains the component balls
and the intervening Taylor segments. The norms in the statement are
taken over $[0,T]$ and both histories. Companion coefficients retain
their pointwise outgoing decay. For $n\geq1$, the Gaussian scales give the sharper
second-derivative bound \eqref{eq:sharp-hessian}.
The ratio of the inner term to the outer term is
$q_{n-1}^{\nu_\Gamma-3}q_n^{\nu_\Gamma}=\exp(-L((4\nu_\Gamma-6)n+3))\leq1$. This sharper bound
will bound the Hessian at index $n+1$ in the units of index $n$.

The segment between two indices can pass through the support at a third
index. To estimate the difference of velocities, group each third source by
whether its radius is larger than both radii, between them, or smaller
than both. Only the first group is estimated by a gradient along the
segment. More explicitly, suppose $R_n\geq R_k$ and write $x_n,x_k$
for the two observed points. If $R_j>R_n$, the whole segment stays
a distance $cR_j$ from source $j$, and
\[
 |u_j(x_n)-u_j(x_k)|\leq C\Gamma_jR_n/R_j^2.
\]
If $R_k<R_j<R_n$, the separate endpoint bounds are
$C\Gamma_j/R_n$ and $C\Gamma_j/R_j$; their sum is at most
$C\Gamma_j/R_j$. If $R_j<R_k$, the sum is at most
$C\Gamma_j/R_k$. Summing these three estimates gives
\[
 \begin{split}
 &CR_nh^\beta\sum_{R_j>R_n}R_j^{\nu_\Gamma-2}
       \leq Ch^\beta R_n,\\
 &Ch^\beta\sum_{R_k<R_j<R_n}R_j^{\nu_\Gamma-1}
       \leq Ch^\beta R_n^{\nu_\Gamma-1},\\
 &Ch^\beta R_k^{-1}\sum_{R_j<R_k}R_j^{\nu_\Gamma}
       \leq Ch^\beta R_k^{\nu_\Gamma-1}.
 \end{split}
\]
The costs are
$Ch^{\beta}\ell_{nk}$, $Ch^{\beta}\ell_{nk}^{\nu_\Gamma-1}$ and
$Ch^{\beta}\min(R_n,R_k)^{\nu_\Gamma-1}$, respectively. Hence the total
cost is $Ch^{\beta}\ell_{nk}$. No gradient is assigned to a segment
crossing a smaller configuration.

The derivative of an exterior gradient along the observed motion also
contains the source's self velocity, of size $\Gamma_k/a_k$ near its
boundary. Estimating that velocity alone would lose a factor
$\eta_k^{-1}$. Its antisymmetry removes this factor. Use the normalized
frozen point velocity $V_i^{\rm pt}$ of the observed
cluster or pair; its physical value is $v_nV_i^{\rm pt}$. For source component $i_{\rm src}$ in configuration $k$, write
$\omega_{k,i_{\rm src}}$ for its signed transported vorticity density.
Exact transport gives
\[
 (\partial_t+v_nV_i^{\rm pt}\cdot\nabla)\nabla(K*\omega_{k,i_{\rm src}})(x)
 =\int \nabla^2K(x-y)(v_nV_i^{\rm pt}-u(y))\omega_{k,i_{\rm src}}(y)\idiff y.
\]
Before estimating its self term, antisymmetrize it. The resulting double
integral, up to its harmless overall sign, is
\[
 \frac12\iint\bigl(\nabla^2K(x-y)-\nabla^2K(x-z)\bigr)
      K(y-z)\omega_{k,i_{\rm src}}(y)
                 \omega_{k,i_{\rm src}}(z)\idiff y\idiff z.
\]
Indeed swapping $y,z$ changes the sign of $K(y-z)$. The segment
between them stays inside the source neighborhood, away from $x$;
the mean-value formula gives
$|\nabla^2K(x-y)-\nabla^2K(x-z)|
\leq C\ell_{nk}^{-4}|y-z|$. Multiplication by
$|K(y-z)|\leq C|y-z|^{-1}$ removes the singularity, and integration
costs $C\Gamma_k^2\ell_{nk}^{-4}$. In particular, the density
$\Gamma_k/a_k^2$ and the boundary self speed never enter this bound.
The remaining speed difference is bounded by
$C(v_n+v_k+h^{\beta}\ell_{nk})$, by the preceding endpoint split.
Taking the supremum over $[0,T]$, the observed neighborhoods, both
histories and the observed components' frozen velocities gives the
first inequality in \eqref{eq:remote-material}.
For clarity, $v_k=h^{\beta-1}R_k^{1+\alpha}$. The source-speed
sum has outer exponent $\nu_\Gamma+1+\alpha-3=4\alpha$
and inner bound $CR_n^{4\alpha}$. For the observer-speed sum use
the Hessian estimate before multiplying by $v_n$:
\[
 v_n\sum_{k\ne n}\Gamma_k\ell_{nk}^{-3}
 \leq Ch^{2\beta-1}R_n^{1+\alpha+\nu_\Gamma-3}
 =Ch^{2\beta-1}R_n^{4\alpha}.
\]
The other two sums have outer exponents
$2\nu_\Gamma-4=6\alpha$ and $\nu_\Gamma-2=3\alpha$,
and inner bounds with these same positive powers. Their prefactors
$h^{2\beta}$ are smaller than $h^{2\beta-1}$. This proves the
last inequality uniformly in the number of sources. This argument is valid in every phase and counts every source
component.

The concentration estimate uses the integral of a gradient, its material
derivative, a Hessian and its square. In physical time, their sum for
the field of other indices is the integral on the left-hand side of
\eqref{eq:physical-cost}.
Indeed \eqref{eq:material-exterior-scaling} multiplies the first and second
spatial derivative norms by $\tau_n$ and $\tau_nd_n$, and the material-gradient
norm by $\tau_n^2$. Since $\diff s_{\rm mat}=\diff t/\tau_n$, these give exactly
the four terms displayed. Their integrands are bounded, respectively, by
\[
 Ch^{\beta},\quad Ch^{\beta+1}R_n^{\alpha},\quad
 Ch^{\beta+1}R_n^{5\alpha},\quad Ch^{\beta+4}R_n^{11\alpha}.
\]
Also $\tau_nM_{1,n}^{\rm rem}\leq Ch^2$ and
$\tau_nd_nM_{2,n}^{\rm rem}\leq Ch^3R_n^{6\alpha}$. Thus
\eqref{eq:physical-cost} is $O(h^{\beta})$ on every physical subinterval
of $[0,1]$. This bound remains small even when the interval has large
length in $t/\tau_n$.

For the analytic comparison of a component at index $k$, its exterior
field includes the other three components at that same index.
Their distances are $\ell_{\rm src}\geq c d_k$, also after the two
neutral pairs separate.
Their nonself speed difference is bounded by
$C(v_k+h^{\beta}\ell_{\rm src})$. The latter contribution to a derivative of
order ${\nu}$ is bounded by $C\Gamma_kh^{\beta}\ell_{\rm src}^{-{\nu}-1}$, which is
at most $C\Gamma_k^2d_k^{-{\nu}-3}$ because
$\tau_kh^{\beta}\leq Ch^2$. Remote sources are farther away and are
summed by the preceding bounds.
Fix an analytic kernel allowance $B_{\rm an}>0$ so that the homogeneous derivatives satisfy
$|\nabla^{\nu}K(x)|\leq CB_{\rm an}^{\nu} {\nu}!|x|^{-{\nu}-1}$. Centered self averages vanish;
the differentiated self-source kernel again gains one spatial
difference. Therefore, at the actual core centroid $b_{k,i}$, with $V$ excluding
that core, one obtains \eqref{eq:analytic-canonical} for every $\nu\geq1$.
The centroid speed contains no fast self speed. Replacing the frozen
velocity by its actual centroid average adds only the Taylor error
controlled in \eqref{eq:centroid-taylor}. Equations
\eqref{eq:remote-spatial}--\eqref{eq:remote-material} bound its
remote contribution. For a derivative of order $\nu$, the same
antisymmetrization uses
$\nabla^{\nu+1}K(x-y)-\nabla^{\nu+1}K(x-z)$ and costs
$CB_{\rm an}^{\nu}(\nu+2)!\ell_{kj}^{-\nu-3}\Gamma_j^2$.
Every additional derivative beyond the gradient costs at most
$C/R_k$ for remote sources and $C/d_k$ for the other components
at index $k$. In particular, the remote estimates are bounded by
\[
 CB_{\rm an}^{\nu}(\nu+2)!
   \left(h^\beta R_k^{1-\nu}
              +h^{2\beta-1}R_k^{1-\nu}\right)
\]
for the respective spatial and material sums. To compare with
\eqref{eq:analytic-canonical}, multiply the first by
$\tau_kd_k^{\nu-1}$ and the second by
$\tau_k^2d_k^{\nu-1}$. The resulting scalar factors are bounded by
$Ch^2(d_k/R_k)^{\nu-1}$ and
$Ch^3R_k^{2\alpha}(d_k/R_k)^{\nu-1}$, respectively.
Both are uniformly bounded. The same-index sums give precisely
$\Gamma_kd_k^{-\nu-1}$ and $\Gamma_k^2d_k^{-\nu-3}$, which
equal the two dimensional factors in \eqref{eq:analytic-canonical}.
Under \eqref{eq:packet-exterior-scaling}, the gradient gains $\tau_k$
and its derivative along the actual centroid gains $\tau_k^2$.
The same estimates hold on two fixed concentric neighborhoods of the
observed centroid with a positive gap, since all sources are separated
from their closures. Thus \eqref{eq:analytic-canonical} verifies
the analytic hypotheses of Section~\ref{app:response}, after the
normalization by $\tau_k$, with the same fixed analytic allowances.
\end{proof}

\subsection{The integral of the rotation defect in the small shape chart}

\begin{proposition}[Concentration before the positive crossing]
\label{prop:geometry-concentration}
Let $[0,T]$ be a current stopped interval on which index $n$ remains
in the small chart and the common bootstrap conditions hold. For the
quiet history, stop also at its first unstable shooting face; for an
active triggered index, stop at its first positive crossing. Then, for
$0\leq t\leq T$,
\[
 \begin{gathered}
 P_{n,j}(t)\leq C_{\rm rot},\qquad
 |\mathsf V_n(t)|\leq C_V(\tau_nh^\beta+\eta_n^2),\\
 |\mathsf A_n(t)-1|\leq C_A(h^\beta+\eta_n^2/\tau_n)t.
 \end{gathered}
\]
The midpoint offset obeys
\begin{equation}\label{eq:compact-impulse}
 |\chi_n|\leq C_\chi h^{\beta+2}R_n^{1+7\alpha},\quad
 |\chi_n|/d_n\leq C_\chi h^{\beta+1}R_n^{5\alpha},\quad
 \tau_n^{-1}(|\chi_n|/d_n)^2\leq C_\chi h^{3\beta}R_n^{9\alpha}.
\end{equation}
Use normalized time $s_{\rm mat}=t/\tau_n$, angular velocity
$\Omega=-1/\mathsf A_n$ and the actual mean-centroid translation in
the frame of Lemma~\ref{lem:concentration}. Its rotation defect
$W_{\rm def}=\max_i|W_i|$ satisfies
\[
 \int_0^{T/\tau_n}W_{\rm def}(s_{\rm mat})\,\diff s_{\rm mat}\leq C_W.
\]
For the active triggered index, these bounds hold on the same stopped
interval as Proposition~\ref{prop:weighted-crossing}.
\end{proposition}

\begin{proof}

In a compact configuration, internal signed impulse cancels exactly. Its center
offset therefore obeys
the following equation. With the centroid convention in
\eqref{eq:chart}, the signed impulse is
\[
 I_n^{\rm imp}=\Gamma_n(b_{n,1}+b_{n,2}-b_{n,3}-b_{n,4})
             =2\Gamma_n\chi_n.
\]
In its derivative, each ordered internal pair cancels the reversed
pair because $K(x-y)=-K(y-x)$. For the remote field $V_n$, expand
about the mean centroid $\boldsymbol c_n$:
\[
 \sum_{i=1}^4\epsilon_i\int V_n(x)\idiff\mu_{n,i}
 =2\nabla V_n(\boldsymbol c_n)\chi_n
  +\sum_{i=1}^4\epsilon_i\int_0^1(1-u)\int
     \nabla^2V_n(\boldsymbol c_n+u(x-\boldsymbol c_n))
          [x-\boldsymbol c_n,x-\boldsymbol c_n]\idiff\mu_{n,i}\idiff u.
\]
The constant term is zero by neutrality, and
$\int|x-\boldsymbol c_n|^2\idiff\mu_{n,i}
=|b_{n,i}-\boldsymbol c_n|^2+I_{n,i}$. Dividing by $2$ gives
\[
 |\chi_n'|\leq M_{1,n}^{\rm rem}|\chi_n|
       +CM_{2,n}^{\rm rem}\left(d_n^2+\sum_{i=1}^4I_{n,i}\right),\qquad\chi_n(0)=0.
\]
Here $I_{n,i}=d_n^2I_i\leq Cd_n^2\eta_n^2=Ca_n^2$ is the physical variance.
On these preliminary bounds and physical duration at most one,
the integrating factor gives
$|\chi_n(t)|\leq Ce^{Ch^\beta t}tM_{2,n}^{\rm rem}d_n^2$.
Here $d_n^2M_{2,n}^{\rm rem}
\leq Ch^{\beta+2}R_n^{\nu_\Gamma-3+2\rho}
=Ch^{\beta+2}R_n^{1+7\alpha}$, giving the first two estimates in
\eqref{eq:compact-impulse}.
Squaring the second and dividing by $\tau_n$ gives its third estimate.
To pass from centroids to shape coordinates, multiply their physical
equations by $\tau_n/d_n$. Their internal point part is then
exactly the $K$-system in Section~\ref{app:point}. Subtracting their
mean removes the constant remote velocity. The remaining remote term
has size $C\tau_nM_{1,n}^{\rm rem}$, the averaging term has
size $C\eta_n^2(1+\tau_nd_nM_{2,n}^{\rm rem})$, and the offset
term has size $C|\chi_n/d_n|^2$. The derivatives of the chart map
are bounded on the fixed separated chart. Consequently the error in
each shape equation, after returning to physical time, is bounded by
\[
 C\left(M_{1,n}^{\rm rem}+\frac{\eta_n^2}{\tau_n}
                +\eta_n^2d_nM_{2,n}^{\rm rem}
                +\frac1{\tau_n}|\chi_n/d_n|^2\right)
 \leq C_{\rm pre}(h^\beta+\eta_n^2/\tau_n).
\]
This proves the following equations, with primes in physical time $t$:
\begin{align}\label{eq:actual-quiet-equations}
 \mathsf U_n'&=\tau_n^{-1}a(\mathsf U_n,\mathsf V_n,\mathsf A_n)\mathsf U_n+F_{U,n},\notag\\
 \mathsf V_n'&=-\tau_n^{-1}a(\mathsf U_n,\mathsf V_n,\mathsf A_n)
                     \mathsf V_n+F_{V,n},\notag\\
 |F_{U,n}|+|F_{V,n}|+|\mathsf A_n'|
       &\leq C_{\rm pre}(h^{\beta}+\eta_n^2/\tau_n).
\end{align}
In particular the action variation on $[0,1]$ is at most
$C_{\rm pre}(h^{\beta}+\eta_*^2h^{9\delta})$.
The stable coordinate starts at zero, so variation of constants bounds
it as follows, using $a\geq\lambda/2$:
\[
 |\mathsf V_n(t)|\leq C_{\rm pre}(h^\beta+\eta_n^2/\tau_n)
        \int_0^t e^{-\lambda(t-v)/(2\tau_n)}\idiff v
 \leq C(\tau_nh^\beta+\eta_n^2).
\]
This bound holds as long as the
configuration stays in the small shape chart.

Use $s_{\rm mat}=t/\tau_n$ and the angular velocity
$\Omega=-1/\mathsf A_n$, translating by the actual mean
centroid velocity. Let $W_i$ be the difference between the frozen point
velocity and this translation and rotation, as defined in
Lemma~\ref{lem:concentration}, and set $W_{\rm def}=\max_i|W_i|$.
The shape equations and the centroid Taylor estimate give
\begin{equation}\label{eq:actual-rotation-defect}
 W_{\rm def}\leq C\left(|\mathsf U_n|+|\mathsf V_n|+|\chi_n|/d_n
       +\tau_n M_{1,n}^{\rm rem}
       +C_{\rm pre}\eta_n^2(1+\tau_nd_nM_{2,n}^{\rm rem})\right).
\end{equation}
To see this estimate in the chosen units, at $\mathsf U=\mathsf V=0$
and $\chi=0$ the internal velocity is rotation at
$\Omega=-1/\mathsf A_n$ in $s_{\rm mat}$. Smoothness of the
point field bounds its difference from this rotation by
$C(|\mathsf U|+|\mathsf V|+|\chi|/d_n)$. The remote field minus
its mean is $O(\tau_nM_{1,n}^{\rm rem})$. Centering at the
actual mean velocity contributes the normalized
centroid averaging error
$C\eta_n^2(1+\tau_nd_nM_{2,n}^{\rm rem})$.
For a candidate reference evolution satisfying
$|\mathsf U_n|\le B_{\rm ref}(\tau_nh^\beta+\eta_n^2)$,
the integral in $s_{\rm mat}=t/\tau_n$ is at most
$C(1+B_{\rm ref})(h^\beta+\eta_*^2h^{9\delta})$.
The exterior cost and action variation have the same allowance.
The total allowance in \eqref{eq:cost} must include the instantaneous
rotation bound. Fix $\lambda_-=\lambda/2$, and let $C_{\rm fr}\geq1$
bound the chart coefficients in \eqref{eq:actual-rotation-defect}.
The entrance time $s_{\rm ent}$ and weighted constants $c_w,C_w$ are
fixed from the leading point-vortex forcing on the phase interval,
before the concentration constants; all patch-dependent terms enter
only the vanishing remainders. Set
\begin{equation}\label{eq:rotation-cost-allowance}
 C_0=10+C_{\rm fr}\left(2+2h_*s_{\rm ent}
                       +\frac{4C_w h_*}{c_w\lambda_-}\right).
\end{equation}
Indeed $|\Omega|\leq2$, the pre-entrance chart contribution is at most
$1+2h_*s_{\rm ent}$, and the reference contribution can be made at
most one by the final choice of $L$. On the provisional weighted
bounds the remaining expanding-coordinate integral is at most
$4C_wh_*/(c_w\lambda_-)$, as shown below. All other coefficient
integrals are smaller than one after that choice of $L$.
Apply the concentration theorem with this fixed $C_0$ and put the
provisional patch constants twice its conclusion constant. Next choose
$B_{\rm ref}$ so that the actual unstable faces point outward and
the stable faces inward in \eqref{eq:actual-quiet-equations}. Finally
increase $L$ so that the total cost is strictly below $C_0$. This improves
the patch, action and impulse bounds in the stated order.

For the mean centroid in the reference evolution, the internal point velocity is
$O(|\chi_n|/\tau_n)$ and its averaging error is $O(v_n\eta_n^2)$.
The endpoint split around $Z$ consequently gives
\[
 |\boldsymbol c_n'-Z'|\leq Ch^{\beta}R_n+Cv_n\eta_n^2,\qquad
 v_n\eta_n^2/R_n=\eta_*^2h^{\beta-1+2E}R_n^{11\alpha}.
\]
Integrating after division by $R_n$ also improves its enclosing-disk
bound. Thus only an unstable face can be the first exit for a candidate
reference evolution. This is the geometric input to
the shooting argument in Section~\ref{sec:finite}.

In the perturbed evolution the event interval has normalized length
of order $\ell_n$, but concentration requires the integral of
$|\mathsf U_n|$, rather than this length times a fixed shape bound.
Write $\mathsf U_n=\mathsf U_n^{\rm quiet}+\Delta\mathsf U_n$ and
use event time $s=(t-t_{n+1}^{\rm out})/\tau_n$. The provisional
weighted bounds in Definition~\ref{def:common-bootstrap}, with lower
constant $c_w/2$ and upper constant $2C_w$, give, for
$s_{\rm ent}\leq v\leq s\leq s_{\rm end}$,
\[
 0<\Delta\mathsf U_n(v)
 \leq\frac{4C_w}{c_w}\frac{\mathfrak w(s)}{\mathfrak w(v)}
          \Delta\mathsf U_n(s)
 \leq\frac{4C_w}{c_w}e^{-\lambda_-(s-v)}\Delta\mathsf U_n(s).
\]
Before the positive crossing $\Delta\mathsf U_n(s)\leq h_*$, so
\[
 \int_{s_{\rm ent}}^{s_{\rm end}}|\Delta\mathsf U_n(v)|\idiff v
       \leq\frac{4C_wh_*}{c_w\lambda_-}
 \qquad(s_{\rm end}\geq s_{\rm ent}).
\]
The reference contribution is bounded by the preceding reference
estimate. Before $t_{n+1}^{\rm out}$, \eqref{eq:tailclock} gives
$t_{n+1}^{\rm out}/\tau_n=o_L(1)$, and the bounded chart controls
that contribution. The entrance interval $[0,s_{\rm ent}]$ has fixed
length and contributes at most $2h_*s_{\rm ent}$; no positive lower
bound is imposed there. Thus, with uniform bounds $M_{\rm quiet}$
and $M_{\rm pre}$ for these separate contributions,
\[
 \int|\mathsf U_n|\idiff s
 \leq M_{\rm quiet}+M_{\rm pre}+\frac{4C_wh_*}{c_w\lambda_-}.
\]
This is the allowance in \eqref{eq:rotation-cost-allowance}. It is
independent of $\ell_n$ and uses only the provisional inequalities,
without assuming either the improved certificate or a future crossing.
Stable-coordinate and action bounds follow from
\eqref{eq:actual-quiet-equations}, which do not require an absolute
reference bound for $\mathsf U$. Equations \eqref{eq:compact-impulse} and
\eqref{eq:actual-rotation-defect} then give the bounded concentration
cost on every current perturbed interval, before its final crossing is known.

We record explicitly the remaining hypotheses of
Lemma~\ref{lem:concentration}. The transported densities have mass
one and remain rearrangements of the initial disks. The chart gives
component distance at least $2c_*d_n$ and bounded normalized
diameter. The supports lie in the stated balls by the provisional
support bound and $t/\tau_n\leq\eta_n^{-1/2}$, since
$\eta_n^{2/9}(1+t/\tau_n)^{2/5}\leq C\eta_n^{1/45}$.
The remote field has zero divergence and curl on their closures.
Its instantaneous normalized gradient is at most $Ch^2$, and its
integrated coefficients are bounded by \eqref{eq:physical-cost}.
The frame has $\Omega=-1/\mathsf A_n$, so $|\Omega|\leq2$ and
$\int|\Omega'|\leq C\int|\mathsf A_n'|$ in normalized time.
The defect integral was just bounded. At physical time zero the
variance, deficit, moment and support assumptions hold for disks;
on a restart they are inherited with their actual elapsed age.
Thus the lemma applies with fixed allowances and gives the claimed
strict patch bounds. Before activation, a larger triggered index
has $t/\tau_k=o_L(1)$ by \eqref{eq:tailclock}; integrating
\eqref{eq:actual-quiet-equations} over this short interval gives a
bounded defect integral directly for every unprocessed index.
\end{proof}

\subsection{Continuation through the separated section}

\begin{proposition}[Continuation from crossing to outgoing pairs]
\label{prop:geometry-continuation}
Suppose the active triggered index reaches
$\Delta\mathsf U_n=h_*$ at $t_n^*$ with the absolute entry bounds
of Lemma~\ref{lem:section-shape}. On a stopped arc interval
$[t_n^*,T]$, the common bootstrap assumptions imply
\[
 P_{n,j}(t)\leq C_{\rm arc},\qquad
 \operatorname{dist}\!\left(Y_n(t),
 \{\operatorname{Rot}(\theta)Y^{\rm arc}(u):
              0\leq u\leq T_+,\ 0\leq\theta\leq2\pi\}\right)
            \leq\delta_{\rm arc},
\]
where $Y_n=((b_{n,j}-\boldsymbol c_n)/d_n)_{j=1}^4$ and the distance
uses the maximum norm over the four components. If the final physical
endpoint allows continuation to $t_n^*+\tau_nT_+$, the outgoing section
is crossed transversely at some $t_n^{\rm out}<t_n^*+\tau_nT_+$, with
\[
 \frac{t_n^{\rm out}-t_n^*}{\tau_n}=T_{\rm arc}+o_L(1).
\]
On every subsequent stopped interval $[t_n^{\rm out},T]$ satisfying
the pair regime, one has
\[
 \begin{gathered}
 P_{n,j}(t)\leq C_{\rm pair},\qquad
 |d_{n,j}^{\rm pair}(t)-d_{n,j}^{\rm pair}(t_n^{\rm out})|
                       \leq\tfrac12\epsilon_0d_n,\\
 2c_0L_n(t)\leq e_n\cdot y_n(t),\qquad
 |y_n(t)|\leq\tfrac12 C_0L_n(t),\\
 |e_n^\perp\cdot y_n(t)|
             \leq\tfrac12\epsilon_0\,e_n\cdot y_n(t).
 \end{gathered}
\]
The constants and $o_L(1)$ are uniform in the stopped endpoint and
the outgoing angle.
\end{proposition}

\begin{proof}

Stop first at the boundary of the $2\delta_{\rm arc}$ tube or
at normalized time $T_+$. On this stopped interval the duration
is bounded, the component distances are at least $2c_*d_n$, and
the diameter is at most $2C_*d_n$. Set $\Omega=0$. The frozen
point speeds, after subtracting the mean, are bounded by a fixed
constant, so their defect has bounded integral. The exterior costs
are bounded by Lemma~\ref{lem:exterior-gradient-controls}. Apply
Lemma~\ref{lem:concentration} with
$s_{\rm pre}=t_n^*/\tau_n$ and the inherited variance, deficit,
moment and support data. All its hypotheses are thus verified on
the stopped tube; this gives $P_{n,j}\leq C_{\rm arc}$ without
assuming the outgoing crossing. The full centroid equation and
the quantitative transverse-section argument in
Lemma~\ref{lem:transverse-section} now improve the tube and supply
that crossing. It uses the concentration estimate on the stopped arc interval.
At $t_n^{\rm out}$ this gives
$I_{n,i}/d_n^2+\mathcal D_{n,i}=O(\eta_n^2)$, where
$\mathcal D_{n,i}$ is the logarithmic deficit of that component
after dilation by $d_n$. The higher-moment and support bounds use
the total elapsed time $t_n^{\rm out}/\tau_n$.
For a newly grouped neutral pair, the mutual logarithmic remainder is
bounded by the two variances; the exterior quadratic stream remainder
is bounded by the gradient times those variances. Its modulated energy
is therefore $O(\eta_n^2)$ as well. These are precisely the initial
bounds required by the pair version of the concentration estimate.
For the interval after $t_n^{\rm out}$ we use
$s_{\rm pre}=t_n^{\rm out}/\tau_n$ with the same transported densities.
Fix a small cone tolerance $\epsilon_0>0$ before choosing
$K_{\rm out}$ and $L$.

Let $A_{\rm comp},H_{\rm comp},B_{\rm comp}$ be the companion contributions
to the gradient, Hessian and material-gradient bounds.
Let ${\boldsymbol y_{\rm sep}}$ be the physical difference of pair midpoints, $\ell=|{\boldsymbol y_{\rm sep}}|$, and
${t_{\rm el}}=t-t_n^{\rm out}$. On the preliminary cone
$\ell\geq c(K_{\rm out}d_n+v_nt_{\rm el})$, the neutral companion gives
\[
 A_{\rm comp}\leq C\Gamma_nd_n\ell^{-3},\quad
 H_{\rm comp}\leq C\Gamma_nd_n\ell^{-4},\quad
 B_{\rm comp}\leq C\left(\Gamma_n^2\ell^{-4}
       +\Gamma_nv_n\ell^{-3}+h^{\beta}\Gamma_n\ell^{-2}\right).
\]
For $b>1$, the integral of $(K_{\rm out}d_n+v_nt_{\rm el})^{-b}$ is at
most $(K_{\rm out}d_n)^{1-b}/((b-1)v_n)$. Its contributions to
the four physical costs are, respectively,
\[
 O(K_{\rm out}^{-2}),\quad
 O(K_{\rm out}^{-3}+K_{\rm out}^{-2}
                       +\tau_nh^{\beta}/K_{\rm out}),\quad
 O(K_{\rm out}^{-3}),\quad O(K_{\rm out}^{-7}).
\]
For example, the gradient cost is
$C\Gamma_nd_n(K_{\rm out}d_n)^{-2}/v_n=CK_{\rm out}^{-2}$,
and the squared-Hessian cost is
\[
 C\tau_nd_n^2\Gamma_n^2d_n^2
       \frac{(K_{\rm out}d_n)^{-7}}{v_n}
 =CK_{\rm out}^{-7},
 \qquad \tau_n=d_n^2/\Gamma_n,\quad v_n=\Gamma_n/d_n.
\]
The three terms in the material-gradient cost similarly give
$\tau_n\Gamma_n^2(K_{\rm out}d_n)^{-3}/v_n$,
$\tau_n\Gamma_n(K_{\rm out}d_n)^{-2}$, and
$\tau_nh^\beta\Gamma_n(K_{\rm out}d_n)^{-1}/v_n$.
They equal, respectively, the three terms on its displayed line.
Each term tends to zero as $K_{\rm out}$ and then $L$ increase.
Together with \eqref{eq:physical-cost}, they meet the small integral
assumption for a neutral pair. Its concentration estimate, with the
elapsed time just specified, therefore holds on the current cone.

Let $C_{\rm conc}$ be the fixed concentration conclusion bound just obtained.
Write $\boldsymbol d_{\ell_{\rm pair}}$ for either positive-minus-negative pair vector. Its mutual
centroid error cancels, so its exact equation gives
\[
 \frac{|\boldsymbol d_{\ell_{\rm pair}}(t)-\boldsymbol d_{\ell_{\rm pair}}(t_n^{\rm out})|}{d_n}
 \leq C\left(K_{\rm out}^{-2}+h^{\beta}
       +C_{\rm conc}\eta_n^2(K_{\rm out}^{-3}+h^{\beta+1})\right).
\]
Indeed, its derivative is the difference of the exterior velocities
averaged over the positive and negative components. On the segment
joining their centroids the companion gradient is at most
$C\Gamma_nd_n\ell^{-3}$ and the remote gradient is at most
$Ch^\beta$. Their Taylor averaging errors are bounded by
$Ca_n^2(H_{\rm comp}+M_{2,n}^{\rm rem})$.
Divide the equation by $d_n$, insert the preliminary upper bound
$|\boldsymbol d_{\ell_{\rm pair}}|\leq Cd_n$, and integrate.
The companion terms cost $CK_{\rm out}^{-2}$ and
$CC_{\rm conc}\eta_n^2K_{\rm out}^{-3}$; the remote terms cost
$Ch^\beta$ and
$CC_{\rm conc}\eta_n^2d_nM_{2,n}^{\rm rem}$.
The latter is at most $CC_{\rm conc}\eta_n^2h^{\beta+1}$.
This preserves the pair-length and direction margins. The self-interaction
contribution to the pair midpoint velocity is
$-\Gamma_nK(\boldsymbol d_{\ell_{\rm pair}})+O(C_{\rm conc}v_n\eta_n^2)$.
The remaining contribution is the arithmetic mean of the companion and
remote fields at the two centroids, up to
$CC_{\rm conc}(H_{\rm comp}+M_{2,n}^{\rm rem})a_n^2$.
The companion mean has size at most $Cv_nK_{\rm out}^{-2}$.
Cross-pair centroid distances are bounded by
$|\boldsymbol y_{\rm sep}|+Cd_n\leq C|\boldsymbol y_{\rm sep}|$ on the cone,
so the difference of the two remote means is at most
$Ch^{\beta}|\boldsymbol y_{\rm sep}|$. Their variance error satisfies
$M_{2,n}^{\rm rem}a_n^2/v_n\leq C\eta_n^2h^3R_n^{6\alpha}$.
If $w_0=-K(\boldsymbol d_1(t_n^{\rm out})/d_n)+K(\boldsymbol d_2(t_n^{\rm out})/d_n)$, it follows that
\[
 |{\boldsymbol y_{\rm sep}}'-v_nw_0|\leq\epsilon_0v_n+Ch^{\beta}|{\boldsymbol y_{\rm sep}}|.
\]
First Gronwall gives $|{\boldsymbol y_{\rm sep}}|\leq C(K_{\rm out}d_n+v_nt_{\rm el})$.
Then $h^{\beta}|{\boldsymbol y_{\rm sep}}|/v_n\leq C(K_{\rm out}h^2+h^{\beta})$.
Projecting onto the entry direction, whose dot product with $w_0$ has
a fixed positive lower bound, strictly improves the lower outgoing cone.
More explicitly, if $e_0\cdot w_0\geq c_1>0$, first choose
$\epsilon_0<c_1/4$ and then $L$ so large that the last displayed
bound is at most $c_1/4$. It follows that
$e_0\cdot\boldsymbol y_{\rm sep}'\geq c_1v_n/2$.
Integration from the section yields the lower bound
$e_0\cdot\boldsymbol y_{\rm sep}(t)
\geq e_0\cdot\boldsymbol y_{\rm sep}(t_n^{\rm out})
+c_1v_nt_{\rm el}/2$ on this same provisional interval.
The perpendicular component stays in a fixed cone by the same estimate.
Thus concentration and the cone close simultaneously.
For the transverse cone, choose the isolated entry and velocity
directions within $\epsilon_0/4$ of $e_n$. In the estimate for
$y_n'-v_nw_0$, take the absolute error smaller than
$c_0\epsilon_0v_n/8$, by first increasing $K_{\rm out}$ and then
$L$. Projecting onto $e_n^\perp$ and integrating gives
\[
 |e_n^\perp\cdot y_n(t)|
 \leq\tfrac12\epsilon_0 e_n\cdot y_n(t).
\]
The entry margins in Definition~\ref{def:common-bootstrap} and
the same projections give $2c_0L_n\leq e_n\cdot y_n$ and
$|y_n|\leq C_0L_n/2$. The pair-vector estimate is smaller than
$\epsilon_0/2$ after the same ordered choices. In applying the
neutral-pair lemma, the exterior consists of the companion pair
and all other indices; it is smooth, curl-free and divergence-free
on the pair balls. The gradient quotient between the two
centroids is bounded by the same companion and remote estimates.
The inherited pair energy is $O(\eta_n^2)$, the pointwise gradient
is bounded, the four integrated coefficients are as small as
required, and the full age is retained. These verify each of the
hypotheses of Lemma~\ref{lem:neutral-pair-control}.
\end{proof}

\subsection{Signed impulse and spatial confinement for separated pairs}

\begin{proposition}[Impulse and confinement]\label{prop:geometry-impulse-confinement}
In the pair stage,
\[
 |\chi_n|/d_n\leq C_\chi h^{2\beta-1}R_n^{4\alpha}.
\]
In every stage $|Z'|\leq Ch^\beta$ and
\[
 \bigcup_{j=1}^4\operatorname{supp}\mu_{n,j}(t)
 \subset B(Z(t)+R_ne_1,\delta_{\rm loc}R_n/2).
\]
Thus the impulse and enclosing-disk conditions strictly improve
their bootstrap versions in both histories.
\end{proposition}

\begin{proof}

After $t_n^{\rm out}$, the signed moment has size proportional to
the distance $\ell$ between the two pair midpoints. Grouping each
neutral pair before Taylor expansion therefore gives
\[
 |\chi_n'|\leq M_{1,n}^{\rm rem}|\chi_n|
      +CM_{2,n}^{\rm rem}\left(d_n(\ell+d_n)+\sum_{i=1}^4I_{n,i}\right).
\]
Consequently $|\chi_n|/d_n\leq C_{K_{\rm out}}h^{2\beta-1}R_n^{4\alpha}$
on $[0,1]$. The earlier bound \eqref{eq:compact-impulse} is used
only before $t_n^{\rm out}$, when all four centroids are within
distance $Cd_n$ of their mean.
For this Gronwall estimate, the cone gives
$\ell\leq C(K_{\rm out}d_n+v_n(t-t_n^{\rm out}))$.
Divide the impulse inequality by $d_n$ and integrate for at most
one physical time unit. Its inhomogeneous term is bounded by
$CM_{2,n}^{\rm rem}(K_{\rm out}d_n+v_n)$.
The first product is at most
$C K_{\rm out}h^{\beta+1}R_n^{5\alpha}$ and the second is
$Ch^{2\beta-1}R_n^{4\alpha}$. Since $\beta<2$ and $R_n\leq1$,
the first is smaller than a constant times the second; the initial
compact impulse and the covariance term are smaller as well.

At $Z(t)$ every index is at distance comparable to $R_k$, so
$|Z'(t)|\leq C\sum_k\Gamma_k/R_k\leq Ch^\beta$.
Finally the self centroid average is zero in every phase. All other
same-configuration components cost $Cv_n$, and the remote endpoint split gives
\[
 |b_{n,i}'-Z'|\leq C(v_n+h^{\beta}R_n),\qquad
 |Z'|\leq Ch^{\beta}.
\]
Since $d_n/R_n\leq h$ and
$v_n/R_n=h^{\beta-1}R_n^{\alpha}\leq h^{\beta-1}$, initial offsets and
the support estimate give
\[
 \sup_{x\text{ in configuration }n}\frac{|x-Z(t)-R_ne_1|}{R_n}
 \leq C_{K_{\rm out}}\left(K_{\rm out}h+h^{\beta-1}+h^{\beta}
                                      +h\eta_*^{1/45}\right)<\delta_{\rm loc}/2.
\]
This improves precisely the disks assumed at the beginning.
\end{proof}

We have obtained $I_{n,i}\leq Ca_n^2$, the support radius
$C\eta_n^{1/45}d_n$, the enclosing-disk bound
$\delta_{\rm loc}R_n/2$, and linear lower and upper bounds for the
separation of the two neutral pairs. The rotation-defect integral is
bounded on the small chart, and the four exterior costs are integrable
on every later interval. These are the quantitative inputs used in
Sections~\ref{sec:release} and \ref{sec:finite}.

\begin{proposition}[Joint continuation of geometry and growth]
\label{prop:bootstrap-closure}
The constants in Definition~\ref{def:common-bootstrap} can be fixed,
and then $L$ chosen independently of $n,N$, so that all its
adjustable bounds have the strict improvements stated above and
in Proposition~\ref{prop:weighted-crossing}. Consequently the first
stop on an active finite step is a quiet unstable shooting face,
a prescribed stage crossing, or the final physical endpoint.
After a stage crossing the next regime starts with strict margins.
In particular, if the quiet history stays in its reference box,
the positive crossing and then the outgoing crossing occur before
their deadlines.
\end{proposition}

\begin{proof}
Choose the constants on the common stopped intervals as follows.
First fix the shape chart, the leading weighted constants $c_w,C_w$,
the entrance time $s_{\rm ent}$, and the direction margins; then choose
$K_{\rm out}$ and its fixed section. Use the provisional factor
$4C_w/c_w$ in \eqref{eq:rotation-cost-allowance} to fix the small-chart
concentration allowance. Obtain the concentration constants for the
small chart, the compact orbit segment and the two separated pairs in
that order.
Set the provisional bounds above twice their maximum and then choose
$L$. Every term still carrying a provisional constant has a positive
power of $h$ or $\eta_*$. Each component uses at most these three
intervals, with its actual elapsed time retained throughout.
Consequently the first failure of any geometric or patch bound would
contradict its strict improvement on the same interval. This is the
continuation implication needed alongside \eqref{eq:certificate}.
The event comparison constants are chosen after the geometric
allowances; its vanishing errors in \eqref{eq:full-error-ledger}
are then made small by the same final choice of $L$. During
$[0,s_{\rm ent}]$ the comparison bounds initialize the lower
weighted bound. After entrance, its provisional lower constant
$c_w/2$ yields a uniform defect allowance, while the weighted
argument improves it to $c_w$ without using that allowance beyond
the common patch conclusions. The patch estimate therefore uses
only a current weighted inequality, and the weighted argument
uses only current geometric estimates. Taking the minimum of
their stopping times gives one valid stopped interval for both.
At a finite first stop, continuity and strict improvement rule
out every accidental boundary simultaneously. At $s=S$ the
lower weighted bound would exceed $h_*$, so that deadline is
also impossible before the positive crossing. On the arc the
transverse argument rules out the tube boundary and $T_+$.
The three successive concentration applications use their own
initial constants and inherit all previous ages; the disk-comparison lemma is applied from the initial disks at time zero.
\end{proof}

\section{Growth forced by the adjacent smaller configuration}\label{sec:release}\label{app:event}

We prove that the actual Euler evolutions satisfy
$c\varepsilon_n\leq\mathfrak w\Delta\mathsf U_n\leq C\varepsilon_n$
after a fixed initial interval, where
$\Delta\mathsf U_n=\mathsf U_n^{\mathrm{trig}}-
\mathsf U_n^{\mathrm{quiet}}$ and
$\mathfrak w(s)=\exp(-\int_0^s\lambda_{\rm eff}(v)\idiff v)$
with $\lambda_{\rm eff}\geq\lambda/2$. This inequality forces
$\Delta\mathsf U_n$ to reach a fixed positive value in time
$O(1+\log\varepsilon_n^{-1})$. The forcing coefficient is the one
computed in Section~\ref{sec:cell}; this section proves that the
integrated remainder is smaller than $\varepsilon_n$.

Two cancellations are essential. The reference and perturbed disks
at indices $k\leq n$ agree initially, so the errors at $k<n$ can be
estimated proportionally to the difference of exterior fields.
Separately, each configuration has zero initial signed impulse. A
centroid error of size $\zeta$ alone would allow impulse
$\zeta\Gamma_{n+1}d_{n+1}$, whose dipole strain divided by the required
quadrupole strain is of order
$\zeta\sigma_nD_n/(\Gamma_{n+1}/\Gamma_n)$. This ratio need not be
small. We instead integrate the exact impulse equation from physical
time zero and retain the resulting dipole term.

Choose a small fixed positive exit $h_*$ in the saddle chart. Fix the
outgoing section and its arc as in Section~\ref{sec:cell}. Let
$m_*=e^{-12}$ and let $\mathcal I=[\psi_-,\psi_+]$ be the increasing
interval on which $\mathcal H$ runs from $m_*$ to $1$. Its length is
less than $1/6$. The choices below fix $\zeta_0>0$ and then one large $L$.

\begin{proposition}[Positive crossing on a common stopped interval]\label{prop:release}
Fix $n<N$ in two finite canonical Euler evolutions. Assume their data
at indices $0,\ldots,n$ agree at physical time zero, all initial
components are the prescribed disks, and every initial signed impulse
is zero. At every index $k\leq n$ assume initially, in both histories,
\[
 \mathsf V_k^\kappa(0)=0,\qquad \mathsf A_k^\kappa(0)=1,\qquad
 \chi_k^\kappa(0)=0,\qquad \vartheta_k^\kappa(0)=\vartheta_k^0.
\]
Their unstable coordinates lie in the reference box.
The reference evolution remains in that box on the interval considered.

Suppose the perturbed configuration at index $n+1$ reaches the section
$\mathsf R=\mathsf R_{\rm out}$ at
$t_{n+1}^{\rm out}$ with normalized centroid discrepancy
$0\le\zeta\le\zeta_0$ from the isolated section state with the same
outgoing direction. Writing $k=n+1$, its transported components obey
\[
 \begin{gathered}
 I_{k,j}^{\rm phys}(t)\leq Ca_k^2,\qquad
 \mathcal D_{k,j}(t)\leq C\eta_k^2,\\
 \int|x-b_{k,j}(t)|^{24}\mu_{k,j}(t,x)\idiff x
       \leq Cd_k^{24}[\eta_k(1+t/\tau_k)]^8,\\
 \operatorname{supp}\mu_{k,j}(t)\subset
 B\bigl(b_{k,j}(t),Cd_k\eta_k^{2/9}(1+t/\tau_k)^{2/5}\bigr).
 \end{gathered}
\]
In the perturbed evolution every index $k>n+1$
remains in the field and satisfies the separated-pair estimates after
its already completed crossing.
Both evolutions satisfy the common bootstrap regime
$\mathcal B$ of Definition~\ref{def:common-bootstrap}. Put
\[
 s=\frac{t-t_{n+1}^{\rm out}}{\tau_n},\qquad
 s_0=\frac{t_{n+1}^{\rm out}}{\tau_n},\qquad
 S=C_{\rm sched}(1+\ell_n),\qquad
 r_n^{\rm out}=\begin{cases}r_{n-1},&n\ge1,\\0,&n=0.\end{cases}
\]
Assume
\[
 s_0\leq CL(n+2)^2r_n,\qquad
 \frac t{\tau_k}\leq r_n^{\rm out}(s_0+S)=o_L(1)\quad(k<n).
\]
The second inequality follows from the first and the schedule bound.
Thus the proportional estimate starts from the original circular patch
and applies throughout the interval required here.

At $t_{n+1}^{\rm out}$ define the relative angle using the actual
reference phase at index $n$ and the actual line of centers:
\[
 \psi_n=\Theta_{n+1}+\vartheta_n^{\rm quiet}(t_{n+1}^{\rm out})+\pi/4
       -2\arg(\boldsymbol c_n^{\rm trig}(t_{n+1}^{\rm out})
                         -\boldsymbol c_{n+1}^{\rm trig}(t_{n+1}^{\rm out})).
\]
Here $\Theta_{n+1}$ is the outgoing-direction coordinate of the
section state at index $n+1$. If $\psi_n\in\mathcal I$ modulo $\pi$,
then on each stopped interval before the first positive crossing, the
coordinate satisfies
\begin{equation}\label{eq:trigger}
 (\Delta\mathsf U_n)'=\lambda_{\rm eff}(s)\Delta\mathsf U_n
             +\varepsilon_n f_{\sigma_n,\psi_n}(s)+\mathcal R(s),
 \qquad\lambda_{\rm eff}\ge\lambda/2,
\end{equation}
and, after one fixed entrance allowance,
\begin{equation}\label{eq:certificate}
 c\varepsilon_n\le
 \exp\!\left(-\int_0^s\lambda_{\rm eff}(v)\idiff v\right)
                    \Delta\mathsf U_n(s)
 \le C\varepsilon_n.
\end{equation}
Both weighted bounds and the auxiliary comparison bounds strictly improve
their provisional versions on the same interval as the bounds of
Section~\ref{sec:geometry}. If no other stop occurs, the crossing
$\Delta\mathsf U_n=h_*$ is forced before time $S$, with a fixed
allowance for continuation along the separated scattering arc.

At the resulting outgoing section, the normalized centroid discrepancy
from the rotated isolated state obeys
\[
 \mathsf R_n(t_n^{\rm out})=\mathsf R_{\rm out},\qquad
 \max_j\left|\frac{b_{n,j}^{\rm trig}(t_n^{\rm out})
                    -\boldsymbol c_n^{\rm trig}(t_n^{\rm out})}{d_n}
                -y_j^{\rm iso,out}(\Theta_n)\right|
       \leq\zeta_{\rm out}(L)\longrightarrow0,
\]
uniformly for $\zeta\leq\zeta_0$. Here $y_j^{\rm iso,out}(\Theta_n)$
denotes the centered normalized isolated section state with direction
$\Theta_n$. With the
absolute normalized section times $s_k^{\rm out}=t_k^{\rm out}/\tau_k$,
the time and direction satisfy
\begin{align}
 s_n^{\rm out}&=r_ns_{n+1}^{\rm out}+\ell_n/\lambda+C_{\rm time}
       -\lambda^{-1}\log\mathcal H(\psi_n)
       +O_{m_*}(\zeta)+o_L(1),\nonumber\\
 \Theta_n&=\vartheta_n^0-s_n^{\rm out}+C_{\rm out}+o_L(1).
       \nonumber
\end{align}
All constants and vanishing errors are uniform in $n,N$ and the
admitted evolutions. The higher-moment and support estimates at this
section use the full elapsed time $t_n^{\rm out}/\tau_n$.
\end{proposition}

We first estimate the signed second moment and impulse of the
configuration at index $n+1$. We then subtract the equations at indices
$k\leq n$ and integrate their coupled inequalities with the weight
$\mathfrak w$. This gives the positive crossing. Finally, absolute
estimates for $\mathsf V_n$, $\mathsf A_n-1$ and $\chi_n/d_n$
give the section shape, time and direction.
Throughout the section suppress $n$ on $\sigma_n,D_n,\varepsilon_n$ and
write $r=r_n$, $\gamma=\Gamma_{n+1}/\Gamma_n=\sigma^\beta$,
$v=\gamma/\sigma$. All estimates concern the current common stopped interval.

\subsection{The signed second moment and the impulse generated since time zero}

The two relevant normalized times are related by
\[
 s_d=\frac{t-t_{n+1}^{\rm out}}{\tau_{n+1}}=\frac{s}{r_n},
 \qquad t=\tau_n(s_0+s).
\]
We first specify the exterior coefficients used in the moment estimate.
Choose $C_{\rm evt}$ after the geometric, concentration and reference-box
constants, and before the final choice of $L$, and put
$A_n=C_{\rm evt}h^2R_n^\alpha$ and
$B_n=C_{\rm evt}h^2R_n^{4\alpha}$.
The remote estimates imply, for all $k\leq n+1$,
\[
 \tau_nM_{1,k}^{\rm rem}\leq A_n,\qquad
 \tau_nR_nM_{2,k}^{\rm rem}\leq B_n.
\]
The same constants bound the reference coordinates and the coefficients
on the whole earlier interval in both evolutions. In particular,
\begin{equation}\label{eq:event-small-coefficients}
 A_n\leq Ch^2R_n^\alpha,\qquad
 B_n\leq Ch^2R_n^{4\alpha}=Cp_n^2.
\end{equation}
For the Hessian, use the sharper estimate
\eqref{eq:sharp-hessian}: if $1\leq k\leq n+1$, then
$R_{k-1}\geq R_n$ and $\nu_\Gamma-3<0$, so
$\tau_nR_nM_{2,k}^{\rm rem}\leq Ch^2R_n^{4\alpha}$.
For $k=0$, direct summation over the inner indices gives
$M_{2,0}^{\rm rem}\leq Ch^\beta$ and the same conclusion.
Thus this comparison uses the positive majorants of the derivatives,
independently of any cancellation between signed fields.

In the coordinates and time for index $n+1$, these bounds become
\[
 \tau_{n+1}M_{1,n+1}^{\rm rem}\leq r_nA_n,\qquad
 \tau_{n+1}d_{n+1}M_{2,n+1}^{\rm rem}
                         \leq r_n\sigma_np_nB_n.
\]
Since $\sigma_np_nB_n/A_n\leq C\sigma_nhR_n^{5\alpha}$, a common
bound for these two normalized derivatives is $M_d\leq Cr_nA_n$.
Its integrated cost up to $s_d=s/r_n$ is $M_ds_d\leq CA_ns=o_L(1)$.

\begin{proposition}[The gradient generated by all smaller indices]
\label{prop:inner-gradient}
Use the assumptions and current stopped interval of
Proposition~\ref{prop:release}. Put
$V_{>n}^\kappa=K*\sum_{k>n}\sum_{j=1}^4
\epsilon_j\Gamma_k\mu_{k,j}^\kappa$ and
\[
 \begin{split}
 \mathscr G_{>n}(s)&=\tau_n\bigl[
 \nabla V_{>n}^{\rm trig}(t,\boldsymbol c_n^{\rm trig})
 -\nabla V_{>n}^{\rm quiet}(t,\boldsymbol c_n^{\rm quiet})\bigr],\\
 X(s)&=(\boldsymbol c_n^{\rm trig}(t)
                -\boldsymbol c_{n+1}^{\rm trig}(t))/d_n,\\
 \mathfrak q_{\rm deep}&=q_{n+1}^{2\nu_\Gamma},\qquad
 I_{\rm deep}=q_n^{\nu_\Gamma-3}q_{n+1}^{2\nu_\Gamma}.
 \end{split}
\]
Then, with $P_{\sigma,\Theta_{n+1}}$ the explicit signed moment
of the rotated isolated orbit from Section~\ref{sec:cell},
\begin{equation}\label{eq:inner-gradient-decomposition}
 \mathscr G_{>n}(s)=\frac{\gamma^2}{2}
       \nabla^3K(X(s)):P_{\sigma,\Theta_{n+1}}(s)
               +\mathscr R_{>n}(s).
\end{equation}
In the common bootstrap regime,
\[
 cD\leq|X(s)|\leq CD,
\]
and
\begin{equation}\label{eq:inner-gradient-remainder}
 \begin{split}
 |\mathscr R_{>n}(s)|\leq C\varepsilon_n(1+s_0+s)^6
 \bigg(\zeta+\eta_{n+1}^2+r+A_n+B_n
       +\mathfrak q_{\rm deep}+B_nI_{\rm deep}
       +\frac{1+\sigma+vs}{D}\bigg).
 \end{split}
\end{equation}
The same estimate holds at any receiver point within $C_*d_n$
of its mean, after adding the already displayed $D^{-1}$ error.
For every source configuration let
\[
 I_k^{\rm imp}(t)=\Gamma_k\sum_{i=1}^4\epsilon_i
                               \int x\mu_{k,i}(t,x)\idiff x.
\]
With $M_{2,k}^{\rm rem}$ the uniform physical exterior-Hessian
majorant on the elapsed interval, its generated impulse satisfies
\begin{equation}\label{eq:packet-impulse-event}
 |I_k^{\rm imp}(t)|\leq CM_{2,k}^{\rm rem}
                 (\Gamma_kd_k^2t+\Gamma_k^2t^2).
\end{equation}
In particular the adjacent-source dipole is included in the remainder.
\end{proposition}

Since $r_n\ll1$, midpoint errors need not remain bounded in $s_d$.
The signed second moment, however, is the product of a midpoint and
a pair vector. We prove a bounded estimate for the vector difference
and an estimate growing linearly for the midpoint difference; their
product has precisely the growth needed for the strain calculation.
Let $\boldsymbol c_{\rm ext}$ follow the actual field of all indices
other than $n+1$, starting at
$\boldsymbol c_{n+1}^{\mathrm{trig}}(t_{n+1}^{\rm out})$.
For either neutral pair at index $n+1$, let $\mu_\pm$ and $b_\pm$
denote its component probabilities and centroids after dilation by
$d_{n+1}$. In the translated coordinates define
\[
 \boldsymbol d=b_+-b_-,\qquad
 \boldsymbol m=\frac{b_++b_-}{2}-\frac{\boldsymbol c_{\rm ext}(t)}{d_{n+1}}.
\]
Let $e_{\rm vec},e_{\rm mid}$ be the maximum discrepancies of these
pair vectors and midpoints from the isolated orbit with the same outgoing direction.
Define the initial section discrepancy $\zeta$ as the maximum of the initial
pair-vector and midpoint discrepancies in these normalized units.
\begin{lemma}[Comparison of the two outgoing neutral pairs]\label{lem:donor-moment}
For the two outgoing pairs $j=1,2$, write
$\ell_j=|\boldsymbol d_j(0)|$. Suppose that in the smaller length units
\[
 \begin{gathered}
 \int|y-b_{j,\pm}|^2\mu_{j,\pm}(y)\idiff y\leq C\eta_{n+1}^2,
 \qquad 0<\ell_j-\epsilon_0\leq|\boldsymbol d_j(s_d)|
                                           \leq\ell_j+\epsilon_0,\\
 c(K_{\rm out}+s_d)\leq|\boldsymbol m_1(s_d)-\boldsymbol m_2(s_d)|
                                    \leq C(K_{\rm out}+s_d),\\
 \|\nabla V_{\rm ext}\|+\|\nabla^2V_{\rm ext}\|\leq M_d.
 \end{gathered}
\]
The exterior norms are taken on the component neighborhoods and the
connecting segments used in the pair equations.
Then, on every such stopped interval, the
upper Dini derivatives of the vector and midpoint discrepancies satisfy
\begin{align}\label{eq:donor-volterra}
 \mathrm D^+e_{\rm vec}&\leq C(K_{\rm out}+{s_d})^{-3}e_{\rm vec}
      +C(K_{\rm out}+{s_d})^{-4}e_{\rm mid}
      +C\eta_{n+1}^2(K_{\rm out}+{s_d})^{-4}
      +CM_d(1+e_{\rm vec}),\notag\\
 \mathrm D^+e_{\rm mid}&\leq Ce_{\rm vec}
      +C(K_{\rm out}+{s_d})^{-3}e_{\rm mid}+C\eta_{n+1}^2
      +CM_d(K_{\rm out}+{s_d}+e_{\rm mid}).
\end{align}
If $M_ds_d\leq c$ for a fixed sufficiently small $c$, and the
initial vector and midpoint discrepancies are at most $\zeta$, then
\begin{equation}\label{eq:donor-comparison-integrated}
 \begin{aligned}
 e_{\rm vec}(s_d)&\leq C(\zeta+\eta_{n+1}^2+M_ds_d),\\
 e_{\rm mid}(s_d)&\leq C\bigl((\zeta+\eta_{n+1}^2)(1+s_d)
                                      +M_d(s_d+s_d^2)\bigr).
 \end{aligned}
\end{equation}
\end{lemma}

\begin{proof}
The mutual averaging error cancels in the first row: exchange the two
integration variables and use the antisymmetry of $K$ to see that the
mutual centroid velocities of a neutral pair are identical. A companion pair has gradient of order
$(K_{\rm out}+{s_d})^{-3}$; moving its midpoint adds another inverse
power. Here the pair equations use the normalized kernel $K$ in the time $s_d$. For a companion with
vector $d_c$ and midpoint $m_c$, its point field at $x$ is
\[
 K(x-m_c-d_c/2)-K(x-m_c+d_c/2)
 =-\int_{-1/2}^{1/2}\nabla K(x-m_c-u d_c)d_c\idiff u.
\]
The pair lengths are bounded above and below. This identity gives
velocity, gradient and Hessian orders $(K_{\rm out}+s_d)^{-2}$,
$-3$ and $-4$. Differentiating in $d_c$ and $m_c$ gives the two
coefficients multiplying $e_{\rm vec}$ and $e_{\rm mid}$ in
\eqref{eq:donor-volterra}. For the observer averaging error use
the Hessian of this grouped field, at cost
$C\eta_{n+1}^2(K_{\rm out}+s_d)^{-4}$.
For a source averaging error, Taylor expansion of each component
costs $C\eta_{n+1}^2\nabla^2K$; subtracting its values at the
two observing centroids costs one more derivative and gives the same
order $-4$. Thus the estimate does not assume that the two observer
covariances agree. The midpoint equation retains the common mutual
averaging error $O(\eta_{n+1}^2)$ and the remote velocity difference
from the exterior trajectory, bounded by
$CM_d(K_{\rm out}+s_d+e_{\rm mid})$.

More explicitly, if $F_c$ is the companion point field, then
\[
 F_c(m+d/2)-F_c(m-d/2)
 =-\int_{-1/2}^{1/2}\int_{-1/2}^{1/2}
 \nabla^2K(m-m_c+u d-vd_c)[d,d_c]\,\mathrm du\,\mathrm dv.
\]
The derivative with respect to either pair vector is
$O((K_{\rm out}+s_d)^{-3})$, and that with respect to either
midpoint is $O((K_{\rm out}+s_d)^{-4})$. In the midpoint equation,
the mutual field $-K(d)$ has a bounded derivative because the pair
length stays away from zero; this gives the term $Ce_{\rm vec}$.
The constant $O(\eta_{n+1}^2)$ mutual averaging error survives only
in the midpoint equation.
To integrate these inequalities, write $T=s_d$, $x=e_{\rm vec}$,
$y=e_{\rm mid}$ and $a_0=\zeta+\eta_{n+1}^2$.
The integrating factor for the second row is bounded above and below
because $\int_0^T(K_{\rm out}+u)^{-3}\idiff u+M_dT\leq C$.
Consequently
\[
 y(T)\leq C\left(a_0(1+T)+\int_0^T x(u)\idiff u
                                  +M_d(T+T^2)\right).
\]
Constants may depend on the fixed $K_{\rm out}$, but not on $T$
or $r_n$. Substitution in the first row is justified by nonnegativity
of the majorants. For its iterated integral, Fubini gives
\[
 \int_0^T(K_{\rm out}+u)^{-4}\int_0^u x(v)\idiff v\idiff u
 \leq\frac13\int_0^T(K_{\rm out}+v)^{-3}x(v)\idiff v.
\]
The remaining feedback integrals obey
\[
 \int_0^{s_d}(K_{\rm out}+s)^{-3}\idiff s\leq CK_{\rm out}^{-2},
 \quad
 \int_0^{s_d}(K_{\rm out}+s)^{-4}(1+s)\idiff s
                                      \leq CK_{\rm out}^{-2},
\]
and $\int_0^T(K_{\rm out}+u)^{-4}(u+u^2)\idiff u
\leq CTK_{\rm out}^{-2}$.
The resulting scalar inequality is
\[
 x(T)\leq C(a_0+M_dT)
   +C\int_0^T\bigl((K_{\rm out}+u)^{-3}+M_d\bigr)x(u)\idiff u.
\]
Its coefficient integral is bounded uniformly when $M_dT$ is small.
Gronwall, followed by substitution into the bound for $y$, proves
\eqref{eq:donor-comparison-integrated}.
\end{proof}

\begin{proof}[Proof of Proposition~\ref{prop:inner-gradient}]
Write $T=s_d$ and $a_0=\zeta+\eta_{n+1}^2$ as in the pair comparison.
The derivative bound $M_d\leq Cr_nA_n$ gives
$M_dT\leq CA_ns=o_L(1)$, so
\eqref{eq:donor-comparison-integrated} applies on the current interval.
For each sign, put
$\operatorname{Cov}_\pm=\int(y-b_\pm)\otimes(y-b_\pm)\idiff\mu_\pm(y)$.
These are physical covariances divided by $d_{n+1}^2$, of size $O(\eta_{n+1}^2)$.
The exact signed moment about the common exterior trajectory is
\[
 \begin{split}
 &\int\left(y-\frac{\boldsymbol c_{\rm ext}(t)}{d_{n+1}}\right)
       \otimes\left(y-\frac{\boldsymbol c_{\rm ext}(t)}{d_{n+1}}\right)\idiff(\mu_+-\mu_-)(y)\\
 &\qquad=\boldsymbol m\otimes\boldsymbol d
      +\boldsymbol d\otimes\boldsymbol m
      +\operatorname{Cov}_+-\operatorname{Cov}_-.
 \end{split}
\]
Componentwise, if $a,b\in\{1,2\}$, the point part is
$(m_a+d_a/2)(m_b+d_b/2)
 -(m_a-d_a/2)(m_b-d_b/2)=m_ad_b+d_am_b$.
The centered first moments vanish, leaving the two covariance matrices.
Their operator norms are bounded by their traces, which are the
variances $C\eta_{n+1}^2$. For actual and isolated pair moments,
with superscript $0$ on the latter, we have
\[
 \begin{split}
 |m\otimes d+d\otimes m-m^0\otimes d^0-d^0\otimes m^0|
 &\leq C\bigl(e_{\rm mid}+(1+T)e_{\rm vec}
                                      +e_{\rm mid}e_{\rm vec}\bigr)\\
 &\leq C\bigl(a_0(1+T)+M_d(T+T^2)\bigr).
 \end{split}
\]
The last inequality uses $a_0+M_dT\leq C$, so the product term has
the same bound. Sum this estimate over both pairs and add their
covariance differences.

We now express the moment of index $n+1$ in the length and circulation
units of index $n$:
\[
 M_{n+1}^{\mathrm{trig}}(s)=\frac{\Gamma_{n+1}}{\Gamma_n d_n^2}
  \sum_{i=1}^4\epsilon_i\int (x-\boldsymbol c_{\rm ext}(t))\otimes
               (x-\boldsymbol c_{\rm ext}(t))\idiff\mu_{n+1,i}^{\mathrm{trig}}(t,x),
 \qquad t=t_{n+1}^{\rm out}+\tau_ns.
\]
Changing units, using
$\gamma\sigma^2=\gamma^2r$ and $\gamma\sigma(\gamma/\sigma)=\gamma^2$, yields
\begin{equation}\label{eq:material-donor}
 \begin{split}
 \frac{|M_{n+1}^{\mathrm{trig}}(s)-\gamma^2P_{\sigma,\theta_d}(s)|}{\gamma^2}
 \leq C\bigl((\zeta+\eta_{n+1}^2)(r+s)
+A_n(rs+s^2)\bigr).
 \end{split}
\end{equation}
Here $\theta_d$ is the outgoing direction of the isolated comparison
orbit at index $n+1$. The identity $\gamma\sigma^2=\gamma^2r_n$
also converts $M_d(s_d+s_d^2)$ into $A_n(r_ns+s^2)$ after
division by $\gamma^2$; no inverse factor $r_n$ remains.
Explicitly, the factor converting a moment in the smaller length units
to $M_{n+1}/\gamma^2$ is $r$, and
\[
 r(1+T)=r+s,\qquad
 r(rA_n)(T+T^2)=A_n(rs+s^2),\qquad T=s/r.
\]
This is why the growing midpoint discrepancy is compatible with an
estimate uniform in the smaller time ratio $r$.
The effective phase
$\psi=\theta_d+\theta_{\rm rec}-2\theta_{\rm line}$ appears after
contracting the kernel and rotating to the coordinates of index $n$.
Here $\theta_{\rm rec}$ is the reference positive-axis angle $\vartheta_n^{\mathrm{quiet}}(t_{n+1}^{\rm out})+\pi/4$,
and $\theta_{\rm line}$ is the line-of-centers angle. The actual reference
positive axis differs from this reference by
$\tfrac12(\phi^{\mathrm{quiet}}-\pi/2)=O(|\mathsf U_n^{\mathrm{quiet}}|+|\mathsf V_n^{\mathrm{quiet}}|)$.
Its projection error is therefore included in the existing
$A_n+\eta_n^2$ allowance for the comparison error.
The estimate just obtained already includes the section impulse in
$\zeta$, because $\zeta$ controls all four centroid coordinates.
The lower-order dipole field requires a sharper impulse estimate,
obtained by integrating its equation from physical time zero.

The physical signed impulse at index $k$ is
$I_k^{\rm imp}(t)=\sum_{i=1}^4\epsilon_i\Gamma_k\int x\idiff\mu_{k,i}(t,x)$.
Let
\[
 V_k=K*\sum_{j\ne k}\sum_{i=1}^4\epsilon_i\Gamma_{j}\mu_{j,i}
\]
be the velocity of all other indices in the same finite evolution.
The impulse has zero initial value at physical time zero and exact derivative
$\sum_{i=1}^4\epsilon_i\Gamma_k\int V_k\idiff \mu_{k,i}$.
Internal interactions cancel. Grouping each neutral pair before Taylor
expansion bounds its second signed moment by
$C(\Gamma_kd_k^2+\Gamma_k^2t)$, including both covariances. Therefore
the required remainder can be estimated without taking the absolute
second moment of the whole configuration. In fact, for a scalar
component $F$ of $V_k$, its affine-subtracted remainder
$R_F(y)=F(c+y)-F(c)-\nabla F(c)y$ satisfies
$|\nabla R_F(y)|\leq M_{2,k}^{\rm rem}|y|$.
For a pair midpoint $m$ and vector $d$, the fundamental theorem of
calculus gives
\[
 |R_F(m+d/2)-R_F(m-d/2)|
 \leq CM_{2,k}^{\rm rem}|d|(|m|+|d|).
\]
The component averaging remainders add at most
$CM_{2,k}^{\rm rem}(I_++I_-)$.
Since $|d|\leq Cd_k$ and
$|m|\leq C(d_k+v_kt)$ in the common moving frame, summing the pairs
gives $CM_{2,k}^{\rm rem}(\Gamma_kd_k^2+\Gamma_k^2t)$.
The constant term cancels by neutrality and the affine term is
$\nabla V_k(c)I_k^{\rm imp}$. Consequently
\[
 \mathrm D^+|I_k^{\rm imp}|
 \leq M_{1,k}^{\rm rem}|I_k^{\rm imp}|
       +CM_{2,k}^{\rm rem}(\Gamma_kd_k^2+\Gamma_k^2t).
\]
Since $I_k^{\rm imp}(0)=0$ and
$\int_0^1M_{1,k}^{\rm rem}\idiff t\leq Ch^\beta$, Gronwall gives
the estimate \eqref{eq:packet-impulse-event}.
Its dipole strain at index $n$, divided by $\varepsilon_n$, is at most
$CB_n(r(s_0+s)+(s_0+s)^2)$. This term is included in the remainder.
Indeed the dimensionless dipole is
$I_{n+1}^{\rm imp}/(\Gamma_nd_n)$, and its strain is bounded by
that norm times $CD^{-3}$. Dividing by $\varepsilon_n=\gamma^2D^{-4}$
and putting $t=\tau_n(s_0+s)$ gives
\[
 \frac{D|I_{n+1}^{\rm imp}|}{\Gamma_nd_n\gamma^2}
 \leq C(\tau_nR_nM_{2,n+1}^{\rm rem})
       \bigl(r(s_0+s)+(s_0+s)^2\bigr),
\]
where $Dd_n=R_n-R_{n+1}\leq R_n$ and
$\Gamma_n\tau_n/d_n^2=1$.
In the reference evolution the second moment at index $n+1$ is
$O(\gamma\sigma^2)$, so its
quadrupole strain costs $O(\varepsilon_n r)$.

The moment in \eqref{eq:material-donor} is first taken about the actual
common exterior trajectory ${\boldsymbol c_{\rm ext}}$. Passing to the actual mean centroid
$\boldsymbol c_{n+1}$ changes it exactly by
\[
 -\frac{(\boldsymbol c_{n+1}-\boldsymbol c_{\rm ext})\otimes I_{n+1}^{\rm imp}
       +I_{n+1}^{\rm imp}\otimes(\boldsymbol c_{n+1}-\boldsymbol c_{\rm ext})}
        {\Gamma_n d_n^2}.
\]
The centroid comparison gives $|\boldsymbol c_{n+1}-{\boldsymbol c_{\rm ext}}|/d_n\leq C(\sigma+vs)$ in
the normalized coordinates of index $n$, where $v=\gamma/\sigma$. Equation
\eqref{eq:packet-impulse-event} therefore bounds this change, divided
by $\gamma^2$, by a small coefficient times a fixed polynomial in $s,{s_0}$.
More precisely that bound is
\[
 CB_n\frac{\sigma+vs}{D}
            \bigl(r(s_0+s)+(s_0+s)^2\bigr).
\]
The center-change identity follows by expanding
$(x-c_{n+1})^{\otimes2}=(x-c_{\rm ext})^{\otimes2}
 -(c_{n+1}-c_{\rm ext})\otimes(x-c_{\rm ext})
 -(x-c_{\rm ext})\otimes(c_{n+1}-c_{\rm ext})
 +(c_{n+1}-c_{\rm ext})^{\otimes2}$.
Its final term integrates to zero by neutrality. Thus changing the
origin does not remove the impulse and does not create a constant-mass
term.
The change of the kernel expansion point costs the further factor
$(\sigma+vs)/D$.

We make explicit where a far-field expansion is used. In the length
units of index $n$, the source centroid diameter is at most
$C(\sigma+vs)$, and its patch radius is at most
$C\sigma\eta_{n+1}^{1/45}$. The observed configuration has bounded
diameter in these units. The actual center separation is comparable
to $D$. The scale margins imply
\[
 \sup_{0\leq s\leq S}\frac{1+\sigma+vs}{D}=o_L(1).
\]
Thus every Taylor segment lies at distance at least $cD$ from zero.
For the gradient of the source velocity, neutrality eliminates the
constant kernel term, the first moment gives the dipole term already
bounded, and the second moment gives
$\tfrac12\nabla^3K(De_1):M_{n+1}$ in center-line coordinates.
For one point pair, subtract the second-order Taylor polynomial of
$\nabla K(X-y)$ and call its remainder $R_2(y)$. Taylor's
formula for its derivative gives
$|\nabla_yR_2(y)|\leq CD^{-5}|y|^2$ on the source ball.
The signed difference at the endpoints of the pair is therefore at most
$CD^{-5}|d_{\rm pair}|(|m_{\rm pair}|+|d_{\rm pair}|)^2$.
The source covariance correction after the quadratic term costs
$CD^{-5}a_{\rm core}^2(|m_{\rm pair}|+|d_{\rm pair}|+a_{\rm sup})$,
where $a_{\rm core}^2=C\sigma^2\eta_{n+1}^2$ bounds each
variance and $a_{\rm sup}\leq C\sigma\eta_{n+1}^{1/45}$ is
the actual support radius. Since $a_{\rm sup}\leq C|d_{\rm pair}|$,
it is absorbed by the pair length.
Using $|d_{\rm pair}|\leq C\sigma$ and the variance
$C\sigma^2\eta_{n+1}^2$ gives, after division by $\varepsilon_n$,
\[
 C\frac{\sigma+vs}{D}(r+s)
    +C\frac{\sigma+vs}{D}r\eta_{n+1}^2.
\]
Here $\gamma\sigma(\sigma+vs)/\gamma^2=r+s$.
The same expansion across the observed diameter adds $CD^{-1}(r+s)$.
These estimates are needed only for $s\leq S$. The integral defining
$\mathcal H_\sigma$ uses the algebraic moment expression extended to
all $s\geq0$, whose exponential tail is estimated separately below.

For $k>n+1$, the sum of the quadrupole contributions divided by the
contribution of index $n+1$ is bounded by
$C\mathfrak q_{\rm deep}$, where $\mathfrak q_{\rm deep}=q_{n+1}^{2\nu_\Gamma}$. The sharp Hessian in
\eqref{eq:sharp-hessian} gives the additional impulse factor
$I_{\rm deep}=q_n^{\nu_\Gamma-3}q_{n+1}^{2\nu_\Gamma}$.
To see these ratios, for $k>n+1$ the dimensional second-moment bound
$C(\Gamma_kd_k^2+\Gamma_k^2t)$, divided by
$\Gamma_{n+1}^2\tau_n$, is at most
\[
 C\left(\frac{R_k}{R_{n+1}}\right)^{2\nu_\Gamma}
       \left(\frac{\tau_k}{\tau_n}+s_0+s\right).
\]
The Gaussian ratios sum from their first value
$q_{n+1}^{2\nu_\Gamma}$ with a uniform geometric bound. For the
impulse an additional factor
$R_{k-1}^{\nu_\Gamma-3}/R_n^{\nu_\Gamma-3}$ occurs in the
Hessian bound. At $k=n+2$ this is $q_n^{\nu_\Gamma-3}$;
each subsequent term still decreases geometrically, because
$2\nu_\Gamma+(\nu_\Gamma-3)=3\nu_\Gamma-3>0$.
Thus the same summation gives $I_{\rm deep}$.
Higher multipoles are also controlled after grouping pairs. At order
$\nu\geq2$ their size is at most
\[
 C\Gamma_kd_k(d_k+v_kt)^{\nu-1}
             +C\Gamma_ka_k^2(d_k+v_kt)^{\nu-2}.
\]
Each extra kernel derivative is accompanied by another ratio of
the configuration diameter to its distance from index $n$,
and
\[
 v_k\tau_n S/R_k
                 \leq ChR_n^{\alpha}R_k^{\alpha}S=o_L(1).
\]
Thus the contribution of every index $k>n+1$ is retained and
summed, both before and after its separated-section crossing.
The kernel contraction is linear in the signed moment, so the
preceding estimates yield exactly \eqref{eq:inner-gradient-decomposition}
and \eqref{eq:inner-gradient-remainder}.
\end{proof}

\subsection{Difference equations for the initially identical configurations}

For $k<n$, an absolute covariance error $O(\eta_k^2)$ need not be
small compared with $\varepsilon_n$. The two patches at each such index,
however, start from the same disk. The proportional estimate of
Section~\ref{app:response} applies from time zero because
\[
 \frac{t}{\tau_k}
 \leq r_n^{\rm out}(s_0+S)=o_L(1),
\]
and \eqref{eq:analytic-canonical} supplies its exterior hypotheses.
Each reduced boundary graph has norm $O(\eta_k^2)$ on a larger
analytic strip. On the common smaller strip of Lemma~\ref{lem:response}, at its
actual elapsed time $t/\tau_k$, let
$e_{k,i}(t)=\|g_{k,i}^{\mathrm{trig}}(t)-g_{k,i}^{\mathrm{quiet}}(t)\|_{\varrho(t/\tau_k),1}$.
In physical time its comparison inequality is
\begin{equation}\label{eq:head-graph-differential}
 \mathrm D^+e_{k,i}\leq C\tau_k^{-1}e_{k,i}
                                      +C\Delta A_{k,i},\qquad e_{k,i}(0)=0.
\end{equation}
Here $\Delta A_{k,i}$ is the supremum norm of the difference of the
physical exterior gradients in matched centered coordinates. The estimate uses
this difference itself, with no time derivative. Indeed, the normalized
velocity is $(\tau_k/d_k)V^{\rm phys}$, so its gradient difference is
$\Delta_A^{\rm int}(t/\tau_k)=\tau_k\Delta A_{k,i}(t)$.
Converting the interaction-time inequality to physical time divides
both terms by $\tau_k$, giving \eqref{eq:head-graph-differential}.
The comparison starts from the original disks and uses their full
elapsed interaction time $t/\tau_k\leq T_*$, without restarting from
a subsequently deformed patch.

We next put the centroid and graph differences in one norm. Write
$\Delta F_k=F_k^{\mathrm{trig}}-F_k^{\mathrm{quiet}}$ for each shape,
phase or offset coordinate, using continuous lifts of the phases.
Define the relative-center difference and the maximum over larger scales by
\[
 \begin{aligned}
 c_{k,n}^{\Delta}
  &=\frac{(\boldsymbol c_k^{\mathrm{trig}}-\boldsymbol c_n^{\mathrm{trig}})
        -(\boldsymbol c_k^{\mathrm{quiet}}-\boldsymbol c_n^{\mathrm{quiet}})}
        {R_k},\\
 \mathcal E_{<n}
   &=\max_{\substack{k<n\\i}}\!
       \left\{e_{k,i},|\Delta\mathsf U_k|,|\Delta\mathsf V_k|,
       |\Delta\mathsf A_k|,|\Delta\vartheta_k|,
       \frac{|\Delta\chi_k|}{d_k},|c_{k,n}^{\Delta}|\right\}.
 \end{aligned}
\]
Set $\mathcal E_{<0}=0$. Comparing centers relative to index $n$
removes their common translation. At index $n$ use the finer
normalization dictated by the impulse equation:
\[
 \xi=\frac{\Delta\chi_n}{d_np_n},\qquad
 \mathcal E_n=\max\{|\Delta\mathsf U_n|,|\Delta\mathsf V_n|,
              |\Delta\mathsf A_n|,|\Delta\vartheta_n|,|\xi|\}.
\]
Since $d_np_n=d_n^2/R_n$, the exterior Hessian in that equation will
have coefficient $\tau_nR_nM_{2,n}^{\rm rem}\leq B_n$.

\begin{lemma}[Coupled comparison with all larger configurations]\label{lem:head-graph}
On a common stopped interval satisfying the bootstrap regime, for the norms just defined, choose
a fixed $C_{\rm env}$ large enough to bound the polynomial
source terms and put $F_{\rm env}(s)=C_{\rm env}(1+s_0+s)^4$.
In the variable $s=(t-t_{n+1}^{\rm out})/\tau_n$, the
centroid and graph comparisons give
\begin{align}\label{eq:head-impulse-system}
 \mathrm D^+\mathcal E_{<n}
  &\leq C(r_n^{\rm out}+A_n)\mathcal E_{<n}
      +C(p_n^3\mathcal E_n+p_n^2|\xi|)
      +C(\varepsilon_n F_{\rm env}+\eta_n^2),\notag\\
 |\xi'|&\leq CA_n|\xi|+CB_n(\mathcal E_n+\mathcal E_{<n})
                  +C\varepsilon_n F_{\rm env}+CB_n\eta_n^2.
\end{align}
\end{lemma}

\begin{proof}
Fix the current stopped interval $0\leq s\leq s_{\rm end}\leq S$ and set
\[
 T=s_0+s=\frac t{\tau_n},\qquad
 \nu_k=\frac{\tau_n}{\tau_k},\qquad
 \widehat\chi_k=\frac{\chi_k}{d_k},\qquad
 y_{k,i}=\frac{b_{k,i}-\boldsymbol c_k}{d_k}.
\]
A prime in this proof denotes differentiation in $s$; a dot denotes
physical-time differentiation. In particular,
\[
 \sum_{i=1}^4y_{k,i}=0,\qquad
 \widehat\chi_k=\frac12\sum_{i=1}^4\epsilon_i y_{k,i},\qquad
 \nu_k\leq r_n^{\rm out}\quad(k<n).
\]
We estimate the components in the following order. The coordinate and
kernel bounds first supply uniform coefficients. The graph rows control
patch deformation, the shape and phase rows control the internal
configuration, and the offset rows control its signed impulse. The
relative-center rows remove common translation; only these rows need
the term $p_n^2|\xi|$. The final fine-impulse estimate closes the
receiver equation in its stronger normalization.

\paragraph{\textbf{The coordinate map and centered source kernels.}}
Use the six real coordinates
\[
 \mathcal X_k=(\mathsf U_k,\mathsf V_k,\mathsf A_k,
                    \vartheta_k,\widehat\chi_k)=\Phi(y_{k,1},\ldots,y_{k,4}).
\]
The inverse chart in \eqref{eq:cell-chart-inverse}, together with
\[
 y_{k,1,2}=\widehat\chi_k/2\pm z_k,\qquad
 y_{k,3,4}=-\widehat\chi_k/2\pm w_k,
\]
gives uniform bounds for $\nabla\Phi$, $\nabla^2\Phi$ and
$\nabla\Phi^{-1}$ on the small chart. These bounds also hold for
continuous lifts of $\vartheta$: the derivatives are periodic in the
common rotation. For example, with variations $h_z,h_w$, the defining
formulas give
\[
 \begin{aligned}
 \nabla\mathsf A[h_z,h_w]&=2z\cdot h_z-2w\cdot h_w,\\
 \nabla\mathsf R[h_z,h_w]
 &=2\frac{|w|}{|z|}z\cdot h_z
                    +2\frac{|z|}{|w|}w\cdot h_w,\\
 \nabla\phi[h_z,h_w]
 &=\frac{z^\perp\cdot h_z}{|z|^2}
                         -\frac{w^\perp\cdot h_w}{|w|^2},\\
 \nabla\vartheta[h_z,h_w]
 &=\frac12\left(\frac{z^\perp\cdot h_z}{|z|^2}
                         +\frac{w^\perp\cdot h_w}{|w|^2}\right),\\
 \nabla\mathsf X
 &=\frac{\mathsf R+\mathsf A}
            {2\sqrt{\mathsf A}\mathsf R^{3/2}}\nabla\mathsf R
   -\frac{\mathsf R+\mathsf A}
            {2\mathsf A^{3/2}\sqrt{\mathsf R}}\nabla\mathsf A,\\
 \nabla\mathsf Y
 &=\frac{\mathsf Y}{2\mathsf R}\nabla\mathsf R
    -\frac{\mathsf Y}{2\mathsf A}\nabla\mathsf A
    +\sqrt{\mathsf R/\mathsf A}\cos(\phi-\pi/2)\nabla\phi.
 \end{aligned}
\]
All denominators have fixed positive lower bounds in the chart.
Differentiating these coefficients gives the asserted second-derivative
bound. Consequently,
\begin{equation}\label{eq:head-chart-displacement}
 \max_i|\Delta y_{k,i}|\leq C\mathcal E_{<n}\quad(k<n),
 \qquad \max_i|\Delta y_{n,i}|\leq C\mathcal E_n,
 \qquad \Delta\widehat\chi_n=p_n\xi.
\end{equation}
For a perturbing centered centroid velocity $E$, the exact identity
\[
 \begin{split}
 \Delta\bigl(\nabla\Phi(y)E\bigr)
  &=\nabla\Phi(y^{\rm trig})\Delta E
       +\bigl(\nabla\Phi(y^{\rm trig})
                  -\nabla\Phi(y^{\rm quiet})\bigr)E^{\rm quiet},\\
 \left|\Delta\bigl(\nabla\Phi(y)E\bigr)\right|
  &\leq C|\Delta E|+C|E^{\rm quiet}|\,|\Delta\mathcal X|
 \end{split}
\]
will account for the nonlinear coordinate change in every scalar row.

Let $\rho_{j,i}^\kappa$ be the centered physical probability measure
obtained from $\mu_{j,i}^\kappa$ by translation by $-b_{j,i}^\kappa$.
Its mass is one and its first moment is zero. The area-coordinate
formula for the two centered graphs, including their constrained low
modes, gives
\begin{equation}\label{eq:head-centered-test}
 \left|\int f(y)\idiff(\rho_{j,i}^{\rm trig}
                       -\rho_{j,i}^{\rm quiet})(y)\right|
 \leq Ca_j^2\|\nabla^2f\|_{L^\infty(B_{2a_j})}e_{j,i},
 \qquad j<n.
\end{equation}
Indeed, subtract $f(0)+\nabla f(0)y$ under the integral. The mass and
centroid constraints cancel this affine polynomial, its remainder is
bounded by $Ca_j^2\|\nabla^2f\|$, and the symmetric difference of the
two normalized graph domains has area at most $Ce_{j,i}$. The low modes
have the same difference bound because they are analytic functions of
the reduced graph on the fixed small graph ball.

For a source component put
\[
 u_{j,i}^\kappa(x)=\epsilon_i\Gamma_j
       \int K(x-b_{j,i}^\kappa-y)\idiff\rho_{j,i}^\kappa(y).
\]
Suppose matched source-observer distances are comparable to $\ell$.
Changing positions with the source graph fixed, then changing the
centered graph with positions fixed, yields, for $m=0,1,2$,
\begin{equation}\label{eq:head-source-kernel}
 \begin{aligned}
 |\Delta\nabla^m u_{j,i}|
 &\leq C_m\Gamma_j\left(
    \ell^{-m-2}|\Delta x-\Delta b_{j,i}|
                   +a_j^2\ell^{-m-3}e_{j,i}\right),&&j<n,\\
 |\Delta\nabla^m u_{n,i}|
 &\leq C_m\Gamma_n\left(
    \ell^{-m-2}|\Delta x-\Delta b_{n,i}|
                   +a_n^2\ell^{-m-3}\right).
 \end{aligned}
\end{equation}
Here $|\nabla^mK(x)|\leq C_m|x|^{-m-1}$ is used on the separated
source-observer neighborhoods. For the second row, Taylor-expand each
source about its own centroid and use its variance bound separately.
For the first row use \eqref{eq:head-centered-test}; hence its covariance
term retains the factor $e_{j,i}$.

Align the two histories by their index-$n$ centers. The exact position
relations are
\[
 \begin{aligned}
 \Delta b_{k,i}&=\Delta\boldsymbol c_n+R_kc_{k,n}^\Delta
                                      +d_k\Delta y_{k,i},&&k<n,\\
 \Delta b_{n,i}&=\Delta\boldsymbol c_n+d_n\Delta y_{n,i}.
 \end{aligned}
\]
The common translation cancels from the displacement in
\eqref{eq:head-source-kernel}. With $\ell_{kj}=\max(R_k,R_j)$,
\begin{equation}\label{eq:head-relative-displacements}
 \begin{aligned}
 |\Delta b_{k,i}-\Delta b_{j,l}|
          &\leq C\ell_{kj}\mathcal E_{<n},&&k,j<n,\\
 |\Delta b_{k,i}-\Delta b_{n,l}|
          &\leq CR_k\mathcal E_{<n}+Cd_n\mathcal E_n,&&k<n.
 \end{aligned}
\end{equation}

For the field of the other configurations define, on the fixed
normalized ball containing the compact state,
\[
 \begin{gathered}
 H_k^\kappa(y)=\frac{\tau_n}{d_k}
 \bigl(V_k^\kappa(\boldsymbol c_k^\kappa+d_ky)
                      -V_k^\kappa(\boldsymbol c_k^\kappa)\bigr),\\
 H_k^\kappa(0)=0,\qquad
 \|\nabla H_k^\kappa\|\leq CA_n,\qquad
 \mathcal G_k=\|\nabla(H_k^{\rm trig}-H_k^{\rm quiet})\|.
 \end{gathered}
\]
For $k<n$, split its sources into $j<n$, $j=n$ and $j>n$.
The coefficient sum for the first group is
\begin{equation}\label{eq:head-gradient-sum}
 \begin{split}
 \tau_n\sum_{\substack{j<n\\j\ne k}}
             \frac{\Gamma_j}{\ell_{kj}^2}
 &\leq C\tau_nh^\beta
    \left(\sum_{j<k}R_j^{\nu_\Gamma-2}
                  +R_k^{-2}\sum_{k<j<n}R_j^{\nu_\Gamma}\right)\\
 &\leq C\tau_nh^\beta\leq CA_n.
 \end{split}
\end{equation}
The exponent $\nu_\Gamma-2=3\alpha>0$ controls both sums uniformly
in the cutoff. Graph contributions have the additional factor
$(a_j/\ell_{kj})^2$ and obey the same bound. The index-$n$ source
has the three contributions
\[
 C\frac{\tau_n\Gamma_n}{R_k^2}\mathcal E_{<n},\qquad
 C\frac{\tau_n\Gamma_nd_n}{R_k^3}\mathcal E_n
       =C\left(\frac{d_n}{R_k}\right)^3\mathcal E_n,
 \qquad
 C\frac{\tau_n\Gamma_na_n^2}{R_k^4}\leq C\eta_n^2.
\]
The first coefficient is at most $CA_n$ by the same scale sum;
the second is at most $Cp_n^3$.
The neutral-source expansion and the impulse estimate of the preceding
subsection give the absolute bounds
\begin{equation}\label{eq:head-inner-absolute}
 \tau_nR_n^{m-1}
       \|\nabla^m V_{>n}^\kappa\|\leq
           C\varepsilon_nF_{\rm env}(s),\qquad m=0,1,2.
\end{equation}
These bounds hold near every observing configuration $k\leq n$.
For the gradient and Hessian they follow from the grouped moment
expansion; for the velocity use the same expansion one derivative
earlier. Each source is neutral, its impulse is generated from zero,
and its signed second moment is bounded by
$C(\Gamma_jd_j^2+\Gamma_j^2t)$. The Gaussian sums used in
Proposition~\ref{prop:inner-gradient} therefore give all three orders.
Increasing the fixed $C_{\rm env}$ absorbs their polynomial bounds.
Together these estimates prove
\begin{equation}\label{eq:head-gradient-mismatch}
 \mathcal G_k\leq CA_n\mathcal E_{<n}
       +Cp_n^3\mathcal E_n+C\eta_n^2
                        +C\varepsilon_nF_{\rm env},\qquad k<n.
\end{equation}

\paragraph{\textbf{The graph rows.}}
The exterior for the component $(k,i)$ includes its other three
same-index components. Define its physical gradient discrepancy at
matched component-centered points by
\[
 \Delta A_{k,i}=\sup_{|y|\leq r_{\rm ext}}
 \left|\nabla V_{k,i}^{\rm ext,trig}(b_{k,i}^{\rm trig}+d_ky)
      -\nabla V_{k,i}^{\rm ext,quiet}(b_{k,i}^{\rm quiet}+d_ky)\right|.
\]
The normalized exterior hypotheses in
\eqref{eq:response-exterior-hypothesis} follow from
\eqref{eq:analytic-canonical}; their constants are uniform in $k$.
Moreover $\nu_k(s_0+S)\leq r_n^{\rm out}(s_0+S)<T_*$, so
Lemma~\ref{lem:response} applies from physical time zero.
The same-index sources in \eqref{eq:head-source-kernel} contribute
$C\nu_k\mathcal E_{<n}$ to $\tau_n\Delta A_{k,i}$, since their
separation is comparable to $d_k$. Translating the remote observing
center from $\boldsymbol c_k$ to $b_{k,i}$ adds at most
\[
 C\tau_nd_kM_{2,k}^{\rm rem}\max_l|\Delta y_{k,l}|
                                      \leq CA_n\mathcal E_{<n}.
\]
The last bound follows from the derivative sum: every remote distance
is at least a fixed multiple of $d_k$. Thus
\[
 \begin{aligned}
 \tau_n\Delta A_{k,i}
   &\leq C(\nu_k+A_n)\mathcal E_{<n}+C\mathcal G_k,\\
 \mathrm D^+e_{k,i}
   &\leq C\nu_ke_{k,i}+C\tau_n\Delta A_{k,i}\\
   &\leq C(\nu_k+A_n)\mathcal E_{<n}
      +Cp_n^3\mathcal E_n+C\varepsilon_nF_{\rm env}+C\eta_n^2.
 \end{aligned}
\]
The circular self frequencies have already been handled in
Lemma~\ref{lem:response}; the coefficient in this row is $\nu_k$.

\paragraph{\textbf{The four shape and phase rows.}}
Let $\Pi$ subtract the arithmetic mean of a four-vector, and put
\[
 \begin{gathered}
 F_i(y)=\sum_{l\ne i}\epsilon_lK(y_i-y_l),\qquad
 \mathcal F(\mathcal X)=\nabla\Phi(y(\mathcal X))\Pi F(y(\mathcal X)),\\
 \mathcal P(\mathcal X,H)=\nabla\Phi(y(\mathcal X))
                         \Pi\bigl(H(y_1),\ldots,H(y_4)\bigr).
 \end{gathered}
\]
On the separated chart $\|\nabla\mathcal F\|\leq C$. For fields with
$H(0)=0$, the coordinate-change identity above implies
\[
 \left|\mathcal P(\mathcal X^{\rm trig},H^{\rm trig})
           -\mathcal P(\mathcal X^{\rm quiet},H^{\rm quiet})\right|
 \leq C\|\nabla H^{\rm quiet}\|\,|\Delta\mathcal X|
                         +C\|\nabla\Delta H\|.
\]
The exact averaged centroid equation, after applying the chart, is
\[
 \mathcal X_k'=\nu_k\mathcal F(\mathcal X_k)
                   +\mathcal P(\mathcal X_k,H_k)+R_k^{\rm patch}.
\]
We first compare its averaging corrections. For a same-index
interaction their velocity remainder is
\[
 \delta_{il}=\iint
 \bigl[K(b_i-b_l+x-y)-K(b_i-b_l)
                     -\nabla K(b_i-b_l)(x-y)\bigr]
                                     \idiff\rho_i(x)\idiff\rho_l(y).
\]
The affine term has zero integral. Changing either graph uses
\eqref{eq:head-centered-test}; changing $b_i-b_l$ differentiates
the quadratic remainder once. Since $|b_i-b_l|\asymp d_k$,
\[
 \frac{\tau_n\Gamma_k}{d_k}|\Delta\delta_{il}|
 \leq C\nu_k\eta_k^2
       \left(\max_j e_{k,j}+\max_j|\Delta y_{k,j}|\right)
 \leq C\nu_k\mathcal E_{<n}.
\]
For the remote observer correction, apply
\eqref{eq:observer-response} in the index-$k$ clock and then multiply
by $\nu_k$. Its contribution in the present velocity units is bounded by
\[
 C\nu_k\eta_k^2\max_i e_{k,i}
             +C\eta_k^4\tau_n\max_i\Delta A_{k,i}
 \leq C(\nu_k+A_n)\mathcal E_{<n}+C\mathcal G_k.
\]
The transformations by $\nabla\Phi$ obey the same bound: the
reference averaging remainder is $O(\nu_k\eta_k^2)$ for internal
interactions, and the difference of the coordinate derivatives costs
$C|\Delta\mathcal X_k|$. Source changes in the remote field are
already included in $\mathcal G_k$. It follows that
\[
 |\Delta R_k^{\rm patch}|
        \leq C(\nu_k+A_n)\mathcal E_{<n}+C\mathcal G_k.
\]
Subtracting the exact vector equations now gives, for each
$Z\in\{\mathsf U,\mathsf V,\mathsf A,\vartheta\}$,
\begin{equation}\label{eq:head-scalar-rows}
 \mathrm D^+|\Delta Z_k|
 \leq C(\nu_k+A_n)\mathcal E_{<n}
      +Cp_n^3\mathcal E_n+C\varepsilon_nF_{\rm env}+C\eta_n^2.
\end{equation}
This conclusion uses a small internal coefficient $\nu_k$ at $k<n$;
the index-$n$ expanding equation will instead be factored exactly below.

\paragraph{\textbf{The coarse offset rows.}}
For the signed sum, all internal patch interactions cancel before
estimating. The exact equation is
\[
 \widehat\chi_k'
    =\frac{\tau_n}{2d_k}\sum_{i=1}^4\epsilon_i
                               \int V_k(x)\idiff\mu_{k,i}(x).
\]
The constant value $V_k(\boldsymbol c_k)$ also cancels. The point
part is consequently the $\widehat\chi$ component of
$\mathcal P(\mathcal X_k,H_k)$, and the observer part obeys the
proportional estimate just proved. Hence
\[
 \mathrm D^+|\Delta\widehat\chi_k|
 \leq C(\nu_k+A_n)\mathcal E_{<n}
      +Cp_n^3\mathcal E_n+C\varepsilon_nF_{\rm env}+C\eta_n^2.
\]
For this row one can drop the internal $\nu_k$ contribution; the
coarser displayed form suffices for the maximum norm.

\paragraph{\textbf{The relative-center rows.}}
Define the averaged velocity of a source configuration $j$ at an
observing configuration $i\ne j$ by
\[
 C_{ij}^\kappa=\frac14\sum_{l=1}^4
                         \int u_j^\kappa(x)\idiff\mu_{i,l}^\kappa(x),
 \qquad u_j^\kappa=\sum_{l=1}^4u_{j,l}^\kappa.
\]
Let $\dot{\boldsymbol c}_i^{\rm int}$ be its same-index mean velocity.
For $k<n$, the exact endpoint decomposition is
\begin{equation}\label{eq:head-center-decomposition}
 \begin{split}
 \dot{\boldsymbol c}_k-\dot{\boldsymbol c}_n
  ={}&\dot{\boldsymbol c}_k^{\rm int}
                   -\dot{\boldsymbol c}_n^{\rm int}
       +\sum_{\substack{j<n\\j\ne k}}(C_{kj}-C_{nj})
       -C_{nk}+C_{kn}\\
     &+\sum_{j>n}(C_{kj}-C_{nj}),\\
 (c_{k,n}^\Delta)'={}&\frac{\tau_n}{R_k}
             \Delta(\dot{\boldsymbol c}_k-\dot{\boldsymbol c}_n).
 \end{split}
\end{equation}
The internal point mean has the smooth normalized form
\[
 \dot{\boldsymbol c}_i^{\rm int,pt}
       =\frac{d_i}{\tau_i}\mathcal M(\mathcal X_i),\qquad
 \mathcal M(\mathcal X)=\frac14\sum_{l=1}^4F_l(y(\mathcal X)).
\]
At $k<n$, smoothness and the proportional internal averaging
estimate give
\[
 \frac{\tau_n}{R_k}
                 |\Delta\dot{\boldsymbol c}_k^{\rm int}|
 \leq C\nu_k\frac{d_k}{R_k}\mathcal E_{<n}
 \leq C\nu_k\mathcal E_{<n}.
\]
At index $n$, central symmetry gives
\[
 \mathcal M(\mathsf U,\mathsf V,\mathsf A,\vartheta,0)=0,
 \quad |\nabla_{\mathsf U,\mathsf V,\mathsf A,\vartheta}\mathcal M|
                      \leq C|\widehat\chi|,
 \quad |\nabla_{\widehat\chi}\mathcal M|\leq C.
\]
Change the offset first, keeping the triggered shape fixed, and then
change the four remaining coordinates. The compact impulse estimate
from physical time zero and $\Delta\widehat\chi_n=p_n\xi$ give
\[
 \begin{aligned}
 |\Delta\mathcal M(\mathcal X_n)|
   &\leq C|\Delta\widehat\chi_n|
                    +C|\widehat\chi_n^{\rm quiet}|\mathcal E_n,
 &|\widehat\chi_n^{\rm quiet}|&\leq Cp_nB_nT,\\
 \frac{\tau_n}{R_k}
          |\Delta\dot{\boldsymbol c}_n^{\rm int,pt}|
   &\leq Cp_n^2|\xi|+Cp_n^2B_nT\mathcal E_n
    \leq Cp_n^2|\xi|+Cp_n^3\mathcal E_n.
 \end{aligned}
\]
The last inequality uses $B_n\leq Cp_n^2$ and
$p_n(s_0+S)=o_L(1)$. Its internal averaging errors, estimated
separately in the two histories, add at most
$C(d_n/R_k)\eta_n^2\leq C\eta_n^2$.

Next consider a common larger source $j<k$. Initially replace the
observer averages by values at the configuration means. In the
source variable $b$, its difference kernel satisfies
\[
 \begin{aligned}
 &\left|\nabla_b\{K(\boldsymbol c_k-b)
                         -K(\boldsymbol c_n-b)\}\right|
     \leq C\frac{R_k}{R_j^3},\\
 &\left|\nabla_b^2\{K(\boldsymbol c_k-b)
                         -K(\boldsymbol c_n-b)\}\right|
     \leq C\frac{R_k}{R_j^4}.
 \end{aligned}
\]
The connecting segment stays a distance comparable to $R_j$ from
this source because $j<k$ and the enclosing-disk margin is fixed.
After alignment, its source displacement is at most
$CR_j\mathcal E_{<n}$, while
$\Delta(\boldsymbol c_k-\boldsymbol c_n)=R_kc_{k,n}^\Delta$.
The endpoint and source changes, including the source graphs via
\eqref{eq:head-centered-test}, therefore give
\[
 \frac1{R_k}\left|\Delta\{u_j(\boldsymbol c_k)
                                   -u_j(\boldsymbol c_n)\}\right|
       \leq C\frac{\Gamma_j}{R_j^2}\mathcal E_{<n},\qquad j<k.
\]
For $k<j<n$, estimate the two endpoints separately. Their distances
from the source are respectively comparable to $R_k$ and $R_j$,
so
\[
 \frac1{R_k}\left|\Delta\{u_j(\boldsymbol c_k)
                                   -u_j(\boldsymbol c_n)\}\right|
 \leq C\frac{\Gamma_j}{R_kR_j}\mathcal E_{<n}
 \leq C\frac{\Gamma_j}{R_j^2}\mathcal E_{<n}.
\]
The source $j=k$, present only at the index-$n$ endpoint, has the
same bound $C\Gamma_kR_k^{-2}\mathcal E_{<n}$. Multiplication by
$\tau_n$ and the sum \eqref{eq:head-gradient-sum} bound all these
mean-point terms by $CA_n\mathcal E_{<n}$.

To restore the observing averages, first retain the arithmetic sum
of the four centroids at index $n$. For any common smooth field $F$,
put
\[
 R_F(x)=F(\boldsymbol c_n+x)-F(\boldsymbol c_n)
                                     -\nabla F(\boldsymbol c_n)x.
\]
Then, exactly,
\[
 \frac14\sum_{i=1}^4F(\boldsymbol c_n+d_ny_{n,i})
    =F(\boldsymbol c_n)+\frac14\sum_{i=1}^4R_F(d_ny_{n,i}),
 \qquad |\nabla R_F(x)|\leq\|\nabla^2F\||x|.
\]
Thus changing the observing shape at fixed centered field costs
$Cd_n^2\|\nabla^2F\|\mathcal E_n$. Its normalized coefficient is
\begin{equation}\label{eq:head-observer-cancellation}
 \frac{\tau_nd_n^2M_{2,n}^{\rm rem}}{R_k}\mathcal E_n
       \leq B_np_n^2\mathcal E_n
       \leq Cp_n^4\mathcal E_n\leq Cp_n^3\mathcal E_n.
\end{equation}
Changes of this centered field itself are proportional to
$\mathcal E_{<n}$ and obey the source estimates already proved.
For observing patches at $k<n$, the corresponding terms retain the
graph factors and their coefficient sum is bounded by
\[
 \frac{\tau_nd_k^2}{R_k}
              \sum_{j\ne k}\frac{\Gamma_j}{\ell_{kj}^3}
       \leq C\tau_n\sum_{j\ne k}\frac{\Gamma_j}{\ell_{kj}^2}
       \leq CA_n.
\]
Here $d_k^2\leq CR_k\ell_{kj}$ was used. At $n$, the finite-patch
observer corrections can be estimated absolutely, at cost
$C\tau_na_n^2M_{2,n}^{\rm rem}/R_k\leq C\eta_n^2$.
This accounts for both the shape and patch parts of the observer
average without losing the cancellation of its affine term.

Finally $C_{kn}$ is the index-$n$ source seen at the larger endpoint.
The $m=0$ row of \eqref{eq:head-source-kernel} gives
\[
 \frac{\tau_n}{R_k}|\Delta C_{kn}|
 \leq CA_n\mathcal E_{<n}
       +C\frac{\tau_n\Gamma_nd_n}{R_k^3}\mathcal E_n+C\eta_n^2
 \leq CA_n\mathcal E_{<n}+Cp_n^3\mathcal E_n+C\eta_n^2.
\]
Equation~\eqref{eq:head-inner-absolute} gives
\[
 \frac{\tau_n}{R_k}
       \left|\Delta\sum_{j>n}(C_{kj}-C_{nj})\right|
         \leq C\varepsilon_nF_{\rm env},
\]
since $R_k\geq R_n$. Combining all groups in
\eqref{eq:head-center-decomposition} proves
\begin{equation}\label{eq:head-center-row}
 \mathrm D^+|c_{k,n}^\Delta|
 \leq C(\nu_k+A_n)\mathcal E_{<n}
       +Cp_n^3\mathcal E_n+Cp_n^2|\xi|
                         +C\varepsilon_nF_{\rm env}+C\eta_n^2.
\end{equation}

\paragraph{\textbf{The fine signed-impulse row at index $n$.}}
Write $H=H_n$, $G^\kappa=\nabla H^\kappa(0)$ and
$R^\kappa(y)=H^\kappa(y)-G^\kappa y$. The exterior estimates give
\[
 |G^\kappa|\leq CA_n,\qquad
 \|\nabla^2H^\kappa\|\leq Cp_nB_n,\qquad
 |\nabla R^\kappa(y)|\leq Cp_nB_n|y|.
\]
The exact signed centroid equation, with all internal interactions
already canceled, is
\begin{equation}\label{eq:head-fine-impulse-exact}
 (\widehat\chi_n^\kappa)'
 =G^\kappa\widehat\chi_n^\kappa
       +\frac12\sum_{i=1}^4\epsilon_iR^\kappa(y_{n,i}^\kappa)
       +r_{\rm cov}^\kappa,
 \qquad |r_{\rm cov}^\kappa|\leq Cp_nB_n\eta_n^2.
\end{equation}
For the sources $j<n$, \eqref{eq:head-source-kernel} at the index-$n$
mean gives a gradient coefficient sum bounded by $CA_n$, and a
Hessian coefficient sum bounded by $CB_n$ after multiplication by
$\tau_nR_n$. Add \eqref{eq:head-inner-absolute} for $j>n$ to get
\begin{equation}\label{eq:head-fine-field-difference}
 \begin{aligned}
 |\Delta G|&\leq CA_n\mathcal E_{<n}
                              +C\varepsilon_nF_{\rm env},\\
 \|\nabla^2\Delta H\|&\leq Cp_nB_n\mathcal E_{<n}
                              +Cp_n\varepsilon_nF_{\rm env}.
 \end{aligned}
\end{equation}
Subtract \eqref{eq:head-fine-impulse-exact} and only then divide by
$p_n$:
\[
 \xi'=G^{\rm trig}\xi
       +\Delta G\frac{\widehat\chi_n^{\rm quiet}}{p_n}
       +\frac1{2p_n}\sum_{i=1}^4\epsilon_i
          \bigl(R^{\rm trig}(y_{n,i}^{\rm trig})
                            -R^{\rm quiet}(y_{n,i}^{\rm quiet})\bigr)
       +\frac{\Delta r_{\rm cov}}{p_n}.
\]
Its four terms have the following bounds:
\[
 \begin{aligned}
 |G^{\rm trig}\xi|&\leq CA_n|\xi|,\\
 \left|\Delta G\frac{\widehat\chi_n^{\rm quiet}}{p_n}\right|
 &\leq CA_nB_nT\mathcal E_{<n}
                   +CB_nT\varepsilon_nF_{\rm env}
 \leq CB_n\mathcal E_{<n}+C\varepsilon_nF_{\rm env},\\
 \frac{|R^{\rm trig}(y_{n,i}^{\rm trig})
                   -R^{\rm trig}(y_{n,i}^{\rm quiet})|}{p_n}
 &\leq CB_n\mathcal E_n,\\
 \frac{|(R^{\rm trig}-R^{\rm quiet})(y_{n,i}^{\rm quiet})|}{p_n}
 &\leq CB_n\mathcal E_{<n}+C\varepsilon_nF_{\rm env},\\
 \frac{|\Delta r_{\rm cov}|}{p_n}&\leq CB_n\eta_n^2.
 \end{aligned}
\]
In the fourth line the field difference has zero value and zero
first derivative at the origin, so its Hessian bound in
\eqref{eq:head-fine-field-difference} applies. The second line uses
$A_nT+B_nT=o_L(1)$ and the compact impulse bound from physical time
zero. These estimates prove the second row of
\eqref{eq:head-impulse-system}.

All components entering $\mathcal E_{<n}$ have now been estimated.
For a finite maximum of the corresponding nonnegative functions,
\[
 \mathrm D^+\max_l f_l(s)
   \leq\max_{\{l:f_l(s)=\max_jf_j(s)\}}\mathrm D^+f_l(s).
\]
Use the graph rows, \eqref{eq:head-scalar-rows}, the coarse offset
rows and \eqref{eq:head-center-row}, and then
$\nu_k\leq r_n^{\rm out}$. This proves the first row of
\eqref{eq:head-impulse-system}. The source coefficients were summed
before taking the maximum, so its constant is independent of the
number of configurations. When $n=0$ the first row is void and the
fine impulse proof still applies with $\mathcal E_{<0}=0$.
\end{proof}

\begin{proposition}[Receiver equation and triangular comparison]
\label{prop:head-comparison}
Use the assumptions of Proposition~\ref{prop:release} and let
$[0,s_{\rm end}]$, $s_{\rm end}\leq S$, be a current common stopped
interval before the positive crossing. There is a decomposition
$\mathcal R=\mathcal R_{\rm comp}+\mathcal R_{\rm inner}$ such that
\[
 (\Delta\mathsf U_n)'=\lambda_{\rm eff}\Delta\mathsf U_n
       +\varepsilon_nf_{\sigma_n,\psi_n}+\mathcal R_{\rm comp}
                                      +\mathcal R_{\rm inner},
 \qquad \lambda_{\rm eff}\geq\lambda/2,
\]
and the comparison contribution satisfies
\begin{equation}\label{eq:comparison-remainder}
 \begin{split}
 |\mathcal R_{\rm comp}|\leq{}&
 C(A_n+\eta_n^2)(1+s_0+s)(\mathcal E_n+\mathcal E_{<n})\\
 &+Cp_n^2\bigl(B_n(s_0+s)|\xi|+|\xi|^2\bigr)+C\eta_n^2.
 \end{split}
\end{equation}
The stable, action and phase differences satisfy
\begin{equation}\label{eq:parent-triangular-rows}
 \begin{gathered}
 (\Delta\mathsf V_n)'=-\lambda_{\rm st}\Delta\mathsf V_n+\rho_V,
 \quad\lambda_{\rm st}\geq\lambda/2,\qquad
 (\Delta\mathsf A_n)'=\rho_A,\\
 (\Delta\vartheta_n)'=G_U\Delta\mathsf U_n+G_V\Delta\mathsf V_n
                          +G_A\Delta\mathsf A_n+\rho_{\rm phase},
 \qquad |G_U|+|G_V|+|G_A|\leq C,\\
 |\rho_V|+|\rho_A|+|\rho_{\rm phase}|
 \leq C(A_n+\eta_n^2)(1+s_0+s)(\mathcal E_n+\mathcal E_{<n})\\
 \hfill{}+Cp_n^2\bigl(B_n(s_0+s)|\xi|+|\xi|^2\bigr)
                  +C(\varepsilon_nF_{\rm env}+\eta_n^2).
 \end{gathered}
\end{equation}
The inner contribution is the difference of the two $k>n$ strain
projections after subtracting $\varepsilon_nf_{\sigma_n,\psi_n}$;
its moment, impulse, Taylor and projection errors are estimated in
Proposition~\ref{prop:inner-gradient} and the following weighted argument.
\end{proposition}

\begin{proof}
Suppress $n$ in the shape variables and put
\[
 q^\kappa=(\mathsf U^\kappa,\mathsf V^\kappa,\mathsf A^\kappa),
 \qquad S^\kappa=(q^\kappa,\vartheta^\kappa),
 \qquad \widehat\chi^\kappa=\chi_n^\kappa/d_n,
 \qquad T=s_0+s.
\]
For $j\in\{\mathsf U,\mathsf V,\mathsf A\}$ define the segment averages
\[
 B_j=\int_0^1\partial_j a\bigl(q^{\rm quiet}
                  +u(q^{\rm trig}-q^{\rm quiet})\bigr)\idiff u.
\]
The fundamental theorem of calculus gives
\[
 a(q^{\rm trig})-a(q^{\rm quiet})
           =B_U\Delta\mathsf U+B_V\Delta\mathsf V+B_A\Delta\mathsf A.
\]
Hence the expanding internal difference factors exactly as
\begin{equation}\label{eq:receiver-exact-factorization}
 \begin{split}
 a(q^{\rm trig})\mathsf U^{\rm trig}
              -a(q^{\rm quiet})\mathsf U^{\rm quiet}
 ={}&\bigl(a(q^{\rm trig})+\mathsf U^{\rm quiet}B_U\bigr)
                                                  \Delta\mathsf U\\
    &+\mathsf U^{\rm quiet}
                         (B_V\Delta\mathsf V+B_A\Delta\mathsf A).
 \end{split}
\end{equation}
Set
\[
 \lambda_{\rm eff}=a(q^{\rm trig})+\mathsf U^{\rm quiet}B_U.
\]
All $B_j$ are bounded on the fixed chart. The reference estimate and
smallness of that chart give
\[
 |\mathsf U^{\rm quiet}|\leq C(A_n+\eta_n^2),\qquad
 \lambda_{\rm eff}\geq\lambda/2,\qquad
 |\mathsf U^{\rm quiet}(B_V\Delta\mathsf V+B_A\Delta\mathsf A)|
                       \leq C(A_n+\eta_n^2)\mathcal E_n.
\]
No time derivative of $a$ or of a segment average is taken in this
factorization.

We next compare the exterior terms. For either source set
$J=\{j<n\}$ or $J=\{j>n\}$, put
\[
 H_J^\kappa(y)=\frac{\tau_n}{d_n}
 \bigl(V_J^\kappa(\boldsymbol c_n^\kappa+d_ny)
                         -V_J^\kappa(\boldsymbol c_n^\kappa)\bigr),
\]
and let $\mathcal P_U(\mathcal X,H)$ be the $\mathsf U$ component
of the centered-field functional in the proof of
Lemma~\ref{lem:head-graph}. The exact coordinate equation reads
\[
 (\mathsf U^\kappa)'=a(q^\kappa)\mathsf U^\kappa
       +O_U(S^\kappa,\widehat\chi^\kappa)
       +\mathcal P_U(\mathcal X^\kappa,H_{<n}^\kappa)
       +\mathcal P_U(\mathcal X^\kappa,H_{>n}^\kappa)
       +R_U^{\rm patch,\kappa}.
\]
For the larger sources, the kernel comparisons give
\[
 \|\nabla H_{<n}^\kappa\|\leq CA_n,\qquad
 \|\nabla\Delta H_{<n}\|\leq CA_n\mathcal E_{<n}.
\]
Change the observing configuration first and then the centered field:
\[
 \begin{split}
 \Delta\mathcal P_U={}&
  \mathcal P_U(\mathcal X^{\rm trig},H_{<n}^{\rm trig})
     -\mathcal P_U(\mathcal X^{\rm quiet},H_{<n}^{\rm trig})\\
 &+\mathcal P_U(\mathcal X^{\rm quiet},H_{<n}^{\rm trig})
     -\mathcal P_U(\mathcal X^{\rm quiet},H_{<n}^{\rm quiet}),\\
 |\Delta\mathcal P_U|\leq{}&CA_n(\mathcal E_n+\mathcal E_{<n}).
 \end{split}
\]
The centered-field functional includes the coordinate derivative.
Thus its bound includes both changes in the evaluated velocity and
the change of $\nabla\Phi$ multiplying the small reference velocity.

For the internal offset correction, the exact point equations are
even in $\widehat\chi$. In particular,
\[
 O_U(S,0)=0,\qquad \nabla_{\widehat\chi}O_U(S,0)=0,
\]
and Taylor's formula gives
\[
 O_U(S,\widehat\chi)
   =\int_0^1(1-u)\nabla_{\widehat\chi}^2
          O_U(S,u\widehat\chi)[\widehat\chi,\widehat\chi]\idiff u.
\]
Its derivatives consequently satisfy
\[
 |\nabla_{\widehat\chi}O_U|\leq C|\widehat\chi|,
 \qquad |\nabla_SO_U|\leq C|\widehat\chi|^2.
\]
Changing $\widehat\chi$ with $S=S^{\rm trig}$ fixed, then changing
$S$ with $\widehat\chi=\widehat\chi^{\rm quiet}$ fixed, proves
\[
 |\Delta O_U|\leq C\bigl(
 |\widehat\chi^{\rm quiet}|\,|\Delta\widehat\chi|
       +|\Delta\widehat\chi|^2
       +|\widehat\chi^{\rm quiet}|^2|\Delta S|\bigr).
\]
Using the compact impulse estimate and the fine difference variable,
\[
 |\widehat\chi^{\rm quiet}|\leq Cp_nB_nT,
 \qquad \Delta\widehat\chi=p_n\xi,
\]
we obtain
\[
 |\Delta O_U|\leq Cp_n^2\bigl(B_nT|\xi|+|\xi|^2\bigr)
                            +Cp_n^2B_n^2T^2\mathcal E_n.
\]
The coefficient of the last term is absorbed as follows:
\[
 p_n^2B_n^2T^2=(p_n^2B_nT)(B_nT)
       \leq CA_nT,
\]
since $p_n^2B_nT=o_L(1)$ and $B_n\leq CA_n$. At index $n$ the
internal averaging errors and exterior observer errors may be
estimated absolutely:
\[
 |R_U^{\rm patch,\kappa}|
 \leq C\eta_n^2(1+\tau_nd_nM_{2,n}^{\rm rem})
 \leq C\eta_n^2.
\]
This is permissible because
$(1+S)^m\eta_n^2/\varepsilon_n\to0$ for every fixed $m$.

The two remainders are now explicitly
\[
 \begin{aligned}
 \mathcal R_{\rm comp}={}&
 \mathsf U^{\rm quiet}(B_V\Delta\mathsf V+B_A\Delta\mathsf A)
       +\Delta O_U+\Delta\mathcal P_U(\mathcal X,H_{<n})
       +\Delta R_U^{\rm patch},\\
 \mathcal R_{\rm inner}={}&
 \mathcal P_U(\mathcal X^{\rm trig},H_{>n}^{\rm trig})
       -\mathcal P_U(\mathcal X^{\rm quiet},H_{>n}^{\rm quiet})
       -\varepsilon_nf_{\sigma_n,\psi_n}.
 \end{aligned}
\]
The preceding four comparison estimates prove
\eqref{eq:comparison-remainder}. The inner term retains the signed
moment forcing identified in Proposition~\ref{prop:inner-gradient};
its rotating projection is compared with $f_{\sigma_n,\psi_n}$ in
the next proposition.

For the remaining rows, factoring the stable internal difference gives
\[
 \begin{split}
 -a(q^{\rm trig})\mathsf V^{\rm trig}
                +a(q^{\rm quiet})\mathsf V^{\rm quiet}
  &=-\lambda_{\rm st}\Delta\mathsf V
       -\mathsf V^{\rm quiet}(B_U\Delta\mathsf U+B_A\Delta\mathsf A),\\
 \lambda_{\rm st}&=a(q^{\rm trig})+\mathsf V^{\rm quiet}B_V
                                                \geq\lambda/2.
 \end{split}
\]
The second term has coefficient $C(A_n+\eta_n^2)$.
The isolated action is constant. If $g(q)$ denotes the phase velocity,
the phase difference is
\[
 g(q^{\rm trig})-g(q^{\rm quiet})
            =G_U\Delta\mathsf U+G_V\Delta\mathsf V+G_A\Delta\mathsf A,
 \qquad G_j=\int_0^1\partial_jg(q^{\rm quiet}+u\Delta q)\idiff u.
\]
Its three coefficients are bounded, while its offset and exterior
remainders have the bounds just proved. Adding the absolute inner
source bound $C\varepsilon_nF_{\rm env}$ gives
\eqref{eq:parent-triangular-rows}. Thus the order-one phase coefficients
enter the other rows only through their displayed small feedback
coefficients.
\end{proof}

\subsection{Exponential weighting and the positive crossing}

\begin{proposition}[Exponential growth and positive crossing]\label{prop:weighted-crossing}
Use the assumptions of Proposition~\ref{prop:release}. On a current
common stopped interval $[0,s_{\rm end}]$, $s_{\rm end}\leq S$, set
\[
 \begin{gathered}
 \mathfrak w(s)=\exp\!\left(-\int_0^s\lambda_{\rm eff}(v)\idiff v\right),
 \qquad \|F\|_{\mathfrak w}
           =\sup_{0\leq s\leq s_{\rm end}}\mathfrak w(s)|F(s)|,\\
 E_U^*=\varepsilon_n^{-1}\|\Delta\mathsf U_n\|_{\mathfrak w},\qquad
 E_\vartheta^*=\|\Delta\vartheta_n\|_{\mathfrak w},\\
 E_{\rm aux}^*=\|\Delta\mathsf V_n\|_{\mathfrak w}
       +\|\Delta\mathsf A_n\|_{\mathfrak w}+\|\xi\|_{\mathfrak w}
       +\|\mathcal E_{<n}\|_{\mathfrak w}.
 \end{gathered}
\]
The scalar equation and its weighted remainder are
\[
 (\Delta\mathsf U_n)'=\lambda_{\rm eff}\Delta\mathsf U_n
             +\varepsilon_nf_{\sigma_n,\psi_n}+\mathcal R,
 \qquad\lambda_{\rm eff}\geq\lambda/2,
\]
\begin{equation}\label{eq:weighted-remainder-bound}
 \varepsilon_n^{-1}\int_0^{s_{\rm end}}
              \mathfrak w(s)|\mathcal R(s)|\idiff s
            \leq C\zeta+\delta_n(1+E_U^*).
\end{equation}
For fixed $c_{\rm decay}>0$, a fixed integer $M_{\rm loss}$ and a
fixed $C_{\rm led}$ chosen in the proof, one can take
\begin{equation}\label{eq:full-error-ledger}
 \begin{split}
 \delta_n=C_{\rm led}(1+S)^{M_{\rm loss}}\bigl(
      &D_n^{-c_{\rm decay}}+r_n+r_n^{\rm out}+A_n+B_n+p_n^2\\
      &+\mathfrak q_{\rm deep}+I_{\rm deep}
                +\eta_{n+1}^2+\eta_n^2/\varepsilon_n\bigr),
 \qquad \sup_n\delta_n\longrightarrow0.
 \end{split}
\end{equation}
The auxiliary comparison satisfies
\begin{equation}\label{eq:weighted-comparison}
 \big\||\Delta\mathsf V_n|+|\Delta\mathsf A_n|
       +|\Delta\vartheta_n|+|\xi|+\mathcal E_{<n}\big\|_{\mathfrak w}
       \leq C(\varepsilon_n+\eta_n^2)(1+E_U^*).
\end{equation}
There are fixed $0<c_w<C_w$, $C_E,C_\vartheta$ and $s_{\rm ent}$
for which
\[
 \begin{gathered}
 E_U^*\leq C_w,\qquad
 c_w\varepsilon_n\leq\mathfrak w(s)\Delta\mathsf U_n(s)
               \leq C_w\varepsilon_n\quad(s_{\rm ent}\leq s\leq s_{\rm end}),\\
 E_{\rm aux}^*\leq C_E(\varepsilon_n+\eta_n^2)(1+E_U^*),\qquad
 E_\vartheta^*\leq C_\vartheta(\varepsilon_n+\eta_n^2)(1+E_U^*),\\
 \sup_{0\leq s\leq s_{\rm end}}|\xi(s)|
 \leq C\bigl((\varepsilon_n+\eta_n^2)(1+S)^5
                         +B_n(1+S)^4h_*\bigr)=o_L(1).
 \end{gathered}
\]
These bounds strictly improve the comparison conditions of
Definition~\ref{def:common-bootstrap}. Once the geometric estimates
are closed jointly and the quiet trajectory stays in its box, the
positive crossing obeys
\[
 \Delta\mathsf U_n(s_{\rm exit})=h_*,\qquad
                   s_{\rm exit}\leq C(1+\ell_n).
\]
\end{proposition}

\begin{proof}
We estimate the stable, action, phase and impulse differences at index
$n$, together with all differences at $k<n$, before determining
the sign of $\Delta\mathsf U_n$.
All these differences vanish at physical time zero. Before
$t_{n+1}^{\rm out}$ the second moment at index $n+1$ remains bounded
in its own units. Integration over that interval gives
\[
 \mathcal E_n(0)+\mathcal E_{<n}(0)
 \leq C\varepsilon_n\!
       \left(rs_0+(\mathfrak q_{\rm deep}+B_nI_{\rm deep})(s_0+s_0^2)\right)
       +C\eta_n^2s_0 .
\]
The argument $0$ here means $t=t_{n+1}^{\rm out}$, whereas the
vanishing initial differences are at physical time zero.
The homogeneous comparison factor on this preliminary interval tends
uniformly to one. Dividing by the transmitted signal gives explicitly
\[
 \frac{|\Delta\mathsf U_n(0)|}{\varepsilon_n}
 \leq C\left(rs_0+(\mathfrak q_{\rm deep}+B_nI_{\rm deep})
                 (s_0+s_0^2)
              +\frac{\eta_n^2}{\varepsilon_n}s_0\right)=o_L(1),
\]
uniformly in $n$ and $N$, by \eqref{eq:tailclock} and
\eqref{eq:polynomial-margin}. This inherited contribution remains in
the variation-of-constants formula; it is not reset at event time zero.

The reference phase over the subsequent interval differs from
$\vartheta_n^{\mathrm{quiet}}(t_{n+1}^{\rm out})-s$ by at most
$C(A_n+\eta_n^2)(1+s_0+s)^2$.
This follows by integrating
$|\vartheta_n'+1|\leq
C(|\mathsf U_n^{\rm quiet}|+|\mathsf V_n^{\rm quiet}|
   +|\mathsf A_n^{\rm quiet}-1|)+C(A_n+\eta_n^2)$,
where the action error at time $s$ is bounded by
$C(A_n+\eta_n^2)(s_0+s)$.
The endpoint calculation in Section~\ref{sec:geometry} bounds the
change of line-of-centers angle by a fixed polynomial in $1+S$
times $A_n+B_np_n^2$, plus
$Cp_n(\eta_n^2+\eta_{n+1}^2)(1+S)^C$ from covariance errors.
Thus the prescribed angle $\psi_n$ gives the leading term
$\varepsilon_n f_{\sigma_n,\psi_n}$ in \eqref{eq:trigger}.
All changes in projection enter $\mathcal R$.
For the center-line angle this also has the elementary direct check
\[
 \left|\frac{\diff}{\diff s}\arg(c_n-c_{n+1})\right|
 \leq\tau_n\frac{|\dot c_n-\dot c_{n+1}|}{|c_n-c_{n+1}|}.
\]
The endpoint split bounds the common remote contribution by $CA_n$.
The separated smaller configuration has internal mean speed at most
$Cv_{n+1}$, giving $Cv/D$ in the displayed derivative; its contribution
over the current interval is at most $CS/D$, since $v\leq1$.
The compact mean and covariance contributions have the small bounds
stated above. Thus the actual line direction can be frozen at the
initial line at a cost covered by a fixed polynomial times
$A_n+B_n+p_n^2+D^{-1}+\eta_n^2+\eta_{n+1}^2$.

The radial factor in the strain can be frozen with the same uniformity.
Here both means belong to the perturbed evolution. Averaging the bound
$|\dot b_{k,i}-\dot Z|\leq C(v_k+h^\beta R_k)$ from
Section~\ref{sec:geometry} at $k=n,n+1$ cancels $Z'$ and gives
$|\dot c_n-\dot c_{n+1}|\leq C(v_n+h^\beta R_n)$, since
$v_{n+1}\leq v_n$. Define, also on the preliminary interval,
\[
 \boldsymbol z(s)=\frac{c_n(t)-c_{n+1}(t)}{R_n-R_{n+1}},
 \qquad -s_0\leq s\leq S,
 \qquad D_{\rm act}(s)=\frac{|c_n(t)-c_{n+1}(t)|}{d_n}.
\]
Initially $\boldsymbol z(-s_0)=e_1$. The identities
$\tau_n v_n/R_n=p_n$ and
$\tau_n h^\beta\leq CA_n$ imply
$|\boldsymbol z'(s)|\leq C(A_n+p_n)$, hence
$|\boldsymbol z(s)-e_1|\leq C(s_0+s)(A_n+p_n)$.
Since $p_n=(1-q_n)/D$ and this last bound is uniformly small,
$|\boldsymbol z|\geq1/2$. Thus
\begin{equation}\label{eq:radial-strain-freezing}
 \left|\left(\frac{D}{D_{\rm act}(s)}\right)^4-1\right|
 \leq C(s_0+s)(A_n+D^{-1}),\qquad
 \left|\frac{\diff}{\diff s}|\boldsymbol z(s)|^{-4}\right|
 =\left|\frac{-4\boldsymbol z\cdot\boldsymbol z'}{|\boldsymbol z|^6}\right|
 \leq C(A_n+D^{-1}).
\end{equation}
The normalized leading moment is bounded by $C(1+s)$, so the first
inequality contributes at most
$C(1+s)(s_0+s)(A_n+D^{-1})$ to the forcing divided by $\varepsilon_n$.
Its integral against any weight bounded by $e^{-\lambda s/2}$ is
$O((1+s_0)(A_n+D^{-1}))$. This accounts for the distance at the
section as well as its subsequent change, using the prescribed
initial centers at physical time zero.

Fix the current stopped endpoint $s_{\rm end}\leq S$ and use the
weighted quantities defined in the proposition.
The following elementary estimate explains why neutral equations do
not produce a factor $S$ in the weighted comparison. For
$0\leq v\leq s$,
$\mathfrak w(s)/\mathfrak w(v)\leq e^{-\lambda(s-v)/2}$; hence, if
$\mu\leq\lambda/4$,
\begin{equation}\label{eq:weighted-convolution}
 \left\|\int_0^s e^{\mu(s-v)}F(v)\idiff v\right\|_{\mathfrak w}
 \leq\frac{\|F\|_{\mathfrak w}}{\lambda/2-\mu}
 \leq\frac4\lambda\|F\|_{\mathfrak w}.
\end{equation}
The stable integrating factor is smaller still. For the rows with
$k<n$, use $\mu=C(r_n^{\rm out}+A_n)<\lambda/4$.
Also $\|F_{\rm env}\|_{\mathfrak w}\leq C$, since $s_0=o_L(1)$
and exponential decay bounds each fixed polynomial.

First control the quadratic term in $\xi$ without weights. On the
stops $|\Delta\mathsf U_n|\leq h_*$ and $|\xi|\leq1$, integrate
the stable row, then the action and phase rows, and finally the
impulse rows and the rows with $k<n$. To exhibit the absorption, put
$T=1+S$, $e=\varepsilon_n+\eta_n^2$ and let $X,P,Y$ be the
unweighted suprema on the current interval of, respectively,
$|\xi|$, $|\Delta\vartheta_n|$, and
$|\Delta\mathsf V_n|+|\Delta\mathsf A_n|+\mathcal E_{<n}$.
The stable integral kernel has integral at most $2/\lambda$;
the neutral kernels have integral at most $T$. The kernel for
$\mathcal E_{<n}$ has an extra factor
$e^{C(r_n^{\rm out}+A_n)T}\leq2$ for large $L$.
The polynomial source has integral $CeT^5$.
Using $X^2\leq X$ on the provisional stop, the displayed differential
equations therefore imply
\[
 \begin{split}
 Y&\leq CeT^5+\delta_0(h_*+Y+P+X),\\
 P&\leq CT(h_*+Y)+CeT^5+\delta_0(Y+P+X),\\
 X&\leq CeT^5+CB_nT(h_*+Y+P)+\delta_0X,
 \end{split}
\]
where one may take
$\delta_0=CT^2(A_n+\eta_n^2+p_n^2)+Cr_n^{\rm out}T$.
Every fixed power of $T$ times $\delta_0$ tends uniformly to zero.
Absorb $\delta_0Y$ and $\delta_0P$ in the first two rows and
substitute them into the third. Their order-one coupling is only
$CTY$ in the phase row, so its feedback is bounded by $C\delta_0T$.
The source contribution to $P$ is at most $CeT^6$; after insertion
in the last row its additional factor $B_nT$ makes it at most
$CeT^5$, since $B_nT^2=o_L(1)$. The $h_*$ contribution to that
row is at most $CB_nT^2h_*$. Absorbing the remaining coefficient
of $X$ gives, in particular, the conservative bound
\[
 \sup_{0\leq s\leq s_{\rm end}}|\xi(s)|
 \leq C\bigl((\varepsilon_n+\eta_n^2)(1+S)^5
                         +B_n(1+S)^4h_*\bigr)=o_L(1).
\]
For example, the direct term $G_U\Delta\mathsf U_n$ contributes at
most $C(1+S)h_*$ to phase, and its subsequent contribution to $\xi$
has the additional coefficient $B_n$ and one time integration.
The remaining integrations and their small feedback terms are bounded
by the displayed larger powers. The estimate improves the stop
$|\xi|\leq1$ and bounds $|\xi|^2$ by $o_L(1)|\xi|$.

Let $\delta_{\rm abs}$ be a sufficiently large fixed polynomial in
$1+S$ times $A_n+B_n+p_n^2+\eta_n^2$. It tends uniformly to zero.
Applying \eqref{eq:weighted-convolution} to
\eqref{eq:head-impulse-system} and \eqref{eq:parent-triangular-rows}
gives
\[
 \begin{split}
 E_{\rm aux}^*&\leq C(\varepsilon_n+\eta_n^2)
       +C\delta_{\rm abs}
            (\varepsilon_nE_U^*+E_{\rm aux}^*+E_\vartheta^*),\\
 E_\vartheta^*&\leq
 C(\varepsilon_n+\eta_n^2+\varepsilon_nE_U^*+E_{\rm aux}^*)
                         +C\delta_{\rm abs}E_\vartheta^* .
 \end{split}
\]
For example, the impulse row contributes
\[
 \|\xi\|_{\mathfrak w}
 \leq C(e+|\xi(0)|)
       +CB_n(\varepsilon_nE_U^*+E_{\rm aux}^*+E_\vartheta^*),
\]
after absorbing its $A_n$ coefficient. The row for larger indices
has the corresponding terms
$C(p_n^3\varepsilon_nE_U^*+p_n^2\|\xi\|_{\mathfrak w})$.
The stable and action rows contribute the small comparison coefficients
from \eqref{eq:parent-triangular-rows}; the phase row has in addition
$C\varepsilon_nE_U^*+C\|\Delta\mathsf V_n\|_{\mathfrak w}
+C\|\Delta\mathsf A_n\|_{\mathfrak w}$.
The unweighted bound just proved replaces $|\xi|^2$ by
$o_L(1)|\xi|$ in all these rows. These are exactly the two displayed
inequalities, with no smallness requirement on $G_U,G_V,G_A$.
First absorb the last term in the second row. Substitution into the
first row leaves a coefficient $C'\delta_{\rm abs}$ multiplying
$E_{\rm aux}^*$, which can also be absorbed. Substituting back gives
\eqref{eq:weighted-comparison}.
This argument permits order-one dependence of phase on shape.
Smallness is needed only for its contribution back to the other rows.

Integrating \eqref{eq:comparison-remainder} against $\mathfrak w$
and dividing by $\varepsilon_n$ now gives a bound
\[
 C(1+S)^C
     (A_n+B_n+p_n^2+\eta_n^2/\varepsilon_n)(1+E_U^*).
\]
The other errors come from the second-moment estimate
\eqref{eq:material-donor}, the impulse, the higher Taylor terms
and the change of the reference projection. To make the last
comparison uniform on the entire interval, split it at
$s_{\rm cut}=c_{\rm cut}\log D_n$.
We do this while the provisional bound $E_U^*\leq C_*$ is still
in place, without using positivity of $\Delta\mathsf U_n$.
On the early interval, the bound on $\lambda_{\rm eff}$ and the
definition of $E_U^*$ give
\[
 |\Delta\mathsf U_n(s)|\leq C_*\varepsilon_ne^{\lambda_+s},
 \qquad
 |\lambda_{\rm eff}(s)-\lambda|
 \leq C\left(\varepsilon_ne^{\lambda_+s}
           +(A_n+\eta_n^2)(1+s_0+s)
           +p_n^2B_n^2(s_0+s)^2\right).
\]
The second inequality follows by the mean-value formula for $a$,
the absolute stable/action estimates, and the even offset dependence.
Its integral is bounded by the expression below after increasing
$C_{\rm Gr}\geq\lambda_+$ and the fixed integer $M_{\rm loss}$.
The projection comparisons have the same bound for an explicit
reason. Equation~\eqref{eq:weighted-comparison}, on the provisional
$E_U^*$ stop, gives
$\mathcal E_n+\mathcal E_{<n}\leq
C(\varepsilon_n+\eta_n^2)e^{\lambda_+s}$.
The map from an affine gradient and the index-$n$ centroid coordinates
to its $\mathsf U$ derivative has bounded first derivatives on the
fixed chart. The reference shape and phase differ from their rigidly
rotating values by $C(A_n+\eta_n^2)(1+s_0+s)^2$; the center-line
change has the bound established above. Applying the mean-value
formula to this map, and using the polynomial bound on
$P_{\sigma,\theta}$, bounds its normalized projection error by
a fixed polynomial times the sum of these coordinate and angle errors.
Together with the integrated rate bound, this gives
\[
 Ce^{C_{\rm Gr}s_{\rm cut}}(1+s_{\rm cut})^{M_{\rm loss}}
       (\varepsilon_n+A_n+\eta_n^2+D_n^{-1}).
\]
Here $C_{\rm Gr}$ is fixed and $M_{\rm loss}$ is one fixed integer
large enough for all polynomial factors in this proof.
The parentheses are $O_q(D_n^{-1/2})$: indeed
$A_n=C_{\rm evt}h^{3/2}(1-q_n)^{1/2}D_n^{-1/2}$,
$\varepsilon_n\leq D_n^{-4}$ and $\eta_n^2\leq\varepsilon_n$.
Choose
$0<c_{\rm cut}<\min\{1/(4C_{\rm Gr}),4/\lambda_+\}$,
where $\lambda_+$ bounds $\lambda_{\rm eff}$ on the stopped chart.
The early comparison then costs
$CD_n^{-1/4}(1+S)^{M_{\rm loss}}$.
In particular, if
$J(s)=\int_0^s(\lambda_{\rm eff}(v)-\lambda)\idiff v$, then
$\sup_{s\leq s_{\rm cut}}|J(s)|$ has this bound and
\[
 |\mathfrak w(s)-e^{-\lambda s}|
 =e^{-\lambda s}|e^{-J(s)}-1|
 \leq Ce^{-\lambda s}D_n^{-1/4}(1+S)^{M_{\rm loss}}.
\]
Multiplying by $|f_{\sigma_n,\psi_n}(s)|\leq C(1+s)$ and
integrating gives the same allowance for the weight replacement.
On the remainder of the interval the normalized forcing is bounded
by $C(1+s)^4$, and its weighted integral is at most
$CD_n^{-\lambda c_{\rm cut}/2}(1+S)^{M_{\rm loss}}$.
Indeed repeated integration by parts, or comparison of the polynomial
with its value at the lower endpoint times a polynomial in
$s-s_{\rm cut}$, gives
$\int_b^\infty e^{-\lambda s/2}(1+s)^4\idiff s
\leq Ce^{-\lambda b/2}(1+b)^4$.
The constant weight has an even smaller tail. This proves the
replacement on every stopped interval, including intervals ending
before $s_{\rm cut}$.
If the current interval ends before $s_{\rm cut}$, only the first
estimate is used. On a bounded $E_U^*$ stop any positive crossing
satisfies $h_*\leq C\varepsilon_ne^{\lambda_+s_{\rm exit}}$.
Since $\ell_n\geq4\log D_n$, our choice also puts $s_{\rm cut}$
before that crossing for large $L$.

Set $c_{\rm decay}=\min\{1/4,\lambda c_{\rm cut}/2\}>0$ and choose a fixed
large $C_{\rm led}$. To assemble the errors, use
$\int_0^\infty e^{-\lambda s/2}(1+s)^j\idiff s\leq C_j$ for
each fixed integer $j$. The separate normalized contributions have the bounds
\[
 \begin{aligned}
 \varepsilon_n^{-1}\int_0^{s_{\rm end}}
       \mathfrak w|\mathcal R_{\rm moment}|\idiff s
 &\leq C\zeta+C\eta_{n+1}^2+Cr_n+C(A_n+B_n)(1+s_0)^C,\\
 \varepsilon_n^{-1}\int_0^{s_{\rm end}}
       \mathfrak w|\mathcal R_{\rm impulse}|\idiff s
 &\leq CB_n(1+s_0)^C,\\
 \varepsilon_n^{-1}\int_0^{s_{\rm end}}
       \mathfrak w|\mathcal R_{\rm deep}|\idiff s
 &\leq C(\mathfrak q_{\rm deep}+B_nI_{\rm deep})(1+s_0)^C,\\
 \varepsilon_n^{-1}\int_0^{s_{\rm end}}
       \mathfrak w|\mathcal R_{\rm Taylor}|\idiff s
 &\leq CD_n^{-1}(1+S)^C,\\
 \varepsilon_n^{-1}\int_0^{s_{\rm end}}
       \mathfrak w|\mathcal R_{\rm comp}|\idiff s
 &\leq C(1+S)^C
       (A_n+B_n+p_n^2+\eta_n^2/\varepsilon_n)(1+E_U^*),\\
 \varepsilon_n^{-1}\int_0^{s_{\rm end}}
       \mathfrak w|\mathcal R_{\rm projection}|\idiff s
 &\leq CD_n^{-c_{\rm decay}}(1+S)^{M_{\rm loss}}.
 \end{aligned}
\]
Here the moment term includes the reference second moment; the impulse
term includes its change of origin. The last allowance also bounds
the replacement of the exponential weight in the model coefficient.
Choose the integer $M_{\rm loss}$ at least as large as every polynomial
power in this list, and use $s_0\leq CL(n+2)^2r_n$.
The finite sum is bounded by $\delta_n$ in
\eqref{eq:full-error-ledger}. This proves the full integral estimate
\eqref{eq:weighted-remainder-bound} on $[0,s_{\rm end}]$.
The replacement of $\mathfrak w f_{\sigma_n,\psi_n}$ by
$e^{-\lambda s}f_{\sigma_n,\psi_n}$ in the integral has the same
vanishing allowance, without the $C\zeta$ term.
The discrepancy $\zeta$ in \eqref{eq:material-donor} multiplies a
fixed polynomial in $s$. For each of its powers $m$,
\[
 C\zeta\int_0^S e^{-\lambda s/2}(1+s)^m\idiff s
 \leq C_m\zeta,
\]
uniformly in $S$.

Every term of $\delta_n$ tends to zero uniformly in $n$. In particular,
the deeper-scale impulse factor must be estimated as a product:
\[
 I_{\rm deep}=q_n^{\nu_\Gamma-3}q_{n+1}^{2\nu_\Gamma}
   =e^{-L[(6\nu_\Gamma-6)n+7\nu_\Gamma-3]}.
\]
Although its first factor can grow, the product decays uniformly,
since $\nu_\Gamma>2$. This decay absorbs every fixed polynomial in
$1+S\leq C L(n+1)^2$. For $\eta_n^2/\varepsilon_n$, the completed
square in \eqref{eq:scale-quadratic} yields exactly the uniform
polynomial margin \eqref{eq:polynomial-margin}. The remaining scale
factors in \eqref{eq:full-error-ledger} have the same uniform margins
from Lemma~\ref{lem:scale-margins}. All polynomial degrees and
comparison constants have been fixed independently of $n,N$ before
increasing $L$.
Variation of constants in \eqref{eq:trigger} consequently gives
\[
 E_U^*\leq C_{\rm boot}+C\zeta+\delta_n(1+E_U^*).
\]
Choose a provisional bound larger than $2(C_{\rm boot}+1)$,
then choose $\zeta_0$ small and $L$ large. This inequality
strictly improves that bound.

For the sign, write $f(v)=f_{\sigma_n,\psi_n}(v)$ and
$\mathcal H_{\sigma_n}(\psi_n)=\int_0^\infty e^{-\lambda v}f(v)\idiff v$.
Variation of constants separates all five errors:
\begin{equation}\label{eq:positive-input-budget}
 \begin{aligned}
 \frac{\mathfrak w(s)\Delta\mathsf U_n(s)}{\varepsilon_n}
 ={}&\mathcal H(\psi_n)
       +\frac{\Delta\mathsf U_n(0)}{\varepsilon_n}
       +(\mathcal H_{\sigma_n}-\mathcal H)(\psi_n)\\
 &-\int_s^\infty e^{-\lambda v}f(v)\idiff v
       +\int_0^s(\mathfrak w(v)-e^{-\lambda v})f(v)\idiff v\\
 &+\varepsilon_n^{-1}\int_0^s\mathfrak w(v)\mathcal R(v)\idiff v.
 \end{aligned}
\end{equation}
On the selected phase interval, $\mathcal H(\psi_n)\geq m_*$.
Choose the fixed entrance time $s_{\rm ent}$ so that the model tail
has absolute value at most $m_*/16$ for every $s\geq s_{\rm ent}$;
the polynomial bound for $f$ makes this uniform for small $r_n$.
Equation~\eqref{eq:coefficient} bounds the coefficient error by $Cr_n$.
The inherited initial term has the normalized bound established above.
The early/late weight comparison controls the fourth term, even though
$f$ can change sign. Finally, \eqref{eq:weighted-remainder-bound}
controls the last term using the current provisional upper bound for
$E_U^*$. First decrease $\zeta_0$ and then increase $L$ so that each
of these four remaining errors also has absolute value at most
$m_*/16$, uniformly in the index and cutoff. Consequently,
\[
 \mathfrak w(s)\Delta\mathsf U_n(s)
       \geq\frac{11m_*}{16}\varepsilon_n
       >\frac{m_*}{4}\varepsilon_n
       \qquad(s_{\rm ent}\leq s\leq s_{\rm end}).
\]
Before $s_{\rm ent}$, fixed-time smallness excludes both exits.
This proves \eqref{eq:certificate} together with its upper bound.
Since $\mathfrak w(s)\leq e^{-\lambda s/2}$, it forces
$\Delta\mathsf U_n=h_*$ by time $C(1+\ell_n)$ unless another
stop occurs. Section~\ref{sec:geometry} improves all the other
bounds on that same interval; the joint continuation argument in
Section~\ref{sec:finite} therefore supplies the crossing.
\end{proof}

\subsection{Section shape, elapsed time and outgoing direction}
\label{subsec:outgoing-laws}

The initial discrepancy $\zeta$ can change the time of the positive
crossing by $O(\zeta)$. For iteration, the shape error at the next
section must nevertheless tend to zero as $L$ grows, uniformly for
$0\leq\zeta\leq\zeta_0$. A bound of the form
$\zeta_{\rm next}\leq C\zeta+o_L(1)$ would permit amplification through
arbitrarily many steps and would not suffice. We instead obtain an
absolute outgoing shape error $\zeta_{\rm out}(L)\to0$, while keeping
the dependence on the incoming discrepancy in the separate timing law.
Absolute estimates for the coordinates
$\mathsf V_n^{\mathrm{trig}}$, $\mathsf A_n^{\mathrm{trig}}-1$
and $\chi_n^{\mathrm{trig}}/d_n$ give the required vanishing shape error.

\begin{lemma}[Absolute shape at the positive crossing]\label{lem:section-shape}
Use the assumptions of Proposition~\ref{prop:release} and let
$s_{\rm exit}$ be the positive crossing supplied by the joint closure.
There is $\epsilon_{\rm out,n}$ such that
\[
 \begin{split}
 &\int_0^{s_{\rm exit}}\bigl(|\mathsf V_n^{\rm trig}|
                         +|\mathsf A_n^{\rm trig}-1|\bigr)\,\mathrm ds\\
 &\quad+\sup_{s\leq s_{\rm exit}}
 \bigl(|\mathsf V_n^{\rm trig}|+|\mathsf A_n^{\rm trig}-1|
                       +|\chi_n^{\rm trig}|/d_n\bigr)
 \leq\epsilon_{\rm out,n},\qquad
 \sup_n\epsilon_{\rm out,n}\longrightarrow0.
 \end{split}
\]
At the crossing,
\[
 \Delta\mathsf U_n(s_{\rm exit})=h_*,\qquad
 (\mathsf U_n^{\rm trig},\mathsf V_n^{\rm trig},
       \mathsf A_n^{\rm trig},\chi_n^{\rm trig}/d_n)(s_{\rm exit})
                         =(h_*,0,1,0)+o_L(1).
\]
The centered full centroid state therefore converges to the isolated
state after matching its rotation.
The convergence is uniform for all incoming discrepancies
$0\leq\zeta\leq\zeta_0$.
\end{lemma}

\begin{proof}
The absolute equations \eqref{eq:actual-quiet-equations} apply to the
perturbed evolution throughout the small shape chart. In the variable
$s=(t-t_{n+1}^{\rm out})/\tau_n$ they imply
\[
 |\mathsf V_n^{\mathrm{trig}}(s)|\leq C(A_n+\eta_n^2),\qquad
 |\mathsf A_n^{\mathrm{trig}}(s)-1|
             \leq C(A_n+\eta_n^2)(s_0+s).
\]
The first estimate uses the negative internal rate and the initial
value $\mathsf V_n(0)=0$ at physical time zero; the second is a direct
time integral from the initial value $\mathsf A_n=1$.
Their integrated cost is
$C(A_n+\eta_n^2)(1+s_0+S)^2=o_L(1)$.
The compact impulse equation, also integrated from zero, gives
$|\chi_n^{\mathrm{trig}}|/d_n\leq Cp_nB_n(s_0+s)$.
Together with the absolute inner forcing bound
$C\varepsilon_n(1+s)^4$, these inequalities yield
\[
 \begin{split}
 &\int_0^{s_{\rm exit}}
    (|\mathsf V_n^{\mathrm{trig}}|
               +|\mathsf A_n^{\mathrm{trig}}-1|)\idiff s\\
 &\quad+\sup_{0\leq s\leq s_{\rm exit}}
    \left(|\mathsf V_n^{\mathrm{trig}}|
       +|\mathsf A_n^{\mathrm{trig}}-1|
       +|\chi_n^{\mathrm{trig}}|/d_n\right)
 \leq\epsilon_{\rm out,n},\qquad
 \sup_n\epsilon_{\rm out,n}\longrightarrow0 .
 \end{split}
\]
Here $s_{\rm exit}$ is the first crossing
$\Delta\mathsf U_n=h_*$.
One may take $\epsilon_{\rm out,n}$ to be a fixed polynomial in
$1+S$ times $\varepsilon_n+\eta_n^2+A_n+B_n+p_n^2$.
It also bounds the integrated centroid and phase remainders.
At the crossing $\mathsf U_n^{\rm trig}=h_*+\mathsf U_n^{\rm quiet}
=h_*+o_L(1)$. The inverse chart has bounded derivatives on the
fixed separated set, so these absolute coordinate estimates give
the claimed full centroid discrepancy.
\end{proof}

\begin{lemma}[Transverse continuation to the outgoing section]
\label{lem:transverse-section}
Start at the positive crossing $t_n^*$ with the entry state from
Lemma~\ref{lem:section-shape}. Suppose the patch bounds of
Proposition~\ref{prop:geometry-continuation} hold on every current stopped
interval in the $2\delta_{\rm arc}$ tube of
Definition~\ref{def:common-bootstrap}. Then the state cannot reach
the tube boundary before crossing $\mathsf R=\mathsf R_{\rm out}$.
There is a unique such crossing near the isolated crossing, it is
transverse, and
\[
 (t_n^{\rm out}-t_n^*)/\tau_n=T_{\rm arc}+o_L(1),\qquad
 \zeta_{\rm out}(L)\longrightarrow0.
\]
The latter is the maximum discrepancy of all four centered
normalized centroids from the rotated isolated outgoing state.
All errors are uniform in $n,N$ and $\zeta\leq\zeta_0$.
\end{lemma}

\begin{proof}
Use $u=(t-t_n^*)/\tau_n$ and the full centered state
$Y=(y_1,\ldots,y_4)$, $y_j=(b_{n,j}-\boldsymbol c_n)/d_n$.
Its point field is
\[
 F_j(Y)=\sum_{\ell\ne j}\epsilon_\ell K(y_j-y_\ell)
 -\frac14\sum_{m=1}^4\sum_{\ell\ne m}
                          \epsilon_\ell K(y_m-y_\ell).
\]
The exact averaged centroid equations imply
\[
 Y'=F(Y)+E(u),\qquad
 |E(u)|\leq C_{\rm arc}\left(
 \tau_nM_{1,n}^{\rm rem}
 +\eta_n^2(1+\tau_nd_nM_{2,n}^{\rm rem})\right)=:\rho_n(L).
\]
Indeed the constant remote velocity is removed by centering, its
variation costs $\tau_nM_1$, and the averaging errors cost the
variance times the relevant Hessian. The normalized offset
$\chi_n/d_n$ is part of $Y$. The remote estimates
and scale margins give $\sup_n\rho_n(L)\to0$.

Let $Y^0(u)$ be the isolated extended orbit with matched entry
rotation. On the closed separated tube, $F$ has a fixed Lipschitz
constant $L_{\rm arc}$. Put $\delta_{\rm in,n}=|Y(0)-Y^0(0)|$.
For every current stopped $u\leq T_+$, Gr\"onwall gives
\[
 |Y(u)-Y^0(u)|\leq
 e^{L_{\rm arc}T_+}\bigl(\delta_{\rm in,n}+T_+\rho_n(L)\bigr)
 =:\delta_{\rm arc,n}(L),\qquad
 \sup_n\delta_{\rm arc,n}(L)\to0.
\]
The exponential factor is fixed; it does not involve $\ell_n$.
Thus the state stays within $\delta_{\rm arc}$ of its matched
orbit and cannot reach the boundary of the larger tube.

For an explicit section argument, set
$g(Y)=\mathsf R(Y)-\mathsf R_{\rm out}$. At $u=T_{\rm arc}$,
$\nabla g(Y^0)\cdot F(Y^0)=g_*>0$. Fix a sufficiently small
$a>0$ with $0<T_{\rm arc}-a<T_{\rm arc}+a<T_+$ such that
$\nabla g\cdot F\geq g_*/2$ throughout the corresponding tube
near the section. Then
\[
 g(Y^0(T_{\rm arc}-a))\leq-g_*a/2,\qquad
 g(Y^0(T_{\rm arc}+a))\geq g_*a/2.
\]
Uniform closeness gives the same strict signs for the actual state
for all large $L$. A crossing exists by continuity. On this section
neighborhood,
\[
 \frac{\mathrm d}{\mathrm du}g(Y(u))
 =\nabla g(Y(u))\cdot(F(Y(u))+E(u))\geq g_*/4,
\]
so the crossing is unique there and transverse. The isolated
trajectory has a strictly negative maximum of $g$ on every earlier
closed subinterval separated from its crossing; hence closeness
also excludes an earlier outgoing crossing. If $u_n^{\rm out}$
is the actual crossing time, evaluating $g$ at $T_{\rm arc}$ and
using the last lower derivative bound gives
\[
 |u_n^{\rm out}-T_{\rm arc}|
 \leq C\delta_{\rm arc,n}(L)/g_*.
\]
The Lipschitz bound on $Y^0$ then gives
$|Y(u_n^{\rm out})-Y^0(T_{\rm arc})|\leq
C\delta_{\rm arc,n}(L)$, which proves the stated outgoing shape
estimate. The geometry proposition supplies concentration on the
stopped tube before this argument; this argument supplies the
crossing needed to end that tube interval.
\end{proof}

\begin{lemma}[Section time]\label{lem:section-clock}
Under the assumptions of Proposition~\ref{prop:release}, for $s_k^{\rm out}=t_k^{\rm out}/\tau_k$, the absolute normalized
outgoing times satisfy
\begin{equation}\label{eq:clock}
 s_n^{\rm out}=r_ns_{n+1}^{\rm out}+\frac{\ell_n}{\lambda}
       +C_{\rm time}-\frac1\lambda\log\mathcal H(\psi_n)
                         +O_{m_*}(\zeta)+o_L(1).
\end{equation}
The final error tends to zero uniformly in $n,N$ and
$0\leq\zeta\leq\zeta_0$.
\end{lemma}

\begin{proof}
To compute the elapsed time more accurately, set
$a_*(v)=a(v,0,1)$ and
\[
 C_{\rm delay}
       =\int_0^{h_*}\frac{a_*(v)-\lambda}{a_*(v)v}\idiff v.
\]
This integral is finite because $a_*(v)-\lambda=O(v)$.
Split the orbit at $\Delta\mathsf U_n=\rho_{\rm split}>0$.
Before its first such crossing, the lower and upper bounds in
\eqref{eq:certificate}, applied backward from that endpoint, give
$\int|\Delta\mathsf U_n|\idiff s\leq C\rho_{\rm split}$.
The absolute estimates above control all other coordinates.
After that crossing the scalar equation converges to
$u'=a_*(u)u$ on the compact interval
$[\rho_{\rm split},h_*]$, where its velocity is positive.
For precision, put $u=\Delta\mathsf U_n$. The reference coordinate
tends uniformly to zero, the stable and action errors have both
vanishing suprema and vanishing integrals, and the unweighted
comparison bounds give at most a fixed polynomial in $1+S$ for the
remaining coordinates. Each coefficient multiplying that polynomial
in \eqref{eq:comparison-remainder} tends to zero after the same
polynomial loss. The moment forcing is bounded by
$C\varepsilon_n(1+S)^4(1+\zeta)$, which tends to zero uniformly
also for $\zeta\leq\zeta_0$. It follows that
\[
 u'=a_*(u)u+o_L(1),\qquad
 \lambda_{\rm eff}=a_*(u)+o_L(1)
\]
uniformly on this part of the orbit. For fixed $\rho_{\rm split}$,
$a_*(u)u\geq c\rho_{\rm split}$, so $u$ is increasing and its
inverse can be used as a time coordinate for all large $L$.
Thus
\[
 \int_{u=\rho_{\rm split}}^{u=h_*}
       (\lambda_{\rm eff}-\lambda)\idiff s
 =\int_{\rho_{\rm split}}^{h_*}
       \frac{a_*(v)-\lambda}{a_*(v)v}\idiff v+o_L(1).
\]
On the preceding part, smoothness gives
$|\lambda_{\rm eff}-\lambda|
\leq C|u|+C(|\mathsf U_n^{\rm quiet}|
                   +|\mathsf V_n^{\rm trig}|
                   +|\mathsf A_n^{\rm trig}-1|)$.
Its integral is $O(\rho_{\rm split})+o_L(1)$, including the fixed
initial interval, where all differences are $O(\varepsilon_n+\eta_n^2)$.
The missing part of the last integral, from zero to
$\rho_{\rm split}$, is also $O(\rho_{\rm split})$ because
$a_*(v)-\lambda=O(v)$.
First let $L\to\infty$, then $\rho_{\rm split}\downarrow0$.
It follows that
\[
 \int_0^{s_{\rm exit}}(\lambda_{\rm eff}(s)-\lambda)\idiff s
                    =C_{\rm delay}+o_L(1),
\]
uniformly for $\zeta\leq\zeta_0$.

The coefficient inside the logarithmic time formula must also be
identified. Define
\[
 B_{\rm ent}
     =\varepsilon_n^{-1}h_*
       \exp\!\left(-\int_0^{s_{\rm exit}}\lambda_{\rm eff}(s)\idiff s\right).
\]
Variation of constants expresses this as
$\varepsilon_n^{-1}\Delta\mathsf U_n(0)$, plus the integral of
$\mathfrak w f_{\sigma_n,\psi_n}$, plus
$\varepsilon_n^{-1}\int\mathfrak w\mathcal R$.
The initial estimate, \eqref{eq:weighted-remainder-bound}, and the
bounded $E_U^*$ control the first and third terms. The comparison
up to $c_{\rm cut}\log D_n$ and the exponential tail estimate
control replacement of the second by $\mathcal H_{\sigma_n}(\psi_n)$.
Thus, increasing the same vanishing error allowance if necessary,
\begin{equation}\label{eq:endpoint-coefficient}
 |B_{\rm ent}-\mathcal H_{\sigma_n}(\psi_n)|
                           \leq C\zeta+\delta_n.
\end{equation}
The endpoint tends uniformly to infinity, since
$h_*\leq C\varepsilon_ne^{\lambda_+s_{\rm exit}}$.
Consequently its omitted exponential tail tends to zero as well.

On $\mathcal I$, the lower bound
$\mathcal H_{\sigma_n}\geq m_*/2$ makes the logarithm Lipschitz
with a constant depending only on $m_*$. The identity
\[
 \lambda s_{\rm exit}
   +\int_0^{s_{\rm exit}}(\lambda_{\rm eff}(s)-\lambda)\idiff s
       =\ell_n+\log h_*-\log B_{\rm ent}
\]
therefore determines the crossing time. Let $C_{\rm time}$ be
$(\log h_*-C_{\rm delay})/\lambda$ plus the fixed time, in the
normalized variable, along the isolated orbit from $\mathsf U=h_*$
to $\mathsf R=\mathsf R_{\rm out}$. Since
$s_0=r_ns_{n+1}^{\rm out}$, adding the time already elapsed gives
\eqref{eq:clock}.

\end{proof}

\begin{lemma}[Outgoing direction]\label{lem:section-angle}
Under the assumptions of Proposition~\ref{prop:release}, the outgoing direction and shape discrepancy satisfy
\begin{equation}\label{eq:angle}
 \Theta_n=\vartheta_n^0-s_n^{\rm out}+C_{\rm out}+o_L(1),
 \qquad \zeta_{\rm out}(L)\longrightarrow0.
\end{equation}
The errors are uniform in $n,N$ and $0\leq\zeta\leq\zeta_0$.
The $O(\zeta)$ dependence of the actual time is retained in
$s_n^{\rm out}$.
\end{lemma}

\begin{proof}
For the angle, the smooth primitive \eqref{eq:phase-primitive}
and the exact perturbed equations give
\[
 \left|(\vartheta_n^{\mathrm{trig}}+s
                         -P(\mathsf U_n^{\mathrm{trig}}))'\right|
 \leq C\left(|\mathsf V_n^{\mathrm{trig}}|
       +|\mathsf A_n^{\mathrm{trig}}-1|
       +(|\chi_n^{\mathrm{trig}}|/d_n)^2+A_n+\eta_n^2\right).
\]
Indeed, if $g(\mathsf U,\mathsf V,\mathsf A)$ denotes the phase
velocity, then $g(u,0,1)=-1+g_*(u)$ and
$P'(u)a_*(u)u=g_*(u)$. In
$\vartheta'+1-P'(\mathsf U)\mathsf U'$, these two terms cancel.
The mean-value theorem in $\mathsf V,\mathsf A$ bounds the remaining
internal terms by $C(|\mathsf V|+|\mathsf A-1|)$, and boundedness
of $P'$ multiplies the exterior and offset remainders by a fixed
constant. The equation is valid at $\mathsf U=0$, since $P'$ has
already been extended smoothly there.
Its integral tends to zero. The same estimates on the earlier
interval from physical time zero show that at the crossing
\[
 \vartheta_n^{\mathrm{trig}}
        +s_0+s_{\rm exit}-P(\mathsf U_n^{\mathrm{trig}})
                        =\vartheta_n^0+o_L(1),
\]
since the initial unstable coordinate, and hence its value of $P$,
tends to zero.

At this crossing the normalized shape and offset converge to
$(\mathsf U,\mathsf V,\mathsf A,\chi/d_n)=(h_*,0,1,0)$.
Lemma~\ref{lem:transverse-section} supplies the outgoing crossing,
with its normalized time and phase converging to the isolated values.
In particular its shape discrepancy is
$\zeta_{\rm out}(L)\to0$, uniformly for $\zeta\leq\zeta_0$.
Define $C_{\rm out}$ on the isolated section to be the outgoing
direction angle minus the mean phase plus $P(\mathsf U)$.
Rotational covariance makes this value independent of the rotation.
The continuous outgoing-direction coordinate applied to the actual
section state now gives \eqref{eq:angle}.
\end{proof}

Proposition~\ref{prop:release} follows from
Propositions~\ref{prop:inner-gradient}, \ref{prop:head-comparison},
\ref{prop:weighted-crossing}, and Lemmas~\ref{lem:section-shape}--\ref{lem:section-angle},
using the joint continuation of Proposition~\ref{prop:bootstrap-closure}.

The finite induction uses the positive weighted bound, the
vanishing outgoing shape discrepancy, and the section time and
angle laws. All estimates above hold up to the current stopping time;
Proposition~\ref{prop:bootstrap-closure} continues the geometry and
comparison conditions together.
\section{Completion of the construction}\label{sec:completion}

Fix the exponent $q=\max(p,3)$ and the scales of
Section~\ref{sec:scales}.  This section combines three independent
steps.  First we establish a compactness criterion which retains exact
transport and also supplies continuous dependence for every fixed finite
family of patches.  We then select a quiet solution by an exit-map
argument and a triggered solution by nested intervals of one initial
rotation angle.  Finally we pass from their index-zero shape gap to one
fixed smooth observable and to a common strong initial datum. The finite
selection uses the analytic facts established first.

\subsection{Compactness with exact vorticity transport}\label{sec:limit}

We prove that a common strong initial limit and a fixed smooth-test
difference at times $t_N\geq t_0>0$ yield the required Euler solutions.
The needed convergence is
\[
 \omega_N^\kappa\longrightarrow\omega^\kappa
       \quad\hbox{in }C([0,1];L^1\cap L^q).
\]
Writing $F_N^\kappa(t)$ for the inverse flow, the identity
$\omega_N^\kappa(t)=\omega_{0,N}^\kappa\circ F_N^\kappa(t)$ reduces this issue to
stability of $F_N^\kappa$ uniformly in the terminal time $t$.

\begin{proposition}[From finite instability to Lagrangian nonuniqueness]
\label{prop:criterion}
Let $q=\max(p,3)$ as above. For each $N$, let $\omega_N^\kappa$,
$\kappa\in\{\mathrm{quiet},\mathrm{trig}\}$, be the canonical
bounded-vorticity Euler evolutions of real bounded initial data
$\omega_{0,N}^\kappa$. Assume that the initial supports lie in one fixed ball,
that
\begin{equation}\label{eq:criterion-data}
 \sup_{N,\kappa}\bigl(\|\omega_{0,N}^\kappa\|_1+\|\omega_{0,N}^\kappa\|_q\bigr)<\infty,
 \qquad \omega_{0,N}^\kappa\longrightarrow\omega_0
                      \quad\hbox{in }L^1\cap L^q,
\end{equation}
and that there are a real $\varphi\in C_c^\infty$, constants
$t_0,\gamma>0$, and times $t_N\in[t_0,1]$ such that
\begin{equation}\label{eq:criterion-gap}
 \int\varphi\bigl(\omega_N^{\mathrm{trig}}(t_N)
                  -\omega_N^{\mathrm{quiet}}(t_N)\bigr)\idiff x\geq\gamma.
\end{equation}
Then a subsequence converges to two canonical Lagrangian Euler evolutions
from $\omega_0$, with
\[
 \omega_N^\kappa\longrightarrow\omega^\kappa
       \quad\hbox{in }C([0,1];L^1\cap L^q),\qquad
 u_N^\kappa\longrightarrow u^\kappa
       \quad\hbox{in }C([0,1];L^2_{\rm loc}).
\]
Their flows are
almost-everywhere invertible and measure-preserving, with exact transport
from time zero, renormalization, and conservation of all initially finite
vorticity norms and absolutely integrable Casimirs. For some $t_*\in[t_0,1]$,
\[
 \int\varphi(\omega^{\mathrm{trig}}(t_*)
             -\omega^{\mathrm{quiet}}(t_*))\idiff x\geq\gamma.
\]
\end{proposition}

We prove the criterion in three steps: obtain a compact canonical
velocity family, control the inverse flows uniformly over terminal times,
and use composition to recover strong vorticity convergence. We first
record the bounded-vorticity solution properties used above. At the end
we deduce the finite-data continuity invoked in Sec.~\ref{sec:finite}.

\subsubsection*{I. Finite solutions and the canonical velocity equation}
\label{sec:finite-framework}

\begin{lemma}[Finite solutions and uniform analytic bounds]
\label{lem:finite-framework}
Every real compactly supported datum $f\in L^1\cap L^\infty$ has a
unique global bounded-vorticity solution, with canonical velocity
$u=K*\omega$.  Its flow maps are measure-preserving homeomorphisms,
and $\omega(t)=f\circ X(t)^{-1}$.  If, in addition,
$\operatorname{supp}f\subset B_{R_0}$ and
$\|f\|_1+\|f\|_q\leq M_0$ for $q>2$, then on $[0,1]$
\begin{equation}\label{eq:uniform-analytic-bounds}
 \begin{gathered}
 \|u\|_{L^\infty_{t,x}}\leq U_{\max}(q,M_0),\qquad
 \operatorname{supp}\omega(t)\subset B_{R_0+U_{\max}},\\
 \|u\|_{L_t^\infty L^6}+\|\nabla u\|_{L_t^\infty L^q}
       +\|P\|_{L_t^\infty L^3}\leq C(q,M_0,R_0).
 \end{gathered}
\end{equation}
With $\mathbb P$ the whole-plane solenoidal projection and $\mathcal R_a$
the Riesz transforms, the pressure and velocity satisfy
\begin{equation}\label{eq:canonical-euler}
 \partial_tu=-\mathbb P\operatorname{div}(u\otimes u),\qquad
 P=\sum_{a,b=1}^2\mathcal R_a\mathcal R_b(u_a u_b)\in L^3.
\end{equation}
\end{lemma}

\begin{proof}
Existence and uniqueness follow from
\cite[Thms.~3.3 and 4.1]{CrippaStefani2024}. The kernel assumptions there
follow from
\[
 |K(x)|=\frac1{|x|},\qquad
 |K(x)-K(y)|=\frac{|x-y|}{|x||y|}.
\]
Its divergence vanishes.
Bounded vorticity belongs to the localized Yudovich class with constant
growth function. The uniqueness integral is $\int_1^\infty\diff r/r$,
and the associated log-Lipschitz modulus is concave.

For these finite solutions, the velocity is bounded and has modulus
$Cr(1+|\log r|)$ at distances $r\leq1$. Osgood uniqueness applies
both forward and backward, so the flow maps are homeomorphisms. Spatial
mollification, Osgood comparison, and passage in the change-of-variables
formula show that the maps and their inverses preserve measure. Hence
each initial patch is transported exactly, every vorticity norm is
conserved, and compact support remains compact. Approximation of the
datum by a smooth compactly supported function also gives strong
continuity in each finite $L^r$ space.

The uniform bounds needed for the limiting argument require only $L^1$
and $L^q$. Splitting the kernel at a radius $r>0$ and then optimizing gives
\begin{equation}\label{eq:velocity}
 \begin{split}
 \|K*f\|_\infty
 &\leq C_q\bigl(\|f\|_q r^{1-2/q}+\|f\|_1/r\bigr)\\
 &\leq C_q\|f\|_1^{(q-2)/(2q-2)}\|f\|_q^{q/(2q-2)}.
 \end{split}
\end{equation}
For nonzero $f$, take $r=(\|f\|_1/\|f\|_q)^{q/(2q-2)}$;
if either norm vanishes, $f=0$. Thus \eqref{eq:criterion-data} and
conservation give one speed bound $U_{\max}$ and one support ball for
both families throughout $[0,1]$.

We also need the velocity equation with its harmonic normalization fixed.
For vorticity in that ball, the far-field bound
$|u(x)|\leq C\|\omega\|_1/|x|$, together with \eqref{eq:velocity}, gives
$u\in L^6$. The singular-integral estimates
\cite[Ch.~III, Sec.~1]{Stein1970} give $\nabla u\in L^q$.
The vorticity equation implies \eqref{eq:canonical-euler}, with
$\mathbb P$ and $\mathcal R_a$ as defined in the statement.
With the normalization of this paper,
$u=2\pi\nabla^\perp\Delta^{-1}\omega$ and
$\operatorname{curl}u=2\pi\omega$.  Thus the two identities needed
for the implication are
\begin{equation}\label{eq:normalized-curl-identities}
 \begin{gathered}
 \partial_tu=-2\pi\nabla^\perp\Delta^{-1}
                         \operatorname{div}(u\omega),\\
 \operatorname{curl}\operatorname{div}(u\otimes u)
                         =2\pi\operatorname{div}(u\omega).
 \end{gathered}
\end{equation}
The first follows from the transport equation and the canonical
Biot--Savart formula; the second follows by spatial smoothing and the
product rule for divergence-free fields.  The order-zero operator
$\nabla^\perp\Delta^{-1}\operatorname{div}$ is bounded on finite
Lebesgue spaces.  Since $u\omega\in L^q$, its application has no
polynomial ambiguity.  Applying the solenoidal multiplier gives the
first identity in \eqref{eq:canonical-euler}; the complementary
multiplier is the displayed Riesz-transform pressure.  Equivalently,
\[
 \partial_tu+\operatorname{div}(u\otimes u)+\nabla P=0.
\]
The factors $2\pi$ cancel in this momentum equation.  In particular
\eqref{eq:uniform-analytic-bounds} uses the normalized vorticity
$\omega$, while the physical curl is $2\pi\omega$.
\end{proof}

\subsubsection*{II. Compactness and uniform inverse-flow stability}

\begin{lemma}[Compactness of canonical velocities]
\label{lem:velocity-compactness}
Assume the common initial support bound and the uniform bounds and
strong initial convergence in \eqref{eq:criterion-data}. Along one
subsequence for both branches there are $u^\kappa$ and
$\overline\omega^\kappa\in L^\infty([0,1];L^q)$ such that
\begin{equation}\label{eq:velocity-compactness}
 u_N^\kappa\longrightarrow u^\kappa\quad\hbox{in }C([0,1];L^2_{\rm loc}),
\end{equation}
and
\[
 \omega_N^\kappa\rightharpoonup^*\overline\omega^\kappa
       \quad\hbox{in }L^\infty([0,1];L^q),\qquad
 u^\kappa=K*\overline\omega^\kappa.
\]
The limits solve unforced Euler with initial velocity $K*\omega_0$,
and satisfy
\[
 \|u^\kappa\|_{L^\infty_{t,x}}\leq U_{\max},\qquad
 \nabla u^\kappa\in L^\infty([0,1];L^q),\qquad
 \operatorname{div}u^\kappa=0.
\]
\end{lemma}

\begin{proof}
Lemma~\ref{lem:finite-framework} gives, uniformly in $N,\kappa$,
\[
 \|u_N^\kappa\|_{L_t^\infty L^6}
 +\|\nabla u_N^\kappa\|_{L_t^\infty L^q}
 +\|P_N^\kappa\|_{L_t^\infty L^3}\leq C.
\]
For a spatial mollifier $\rho_\epsilon$, put
$T_N^\kappa=u_N^\kappa\otimes u_N^\kappa+P_N^\kappa I$.
Then $\|T_N^\kappa\|_{L_t^\infty L^3}\leq C$ and
$\partial_tu_N^\kappa=-\operatorname{div}T_N^\kappa$.  On every
fixed ball $B_R$, Hölder's inequality and $q\geq3$ give
\[
 \sup_t\|u_N^\kappa-\rho_\epsilon*u_N^\kappa\|_{L^2(B_R)}
       \leq C_R\epsilon,\qquad
 \|\partial_t(\rho_\epsilon*u_N^\kappa)\|_{L^\infty_tL^2(B_R)}
       \leq C_R\epsilon^{-1}.
\]
The second estimate follows by convolving $T_N^\kappa$ with
$\nabla\rho_\epsilon$ and restricting its $L^3$ estimate to $B_R$.
Taking $\epsilon=|t-s|^{1/2}$ yields the explicit common modulus
\begin{equation}\label{eq:velocity-time-modulus}
 \|u_N^\kappa(t)-u_N^\kappa(s)\|_{L^2(B_R)}
                        \leq C_R|t-s|^{1/2}.
\end{equation}
At each time the local $W^{1,q}$ bound is compact in $L^2(B_R)$.
A finite time mesh, \eqref{eq:velocity-time-modulus}, and a diagonal
subsequence over $R\in\mathbb N$ and the two branches give
\eqref{eq:velocity-compactness}.
The quadratic products converge in local spacetime $L^1$, so weak
pressure convergence passes the unforced Euler equation. A weak-*
vorticity subsequence is bounded in $L_t^\infty L^q$ and supported in
the common ball. Pairing with $K$ convolved against a compact smooth
test identifies $u^\kappa=K*\overline\omega^\kappa$.  More explicitly,
for $\Phi\in C_c^\infty(\mathbb R^2;\mathbb R^2)$,
\[
 \int u_N^\kappa(x)\cdot\Phi(x)\idiff x
  =\int\omega_N^\kappa(y)
        \left(\int K(x-y)\cdot\Phi(x)\idiff x\right)\idiff y.
\]
The inner integral is bounded on the common vorticity support; the
same identity with a smooth time test passes to the weak-* limit. The initial velocity is $K*\omega_0$ by
\eqref{eq:velocity}. The limits retain bounded speed, zero divergence,
and $\nabla u^\kappa\in L_t^\infty L^q$.
\end{proof}

\begin{lemma}[Inverse-flow stability uniform in terminal time]
\label{lem:uniform-inverse-stability}
Let $u_N,u$ be bounded divergence-free drifts on $[0,1]\times\mathbb R^2$
with speed at most $U_{\max}$ and locally Sobolev spatial regularity.
Suppose, for some $q>1$,
\[
 \sup_N\|\nabla u_N\|_{L_t^1L^q}+\|\nabla u\|_{L_t^1L^q}<\infty,
 \qquad u_N\longrightarrow u\quad\hbox{in }L^1([0,1]\times B_R)
                       \quad\hbox{for every }R>0.
\]
Their measure-preserving regular Lagrangian flows have almost-everywhere
inverses $F_N(t)$ and $F(t)$, and for every $R,\epsilon>0$,
\begin{equation}\label{eq:inverse}
 \sup_{0\leq t\leq1}
 |\{x\in B_R:|F_N(t,x)-F(t,x)|>\epsilon\}|
 \longrightarrow0.
\end{equation}
\end{lemma}

\begin{proof}
For each limiting drift, the Sobolev-flow theorem
\cite[Cor.~3.7 and Cor.~3.10]{CrippaDeLellis2008} supplies a unique
regular Lagrangian flow. Its hypotheses hold because $u^\kappa$ is locally
$W^{1,q}$, has zero divergence, and
$u^\kappa/(1+|x|)\in L_t^1(L^1+L^\infty)$ by bounded speed.
Zero divergence gives measure preservation. Apply the same construction
to time-reversed drifts; uniqueness and composition give inverse time
maps, denoted $F^\kappa(t,\cdot)$. Write $F_N^\kappa$ for the finite inverse maps.

The key point is that inverse convergence must be uniform in the chosen
terminal time. On a ball $B_R$ of terminal positions, backward paths
stay in $B_{R+U_{\max}}$. Fix a terminal time $t$ and use $a\in[0,t]$
for elapsed backward time. Write $Y_N^t(a,x)$ and $Y^t(a,x)$ for
the two backward paths, suppressing the superscript $\kappa$. They satisfy
\begin{equation}\label{eq:backward-paths}
 \begin{gathered}
 \partial_aY_N^t(a,x)=-u_N(t-a,Y_N^t(a,x)),\qquad Y_N^t(0,x)=x,\\
 \partial_aY^t(a,x)=-u(t-a,Y^t(a,x)),\qquad Y^t(0,x)=x,\\
 Y_N^t(t,x)=F_N(t,x),\qquad Y^t(t,x)=F(t,x).
 \end{gathered}
\end{equation}
Let $\mathcal M_{\rm HL}$ be the Hardy--Littlewood maximal operator
and set $G_N(s)=\mathcal M_{\rm HL}(|\nabla u_N(s)|)$.
The Sobolev pointwise inequality gives, outside negligible sets,
\[
 |u_N(s,y)-u_N(s,z)|
 \leq C|y-z|\bigl(G_N(s,y)+G_N(s,z)\bigr).
\]
Subtract \eqref{eq:backward-paths}, insert and subtract
$u_N(t-a,Y^t)$, and put $d(a,x)=|Y_N^t(a,x)-Y^t(a,x)|$.
For every $\delta_f>0$, absolute continuity gives, for almost every
backward time and almost every terminal label,
\[
 \begin{split}
 \frac{\diff}{\diff a}\log(1+d/\delta_f)
 &\leq C\bigl(G_N(t-a,Y_N^t)+G_N(t-a,Y^t)\bigr)\\
 &\quad+\delta_f^{-1}|u_N-u|(t-a,Y^t).
 \end{split}
\]
The pointwise inequality is applicable along these paths by measure
preservation and Fubini. Integrate over $a\in[0,t]$ and $x\in B_R$.
At each backward time both maps preserve measure, so their maximal-function
terms have integrals at most
$C_R\|G_N(s)\|_q\leq C_{R,q}\|\nabla u_N(s)\|_q$.
The last term is bounded by the $L^1$ difference of the drifts on
the common trajectory ball. All constants are independent of the
terminal time, and the initial logarithm is zero. The
maximal-function $L^q$ estimate and measure preservation therefore give,
restoring the superscript $\kappa$,
\begin{equation}\label{eq:inverse-log}
 \begin{split}
 \sup_{0\leq t\leq1}\int_{B_R}
 \log\!\left(1+\frac{|F_N^\kappa(t,x)-F^\kappa(t,x)|}{\delta_f}\right)\idiff x
 &\leq C_R\int_0^1\|\nabla u_N^\kappa(s)\|_q\idiff s\\
 &\quad+\delta_f^{-1}
 \|u_N^\kappa-u^\kappa\|_{L^1([0,1]\times B_{R+U_{\max}})}.
 \end{split}
\end{equation}
This is the logarithmic argument of
\cite[Thm.~3.8 and Cor.~3.9]{CrippaDeLellis2008}, with constants uniform
on all intervals $[0,t]$.

Choose $\delta_f=N^{-1}+\|u_N^\kappa-u^\kappa\|_{L^1([0,1]\times B_{R+U_{\max}})}$.
The right-hand side stays bounded, whereas $\delta_f\to0$.
For every $\epsilon>0$, division by
$\log(1+\epsilon/\delta_f)$ proves \eqref{eq:inverse} for each branch.
The analogous forward estimate holds starting at time zero.
\end{proof}

\begin{lemma}[Strong convergence by exact transport]
\label{lem:strong-transport}
Under the assumptions of Lemma~\ref{lem:velocity-compactness}, every
subsequence selected there satisfies, after identifying its weak
vorticity limit,
\begin{equation}\label{eq:strong}
 \omega_N^\kappa\longrightarrow\omega^\kappa=\omega_0\circ F^\kappa
 \quad\hbox{in }C([0,1];L^1\cap L^q).
\end{equation}
Here $F^\kappa(t)$ is the almost-everywhere inverse of the
measure-preserving flow of $u^\kappa$. The limit obeys the renormalized
transport equation and conserves every initially finite vorticity norm
and absolutely integrable Casimir.
\end{lemma}

\begin{proof}
The drifts in Lemma~\ref{lem:velocity-compactness} satisfy all the
hypotheses of Lemma~\ref{lem:uniform-inverse-stability}, branch by
branch. For $r=1,q$, measure preservation first gives
\[
 \|(\omega_{0,N}^\kappa-\omega_0)\circ F_N^\kappa(t)\|_r
 =\|\omega_{0,N}^\kappa-\omega_0\|_r.
\]
Approximate $\omega_0$ in both norms by one compact smooth function.
Its compositions converge in $L^r$, uniformly in time, by
\eqref{eq:inverse}, uniform continuity, and bounded displacement.
The approximation errors retain their norms under all the maps.
For an explicit estimate let that smooth function be $g$, and choose
$B_{R_g}$ containing the supports of all $g\circ F_N^\kappa(t)$ and
$g\circ F^\kappa(t)$; bounded displacement permits a single such ball.
For $\epsilon>0$ set
$E_{N,t,\epsilon}=\{x\in B_{R_g}:|F_N^\kappa(t,x)-F^\kappa(t,x)|>\epsilon\}$.
Then
\[
 \sup_t\|g\circ F_N^\kappa(t)-g\circ F^\kappa(t)\|_r^r
 \leq |B_{R_g}|(\|\nabla g\|_\infty\epsilon)^r
    +(2\|g\|_\infty)^r\sup_t|E_{N,t,\epsilon}|.
\]
First let $N\to\infty$ with $g,\epsilon$ fixed, then let
$\epsilon\to0$. The full difference is bounded by this composition
norm plus
$\|\omega_{0,N}^\kappa-\omega_0\|_r+2\|\omega_0-g\|_r$.
Finally approximate $\omega_0$ arbitrarily well. This order of limits
proves uniform-in-time convergence for $r=1$ and $r=q$ simultaneously.
This proves \eqref{eq:strong}.
This strong limit agrees with the weak-* limit
$\overline\omega^\kappa$ of Lemma~\ref{lem:velocity-compactness}, so
$u^\kappa=K*\omega^\kappa$.  In particular the limit has a strong
trace and, by invertibility, exact transport from time zero. Equation~\eqref{eq:velocity} also gives
$C_tL^\infty$ convergence of the velocities.

Let $B\in C^1$ be bounded with $B(0)=0$ and compactly
supported derivative. Apply the pathwise chain rule to
$\Phi(t,X^\kappa(t,x_0))$, multiply by $B(\omega_0(x_0))$, and integrate.
For a compact smooth spacetime test vanishing at time one, this gives
\[
 \int_0^1\!\int B(\omega^\kappa)
 (\partial_t\Phi+u^\kappa\cdot\nabla\Phi)\idiff x\idiff t
 +\int B(\omega_0)\Phi(0,x)\idiff x=0.
\]
Bounded speed and integrability justify Fubini. Smooth truncations give
the usual renormalized identities; using $\omega_0$ directly is also
legitimate since $u^\kappa\omega^\kappa\in L_t^\infty L^1$. Change of variables
conserves every initially finite vorticity norm and every absolutely
integrable Casimir. The same composition argument applies to any fixed
material component.

\end{proof}

\begin{proof}[Proof of Proposition~\ref{prop:criterion}]
Lemmas~\ref{lem:velocity-compactness} and \ref{lem:strong-transport}
give the asserted limits, Euler equation and transport properties.
Take a further subsequence with $t_N\to t_*\in[t_0,1]$.
Strong convergence in \eqref{eq:strong} and continuity of the limits give
\[
 \int\varphi(\omega_N^{\mathrm{trig}}(t_N)
                -\omega_N^{\mathrm{quiet}}(t_N))\idiff x
 \longrightarrow
 \int\varphi(\omega^{\mathrm{trig}}(t_*)
                -\omega^{\mathrm{quiet}}(t_*))\idiff x.
\]
This preserves \eqref{eq:criterion-gap} and proves the criterion.
\end{proof}

\subsubsection*{III. Continuity with respect to finite patch data}
\label{sec:finite-continuity}

The parameter choices in Section~\ref{sec:finite} vary finitely many
disk centers. We verify that their trajectories, centroids and
transverse crossing times depend continuously on those centers.

\begin{corollary}\label{cor:finite-continuity}
For a fixed finite family of disks with fixed radii and signed strengths,
whose centers vary in a compact set with disjoint initial disks, the
canonical solutions depend continuously on the centers in
$C_t(L^1\cap L^q)$, and the velocities in $C_tL^\infty$.
Forward and inverse maps converge uniformly on compact sets. The tagged
supports converge in Hausdorff distance, and their centroids, moments,
and centroid velocities vary continuously. Isolated transverse section
crossings consequently depend continuously on the initial centers.
\end{corollary}

\begin{proof}
Translation continuity gives $L^1$ continuity of the data. The fixed
finite family has a common $L^\infty$ bound and compact support, so
interpolation gives $L^q$ continuity. Lemmas~\ref{lem:velocity-compactness} and \ref{lem:strong-transport}
apply to every parameter sequence without \eqref{eq:criterion-gap}. Every limiting
solution retains the $L^\infty$ bound and equals the solution of the
limiting datum by bounded-vorticity uniqueness. Thus the whole sequence
converges in the claimed norms. A common Osgood modulus gives uniform
forward and backward trajectory convergence: if two trajectories have
separation $d(t)$, then
\[
 D^+d(t)\leq\|u^{(m)}(t)-u(t)\|_\infty
                    +C d(t)(1+|\log d(t)|)
\]
as long as $0<d(t)\leq1/2$, and Osgood comparison sends the uniform
separation to zero with the initial-center and velocity errors.
This implies Hausdorff convergence of the transported closed disks.
For each fixed tag the composition argument also gives
$\mu_j^{(m)}\to\mu_j$ in $C_tL^1$. Bounded support then gives
convergence of all tagged polynomial moments, and
\[
 \dot b_j(t)=\int u(t,x)\mu_j(t,x)\idiff x,\qquad
 |\dot b_j^{(m)}-\dot b_j|
     \leq\|u^{(m)}-u\|_\infty
                       +U_{\max}\|\mu_j^{(m)}-\mu_j\|_1.
\]
Consequently the centroid curves converge in $C^1$.
For a $C^1$ section function $H$, an isolated crossing satisfying
$H(b(t_*))=0$ and
$\nabla H(b(t_*))\cdot\dot b(t_*)\ne0$ has a unique nearby crossing
for all sufficiently close parameters.  The latter crossing times
converge to $t_*$, by the uniform sign of the derivative on a small
interval and the intermediate value theorem.
\end{proof}

The two outputs have different uses. The criterion gives strong
$C_t(L^1\cap L^q)$ convergence and preserves a fixed test-function gap
at a positive limiting time. The corollary gives continuity of finite
parameter-dependent section crossings. The former closes the theorem
below; the latter justifies the finite exit maps constructed next.

\subsection{Selecting finite data with a fixed shape difference}\label{sec:finite}

We now construct, for every $N$, two finite Euler solutions whose initial
patches agree at indices $0,\ldots,N-1$, but for which
$\Delta\mathsf U_0=h_*$ at a time below $1/4$. The initial difference
is confined to index $N$, so its norm will tend to zero. The construction
has two choices: a point in a finite-dimensional box gives the reference
data, and nested intervals of one rotation parameter give the perturbed
data. Both choices use actual bounded-vorticity solutions and their
continuous dependence on disk centers, proved independently in
Sec.~\ref{sec:finite-continuity}.

We use the signs $\epsilon_j=(1,1,-1,-1)_j$ and make the cutoff
explicit only where needed.  Thus every finite solution below has the
material decomposition
\begin{equation}\label{eq:finite-material-decomposition}
 \begin{gathered}
 \omega_N^\kappa(t,x)=\sum_{n=0}^{N}\sum_{j=1}^{4}
          \epsilon_j\Gamma_n\mu_{n,j,N}^\kappa(t,x),\\
 \mu_{n,j,N}^\kappa(t,x)=\frac{\ind_{\Omega_{n,j,N}^\kappa(t)}(x)}
                                      {\pi a_n^2},\qquad
 \int\mu_{n,j,N}^\kappa(t,x)\idiff x=1.
 \end{gathered}
\end{equation}
Here $\Omega_{n,j,N}^\kappa(t)$ is the image of its initial disk under
the complete Euler flow, and
$b_{n,j,N}^\kappa(t)=\int x\mu_{n,j,N}^\kappa(t,x)\idiff x$.
Every tag remains in the velocity after its configuration crosses an
outgoing section.

\begin{proposition}[Finite amplification]\label{prop:finite}
For one sufficiently large fixed $L$ in \eqref{eq:scales} and every
$N\geq1$, there are two canonical global bounded-vorticity solutions
$\omega_N^{\mathrm{quiet}}$ and $\omega_N^{\mathrm{trig}}$ whose initial
data consist of four signed disks at each index $0,\ldots,N$.
Their initial patches agree at all indices $n<N$.
For a fixed $B_{\rm ref}$, the reference solution satisfies
\begin{equation}\label{eq:quietbox}
 |\mathsf U_{n,N}^{\mathrm{quiet}}(t)|,
 |\mathsf V_{n,N}^{\mathrm{quiet}}(t)|
 \leq B_{\rm ref}(\tau_n h^\beta+\eta_n^2),\qquad 0\leq t\leq1.
\end{equation}
There are fixed $h_*,c_0,\delta_{\rm loc}>0$ and a time
$t_N^{\rm exit}\in(0,1/4)$ such that
\begin{equation}\label{eq:exit}
 \mathsf U_{0,N}^{\mathrm{trig}}(t_N^{\rm exit})
 -\mathsf U_{0,N}^{\mathrm{quiet}}(t_N^{\rm exit})=h_*.
\end{equation}
For $0\leq n\leq N$, $1\leq j\leq4$, and
$0\leq t\leq t_N^{\rm exit}$,
\begin{equation}\label{eq:finite-support-confinement}
 \begin{gathered}
 \operatorname{supp}\mu_{n,j,N}^\kappa(t)
    \subset B(b_{n,j,N}^\kappa(t),Cd_n\eta_n^{1/45})
       \cap B(Z_N^\kappa(t)+R_ne_1,\delta_{\rm loc}R_n),\\
 Z_N^\kappa(0)=0,\qquad
 (Z_N^\kappa)'=u_N^\kappa(t,Z_N^\kappa),\qquad
 |(Z_N^\kappa)'|\leq Ch^\beta,\qquad \delta_{\rm loc}<1/16.
 \end{gathered}
\end{equation}
At $t=t_N^{\rm exit}$ both index-zero centroid configurations satisfy
\begin{equation}\label{eq:observation-chart}
 \begin{gathered}
 |\mathsf U_0^\kappa|,|\mathsf V_0^\kappa|\leq2h_*,\qquad
 |\mathsf A_0^\kappa-1|\leq\tfrac14,\qquad
 |\chi_0^\kappa|/d_0\leq\tfrac1{32},\\
 \min_{j\ne k}|b_{0,j,N}^\kappa-b_{0,k,N}^\kappa|\geq c_0d_0,
 \end{gathered}
\end{equation}
where $h_*\leq1/16$.
All constants are independent of $N$; the other centroid coordinates
are defined by the chart of Section~\ref{sec:cell}.
\end{proposition}

\begin{proof}
\emph{Reference data: an exit map from a box.}
Set the initial configuration centers to $R_ne_1$ and prescribe
$\chi_n=0$, $\mathsf A_n=1$, $\mathsf V_n=0$, and $\vartheta_n=0$.
Vary only the initial unstable coordinates in
\[
 \mathcal B_N=\prod_{n=0}^N
 [-B_{\rm ref}(\tau_nh^\beta+\eta_n^2),
    B_{\rm ref}(\tau_nh^\beta+\eta_n^2)].
\]
The equations \eqref{eq:actual-quiet-equations} have positive coefficient
for $\mathsf U_n$ and negative coefficient for $\mathsf V_n$.
Their error is $O(\tau_nh^\beta+\eta_n^2)$. In time
$s=t/\tau_n$, the derivative on either unstable face therefore
satisfies
\[
 \operatorname{sgn}(\mathsf U_n)\mathsf U_n'
 \geq (\lambda B_{\rm ref}/2-C)(\tau_nh^\beta+\eta_n^2)>0.
\]
Choose, for example, $B_{\rm ref}\geq4(C+1)/\lambda$.
The stable inequality
$D^+|\mathsf V_n|\leq-\lambda|\mathsf V_n|/2
+C(\tau_nh^\beta+\eta_n^2)$ preserves its corresponding bound.
Section~\ref{sec:geometry} improves the support, separation and other
coordinate bounds through time one as long as $\mathsf U$ stays in
$\mathcal B_N$.

To make the topological argument precise, denote the initial vector of
unstable coordinates by $x\in\mathcal B_N$, and its subsequent vector
by $\mathsf U(t;x)$. If every $x$ leaves the box before time one,
let $T(x)$ be its first exit time and set
\[
 \mathcal R_N(x)=\mathsf U(T(x);x)\in\partial\mathcal B_N.
\]
For an interior starting point, a compact subinterval before $T(x)$
has positive distance from every unstable face. Continuous dependence
therefore prevents a nearby datum from exiting much earlier. At the
first contact at least one face is reached, and its strictly outward
derivative forces nearby data to cross that face within an arbitrarily
short further time. These two observations give continuity of $T$.
The same one-sided argument applies at a corner where several faces
are reached: the first of their crossings still occurs in that short
time interval. At a boundary starting point the outward inequality
gives $T(x)=0$ and $T(x')\to0$ as interior $x'\to x$.
Consequently $\mathcal R_N$ is continuous and equals the identity on
$\partial\mathcal B_N$. This is a retraction of a closed box onto its
boundary, which the no-retraction theorem excludes. All other stopping
conditions have already been strictly improved, so the alternative is
an actual solution staying in the box through time one. Fix its initial data for this $N$.

\emph{Perturbed data: one rotation parameter.}
Replace only the four disks at index $N$ by disks centered on the fixed
section $\mathsf R=\mathsf R_{\rm out}$ of
$\mathsf V=0$, $\mathsf A=1$, $\mathsf U>0$.
Rotate that configuration by one free angle. Its signed impulse is zero,
and all initial patches at indices $n<N$ remain the reference patches.
Parametrize the rotation by $\theta_{\rm in}$ so that its direction
coordinate equals
$\theta_{\rm in}-\vartheta_{N-1}^{\mathrm{quiet}}(0)-\pi/4$.
The initial line between the centers is horizontal, so
$\theta_{\rm in}$ is exactly the angle used in the forcing of
$\Delta\mathsf U_{N-1}$. By construction $t_N^{\rm out}=0$.
Each value of this parameter specifies one complete Euler solution.
The subsequent choices restrict this initial parameter set.

Choose an interval $\mathcal I=[\psi_-,\psi_+]$ on which
$\mathcal H$ increases and
\[
 \mathcal H(\psi_-)=m_*:=e^{-12},\qquad \mathcal H(\psi_+)=1.
\]
The formula \eqref{eq:coefficient} permits $|\mathcal I|<1/6$.
For example, write $\theta_H=\arctan(1/\sqrt2)$, so that
$\mathcal H(\psi)=2\sqrt3\cos(2\psi-\theta_H)$, and put
$\psi_c=(\theta_H-\pi/2)/2$. On the interval immediately to the
right of $\psi_c$,
$\mathcal H(\psi)=2\sqrt3\sin(2(\psi-\psi_c))$ is increasing.
We may choose
\[
 \psi_-=\psi_c+\tfrac12\arcsin(m_*/(2\sqrt3)),\qquad
 \psi_+=\psi_c+\tfrac12\arcsin(1/(2\sqrt3)).
\]
Both arguments lie in $(0,1)$, and
$|\mathcal I|<\frac12\arcsin(1/(2\sqrt3))<1/6$.
Initially restrict $\theta_{\rm in}$ to $\mathcal I$.

\emph{From a section crossing at $n+1$ to one at $n$.}
Suppose index $n+1$ has crossed its separated section at
$t_{n+1}^{\rm out}$ and every index $k>n+1$ satisfies the separated-pair
bounds. The induction retains the original tagged patches, all elapsed
rescaled times, zero initial signed impulses, and identical initial data
at indices $k\leq n$. The section times will satisfy
\begin{equation}\label{eq:cascade-schedule}
 t_n^{\rm out}\leq C_{\rm sched}
       \sum_{k=n}^{N-1}\tau_k(1+\ell_k),\qquad t_N^{\rm out}=0.
\end{equation}
At this step only its instance for $n+1$ is assumed. Indeed,
\eqref{eq:tailclock} then gives
$t_{n+1}^{\rm out}/\tau_n
\leq CL(n+2)^2r_n=o_L(1)$.
For $k<n$, the total rescaled time from zero during the present interval
is at most $r_{n-1}(s_0+S)=o_L(1)$. These are the time hypotheses of
Proposition~\ref{prop:release}; at $n=0$ there is no such larger index.

Apply that proposition on the interval stopped at the first possible
failure of any stated geometric, support or comparison bound. For
$\psi_n\in\mathcal I$ it gives \eqref{eq:certificate}.
In particular, if $s_a$ is the fixed initial allowance and
$s_a\leq v\leq s$, its upper and lower bounds imply
\begin{equation}\label{eq:finite-integral}
 \Delta\mathsf U_n(v)
 \leq\frac Cc e^{-\lambda(s-v)/2}\Delta\mathsf U_n(s),\qquad
 \int_{s_a}^s|\Delta\mathsf U_n(v)|\idiff v
             \leq\frac{2C}{c\lambda}h_*.
\end{equation}
The bounded initial interval is estimated separately. Thus the departure
from rigid rotation has bounded integral on each current interval.
For clarity, the quotient of the two bounds in \eqref{eq:certificate}
is
$\Delta\mathsf U_n(v)/\Delta\mathsf U_n(s)
\leq(C/c)\mathfrak w(s)/\mathfrak w(v)$, and
$\mathfrak w(s)/\mathfrak w(v)
=\exp(-\int_v^s\lambda_{\rm eff})\leq e^{-\lambda(s-v)/2}$.
The lower bound also shows that the difference is positive on this
interval. Integrating the exponential and using
$\Delta\mathsf U_n(s)\leq h_*$ before the crossing proves
\eqref{eq:finite-integral}. In particular, this estimate holds at each current stopped endpoint.
Section~\ref{sec:geometry} consequently improves the patch support,
configuration-position and separated-pair bounds. The difference equations
in Sec.~\ref{sec:release} improve the remaining coordinate and analytic
bounds on that same interval. None can be the first failing bound before
$\Delta\mathsf U_n=h_*$.

Because the exponential lower bound forces this crossing by
$s=C(1+\ell_n)$, the only remaining continuation is along the fixed
separated point-vortex arc. It gives a transverse crossing of
$\mathsf R=\mathsf R_{\rm out}$, with centroid discrepancy
$\zeta_{\rm out}(L)\to0$ uniformly for $0\leq\zeta\leq\zeta_0$.
Increasing one fixed
$C_{\rm sched}$ to include the arc proves
\eqref{eq:cascade-schedule} for $n$. The time assumption is therefore
propagated by induction on completed crossings.

\emph{Selecting all angles from the original parameter interval.}
Put $s_n^{\rm out}=t_n^{\rm out}/\tau_n$ and let $\Theta_n$
be the direction coordinate at the separated section. The time and
angle laws \eqref{eq:clock}--\eqref{eq:angle} determine
\[
 \psi_{n-1}=\Theta_n+\vartheta_{n-1}^{\mathrm{quiet}}(t_n^{\rm out})
 +\pi/4-2\arg(\boldsymbol c_{n-1}^{\mathrm{trig}}(t_n^{\rm out})
                   -\boldsymbol c_n^{\mathrm{trig}}(t_n^{\rm out})).
\]
This angle depends continuously on the original parameter. Write
$\vartheta_n^0=\vartheta_n(0)$ and
$A_n^0=\vartheta_n^0+\vartheta_{n+1}^0+\pi/4+C_{\rm out}$.
The bounds on $t/\tau_{n-1}$ and the centers give
$\psi_{n-1}=A_{n-1}^0-(1+r_{n-1})s_n^{\rm out}+o_L(1)$.
The coefficient $1+r_{n-1}$ has two distinct sources:
$\Theta_n=\vartheta_n^0-s_n^{\rm out}+C_{\rm out}+o_L(1)$,
while the reference configuration at index $n-1$ accumulates the
rotation
$-t_n^{\rm out}/\tau_{n-1}=-r_{n-1}s_n^{\rm out}$.
The line-of-centers angle tends to zero uniformly by the geometric
estimates, and its contribution is included in the displayed error.
If $n+1<N$, its earlier section crossing also gives
$\psi_n=A_n^0-(1+r_n)s_{n+1}^{\rm out}+o_L(1)$.
For these indices the elimination can be written in two steps. Put
$T_n^0=\ell_n/\lambda+C_{\rm time}$ and write the time error as
$d_n^{\rm time}=O_{m_*}(\zeta)+o_L(1)$. Then
\[
 \begin{split}
 s_n^{\rm out}
 &=r_ns_{n+1}^{\rm out}+T_n^0
       -\lambda^{-1}\log\mathcal H(\psi_n)+d_n^{\rm time},\\
 r_ns_{n+1}^{\rm out}
 &=\frac{r_n}{1+r_n}(A_n^0-\psi_n)+o_L(1).
 \end{split}
\]
Substitution into the expression for $\psi_{n-1}$ yields, for $n\geq1$,
\begin{equation}\label{eq:covermap}
 \begin{gathered}
 \psi_{n-1}=B_n^{\rm phase}+\alpha_n\psi_n
                  +c_n\log\mathcal H(\psi_n)+e_n^{\rm phase},\\
 \alpha_n=\frac{r_n(1+r_{n-1})}{1+r_n}\geq0,\qquad
 c_n=\frac{1+r_{n-1}}\lambda,\qquad |e_n^{\rm phase}|\leq1/8,\\
 B_n^{\rm phase}=A_{n-1}^0-\alpha_n A_n^0
       -(1+r_{n-1})(\ell_n/\lambda+C_{\rm time}).
 \end{gathered}
\end{equation}
The time error is multiplied by $1+r_{n-1}\leq2$; all other
errors are multiplied by uniformly bounded ratios, with no factor
$r_n^{-1}$. Quantitatively, write the resulting allowance as
$C_{\rm phase}(m_*)\zeta+\rho_{{\rm phase},n}(L)$, where
$\sup_n\rho_{{\rm phase},n}(L)\to0$. Since the fixed positive margin
$m_*=e^{-12}$ can enter the constant in the logarithmic timing law,
choose the admitted discrepancy and then $L$ so that
\[
 C_{\rm phase}(m_*)\zeta_0\leq1/16,\qquad
 \sup_n\rho_{{\rm phase},n}(L)\leq1/16.
\]
This gives $|e_n^{\rm phase}|\leq1/8$ with an allowance uniform in
the depth. The coefficient $c_n\geq1/\lambda$ retains the required
logarithmic variation even when $r_n$ is extremely small.
All angles are lifted continuously to the real line on the interval
under consideration. The offset $B_n^{\rm phase}$ is constant on the
original parameter family. At $n=N-1$, the initial value
$s_N^{\rm out}=0$ replaces the preceding crossing law: take
$\alpha_n=0$ and omit $\alpha_nA_n^0$.
For $N=1$ there is only the crossing at index zero and no next-angle map.

Suppose a closed interval of the original parameter has its current
angle in $\mathcal I$ and covers that interval, with its endpoints
mapping to the two endpoints in order. A translated interval
$\mathcal I+m\pi$ is treated by subtracting $m\pi$ from the
continuous lift and absorbing the resulting constant into
$B_n^{\rm phase}$. At endpoints with values $\psi_-$ and
$\psi_+$, the non-error part of \eqref{eq:covermap} increases by at least
\[
 \frac{12}{\lambda}=3\sqrt2>\pi+|\mathcal I|+4/8.
\]
The two endpoint errors have total size at most $2/8$. The image of the
next angle therefore contains $\mathcal I+m\pi$ for some integer $m$.
To see why this length suffices, choose the first point of the lattice
$\psi_-+\pi\mathbb Z$ above the lower endpoint of the next-angle
range. It lies less than $\pi$ above that endpoint; an additional
length $|\mathcal I|$ contains its entire translated interval.
Take its first crossing of the upper endpoint and its last preceding
crossing of the lower endpoint. Between those parameters the angle stays
in $\mathcal I+m\pi$ and covers that interval. Continuity suffices for this construction. The smaller closed parameter interval
retains all previously verified crossings. Applying this step successively
selects one initial $\theta_{\rm in}$ for which every required crossing
through index zero occurs.

\emph{Uniformity in $N$.}
All lower bounds on $L$ in this section, including those for the
fixed observable below, are imposed before making any finite selection.
First fix $\mathcal I$, $h_*$, the separated section, its connecting arc
and the pair-separation margins. Fix $s_{\rm ent},c_w,C_w$ from the
leading weighted input and use \eqref{eq:rotation-cost-allowance},
including its provisional factor $4C_w/c_w$, for the small-chart cost.
This allowance and the arc and pair coefficient integrals determine
the constants in the patch estimates. Set the provisional bounds above
these conclusions; then choose $\zeta_0$ small and finally $L$ large.
Require $\zeta_{\rm out}(L)\leq\zeta_0/2$, the weighted remainder below
the positive forcing margin, and every required full elapsed interaction
time for Lemma~\ref{lem:response} below its fixed $T_*$. Also require
\[
 C_{\rm sched}490Le^{-\delta L}<1/4.
\]
Lemma~\ref{lem:scale-margins} gives all these conditions simultaneously.
For the later observation argument also impose
\[
 B_{\rm ref}\sup_{n\geq0}(\tau_nh^\beta+\eta_n^2)\leq h_*/4,
 \qquad
 C\eta_0^{1/45}<\frac18\min\{c_0/16,h_*/(32C_{\rm gap}),1/8\}.
\]
Here $C_{\rm gap}$ is the dimensionless chart constant in
Lemma~\ref{lem:shape-statistic}, which is independent of $L$.
The constants in the variance bounds multiply $\eta_n^2$, so increasing
$L$ suffices without changing the already chosen section. Each index uses
the same bounds. The last display and \eqref{eq:cascade-schedule} place
the index-zero crossing before time $1/4$; the geometric estimates give
the asserted support and separation bounds.
\end{proof}

We have obtained \eqref{eq:quietbox} and the fixed difference
\eqref{eq:exit}, with all four patches at each index within the specified
balls. The only changed initial disks have index $N$. These two facts
will give, respectively, a nonzero smooth-test difference and a common
strong initial limit in Subsection~\ref{subsec:common-datum-observable}.

\subsection{The common datum and a fixed vorticity observable}\label{subsec:common-datum-observable}

We verify the remaining hypotheses of Proposition~\ref{prop:criterion}:
both initial families converge in $L^1\cap L^q$ to the same compactly
supported $\omega_0$, and one smooth test has a positive pairing with
their difference at the times $t_N^{\rm exit}$. The norm tails give the
first claim. For the second, \eqref{eq:exit} concerns a shape coordinate;
we must show that it implies a difference in the signed vorticity itself.

\subsubsection*{I. A common strong initial limit}

\begin{proposition}[A common strong datum]
\label{prop:common-strong-datum}
Along a subsequence $N\to\infty$, the two initial families of
Proposition~\ref{prop:finite} converge in $L^1\cap L^q$ to the same
compactly supported datum
\begin{equation}\label{eq:datum}
 \omega_0(x)=\sum_{n\geq0}\sum_{j=1}^4
                 \epsilon_j\Gamma_n\mu_{n,j}^0(x),\qquad
 \mu_{n,j}^0(x)=\frac1{\pi a_n^2}\ind_{D_{n,j}^0}(x),\qquad
 D_{n,j}^0=B(b_{n,j}^0,a_n),
\end{equation}
where $b_{n,j}^0$ are limiting initial centers. The limiting disks are
pairwise disjoint modulo null sets, and their only possible accumulation
point is the origin.
\end{proposition}

\begin{proof}
For each fixed $n$, the reference coordinates
$\mathsf U_{n,N}^{\mathrm{quiet}}(0)$ lie in the compact interval
\eqref{eq:quietbox}. A diagonal subsequence gives limits $\mathsf U_n^0$.
The inverse chart \eqref{eq:chart}, with
$\boldsymbol c_n=R_ne_1$, $\chi_n=0$, $\mathsf A_n=1$,
$\mathsf V_n=0$, and $\vartheta_n=0$, then defines the centers $b_{n,j}^0$ and the datum in \eqref{eq:datum}.
The disks are disjoint up to null sets: an intersection of two limiting
interiors would persist at finite cutoffs. The centers accumulate only
at the origin, and the union of the disks has compact closure. Equation~\eqref{eq:norms}
gives $\omega_0\in L^1\cap L^q$.

Convergence of the centers of finitely many disks gives convergence of
their indicator functions in both norms. The summable tails
\eqref{eq:norm-tail} control all remaining disks, uniformly in the
cutoff. Finally the change of the four initial disks at index $N$ has vanishing norm:
\begin{equation}\label{eq:seed-norm}
 \|\omega_{0,N}^{\mathrm{trig}}-\omega_{0,N}^{\mathrm{quiet}}\|_1
       \leq8\Gamma_N,\qquad
 \|\omega_{0,N}^{\mathrm{trig}}-\omega_{0,N}^{\mathrm{quiet}}\|_q^q
       \leq C_q\eta_*^{-2(q-1)}d_N^{\kappa_q}.
\end{equation}
More explicitly, fix an integer $J$ and separate the indices $n<J$
from the others. On the chosen diagonal subsequence every disk among
these first $4J$ disks converges by translation continuity. For the tails,
disjointness gives
\[
 \begin{split}
 \sup_{N\geq J}\|\omega_{0,N}^{\rm quiet}\big|_{n\geq J}\|_1
   &\leq4\sum_{n\geq J}\Gamma_n,\\
 \sup_{N\geq J}\|\omega_{0,N}^{\rm quiet}\big|_{n\geq J}\|_q^q
   &\leq4\pi^{1-q}\eta_*^{-2(q-1)}
                       \sum_{n\geq J}d_n^{\kappa_q}.
 \end{split}
\]
Here the restriction denotes the sum of the tagged components with
the indicated indices.
The limiting datum has the same two tail bounds. First let $N$
tend to infinity with $J$ fixed, then let $J$ tend to infinity.
This proves convergence of the reference data in both norms.
For the perturbed data, the $L^1$ difference of the two choices of four
disks is at most their combined unsigned mass $8\Gamma_N$.
For the $L^q$ estimate use
$|f-g|^q\leq2^{q-1}(|f|^q+|g|^q)$ and the four-disk norm
at index $N$. These translation-independent bounds prove
\eqref{eq:seed-norm}.
Both families therefore converge strongly to \eqref{eq:datum} in
$L^1\cap L^q$. Their complete supports and norms have the uniform bounds
required by the criterion.  Since $\beta-2-2E=-11\delta<0$,
\[
 |\omega_0|\big|_{D_{n,j}^0}
     =\pi^{-1}\eta_*^{-2}d_n^{-11\delta}\longrightarrow\infty.
\]
Thus the common datum lies outside $L^\infty$.
\end{proof}

\subsubsection*{II. From a shape gap to a fixed spatial test}\label{sec:separation}

The signed measure is invariant under interchange of the two components of each sign.
We therefore use a statistic invariant under precisely those permutations.

\begin{lemma}[A permutation-invariant shape statistic]
\label{lem:shape-statistic}
For index-zero centroid configurations satisfying
\eqref{eq:observation-chart}, put
\[
 \mathcal M=|\mathsf X|+|\mathsf Y|
          =\max(|\mathsf U|,|\mathsf V|).
\]
With $z=(b_1-b_2)/(2d_0)$ and $w=(b_3-b_4)/(2d_0)$, it equals
\begin{equation}\label{eq:shape-statistic}
 \mathcal M=
 \frac{|\mathsf R-\mathsf A|}{\sqrt{\mathsf A\mathsf R}}
 +\sqrt{\frac{\mathsf R}{\mathsf A}}\,
        \frac{|z\cdot w|}{|z||w|},
 \qquad
 \mathsf A=|z|^2-|w|^2,\quad \mathsf R=2|z||w|.
\end{equation}
This formula is unchanged by translations, rotations and permutations
of the two positive or the two negative labels.  There is a constant
$C_{\rm gap}>0$, independent of $N,L$, such that for any two such
configurations $b,\widetilde b$ and any permutation $\pi$ preserving
each sign,
\begin{equation}\label{eq:shape-statistic-lipschitz}
 |\mathcal M(b)-\mathcal M(\widetilde b)|
       \leq\frac{C_{\rm gap}}{d_0}
                         \max_{1\leq j\leq4}|b_j-\widetilde b_{\pi(j)}|.
\end{equation}
\end{lemma}

\begin{proof}
Use $z=(b_1-b_2)/(2d_0)$ and $w=(b_3-b_4)/(2d_0)$.
The definition \eqref{eq:chart} gives \eqref{eq:shape-statistic}.
We used $\sin(\phi-\pi/2)=-\cos\phi$ and
$\cos\phi=(z\cdot w)/(|z||w|)$. Interchanging either same-sign
pair replaces $z$ or $w$ by its negative, which leaves this expression
unchanged. Translation cancels in the half-differences, and rotation
preserves their lengths and scalar product. In the permitted compact
chart, the explicit inverse formulas imply the numerical bounds
\[
 \tfrac12\leq\mathsf A\leq\tfrac32,\qquad
 \tfrac18\leq\mathsf R\leq6,\qquad
 \tfrac1{32}\leq|z|,|w|\leq3.
\]
For example, $|\mathsf X|\leq2h_*\leq1/8$ implies
$1/2\leq q_0\leq2$ in \eqref{eq:cell-chart-inverse}, and
$|w|^2=\mathsf R^2/(2(\sqrt{\mathsf A^2+\mathsf R^2}+\mathsf A))$
gives the displayed lower bound.  On these scalar ranges, norms, scalar
products, square roots and reciprocals in \eqref{eq:shape-statistic}
are Lipschitz with uniform constants.  The inequalities
$||z|-|\widetilde z||\leq|z-\widetilde z|$ and
$||a|-|b||\leq|a-b|$ therefore give a uniform Lipschitz estimate
directly, without an angle branch.  Reconstructing the half-differences
introduces $1/d_0$.  The formula is unchanged by $z\mapsto-z$ or
$w\mapsto-w$, so the same constant works for every sign-preserving
permutation. This proves \eqref{eq:shape-statistic-lipschitz}.
\end{proof}

\begin{proposition}[One fixed observable at positive times]
\label{prop:fixed-observable}
There are a real $\varphi_*\in C_c^\infty(\mathbb R^2)$, a number
$t_0>0$, and an infinite subsequence of the finite pairs such that
\begin{equation}\label{eq:fixed-observable-gap}
 t_N^{\rm exit}\in[t_0,1/4],\qquad
 \int\varphi_*(\omega_N^{\rm trig}(t_N^{\rm exit})
              -\omega_N^{\rm quiet}(t_N^{\rm exit}))\idiff x
                 \geq\Gamma_0.
\end{equation}
The constants and the test are fixed before passing to a limit.
\end{proposition}

\begin{proof}
The final choice of $L$ in Proposition~\ref{prop:finite} makes the
reference bound \eqref{eq:quietbox} at most $h_*/4$.
Hence $\mathsf U_0^{\rm trig}\geq3h_*/4$ and
$\mathcal M_0^{\rm quiet}\leq h_*/4$ at $t_N^{\rm exit}$, so
\eqref{eq:exit} gives
\begin{equation}\label{eq:gap}
 \mathcal M_N^{\mathrm{trig}}(t_N^{\rm exit})
 -\mathcal M_N^{\mathrm{quiet}}(t_N^{\rm exit})\geq h_*/2.
\end{equation}

The corresponding index-zero centers agree initially and move at speed
at most $U_{\max}$ in each solution. Thus the observation times cannot
approach zero:
\begin{equation}\label{eq:positive-observation-time}
 t_N^{\rm exit}\geq
 t_0:=\frac{h_*d_0}{4C_{\rm gap}U_{\max}}>0.
\end{equation}
Indeed each corresponding centroid has moved by at most
$U_{\max}t$ from its common starting point, so the maximum displacement
between the two configurations is at most $2U_{\max}t$.
Applying the Lipschitz bound to \eqref{eq:gap} gives
$h_*/2\leq2C_{\rm gap}U_{\max}t/d_0$, which is the displayed lower bound.
The case $U_{\max}=0$ is excluded by \eqref{eq:gap}.

To detect this difference, choose
\[
 \delta_{\rm iso}=\min\{c_0d_0/16,
                   h_*d_0/(32C_{\rm gap}),1/8\}.
\]
The index-zero support radius $Cd_0\eta_0^{1/45}$ is less than
$\delta_{\rm iso}/8$ for large $L$. This condition is included in the
single final choice of $L$; the ratio $\delta_{\rm iso}/d_0$ is eventually
a fixed positive chart fraction.

There is a centroid of the perturbed solution $b_*$ of sign $\varepsilon_*$ at distance at
least $\delta_{\rm iso}$ from both reference centroids of that sign.
Otherwise the internal separation makes the close same-sign matchings
bijective, and the Lipschitz bound contradicts \eqref{eq:gap}.
To verify the matching assertion, assign to each of the two perturbed
centroids of a given sign a reference centroid within
$\delta_{\rm iso}$. Two perturbed centroids cannot be assigned to
the same reference centroid: their mutual distance would be less than
$2\delta_{\rm iso}<c_0d_0$. Thus the assignment is a permutation
for each sign. In this ordering the maximum displacement is less than
$\delta_{\rm iso}$, and permutation invariance gives
$|\mathcal M^{\rm trig}-\mathcal M^{\rm quiet}|
\leq C_{\rm gap}\delta_{\rm iso}/d_0\leq h_*/32$,
contradicting \eqref{eq:gap}.
The choice of this centroid may depend on $N$; only the finite family
of tests constructed next must be fixed.

The support balls in Proposition~\ref{prop:finite}, together with
$|Z_N^\kappa(t)|\leq Ch^\beta$, place the patches with $n\geq1$ in
$B(0,1/4)$ and those with $n=0$ in $B(e_1,1/4)$ at the crossing times.
Choose a fixed smooth function $0\leq\chi_{\rm bump}\leq1$ that is one on
$B_{1/4}$ and supported in $B_{1/2}$, and a finite
$\delta_{\rm iso}/32$-net $\mathcal G$ of $B(e_1,1/4)$. For
$\xi\in\mathcal G$ and $\varsigma\in\{-1,1\}$ put
\begin{equation}\label{eq:test-family}
 \varphi_{\xi,\varsigma}(x)
       =\varsigma\chi_{\rm bump}((x-\xi)/\delta_{\rm iso}).
\end{equation}
A node near $b_*$ and the sign $\varsigma=\varepsilon_*$ give a test equal
to $\varepsilon_*$ on the entire selected patch. Its support excludes
opposite-sign patches of the perturbed solution, same-sign patches of
the reference solution, and all patches with $n\geq1$.
For the first statement,
$|\xi-b_*|\leq\delta_{\rm iso}/32$ and the selected support radius
is less than $\delta_{\rm iso}/8$, so the patch lies in
$B(\xi,5\delta_{\rm iso}/32)\subset B(\xi,\delta_{\rm iso}/4)$.
Every reference centroid of sign $\varepsilon_*$ is at distance at least
$\delta_{\rm iso}$ from $b_*$; subtracting the same two allowances
still leaves distance greater than $\delta_{\rm iso}/2$ from $\xi$.
Opposite-sign perturbed centroids are even farther away, by
$c_0d_0\geq16\delta_{\rm iso}$. These inequalities exclude their
entire supports from the bump. Finally the two fixed balls
$B(0,1/4)$ and $B(e_1,1/4)$, and
$\delta_{\rm iso}\leq1/8$, exclude every index $n\geq1$.
Each remaining
reference contribution has sign $-\varepsilon_*$; subtracting it in the
vorticity difference therefore adds a nonnegative amount to the pairing. Hence one of this fixed
finite family satisfies
\[
 \int\varphi_{\xi,\varsigma}
  (\omega_N^{\mathrm{trig}}(t_N^{\rm exit})
   -\omega_N^{\mathrm{quiet}}(t_N^{\rm exit}))\idiff x\geq\Gamma_0.
\]
Retain an infinite subsequence on which this spatial test is one fixed
$\varphi_*$. The observation times remain in $[t_0,1/4]$.
A further subsequence satisfies $t_N^{\rm exit}\to t_*\in[t_0,1/4]$.
Keep $L$ and all resulting scales fixed during this cutoff limit.
For each branch, Proposition~\ref{prop:criterion} gives
\[
 \begin{split}
 \|\omega_N^\kappa(t_N^{\rm exit})-\omega^\kappa(t_*)\|_1
 &\leq\sup_{t\in[0,1]}\|\omega_N^\kappa(t)-\omega^\kappa(t)\|_1\\
 &\quad+\|\omega^\kappa(t_N^{\rm exit})-\omega^\kappa(t_*)\|_1
 \longrightarrow0.
 \end{split}
\]
Pairing with the fixed bounded test preserves the gap at $t_*>0$.
This is the observable distinction represented schematically in
Figure~\ref{fig:material-histories}.
\end{proof}

\subsection{Completion of the theorem}

Apply Proposition~\ref{prop:criterion} with the common datum
\eqref{eq:datum}, the test $\varphi_*$, and $\gamma=\Gamma_0$.
It supplies two canonical unforced Euler evolutions, their exact
invertible measure-preserving flows, and a time $t_*>0$ at which the
observable gap is at least $\Gamma_0$. Choose
$\gamma_*=\Gamma_0/2$. Their vorticities are strongly continuous in
$L^1\cap L^q$ and their velocities in $L^2_{\rm loc}$, with the
common initial traces required in the theorem.

If $p<q$, set
\begin{equation}\label{eq:final-interpolation}
 \theta_{\rm int}=\frac{q-p}{p(q-1)},\qquad
 \|f\|_p\leq\|f\|_1^{\theta_{\rm int}}
                    \|f\|_q^{1-\theta_{\rm int}}.
\end{equation}
Apply this to time increments and convergence differences. It gives
the required strong $L^p$ continuity; when $p=q$ no interpolation is
needed. The material decomposition is also retained explicitly.  Define
\begin{equation}\label{eq:infinite-material-decomposition}
 \begin{gathered}
 E_{n,j}^\kappa(t)=\{x:F^\kappa(t,x)\in D_{n,j}^0\},\qquad
 \mu_{n,j}^\kappa(t,x)=\mu_{n,j}^0(F^\kappa(t,x))
                =\frac{\ind_{E_{n,j}^\kappa(t)}(x)}{\pi a_n^2},\\
 \omega^\kappa(t,x)=\sum_{n\geq0}\sum_{j=1}^{4}
              \epsilon_j\Gamma_n\mu_{n,j}^\kappa(t,x),\qquad
 |E_{n,j}^\kappa(t)|=\pi a_n^2.
 \end{gathered}
\end{equation}
The sets $E_{n,j}^\kappa(t)$ are pairwise disjoint modulo null sets.
The series converges in $L^1\cap L^q$, uniformly in time, because
measure preservation gives exactly the same tail norms as at time zero.
Thus both solutions transport every one of the same initial material
components, with the same signs, values and areas.  Their arrangements
are separated by the fixed observable at $t_*$.
All quantities were chosen for the prescribed exponent before the
cutoff limit. This proves Theorem~\ref{thm:main}.

\section{Proofs of the three patch estimates}
\label{sec:technical-proofs}\label{sec:concentration}\label{sec:local}\label{app:concentration}\label{sec:response}\label{app:response}

This section proves Lemmas~\ref{lem:concentration}, \ref{lem:neutral-pair}
and~\ref{lem:response}. Throughout, the kernel is
$K(x)=x^\perp/|x|^2$, and all interaction times are normalized by
$d^2/\Gamma$. The densities have mass one; their signed circulation
is supplied by the coefficient multiplying them in the velocity.

\subsection{Concentration in a translating and rotating frame}
\label{subsec:proof-concentration}
\begin{proof}[Proof of Lemma~\ref{lem:concentration}]
Initially the claimed support radius tends to zero with $\eta$, even
when $s_{\rm pre}$ has its largest allowed value. We may therefore stop
the argument when a centered support first reaches the boundary of a
ball $B_{r_*}$, where $0<r_*\leq r/2$ is chosen so that
$C(\ell_-,\ell_+,C_0)r_*^2\leq1/100$; here $r$ is the tube radius
in the statement of the lemma. On this
interval we first control the centroid error through the variance,
then prove the variance bound by a corrected energy. A high moment
and a scalar comparison will finally keep the support strictly inside
$B_{r_*}$ and remove the stop.

\subsubsection*{The centroid error and the logarithmic deficit}

The exact centroid equation indicates the estimate we need. Split its
averaging error into interactions between patches and the exterior:
\[
 \begin{aligned}
 \delta_i^{\rm int}&=\sum_{j\ne i}\epsilon_j
 \left(\iint K(x-y)\idiff\mu_i(x)\idiff\mu_j(y)-K(b_i-b_j)\right),\\
 \delta_i^{\rm ext}&=\int(V(x)-V(b_i))\idiff\mu_i(x).
 \end{aligned}
\]
Oddness of $K$ makes the self average zero. In each remaining integral,
Taylor expansion at the centroids has no linear contribution because
$\int(x-b_i)\idiff\mu_i=0$. The fixed separation bounds the Hessians
of the interaction kernels. More explicitly, put $z=b_i-b_j$ and
$h=(x-b_i)-(y-b_j)$. Then
\[
 K(z+h)-K(z)-\nabla K(z)h
 =\int_0^1(1-v)\nabla^2K(z+vh)[h,h]\idiff v.
\]
All segments in this formula remain away from zero on the stopped
support neighborhoods, and
$\iint|h|^2\idiff\mu_i\idiff\mu_j=I_i+I_j$.
The same formula for $V(b_i+h)$ has remainder at most $H|h|^2/2$.
Thus
\begin{equation}\label{eq:centroid-taylor}
 b_i'=V_i^{\rm pt}+\delta_i^{\rm int}+\delta_i^{\rm ext},\qquad
 |\delta_i^{\rm int}|\leq C\sum_jI_j,\qquad
 |\delta_i^{\rm ext}|\leq H I_i.
\end{equation}

It therefore suffices to prove $\sum_i I_i\leq C\eta^2$ to obtain an
$O(\eta^2)$ internal averaging error on this interval.

The rearrangement deficit supplies the required control of $I_i$.
For a probability rearrangement of $f_\eta$, put
$A_\mu(r)=\iint\ind_{\{|x-y|<r\}}\idiff\mu(x)\idiff\mu(y)$.
The rearrangement inequality for the ball kernel
\cite{BrascampLiebLuttinger1974} gives $A_\mu(r)\leq A_{f_\eta}(r)$.
Its three nonnegative factors here are $\mu(x)$, $\mu(y)$ and
$\ind_{B_r}(x-y)$; their linear arguments have scalar coefficients
$(1,0)$, $(0,1)$ and $(1,-1)$ in the two vector variables.
The first two rearrange to $f_\eta$, and the ball indicator is
already radially decreasing.
Apply the layer representation first to bounded logarithmic kernels and
then pass to the limit. Since $A_{f_\eta}(r)=1$ for $r\geq2\eta$,
\begin{equation}\label{eq:layer-cake}
 \mathcal D(\mu)=\int_0^\infty
   \frac{A_{f_\eta}(r)-A_\mu(r)}r\idiff r
 \geq\iint\log_+\frac{|x-y|}{2\eta}\idiff\mu(x)\idiff\mu(y).
\end{equation}
On $0\leq r\leq2r_*$, elementary estimate of
$x^2 \leq 1 + C\log_+ x$ for $0 < x = r/(2\eta) < r_*/\eta$ gives
\[
 r^2\leq4\eta^2+\frac{4r_*^2}{\log(r_*/\eta)}
                   \log_+\frac r{2\eta},\qquad 0 < \eta < r_*.
\]
Using $2I_i=\iint|x-y|^2\idiff\mu_i(x)\idiff\mu_i(y)$, we obtain
\begin{equation}\label{eq:coercivity}
 \sum_iI_i\leq C\eta^2+\epsilon_\eta\sum_i\mathcal D_i,
 \qquad \epsilon_\eta=\frac{4r_*^2}{\log(r_*/\eta)}\longrightarrow0.
\end{equation}
The coefficient $\epsilon_\eta\to0$ is decisive: it will allow all
quadratic interaction errors to be absorbed after estimating the
logarithmic deficit by an energy identity.

\subsubsection*{Canceling the rigid rotation in the energy identity}

Set
$G_0=\log|x|$, write $V=(\nabla\Psi)^\perp$, and define
\[
 E_{\rm cl}=\frac1{4\pi}\left\|K*\sum_i\epsilon_i\mu_i\right\|_2^2,
 \qquad H_{\rm pt}=-\sum_{i<j}\epsilon_i\epsilon_jG_0(b_i-b_j).
\]
Neutrality makes $E_{\rm cl}$ finite. To compare this energy with the
energy of the centroids, subtract the point Hamiltonian, subtract the
fixed disk self energies, and center the exterior stream terms. With
$m$ components, the resulting quantity is
\[
 \mathcal Q=E_{\rm cl}-\sum_i\epsilon_i\int\Psi\idiff\mu_i-H_{\rm pt}
       +\sum_i\epsilon_i\Psi(b_i)-\frac{m}{2}\mathcal S(f_\eta).
\]
Here is the sign computation. Put $\omega_{\rm cl}=\sum_i\epsilon_i\mu_i$
and $\psi_{\rm cl}=G_0*\omega_{\rm cl}$. Since $\Delta\psi_{\rm cl}
=2\pi\omega_{\rm cl}$ and $K*\omega_{\rm cl}=(\nabla\psi_{\rm cl})^\perp$,
integration by parts gives
\begin{equation}\label{eq:material-energy-expansion}
 E_{\rm cl}=-\frac12\int\psi_{\rm cl}\omega_{\rm cl}
 =\frac12\sum_i\mathcal S(\mu_i)
 -\sum_{i<j}\epsilon_i\epsilon_j
       \iint G_0(x-y)\idiff\mu_i(x)\idiff\mu_j(y).
\end{equation}
The boundary term at infinity vanishes because the total circulation
is zero: $\nabla\psi_{\rm cl}=O(|x|^{-2})$. Subtracting the fixed disk
self energies in this identity gives exactly
$-\tfrac12\sum_i\mathcal D_i$.
Every other difference is a quadratic Taylor remainder, since the
first moments about $b_i$ vanish. Thus
\[
 \mathcal Q=-\frac12\sum_i\mathcal D_i+O\left(\sum_iI_i\right).
\]
An affine change of a local stream does not change its centered
remainder, so this definition uses only streams on the stated
neighborhoods.

Write $\widehat\mu_i(x)=\mu_i(b_i+x)$ and
$\mathscr R_i f(x)=f(b_i+x)-f(b_i)-\nabla f(b_i)\cdot x$.
We next differentiate before estimating any velocity. Exact transport
by $u_{\rm cl}+V$, where $u_{\rm cl}=(\nabla\psi_{\rm cl})^\perp$, gives
\[
 E_{\rm cl}'=-\int\nabla\psi_{\rm cl}\cdot V\,\omega_{\rm cl},\qquad
 \left(\int\Psi\omega_{\rm cl}\right)'
 =\int\Psi_s\omega_{\rm cl}
      +\int\nabla\Psi\cdot u_{\rm cl}\,\omega_{\rm cl}.
\]
The self contribution to the first identity and the $V$ contribution
to the second vanish by $a\cdot a^\perp=0$; the two remaining transport
terms cancel because
$\nabla\Psi\cdot (\nabla\psi_{\rm cl})^\perp
=-\nabla\psi_{\rm cl}\cdot (\nabla\Psi)^\perp$.
For the point expression, use $K^\perp=-\nabla G_0$ and $V^\perp=-\nabla\Psi$.
We obtain
\[
 \left(E_{\rm cl}-\sum_i\epsilon_i\int\Psi\idiff\mu_i\right)'
       =-\sum_i\epsilon_i\int\Psi_s\idiff\mu_i,
 \qquad
 \nabla_{b_i}\left(H_{\rm pt}-\sum_j\epsilon_j\Psi(b_j)\right)
       =\epsilon_i(V_i^{\rm pt})^\perp.
\]
In particular, its derivative along the actual centroids is
\[
 \left(H_{\rm pt}-\sum_i\epsilon_i\Psi(b_i)\right)'
 =\sum_i\epsilon_i(V_i^{\rm pt})^\perp\cdot
                  (\delta_i^{\rm int}+\delta_i^{\rm ext})
                       -\sum_i\epsilon_i\Psi_s(b_i),
\]
since $(V_i^{\rm pt})^\perp\cdot V_i^{\rm pt}=0$. Centering each stream
remainder removes its linear part as well as its constant part.
The exterior error combines with the stream derivative because
$(V_i^{\rm pt})^\perp\cdot\delta_i^{\rm ext}
=V_i^{\rm pt}\cdot\int(\nabla\Psi(b_i+x)-\nabla\Psi(b_i))\idiff\widehat\mu_i$.
Consequently
\begin{equation}\label{eq:cluster-work}
 \mathcal Q'=-\sum_i\epsilon_i\int
 \mathscr R_i\bigl((\partial_s+V_i^{\rm pt}\cdot\nabla)\Psi\bigr)
       \idiff\widehat\mu_i
       -\sum_i\epsilon_i(V_i^{\rm pt})^\perp\cdot\delta_i^{\rm int}.
\end{equation}
The last term of \eqref{eq:cluster-work} contains the full point
velocities, including the rigid rotation. Its translation and rotation
parts can be canceled exactly. Internal linear and angular impulse
conservation give
\[
 \sum_i\epsilon_i\delta_i^{\rm int}=0,\qquad
 \sum_i\epsilon_i b_i\cdot\delta_i^{\rm int}
       =-\frac12\sum_i\epsilon_i(I_i')_{\rm int}.
\]
Indeed the angular impulse of the components is
$\sum_i\epsilon_i(|b_i|^2+I_i)$; subtract the point angular-impulse identity
from its internal derivative. Both internal derivatives vanish, since
pairing the terms with indices $i,j$ leaves
$(x-y)\cdot K(x-y)=0$ in the material identity and
$(b_i-b_j)\cdot K(b_i-b_j)=0$ in the point identity.
In the last term of \eqref{eq:cluster-work}, translation contributes
$-(\boldsymbol c')^\perp\cdot\sum_i\epsilon_i\delta_i^{\rm int}=0$,
whereas rotation contributes
\[
 -\sum_i\epsilon_i\bigl(\Omega (b_i-\boldsymbol c)^\perp\bigr)^\perp
                   \cdot\delta_i^{\rm int}
 =\Omega\sum_i\epsilon_i b_i\cdot\delta_i^{\rm int}
 =-\frac\Omega2\sum_i\epsilon_i(I_i')_{\rm int}.
\]
This identity determines the sign of the
required rotation correction:
$\widetilde{\mathcal Q}=\mathcal Q+(\Omega/2)\sum_i\epsilon_iI_i$.
Substitution of the rigid-frame decomposition into \eqref{eq:cluster-work}
cancels both the translation and the internal rotation terms:
\begin{align}\label{eq:rotation-work}
 \widetilde{\mathcal Q}'={}&-\sum_i\epsilon_i\int
 \mathscr R_i\bigl((\partial_s+V_i^{\rm pt}\cdot\nabla)\Psi\bigr)
       \idiff\widehat\mu_i-\sum_i\epsilon_iW_i^\perp\cdot\delta_i^{\rm int}\notag\\
 &+\frac{\Omega'}2\sum_i\epsilon_iI_i
       +\frac\Omega2\sum_i\epsilon_i(I_i')_{\rm ext}.
\end{align}
For the first term in \eqref{eq:rotation-work}, the Hessian of the
stream being centered is
$-\bigl((\partial_s+V_i^{\rm pt}\cdot\nabla)\nabla V\bigr)^\perp$,
where the perpendicular operation acts on each output column.
The frozen vector $V_i^{\rm pt}$ is constant when taking this Hessian,
so its norm is bounded by $B$. Also
\[
 (I_i')_{\rm ext}
 =2\int (x-b_i)\cdot(V(x)-V(b_i))\idiff\mu_i(x),
 \qquad |(I_i')_{\rm ext}|\leq2AI_i.
\]
Together with \eqref{eq:centroid-taylor}, these bounds give
\[
 |\widetilde{\mathcal Q}'|
 \leq C(B+\max_i|W_i|+|\Omega'|+|\Omega|A)\sum_iI_i.
\]
Also $\sum_i\mathcal D_i\leq C\sum_iI_i+C|\widetilde{\mathcal Q}|$.
To see the absorption explicitly, denote $I=\sum_iI_i$, $D=\sum_i\mathcal D_i$
and $Q=|\widetilde{\mathcal Q}|$. The two inequalities are
$I\leq C\eta^2+\epsilon_\eta D$ and $D\leq C(I+Q)$.
Choose $\eta$ so that $C\epsilon_\eta\leq1/2$; substitution gives
$I\leq C\eta^2+C\epsilon_\eta Q$. If
$k=B+\max_i|W_i|+|\Omega'|+|\Omega|A$, then
\[
 \mathrm D^+Q\leq Ck\eta^2+C\epsilon_\eta kQ,\qquad
 Q(s)\leq\left(Q(0)+C\eta^2\int_0^s k\right)
                   \exp\left(C\epsilon_\eta\int_0^s k\right).
\]
The integral of $k$ is bounded by \eqref{eq:cost}, including the
assumed bound on $|\Omega|$. Since $Q(0)\leq C\eta^2$, first $Q$,
then $I$, and finally $D$ are bounded by $C\eta^2$. Hence
\begin{equation}\label{eq:variance-deficit}
 \sum_i(I_i+\mathcal D_i)+|\widetilde{\mathcal Q}|\leq C\eta^2.
\end{equation}
This proves the variance estimate independently of $S$: the integral
of $|\Omega|$ never occurs in the bound. Inserting it in
\eqref{eq:centroid-taylor} also bounds the actual centroid errors by
$C\eta^2(1+H)$.

\subsubsection*{Adapted moments control the tails}

Small variance alone permits a small amount of vorticity far from the
centroid. We next bound a higher moment, using level sets along which
the frozen point velocity has zero normal component. Choose
$\theta'=\Omega$ and put $x=b_i+\operatorname{Rot}(\theta)y$.
Along a material trajectory this change of coordinates satisfies
\begin{equation}\label{eq:material-moving-coordinates}
 y'=\operatorname{Rot}(-\theta)(u(x)-b_i')-\Omega y^\perp.
\end{equation}
Put $R=\operatorname{Rot}(\theta)$ temporarily, and let
$P_i(x)=\sum_{j\ne i}\epsilon_jG_0(x-b_j)+\Psi(x)$ be the regular
point stream near $b_i$. Freezing the other patches at their centroids
and centering their velocity at $b_i$ gives
\[
 v_i^{\rm fr}(y)=\epsilon_iK(y)
       +R^{-1}\bigl(\nabla P_i(b_i+Ry)-\nabla P_i(b_i)\bigr)^\perp
       -\Omega y^\perp.
\]
Rotations commute with the perpendicular operation, and the chain rule gives
$\nabla_y[\mathscr R_iP_i(Ry)]=R^{-1}(\nabla P_i(b_i+Ry)-\nabla P_i(b_i))$.
Subtract the constant and linear parts of the regular stream, and include
the stream of the rotating coordinates. Define
\begin{equation}\label{eq:adapted-radius}
 \log\rho_i(y,s)=\log|y|+\frac1{\epsilon_i}
 \left[\mathscr R_i\left(\sum_{j\ne i}\epsilon_jG_0(\cdot-b_j)+\Psi\right)
       (\operatorname{Rot}(\theta)y)-\frac\Omega2|y|^2\right].
\end{equation}
In these coordinates the resulting frozen velocity is
$\epsilon_i(\nabla\log\rho_i)^\perp$. Its scalar product with
$\nabla\rho_i$ is zero. Thus differentiating a moment of $\rho_i$
eliminates this part of the velocity exactly.
The quadratic correction is chosen from this identity: the regular
field generally has a nonzero component in the direction $y$,
whereas its contribution to $v_i^{\rm fr}\cdot\nabla\rho_i$ cancels.
For a fixed even integer $\nu\geq4$, set
$F_{i,\nu}=\rho_i^\nu-|y|^\nu$. The regular term in
\eqref{eq:adapted-radius} is quadratic, so on the fixed small ball
\begin{equation}\label{eq:adapted-derivatives}
 \rho_i\asymp|y|,\quad |\nabla\rho_i|\leq C,\quad
 |\nabla F_{i,\nu}|\leq C_\nu|y|^{\nu+1},\quad
 |\nabla^2F_{i,\nu}|\leq C_\nu|y|^\nu.
\end{equation}
For example, write $\log\rho_i=\log|y|+q_i$ with
$q_i=O(|y|^2)$, $\nabla q_i=O(|y|)$ and $\nabla^2q_i=O(1)$.
Then $F_{i,\nu}=|y|^\nu(e^{\nu q_i}-1)$; differentiating this product
once and twice gives the powers $\nu+1$ and $\nu$ in
\eqref{eq:adapted-derivatives}.
We must also control the time dependence of these level sets. The bound is
\[
 |\partial_s\log\rho_i|\leq CK_0|y|^2,\qquad
 K_0=\max_i|W_i|+B+|\Omega|A+|\Omega'|
                   +C\eta^2(1+H+H^2).
\]
Indeed, the rotated relative centroid velocities equal
$\operatorname{Rot}(\theta)^{-1}(W_i-W_j+\delta_i-\delta_j)$,
where $\delta_i=\delta_i^{\rm int}+\delta_i^{\rm ext}$.
For clarity, put $a_{ij}=R^{-1}(b_i-b_j)$. Then
$a_{ij}'=R^{-1}(W_i-W_j+\delta_i-\delta_j)$, because the
translation cancels and differentiation of $R^{-1}$ cancels
$\Omega (b_i-b_j)^\perp$. The interaction remainder can be written as
\[
 G_0(a_{ij}+y)-G_0(a_{ij})-\nabla G_0(a_{ij})\cdot y
 =\int_0^1(1-v)\nabla^2G_0(a_{ij}+vy)[y,y]\idiff v.
\]
Its time derivative is at most $C|a_{ij}'||y|^2$ by the fixed
separation. The rigid rotation itself therefore makes no contribution
to the interaction part of $\partial_s\log\rho_i$.
For the exterior remainder the exact derivative is
\[
 \mathscr R_i((\partial_s+b_i'\cdot\nabla)\Psi)(\operatorname{Rot}(\theta)y)
 +\Omega\nabla\mathscr R_i\Psi(\operatorname{Rot}(\theta)y)
                      \cdot (\operatorname{Rot}(\theta)y)^\perp.
\]
Replacing $b_i'$ by $V_i^{\rm pt}$ costs $C\eta^2(H+H^2)|y|^2$;
the last displayed term costs $C|\Omega|A|y|^2$.
All terms of $K_0$ therefore have bounded integrals under the hypotheses:
the only additional constant term contributes at most
$C\eta^2S\leq C\eta^{3/2}$.

Let $\mu_i^{\rm rot}$ denote the probability density in these centered,
rotating coordinates. The adapted and ordinary moments are
$m_{i,\nu}=\int\rho_i^\nu\idiff\mu_i^{\rm rot}$,
$M_{i,\nu}=\int|y|^\nu\idiff\mu_i^{\rm rot}$.
By \eqref{eq:centroid-taylor} and the variance bound, the actual centered
velocity differs from the frozen tangent field by
$\epsilon_i(K*\mu_i^{\rm rot}-K)$ and a remainder bounded by
$C\eta^2(1+H)$. Indeed, replacing another patch by its centroid at
$b_i+Ry$ costs its variance times a bounded kernel Hessian, and
replacing $b_i'$ by $V_i^{\rm pt}$ costs \eqref{eq:centroid-taylor}.
It remains to estimate the self-interaction contribution
to $m_{i,\nu}'$. For its radial part, exchange $y$ and $z$ in half of
the double integral. The resulting kernel is
\[
 \frac\nu2(|y|^{\nu-2}y-|z|^{\nu-2}z)\cdot K(y-z),
\]
bounded by $C_\nu\min(|y|,|z|)\max(|y|,|z|)^{\nu-3}$.
For $|z|\leq|y|/2$, the scalar product, apart from its fixed factor,
is exactly
\[
 -\frac{(|y|^{\nu-2}-|z|^{\nu-2})y\cdot z^\perp}{|y-z|^2}.
\]
Its absolute value is at most
$C_\nu |y|^{\nu-2}|y||z|/|y|^2
=C_\nu|z||y|^{\nu-3}$.
For comparable radii use both
\[
 \bigl||y|^{\nu-2}-|z|^{\nu-2}\bigr|
 \leq C_\nu\max(|y|,|z|)^{\nu-3}|y-z|,
 \quad |y\cdot z^\perp|\leq\max(|y|,|z|)|y-z|.
\]
The correction $F=F_{i,\nu}$ must include the subtracted point velocity
at both variables. Its exact symmetric kernel is
\[
 \frac12\bigl[(\nabla F(y)-\nabla F(z))\cdot K(y-z)
             -\nabla F(y)\cdot K(y)-\nabla F(z)\cdot K(z)\bigr].
\]
It is bounded by $C_\nu\min(|y|,|z|)\max(|y|,|z|)^{\nu-1}$.
For $|z|\leq|y|/2$, rewrite twice this kernel as
\[
 \nabla F(y)\cdot(K(y-z)-K(y))
       -\nabla F(z)\cdot K(y-z)-\nabla F(z)\cdot K(z).
\]
The three absolute bounds are, respectively,
$C|y|^{\nu+1}|z|/|y|^2$,
$C|z|^{\nu+1}/|y|$, and $C|z|^\nu$.
Each is at most $C|z||y|^{\nu-1}$. Values at $z=0$ are understood
by continuity, since $\nabla F(z)=O(|z|^{\nu+1})$.
For comparable radii the Hessian bound in \eqref{eq:adapted-derivatives}
cancels the diagonal singularity: along the segment from $z$ to $y$,
$|\nabla F(y)-\nabla F(z)|\leq C\max(|y|,|z|)^\nu|y-z|$.
The other two terms are $O(\max(|y|,|z|)^\nu)$, which is the required
bound because the two radii are comparable.

To integrate the radial bound, split the product space into
$|z|\leq|y|$ and its transpose. On the first part enlarge the inner
integration to all $z$, obtaining
\[
 \iint\min(|y|,|z|)\max(|y|,|z|)^{\nu-3}
                 \idiff\mu_i^{\rm rot}(y)\idiff\mu_i^{\rm rot}(z)
 \leq2M_{i,1}M_{i,\nu-3}
 \leq C\eta m_{i,\nu}^{1-3/\nu}.
\]
Here $M_{i,1}\leq\sqrt{I_i}\leq C\eta$, probability interpolation
gives $M_{i,\nu-3}\leq M_{i,\nu}^{1-3/\nu}$, and
$M_{i,\nu}\asymp m_{i,\nu}$. The correction kernel is bounded the
same way after using $\max(|y|,|z|)^2\leq r_*^2$.
The regular velocity remainder contributes at most
$C\eta^2(1+H)M_{i,\nu-1}$; use
$M_{i,\nu-1}\leq r_*^2M_{i,\nu-3}$ and
$\eta^2(1+H)\leq\eta(1+\eta H)$ for $\eta\leq1$.
Finally $\partial_s(\rho_i^\nu)=\nu\rho_i^\nu\partial_s\log\rho_i$
contributes at most $C_\nu K_0m_{i,\nu}$.
Thus
\begin{equation}\label{eq:adapted-moment}
 m_{i,\nu}'\leq C_\nu\eta(1+\eta H)m_{i,\nu}^{1-3/\nu}
                              +C_\nu K_0m_{i,\nu}.
\end{equation}
For $\nu=24$, the power on the right is $7/8$. Multiplication by
$m_{i,24}^{-7/8}$, justified first with a positive regularization,
therefore gives a linear inequality for $m_{i,24}^{1/8}$. Its initial
value is bounded by $C\eta(1+s_{\rm pre})$. The bounded coefficient
integrals can be used without a factor exponential in $S$: with
$z=m_{i,24}^{1/8}$, the inequality reads
\[
 z'\leq C\eta(1+\eta H)+CK_0z,\qquad
 z(s)\leq C\left(z(0)+\eta s+\eta^2\int_0^sH\right)
                       \exp\left(C\int_0^sK_0\right)
 \leq C\eta(1+s_{\rm pre}+s).
\]
Together with $\rho_i\asymp|y|$, this yields
\begin{equation}\label{eq:aged-moment}
 M_{i,24}(s)\leq C[\eta(1+s_{\rm pre}+s)]^8.
\end{equation}

\subsubsection*{The largest support radius}

To pass from a moment bound to the full support, let $R(s)$ be the
maximum of $\rho_i$ on the transported support. At a maximizing point
$y$, split the self integral into $|z|\leq|y|/2$ and its complement.
In the first region subtracting $K(y)$ leaves at most
$C\eta|y|^{-2}$. The complementary mass is bounded by
$CM_{i,24}|y|^{-24}$. For any nonnegative density, splitting the singular
integral into a disk of radius $r$ and its complement and minimizing in
$r$ gives
\[
 \int|y-z|^{-1}f(z)\idiff z
       \leq C\|f\|_\infty^{1/2}\|f\|_1^{1/2}\qquad(f\geq0).
\]
In detail, the near and far terms are at most
$2\pi r\|f\|_\infty+r^{-1}\|f\|_1$; choose
$r=(\|f\|_1/\|f\|_\infty)^{1/2}$ when $f\ne0$.
The preserved density bound
$\|\mu_i^{\rm rot}\|_\infty=(\pi\eta^2)^{-1}$ then controls the
contribution of this complementary mass. Inserting
\eqref{eq:aged-moment} and the regular errors gives
\begin{equation}\label{eq:support-dini}
 \mathrm D^+R\leq C\eta R^{-2}
       +C\eta^3(1+s_{\rm pre}+s)^4R^{-12}
       +C\eta^2(1+H)+CK_0R^3.
\end{equation}
Indeed, the outer mass $m_{\rm tail}$ satisfies
$m_{\rm tail}\leq CM_{i,24}R^{-24}$, and its singular integral is
$C\eta^{-1}m_{\rm tail}^{1/2}
 \leq C\eta^3(1+s_{\rm pre}+s)^4R^{-12}$.
The outer subtracted point field obeys the same bound: its magnitude is
at most $Cm_{\rm tail}/|y|\leq C\eta^{-1}m_{\rm tail}^{1/2}$, since
$m_{\rm tail}\leq1$ and the contact radii below satisfy $|y|\geq\eta$.

The powers in the following comparison function are chosen so that its
derivative dominates both inverse powers of $R$ in
\eqref{eq:support-dini} throughout the permitted interval:
\[
 R_{\rm bar}(s)=C_{\rm bar}\eta^{2/9}(1+s_{\rm pre}+s)^{2/5}
       \exp\left(C_{\rm exp}\int_0^s(K_0+H)\idiff r\right).
\]
Choose $C_{\rm exp}$ to absorb the coefficient terms in
\eqref{eq:support-dini}. Relative to the derivative of the polynomial
factor, namely $(2/5)R_{\rm bar}/(1+s_{\rm pre}+s)$, the remaining
terms are bounded, respectively, by
\[
 CC_{\rm bar}^{-3}\eta^{1/3}T^{-1/5},\qquad
 CC_{\rm bar}^{-13}\eta^{1/9}T^{-1/5},\qquad
 CC_{\rm bar}^{-1}\eta^{16/9}T^{3/5},\quad T=1+s_{\rm pre}+s.
\]
All three ratios tend uniformly to zero for
$s_{\rm pre}+s\leq\eta^{-1/2}$. Choose $C_{\rm bar}$ above the initial
support allowance and then choose $\eta$ small. The maximal adapted
radius is absolutely continuous: the transported support is a compact
image of the initial support, its trajectories have a common integrable
speed bound, and the time dependence of $\rho_i$ has the same property.
Scalar comparison, first with a strict margin and then by taking its
limit, gives $R\leq R_{\rm bar}$. It applies also when the coefficients
are merely integrable. Since the exponential factor is bounded, the
largest resulting radius is
$O(\eta^{2/9-(1/2)(2/5)})=O(\eta^{1/45})$. For small $\eta$ this is
strictly less than $r_*$, so the original support stop cannot occur.

To justify the identities for transported patches, regularize the
logarithmic kernel, apply exact transport to the integrals, and then
remove the regularization. The first
unsymmetrized singularity is $|x-y|^{-1}$, which is integrable against
bounded, compactly supported densities. The symmetrized moment kernels
have the bounds proved above. The continuous flow of the finite
bounded-vorticity solution supplies the compact support images used in
the scalar comparison. These facts justify the argument for the actual
transported patches and complete the proof.
\end{proof}

\subsection{Concentration for a translating neutral pair}
\label{subsec:proof-pair}
\begin{proof}[Proof of Lemma~\ref{lem:neutral-pair}]
Put $\boldsymbol d=b_+-b_-$ and define
\[
 \delta v=-\iint K(x-y)\idiff\mu_+(x)\idiff\mu_-(y)+K(\boldsymbol d),
 \qquad \delta V_\pm=\int(V(x)-V(b_\pm))\idiff\mu_\pm(x).
\]
Oddness of $K$ shows that the same $\delta v$ occurs in both centroid
equations. Subtracting them gives
\[
 b_\pm'=-K(\boldsymbol d)+V(b_\pm)+\delta v+\delta V_\pm,
 \qquad
 \boldsymbol d'=V(b_+)-V(b_-)+\delta V_+-\delta V_-.
\]
In particular, the mutual interaction does not change $\boldsymbol d$,
even when the patches are not disks. Taylor expansion gives
$|\delta v|\leq C|\boldsymbol d|^{-3}(I_++I_-)$ and
$|\delta V_\pm|\leq HI_\pm$. The two-patch specialization of the energy is
\[
 \mathcal Q_{\rm pair}=\frac1{4\pi}\|K*(\mu_+-\mu_-)\|_2^2
 -\int\Psi\idiff(\mu_+-\mu_-)
 -\left(\log|\boldsymbol d|-\Psi(b_+)+\Psi(b_-)\right)
 -\mathcal S(f_\eta).
\]
The last term is the self energy of the original disks and is fixed in
time. Differentiating as in \eqref{eq:cluster-work} gives
\[
 \mathcal Q_{\rm pair}'=-\sum_{i=\pm}\epsilon_i\int
 \mathscr R_i((\partial_s+V_i^{\rm pt}\cdot\nabla)\Psi)\idiff\widehat\mu_i
 -(V(b_+)-V(b_-))^\perp\cdot\delta v,
 \quad V_i^{\rm pt}=-K(\boldsymbol d)+V(b_i).
\]
The common velocity $-K(\boldsymbol d)$ has canceled from this identity;
only the exterior difference multiplies $\delta v$. Include the quotient
$|V(b_+)-V(b_-)|/|\boldsymbol d|$ in $A$, since gradients on disjoint
neighborhoods alone need not bound that quotient.
Stop provisionally at the first time that
$|\boldsymbol d|\notin(\ell_-/2,2\ell_+)$ or a centered support reaches
$r_*\leq r/2$. On this stopped interval all kernel derivatives used
above are bounded in terms of $\ell_-$ and $r$. In particular, with
$I=I_++I_-$ and $D=\mathcal D_++\mathcal D_-$,
\[
 \left|\mathcal Q_{\rm pair}+\tfrac12D\right|\leq CI,
 \qquad |\mathcal Q_{\rm pair}'|\leq C(A+B)I,
 \qquad I\leq C\eta^2+\epsilon_\eta D.
\]
The first inequality is the same centered Taylor expansion as in
\eqref{eq:material-energy-expansion}; the second uses
$|V(b_+)-V(b_-)|\leq A|\boldsymbol d|$ and the bound for $\delta v$.
Choose $\eta$ so that $C\epsilon_\eta\leq1/2$ and apply the scalar
Gronwall calculation following \eqref{eq:rotation-work}. Since
$\int_0^S(A+B)\leq\delta$, it gives
\begin{equation}\label{eq:pair-energy-bound}
 I+D+|\mathcal Q_{\rm pair}|\leq C\eta^2.
\end{equation}
No integral of the translation speed $|K(\boldsymbol d)|$ occurs.

The exact difference equation and \eqref{eq:pair-energy-bound} give
\[
 |\boldsymbol d'|\leq A|\boldsymbol d|+C\eta^2H,
 \qquad
 |\boldsymbol d(s)-\boldsymbol d(0)|
 \leq e^\delta\bigl(|\boldsymbol d(0)|+C\eta^2\bigr)\delta.
\]
For example, first apply Gronwall to $|\boldsymbol d|$ and then
integrate its derivative. Choosing $\delta_0$ so that
$e^{\delta_0}(\ell_++C)\delta_0\leq\ell_-/4$ improves the separation
stop to $3\ell_-/4\leq|\boldsymbol d|\leq\ell_++\ell_-/4$.

For the moment estimate, take $\Omega=0$. The adapted radius at the
positive center is explicitly
\[
 \log\rho_+(x)=\log|x|-\log|\boldsymbol d+x|+\log|\boldsymbol d|
       +\frac{\boldsymbol d\cdot x}{|\boldsymbol d|^2}
        +\mathscr R_+\Psi(x),
\]
At the negative center the corresponding expression is
\[
 \log\rho_-(x)=\log|x|-\log|\boldsymbol d-x|+\log|\boldsymbol d|
       -\frac{\boldsymbol d\cdot x}{|\boldsymbol d|^2}
       -\mathscr R_-\Psi(x).
\]
The frozen pair field is tangent to its levels. Differentiating its
regular correction costs
$C(A+B+H+H^2)|x|^2$ and the already controlled variance errors.
More precisely, its coefficient $K_{\rm pair}$ satisfies
\[
 K_{\rm pair}\leq C(A+B+H+H^2)+C\eta^2(1+H+H^2),
 \qquad \int_0^S K_{\rm pair}\leq C(\delta+\eta^{3/2}).
\]
Use this coefficient in \eqref{eq:adapted-moment} with $\nu=24$.
It gives \eqref{eq:aged-moment}, and the comparison function following
\eqref{eq:support-dini} improves the support stop to
$C\eta^{1/45}<r_*/2$. The two strict improvements remove both stops.
The moment and support argument applies with $s_{\rm pre}$
equal to the full preceding elapsed time and with the same normalizing
length $d$. Passing from four
patches to two pairs changes the coordinates used for the estimate;
the patches continue to be transported by their original Euler flow.

\end{proof}

\subsection{Comparison of initially identical disks}
\label{subsec:proof-response}
\begin{proof}[Proof of Lemma~\ref{lem:response}]
We write $\sigma=s=t/\tau$ for the interaction time in
Lemma~\ref{lem:response}, and $\zeta=\sigma/\eta^2=\Gamma t/a^2$
for the fast core time; $\tau=d^2/\Gamma$ remains the physical
interaction scale. The sign is $\epsilon\in\{-1,1\}$.
The proof first derives the boundary equation and identifies the
nonzero frequencies of its linear part. Its nonlinear terms lose one
analytic derivative. A decreasing analytic radius absorbs that loss.
The change of variable $z=g+\mathcal L^{-1}f$, defined below, removes
the leading exterior forcing from each separate equation and gives
$\|g^j\|=O(\eta^2)$. With this bound established, subtraction of the
original equations gives \eqref{eq:response} using only $\Delta_A$.
We also justify that the analytic solutions constructed in this way are
the actual transported patches.

\subsubsection*{Centering the exact boundary equation}

A radial graph makes the boundary velocity and its constraints explicit.
Write
$Z_r(\theta)=(1+r(\theta))e(\theta)$, with $e(\theta)=(\cos\theta,\sin\theta)$,
and let $D_r$ be its enclosed domain. In the coordinates $(x-db)/a$ and
fast time $\zeta$, the patch has normalized density $\epsilon/\pi$,
curl $2\epsilon$, and area $\pi$. Indeed, its physical-coordinate
normalized density is
$\epsilon\Gamma(\pi a^2)^{-1}\ind_{D^{\rm phys}}$.
Under $x=db+a y$ and $\diff t=(a^2/\Gamma)\diff\zeta$,
homogeneity of $K$ changes its self velocity to
\[
 \frac{a}{\Gamma}\frac{\epsilon\Gamma}{\pi a^2}
       \int_{D^{\rm phys}}K(db+ay-x')\idiff x'
 =\frac{\epsilon}{\pi}\int_{D_r}K(y-z)\idiff z.
\]
The disk therefore has velocity $\epsilon y^\perp$ inside it and
$\epsilon K(y)$ outside it; both equal $\epsilon y^\perp$ on its
unit boundary.
The divergence theorem gives its exact
self velocity on the boundary:
\begin{equation}\label{eq:boundary-velocity}
 V_r(Z_r(\theta))=-\frac\epsilon\pi\int_0^{2\pi}
       \log|Z_r(\theta)-Z_r(\beta)|Z_r'(\beta)\idiff\beta.
\end{equation}
Indeed $(x-y)^\perp/|x-y|^2=-(\nabla_y\log|x-y|)^\perp$, and the positively oriented
tangent is the outward normal rotated by $\pi/2$. The logarithm is integrable.

In these coordinates, before subtracting the centroid velocity, the
exterior field is
\begin{equation}\label{eq:core-exterior-scaling}
 U(\zeta,y)=\eta V_{\rm ext}(\eta^2\zeta,b(\eta^2\zeta)+\eta y).
\end{equation}
Its mean $\overline U_r=\pi^{-1}\int_{D_r}U$ equals
$\eta b'(\sigma)$, because the self average is zero. The centered radial
equation is therefore
\begin{equation}\label{eq:radial-equation}
 r_\zeta=\left(e(\theta)-\frac{r_\theta}{1+r}e(\theta)^\perp\right)
       \cdot\bigl(V_r+U-\overline U_r\bigr)(Z_r(\theta)).
\end{equation}
In fact, differentiation of the coordinate map gives
$y_\zeta=V_r(y)+U(\zeta,y)-\eta b'(\sigma)$:
the centroid term is $(d/a)(\diff\sigma/\diff\zeta)b'=\eta b'$.
The vector $e-r_\theta(1+r)^{-1}e^\perp$ has scalar product zero with
$Z_r'=r_\theta e+(1+r)e^\perp$ and scalar product one with $e$.
Taking its scalar product with the boundary velocity thus gives
exactly $r_\zeta$, independently of tangential parametrization.
The factor containing $r_\theta$ converts the normal velocity of the
curve to the derivative of its radial graph. Through this factor, the
tangential circular self velocity contributes to the linearization.

The area and centroid constraints remove exactly the modes on which
the linearized self equation will have zero frequency. In radial
coordinates they read
\begin{equation}\label{eq:radial-constraints}
 r_0=-\tfrac12(r^2)_0,\qquad
 r_{\pm1}=-(r^2)_{\pm1}-\tfrac13(r^3)_{\pm1}.
\end{equation}
More precisely, the polar formulas for area and first moment are
\[
 |D_r|=\frac12\int_0^{2\pi}(1+r)^2\idiff\theta=\pi,
 \qquad \int_{D_r}y\idiff y
    =\frac13\int_0^{2\pi}e(\theta)(1+r)^3\idiff\theta=0.
\]
The zero coefficient of the first equation and the coefficients
$\pm1$ of the second give \eqref{eq:radial-constraints}.
For $v=\Pi_{\geq2}r$ small in $\|\cdot\|_{\varrho,1}$, these three
equations give $r=v+\ell(v)$ by contraction in the three low modes.
Writing their right sides as $\Phi(v,\ell)$, one has
$\Phi(0,0)=0$ and $\nabla\Phi(0,0)=0$; differentiating
$\ell=\Phi(v,\ell)$ gives
$\nabla\ell=(1-\nabla_\ell\Phi)^{-1}\nabla_v\Phi$.
For example, on a sufficiently small fixed ball the derivative of their
right side in those modes has norm less than $1/4$. The map $\ell$ is
analytic and quadratic, with
\[
 \|\ell(v)\|_{\varrho,2}\leq C\|v\|_{\varrho,1}^2,
 \qquad \|\nabla\ell(v)w\|_{\varrho,1}
          \leq C\|v\|_{\varrho,1}\|w\|_{\varrho,1}.
\]
Incompressibility preserves the area, and subtracting the actual
centroid velocity preserves the centered first moments. Hence a solution
for the modes $|m|\geq2$, with these three constraints imposed, determines
the full contour. The modes $0,\pm1$ need no separate evolution equation.

\subsubsection*{The circular spectrum and the nonlinear derivative estimate}

We now isolate the singularity of the boundary kernel so that both the
linear frequencies and the number of derivatives in the nonlinear
terms can be checked. For functions of two angles, use
\[
 \|F\|_{\mathscr B_\varrho^k}
 =\sum_{m,n}(1+|m|+|n|)^k e^{\varrho(|m|+|n|)}|F_{m,n}|,
\]
and the divided differences
\[
 T_\pm r=\frac{e^{\pm\iunit\theta}r(\theta)-e^{\pm\iunit\beta}r(\beta)}
                    {e^{\pm\iunit\theta}-e^{\pm\iunit\beta}}.
\]
Applied to $e^{\iunit m\theta}$, each divided difference is a sum of
$|m\pm1|$ monomials, each of total absolute degree $|m|$. Summing their
Fourier weights gives
\[
 \|T_\pm r\|_{\mathscr B_\varrho^0}\leq\|r\|_{\varrho,1},\qquad
 \|T_\pm r\|_{\mathscr B_\varrho^1}\leq\|r\|_{\varrho,2}.
\]
For example, with $z=e^{\iunit\theta}$ and $w=e^{\iunit\beta}$,
\[
 T_+(e^{\iunit m\cdot})=
 \begin{cases}
 \displaystyle\sum_{j=0}^{m}z^{m-j}w^j,&m\geq0,\\
 0,&m=-1,\\
 \displaystyle-\sum_{j=0}^{-m-2}z^{-1-j}w^{m+1+j},&m\leq-2.
 \end{cases}
\]
Reflection of both angles gives the formula for $T_-$. Every summand
has Fourier weight $e^{\varrho|m|}$ and polynomial weight
$(1+|m|)^k$, so counting the summands proves the displayed estimates.
The logarithm factors exactly as
\begin{equation}\label{eq:log-factorization}
 \log|Z_r(\theta)-Z_r(\beta)|
 =L_0(\theta-\beta)+\tfrac12\log(1+T_+r)+\tfrac12\log(1+T_-r),
\end{equation}
where $L_0(\theta)=\log|e^{\iunit\theta}-1|$ has coefficients
$(L_0)_0=0$, $(L_0)_m=-1/(2|m|)$ for $m\ne0$.
For a real graph the last two terms are conjugates. On a small ball of
the Wiener space, their power series converge in the two-angle norm,
which also defines their analytic extension to complex graphs.

The circle velocity is $\iunit\epsilon e^{\iunit\theta}$. Keeping the linear
terms in \eqref{eq:boundary-velocity} gives
\[
 V_{\rm lin}(r)=\iunit\epsilon\sum_m\sgn{(m+1)}r_m
                       e^{\iunit(m+1)\theta}-\iunit\epsilon r_{-1}.
\]
To verify the two terms separately, convolve $L_0$ with
$Z_r'-Z_0'=\sum_m\iunit(m+1)r_m e^{\iunit(m+1)\beta}$ and multiply
by $-\epsilon/\pi$. For $m+1\ne0$ the result is
$\iunit\epsilon\sgn(m+1)r_m e^{\iunit(m+1)\theta}$; the zero mode
of $L_0$ makes the $m=-1$ term vanish. For the regular logarithm,
the finite sums above give
\[
 \int_0^{2\pi}\tfrac12(T_++T_-)(e^{\iunit m\cdot})
                         \iunit e^{\iunit\beta}\idiff\beta
 =\begin{cases}\iunit\pi,&m=-1,\\0,&m\ne-1.\end{cases}
\]
For $m\leq-2$ the two possible constant terms in $\beta$ cancel;
for $m\geq0$ neither contributes. This proves the remaining term
$-\iunit\epsilon r_{-1}$ in $V_{\rm lin}$.
If $|m|\geq2$, the $m$th coefficient of its radial projection is
\[
 \bigl(\operatorname{Re}(e^{-\iunit\theta}V_{\rm lin})\bigr)_m
 =\frac{\iunit\epsilon}{2}
       (\sgn(m+1)-\sgn(1-m))r_m
 =\iunit\epsilon\sgn(m)r_m.
\]
The circle velocity contributes $-\epsilon r_\theta$ through the
normal factor. Adding these terms in \eqref{eq:radial-equation} yields
\begin{equation}\label{eq:circular-spectrum}
 (\mathcal L v)_m=-\iunit\epsilon(m-\sgn m)v_m,\qquad |m|\geq2.
\end{equation}
For $m=\pm1$ the constant correction in $V_{\rm lin}$ instead gives
radial coefficient $\iunit\epsilon m r_m$, which cancels
$-\epsilon(r_\theta)_m$. The constant radial coefficient is zero as
well. Thus the modes $0,\pm1$ have eigenvalue zero, as required by
the area and translation constraints. On $|m|\geq2$ every eigenvalue is purely
imaginary and nonzero. Therefore $e^{\zeta\mathcal L}$ is an isometry,
and $|m-\sgn m|=|m|-1$ implies
\[
 \|\mathcal L^{-1}f\|_{\varrho,1}
       +\|\mathcal L^{-1}f\|_{\varrho,2}\leq C\|f\|_{\varrho,1}.
\]

This is the cancellation needed on an interval of length
$O(\eta^{-2})$ in $\zeta$. For example, the linear equation
$g_\zeta=\mathcal Lg+f$ with constant $f$ and $g(0)=0$ has the exact solution
\[
 g(\zeta)=\mathcal L^{-1}(e^{\zeta\mathcal L}-1)f,
 \qquad \|g(\zeta)\|_{\varrho,1}\leq C\|f\|_{\varrho,1}.
\]
A force of size $\eta^2$ thus produces a deformation of size
$O(\eta^2)$ uniformly in elapsed core time.
We next estimate the nonlinear terms; later the same cancellation will
be used for the actual time-dependent forcing.

Put $E_k=\|r\|_{\varrho,k}$ and split the regular part of
\eqref{eq:log-factorization} into its linear term $\Lambda_1$ and its
remainder $\Lambda_{\geq2}$. The zero-order two-angle norm is an algebra
norm and the first-order norm satisfies the product rule. The logarithmic
series therefore gives
\[
 \|\Lambda_{\geq2}\|_{\mathscr B_\varrho^0}\leq CE_1^2,\qquad
 \|\Lambda_{\geq2}\|_{\mathscr B_\varrho^1}\leq CE_1E_2.
\]
For example, the Fourier convolution inequality is
\[
 \|FG\|_{\mathscr B_\varrho^1}
 \leq\|F\|_{\mathscr B_\varrho^1}\|G\|_{\mathscr B_\varrho^0}
       +\|F\|_{\mathscr B_\varrho^0}\|G\|_{\mathscr B_\varrho^1}.
\]
It follows from the triangle inequalities for the exponential weight
and for $1+|m|+|n|$. In the series
$\log(1+T)-T=\sum_{j\geq2}(-1)^{j+1}T^j/j$, apply this product
bound to one factor and the algebra bound to the other $j-1$ factors.
The geometric sum is $O(E_1E_2)$ when $E_1$ is small.
Since $Z_r'$ is affine in $r$, the nonlinear velocity is the integral of
\[
 \Lambda_{\geq2}Z_0'+\Lambda_1(Z_r'-Z_0')
                         +\Lambda_{\geq2}(Z_r'-Z_0').
\]
Its zero- and first-order bounds are $CE_1^2$ and $CE_1E_2$.
The Fourier coefficients of $L_0$ are proportional to $|m|^{-1}$,
so its convolution cancels the derivative in $Z_r'$. The geometric series for
$r_\theta/(1+r)-r_\theta$ has the same quadratic bounds.
Subtracting the linearized equation and using the low-mode chart gives
the reduced radial remainder $\mathcal N_{\rm rad}$ with
\begin{align}\label{eq:contour-tame}
 \|\mathcal N_{\rm rad}(v)\|_{\varrho,1}
 &\leq C\|v\|_{\varrho,1}\|v\|_{\varrho,2},\notag\\
 \|\mathcal N_{\rm rad}(v_1)-\mathcal N_{\rm rad}(v_2)\|_{\varrho,1}
 &\leq C(\|v_1\|_{\varrho,1}+\|v_2\|_{\varrho,1})
                                      \|v_1-v_2\|_{\varrho,2}\notag\\
 &\quad+C(\|v_1\|_{\varrho,2}+\|v_2\|_{\varrho,2})
                                      \|v_1-v_2\|_{\varrho,1}.
\end{align}
Both bounds lose exactly one angular derivative. In the difference
estimate, that derivative can fall on $v_1-v_2$ or on one of the two
individual graphs, which accounts for its two terms. To check this
subtraction for the logarithmic series, use for every power
\[
 T_1^j-T_2^j=\sum_{a=0}^{j-1}T_1^a(T_1-T_2)T_2^{j-1-a}.
\]
The first-order weight can fall on the difference factor, costing
$\|v_1-v_2\|_{\varrho,2}$, or on a background factor, costing
$\|v_1\|_{\varrho,2}+\|v_2\|_{\varrho,2}$.
All other factors sum in a convergent geometric series. Integration
in $\beta$ only selects its zero Fourier coefficient and cannot
increase these norms. The same product subtraction applies to
$Z_r'$ and to the geometric series for $(1+r)^{-1}$.

The comparison of physical domain integrals will be simplest in the area
variable $\mathfrak h=r+r^2/2$. On the modes $|m|\geq2$ this is an
analytic change of coordinates tangent to the identity, with an inverse
obeying the same product bounds. To check the derivative count after
this change, write the full radial self equation as
$r_\zeta=\mathcal Lr+\mathcal N_{\rm full}(r)$. The area remainder is
\[
 (1+r)\mathcal N_{\rm full}(r)+(1+r)\mathcal Lr
                              -\mathcal L(r+r^2/2).
\]
In the final two terms the derivative part of $\mathcal L$ cancels by
the product rule; the remaining sign multiplier is bounded in each
Wiener norm. Thus \eqref{eq:contour-tame} also holds for the reduced
area remainder $\mathcal N$ and the graph $g$.
The area low modes satisfy $\mathfrak h_0=0$ and
$\int e(\theta)(1+2\mathfrak h)^{3/2}\idiff\theta=0$, giving the claimed
quadratic centroid chart. This also verifies that using area coordinates
introduces no extra derivative loss.
For reference, the leading centroid relation is
$\mathfrak h_{\pm1}=-\tfrac12(\mathfrak h^2)_{\pm1}
 +O(\|\mathfrak h\|_{\varrho,1}^3)$, obtained by expanding
$(1+2\mathfrak h)^{3/2}=1+3\mathfrak h+\tfrac32\mathfrak h^2
 +O(\mathfrak h^3)$ in the centroid constraint. The exact centroid constraint is imposed throughout.

\subsubsection*{Size and time variation of the exterior forcing}

The exterior gradient is of order one in the time $\sigma$ and the
length-$d$ coordinates, but of order $\eta^2$ in the boundary equation.
To prove this and the needed difference estimates, identify the plane
with $\mathbb C$. Divergence and curl zero imply that the complex
velocity is antiholomorphic. On a disk larger than the normalized patch,
write
$U_1+\iunit U_2=\sum_{n\geq0}c_n\bar z^n$. Its centroid mean is exactly
\[
 \overline U_r=c_0+\sum_{n\geq2}c_n C_n(r),\qquad
 C_n(r)=\frac2{n+2}\bigl((1+r)^{n+2}\bigr)_n,
\]
where $C_0=1$ and $C_1=0$ by area and centering. Subtracting this mean
therefore removes the constant field for every graph. With
\[
 W_r=e^{-\iunit\theta}\left(U((1+r)e^{\iunit\theta})-\overline U_r\right),
\]
the radial exterior operator is
$\operatorname{Re}W_r-r_\theta(1+r)^{-1}\operatorname{Im}W_r$.
The coefficient formula for $C_n$ follows from
\[
 \frac1\pi\int_{D_r}\bar z^n\idiff z
 =\frac1{\pi(n+2)}\int_0^{2\pi}
       e^{-\iunit n\theta}(1+r(\theta))^{n+2}\idiff\theta.
\]
For the disk, all these integrals vanish when $n\geq1$.
Consequently the disk forcing before projection is exactly
$\operatorname{Re}\sum_{n\geq1}c_ne^{-\iunit(n+1)\theta}$;
its modes already satisfy $|m|\geq2$. Multiplication by $1+r$
when passing to the area variable does not change this value at $r=0$.
Take $0\leq\varrho\leq1/4$. If
$E_U=\sum_{n\geq1}2^n|c_n|$, then for every fixed $k$
\[
 \sum_{n\geq1}(n+1)^k|c_n|
       \bigl(e^\varrho(1+\|r\|_{\varrho,0})\bigr)^n\leq C_kE_U,
\]
because the parenthesized base is less than $3/2$ for small graphs.
The strict inequality between these bases proves convergence of the
series after any fixed number of coefficient derivatives. In particular,
its first two graph derivatives converge, with at most the one angular
derivative already present in the radial equation.

The scaling \eqref{eq:core-exterior-scaling} gives two factors of
$\eta$ when differentiating $U$ in space, and harmonic interior
estimates then give $E_U\leq C\eta^2M_{\rm ext}$.
In coefficient form, if the centered exterior in the length-$d$
variable is
$\widetilde V_1+\iunit\widetilde V_2=\sum_{n\geq1}a_n(\sigma)\bar z^n$,
then
\begin{equation}\label{eq:response-coefficient-scaling}
 c_n(\zeta)=\eta^{n+1}a_n(\eta^2\zeta),\qquad n\geq1.
\end{equation}
For a fixed disk of radius $R_0$ in the larger neighborhood, Cauchy's
estimate for its antiholomorphic gradient gives
$n|a_n|\leq C M_{\rm ext}R_0^{1-n}$ on a slightly smaller disk.
The sum $\sum_{n\geq1}2^n\eta^{n+1}|a_n|$ is therefore
$O(\eta^2M_{\rm ext})$ when $\eta$ is small compared with $R_0$.
Differentiation in
$\zeta$ adds a further factor $\eta^2$, since
$\sigma=\eta^2\zeta$. The derivative of the centered exterior gradient
is evaluated along $b(\sigma)$: equivalently, the Hessian of the
corresponding stream derivative is a rotation of
$(\partial_\sigma+b'\cdot\nabla)\nabla V_{\rm ext}$.
Constant velocity and affine stream terms disappear after centering
and projection, so a bound on $b''$ is unnecessary.
More explicitly, at fixed centered $z$,
\[
 \partial_\sigma\nabla_z\widetilde V(\sigma,z)
 =\bigl[(\partial_\sigma+b'\cdot\nabla)
            \nabla V_{\rm ext}\bigr](\sigma,b(\sigma)+z).
\]
Apply the same Cauchy estimate to this gradient and differentiate
\eqref{eq:response-coefficient-scaling}; this gives
$\sum_{n\geq1}2^n|\partial_\zeta c_n|\leq C\eta^4M_{\rm ext}$.
No derivative of the subtracted constant is needed in this calculation.
We have now derived
the exact reduced equation
\begin{equation}\label{eq:contour}
 g_\zeta=\mathcal Lg+\mathcal N(g)+f(\zeta)+\mathcal E(\zeta,g),
 \qquad \mathcal E(\zeta,0)=0.
\end{equation}
Here $f$ is the exterior term evaluated at the disk, $\mathcal N$ is
the nonlinear self term, and $\mathcal E$ is the change in the exterior
term when the disk is deformed. The preceding bounds give
\begin{equation}\label{eq:exterior-core}
 \|f\|_{\varrho,1}\leq C\eta^2,\qquad
 \|f_\zeta\|_{\varrho,1}\leq C\eta^4,\qquad
 \|\mathcal E(\zeta,g)\|_{\varrho,1}\leq C\eta^2\|g\|_{\varrho,2}.
\end{equation}
We will compare two solutions by separating changes of field from changes
of graph. Holding the graph fixed, linearity in the exterior field and
the same convergent series give
\begin{equation}\label{eq:exterior-difference}
 \|f^1-f^2+\mathcal E^1(\zeta,g)-\mathcal E^2(\zeta,g)\|_{\varrho,1}
       \leq C\eta^2\Delta_A(1+\|g\|_{\varrho,2}).
\end{equation}
Holding the field fixed and subtracting the graph products gives
\begin{equation}\label{eq:exterior-same-field}
 \begin{split}
 \|\mathcal E(\zeta,g_1)-\mathcal E(\zeta,g_2)\|_{\varrho,1}
 \leq C\eta^2\bigl[&\|g_1-g_2\|_{\varrho,2}\\
 &+(\|g_1\|_{\varrho,2}+\|g_2\|_{\varrho,2})
                                  \|g_1-g_2\|_{\varrho,1}\bigr].
 \end{split}
\end{equation}
Thus every field-difference term contains the instantaneous quantity
$\Delta_A$. Its time derivative is not required.

\subsubsection*{The analytic graph is the actual patch}

The derivative hypothesis in the lemma is stated along the actual
centroid. We must therefore identify the analytic solution with the
transported patch before applying that hypothesis. First construct an
analytic solution on a short interval using only instantaneous bounds
on the exterior gradient. Truncate the modes $|m|\geq2$ in
\eqref{eq:contour}, evolve the centroid by its exact average-velocity
equation, and reconstruct the other modes by the constraints. On a
decreasing analytic radius $\varrho_b-v\zeta$, the instantaneous estimates give
\[
 E_1'\leq\left(-\tfrac23v+C(E_1+\eta^2M_{\rm loc})\right)E_2
                                      +C\eta^2M_{\rm loc},
\]
where $M_{\rm loc}$ bounds the exterior gradient on a fixed larger
neighborhood, and $E_k$ denotes the graph norm of order $k$. Choose
$v$ so that the negative coefficient absorbs the term containing $E_2$
while $E_1$ is small, and then choose a short interval on which $E_1$
remains small. The estimate is independent of the Fourier cutoff. On
smaller positive radii the equation also bounds time derivatives, so
compactness produces an exact analytic contour. Differentiating the
constraints reconstructs the full radial equation.
Here the compactness is seen directly from the Fourier weights: if
$\varrho_1<\varrho_2$, a bounded set in $\|\cdot\|_{\varrho_2,1}$
has uniformly summable tails in $\|\cdot\|_{\varrho_1,2}$ because
$(1+|m|)e^{-(\varrho_2-\varrho_1)|m|}\to0$.
The equation gives equicontinuity on a further smaller strip.
Modewise compactness and these tail bounds give strong convergence
in a norm in which all the displayed nonlinear products converge.
The exact radial vector field preserves the three integral constraints.
Its high modes agree with the limit equation, so its remaining three
modes agree as well: the derivative of the constraints with respect
to the low modes is invertible in the small graph chart.

Now prescribe the exterior field of the given Euler solution. Both its
actual patch and the constructed contour solve the same active patch
equation with that prescribed exterior. We verify uniqueness directly
on the common tube containing their supports. Extend the exterior stream
from a smaller tube and subtract the common translation. For two patch
solutions with the same initial datum, let
$w=u_{{\rm self},1}-u_{{\rm self},2}$ in length-$d$ coordinates and time
$\sigma$, and put
$\zeta_{\rm ext}=\operatorname{curl}V_{\rm ext}$. Their equal circulation
makes $w\in L^2$. Taking curl of the cross-advection terms and applying
the whole-plane solenoidal projection gives
\[
 \partial_\sigma w+\mathcal P\bigl((u_{{\rm self},1}+V_{\rm ext})\cdot\nabla w
              +w\cdot\nabla(u_{{\rm self},2}+V_{\rm ext})\bigr)
       =\frac1{2\pi}K*\operatorname{div}(\zeta_{\rm ext}w).
\]
The extension may have nonzero curl outside the original tube; this
accounts for the last term. Its operator has order zero on $L^2$.
To check that term and its sign, put
$\delta\omega=\operatorname{curl}w$. The difference vorticity equation is
\[
 \partial_\sigma\delta\omega
 +(u_{{\rm self},1}+V_{\rm ext})\cdot\nabla\delta\omega
             +w\cdot\nabla\omega_{{\rm self},2}=0.
\]
The curl of the two transport terms in the projected velocity equation
has one additional term $w\cdot\nabla\zeta_{\rm ext}$.
Because $\operatorname{div}w=0$, it equals
$\operatorname{div}(\zeta_{\rm ext}w)$, exactly the curl of the
displayed right side. The Fourier multiplier of $K*\operatorname{div}$
is uniformly bounded at nonzero frequencies, so
$\|K*\operatorname{div}(\zeta_{\rm ext}w)\|_2
 \leq C\|\zeta_{\rm ext}\|_\infty\|w\|_2$.
Testing the equation with $w$, all exterior terms are bounded by
$C(\|\nabla V_{\rm ext}\|_\infty+\|\zeta_{\rm ext}\|_\infty)\|w\|_2^2$.
For the self term, Calder\'on--Zygmund bounds on the bounded compact
vorticity give $\|\nabla u_{{\rm self},2}\|_r\leq Cr$, $r\geq2$.
The potentially unbounded self-gradient term is estimated by
\[
 \int|\nabla u_{{\rm self},2}||w|^2
 \leq Cr\|w\|_{2r/(r-1)}^2
 \leq Cr\|w\|_2^{2-2/r}\|w\|_\infty^{2/r}
 \leq Cr E^{1-1/r},\qquad E=\|w\|_2^2.
\]
The bounded velocity norm and its fixed-radius dependence have been
absorbed into $C$.
Taking $r=\max(2,\log(C/E))$ yields
\[
 E'\leq CE(1+\log(C/E)).
\]
Osgood's inequality gives $E=0$ from $E(0)=0$, identifying the analytic
contour with the actual patch. The constants in this uniqueness argument
may depend on the fixed disk radius. Uniformity in $\eta$ will come from
the analytic estimates below, which now apply to that actual patch.

\subsubsection*{Uniform deformation bound on a fixed interval in $\sigma$}

The exterior derivative bound is now available along the centroid of
the analytic graph. If one integrated the estimate $\|f\|=O(\eta^2)$
directly over $\zeta=O(\eta^{-2})$, it would only give a bound of order
one. Instead use the invertible imaginary operator $\mathcal L$ to set
$z=g+\mathcal L^{-1}f$. The exact equation becomes
\[
 z_\zeta=\mathcal Lz+\mathcal N(g)+\mathcal E(\zeta,g)
                                  +\mathcal L^{-1}f_\zeta.
\]
The new forcing $\mathcal L^{-1}f_\zeta$ is $O(\eta^4)$ and gains one
derivative. On the radius
$\varrho_{\rm big}=\varrho_*-\Lambda\eta^2\zeta$, put
$Q_k=\|z\|_{\varrho_{\rm big},k}$. The imaginary diagonal part does
not increase any Fourier absolute value. Moreover,
$|m|(1+|m|)\geq(2/3)(1+|m|)^2$ for $|m|\geq2$; differentiating the strip
weight therefore gives
\[
 Q_1'\leq\left(-\tfrac23\Lambda\eta^2+C(Q_1+\eta^2)\right)Q_2
                                       +C\eta^2(Q_1+\eta^2).
\]
For the nonlinear terms this follows by writing
$g=z-\mathcal L^{-1}f$, whence
$\|g\|_{\varrho_{\rm big},1}\leq Q_1+C\eta^2$ and
$\|g\|_{\varrho_{\rm big},2}\leq Q_2+C\eta^2$.
The self remainder is consequently bounded by
$C(Q_1+\eta^2)(Q_2+\eta^2)$, the exterior remainder by
$C\eta^2(Q_2+\eta^2)$, and the new forcing by $C\eta^4$.
Collecting the $Q_2$ terms gives precisely the displayed inequality.
Initially $g=0$, so $Q_1(0)\leq C_{\rm init}\eta^2$.
Choose $K>2C_{\rm init}$ and impose provisionally $Q_1\leq K\eta^2$.
Fix $\Lambda>3C(K+1)$ to absorb $Q_2$, enlarging this same constant
if needed for the difference estimate below. Choose $T_*>0$, depending
only on the fixed lemma data, so that
\[
 \Lambda T_*\leq\varrho_*/4,\qquad
 C(K+1)T_*<K-C_{\rm init}.
\]
For the prescribed horizon $T_0$, set $T=\min\{T_0,T_*\}$ as in
Lemma~\ref{lem:response}. On $0\leq\zeta\leq T/\eta^2$, the
remaining derivative bound changes $Q_1$ by at most
$C(K+1)T_*\eta^2$, strictly improving the provisional bound.
The larger analytic radius remains at least $3\varrho_*/4$.
The local construction therefore continues the actual patch through
$0\leq\sigma\leq T$ and proves its $O(\eta^2)$ deformation bound.
The prescribed horizon $T_0$ has not been changed.

\subsubsection*{Subtracting two equations from the original disks}

For the difference estimate use the smaller radius
$\varrho=\varrho_*/2-\Lambda\eta^2\zeta$. Its fixed gap from
$\varrho_{\rm big}$ absorbs an additional factor $1+|m|$ in the
Fourier sum, giving $\|g^j\|_{\varrho,2}\leq C\eta^2$.
Explicitly, that gap is $\varrho_*/2$ and
\[
 \|g^j\|_{\varrho,2}
 \leq\sup_{m\geq2}(1+m)e^{-\varrho_*m/2}
                          \|g^j\|_{\varrho_{\rm big},1}
 \leq C\eta^2.
\]
Subtract the original equations \eqref{eq:contour}, and set
$D_k=\|g^1-g^2\|_{\varrho,k}$.
In the exterior terms make the exact split
\[
 \begin{split}
 f^1-f^2+\mathcal E^1(g^1)-\mathcal E^2(g^2)
 ={}&\mathcal E^1(g^1)-\mathcal E^1(g^2)\\
 &+f^1-f^2+\mathcal E^1(g^2)-\mathcal E^2(g^2).
 \end{split}
\]
The first row is a graph difference with one field fixed; the second
is a field difference with one graph fixed. In particular its only
inhomogeneous term is bounded by $C\eta^2\Delta_A$.
The tame bounds and
\eqref{eq:exterior-difference}--\eqref{eq:exterior-same-field} yield
\[
 D_1'\leq\left(-\tfrac23\Lambda\eta^2+C\eta^2\right)D_2
                   +C\eta^2D_1+C\eta^2\Delta_A(\eta^2\zeta).
\]
With the final $\Lambda$ and $T_*$ chosen above, the smaller radius
is at least $\varrho_*/4$ on $0\leq\sigma\leq T$. Absorb $D_2$, and use
$\partial_\sigma=\eta^{-2}\partial_\zeta$. The resulting inequality
is \eqref{eq:response}. Gronwall gives its integrated form
\begin{equation}\label{eq:proved-response}
 \|g^1(\sigma)-g^2(\sigma)\|_{\varrho(\sigma),1}
 \leq Ce^{C\sigma}\left(\|g^1(0)-g^2(0)\|_{\varrho_*/2,1}
                           +\int_0^\sigma\Delta_A(s)\idiff s\right).
\end{equation}
For the initially identical disks the initial difference is zero.
Since $\sigma\leq T\leq T_*$, the exponential is bounded by
$e^{CT_*}$ and is part of the constant depending only on the fixed
lemma data.
The change to $z$ establishes the two separate small graph bounds;
subtraction in the original variable $g$ preserves dependence on
$\Delta_A$ alone.

\subsubsection*{Consequences for the centroid equations}

First compare the velocities produced by the two patches after
translating them to the same physical center $db$. In polar area
coordinates, put $\xi=(|y|^2-1)/2$ and
$Y(\xi,\theta)=\sqrt{1+2\xi}\,e(\theta)$, so that
$\idiff y=\idiff\xi\idiff\theta$. If $D_{\mathfrak h_j}$ is the
centered unit-scale domain with graph $\mathfrak h_j$, then for every
smooth $F$ on their union
\begin{equation}\label{eq:polar-area-difference}
 \frac1\pi\left(\int_{D_{\mathfrak h_1}}F(y)\idiff y
                     -\int_{D_{\mathfrak h_2}}F(y)\idiff y\right)
 =\frac1\pi\int_0^{2\pi}\int_{\mathfrak h_2(\theta)}^{\mathfrak h_1(\theta)}
                 F(Y(\xi,\theta))\idiff\xi\idiff\theta.
\end{equation}
The inner integral is signed when the graphs cross. Its absolute value
is bounded by $C\|F\|_\infty\|\mathfrak h_1-\mathfrak h_2\|_{\varrho,1}$,
and the low-mode chart bounds the last graph difference by
$C\|g^1-g^2\|_{\varrho,1}$.
For the velocity use
$F(y)=\nabla^\nu K(x-db-ay)$. The constant and linear Taylor terms of
$\nabla^\nu K(x-db-ay)$ in $ay$ integrate to zero, because both
patches have equal area and zero centered first moments. For
$|x-db|\geq4a$ and graphs in the fixed small ball, every Taylor
segment has distance at least $|x-db|/2$ from the origin. Its remainder is
\[
 a^2\int_0^1(1-v)\nabla^{\nu+2}K(x-db-vay)[y,y]\idiff v,
\]
bounded by
$C_\nu a^2|x-db|^{-\nu-3}$. The area between the graphs is bounded by
$C\|g^1-g^2\|_{\varrho,1}$, including their constrained low modes.
Multiplication by the physical circulation gives
\eqref{eq:remotegraph}.

Next compare the averaging errors of the exterior fields over the two
patches. The average of a harmonic field over an undeformed disk is
exactly its value at the center. Subtract this identity before comparing
the deformed patches. For the centered field
$W(y)=V_{\rm ext}(b+y)-V_{\rm ext}(b)$ set
\[
 T_W(y)=W(\eta y)-\eta\nabla W(0)y,
 \qquad |T_W(y)|\leq C\eta^2\|\nabla^2W\|.
\]
The exact observer error for a graph $\mathfrak h$ is therefore
\begin{equation}\label{eq:observer-area-identity}
 \delta_{\rm obs}(W,\mathfrak h)
 =\frac1\pi\int_0^{2\pi}\int_0^{\mathfrak h(\theta)}
                   T_W(Y(\xi,\theta))\idiff\xi\idiff\theta.
\end{equation}
Indeed the disk integral of $W(\eta y)$ is zero by harmonicity;
the linear term has zero integral over both the disk and the centered
patch. Thus this remainder for one graph is bounded by
$C\eta^2\|\nabla^2V_{\rm ext}\|\|g\|_{\varrho,1}$.
In the difference of \eqref{eq:observer-area-identity}, first change
$\mathfrak h_1$ to $\mathfrak h_2$ with $W_1$ fixed, and then change
$W_1$ to $W_2$ with $\mathfrak h_2$ fixed. The change of graph contributes
$C\eta^2M_{\rm ext}\|g^1-g^2\|_{\varrho,1}$, and the change of field contributes
$C\eta^2\Delta_A\|g^2\|_{\varrho,1}\leq C\eta^4\Delta_A$.
Interior harmonic estimates on the larger neighborhood bound the
required second derivatives by the given gradient norms. This proves
\eqref{eq:observer-response} and the lemma.
\end{proof}

\vspace{0.2in}
\noindent{\bf Acknowledgments.}
This work was supported in part by NSF Grant DMS-2508463,
the Choi Family Gift Fund, and the Mike Yan Gift Fund.
Before posting this manuscript on arXiv, we communicated
with Hyungjun Choi regarding the independent development
of our respective works and the substantial overlap
between their main constructions. We appreciate his
thoughtful and collegial response. Following this exchange,
we agreed to proceed with separate papers to recognize
the independent development of our respective results.
We have acknowledged the close relationship between
the two works in the Introduction.

\bibliographystyle{unsrtnat}
\bibliography{reference}
\end{document}